\documentclass{article}
\usepackage{amssymb,a4wide}
\usepackage{amsmath}
\usepackage{pifont}
\usepackage{amsfonts}
\usepackage{latexsym}
\usepackage{verbatim}
\usepackage{float}
\usepackage{exscale}

\renewcommand{\ldots}{\ensuremath{\dotsc}}

\renewcommand{\theequation}{\arabic{section}.\arabic{equation}}

\newcommand{\BIGOP}[1]{\mathop{\mathchoice%
{\raise-0.22em\hbox{\huge $#1$}}%
{\raise-0.05em\hbox{\Large $#1$}}{\hbox{\large $#1$}}{#1}}}
\newcommand{\bigtimes}{\BIGOP{\times}}
\newcommand{\BIGboxplus}{\mathop{\mathchoice%
{\raise-0.35em\hbox{\huge $\boxplus$}}%
{\raise-0.15em\hbox{\Large $\boxplus$}}{\hbox{\large $\boxplus$}}{\boxplus}}}

\newcommand{\Rplus}{{\mathbb R}_{>0}}

\def\epsilon{\varepsilon}
\def\hat{\widehat}

\def\undertilde#1{{\baselineskip=0pt\vtop
{\hbox{$#1$}\hbox{$\scriptscriptstyle\sim$}}}{}}
\def\underdtilde#1{{\baselineskip=0pt\vtop
{\hbox{$#1$}\hbox{$\scriptscriptstyle\approx$}}}{}}
\def\epsilon{\varepsilon}
\def\hat{\widehat}

\def\nabq{\undertilde{\nabla}_{q}}
\def\nabqi{\undertilde{\nabla}_{q_i}}
\def\nabqj{\undertilde{\nabla}_{q_j}}
\def\nabx{\undertilde{\nabla}_{x}}
\def\nabr1{\undertilde{\nabla}_{r_1}}
\def\nabr2{\undertilde{\nabla}_{r_2}}
\def\delx{\Delta_{x}}

\def\lae{\ell_{a}}

\def\ut{\undertilde{u}}
\def\utae{\undertilde{u}_{\epsilon,L}}
\def\utaed{\undertilde{u}_{\epsilon,L,\delta}}
\def\uta{\undertilde{u}_{\epsilon,L}}
\def\ute{\undertilde{u}_{\epsilon}}

\def\utaeD{\utae^{\Delta t}}
\def\utaeDm{\utae^{\Delta t,-}}
\def\utaeDp{\utae^{\Delta t,+}}

\def\vt{\undertilde{v}}
\def\wt{\undertilde{w}}
\def\xt{\undertilde{x}}
\def\qt{\undertilde{q}}

\def\ft{\undertilde{f}}
\def\gt{\undertilde{g}}

\def\nt{\undertilde{n}}
\def\yt{\undertilde{y}}
\def\zt{\undertilde{z}}
\def\etat{\undertilde{\eta}}

\def\chit{\undertilde{\chi}}

\def\Ct{\undertilde{C}}
\def\Ft{\undertilde{F}}

\def\Ht{\undertilde{H}}

\def\Lt{\undertilde{L}}

\def\St{\undertilde{S}}

\def\Vt{\undertilde{V}}

\def\Yt{\undertilde{Y}}
\def\dq{\,{\rm d}\undertilde{q}}
\def\dx{\,{\rm d}\undertilde{x}}

\def\dt{\,{\rm d}t}
\def\zerot{\undertilde{0}}
\def\tautt{\underdtilde{\tau}}

\def\Att{\underdtilde{A}}

\def\Ctt{\underdtilde{C}}

\def\Itt{\underdtilde{I}}

\def\Ltt{\underdtilde{L}}

\def\nabxtt{\underdtilde{\nabla}_{x}\,}

\def\unabtt{\underdtilde{\nabla}_{x}\,\ut}

\def\vnabtt{\underdtilde{\nabla}_{x}\,\vt}

\def\wnabtt{\underdtilde{\nabla}_{x}\,\wt}
\def\sigtt{\underdtilde{\sigma}}

\def\pae{p_{\epsilon,L}}
\def\hpsiae{\hat \psi_{\epsilon,L}}
\def\hpsiaed{\hat \psi_{\epsilon,L,\delta}}
\def\psiae{\psi_{\epsilon,L}}
\def\psia{\hat \psi_{\epsilon,L}}

\def\ptut{\frac{\partial \ut}{\partial t}}

\def\ptutae{\frac{\partial \utae}{\partial t}}

\def\dd {{\,\rm d}}

\def\ocirc#1{\ifmmode\setbox0=\hbox{$#1$}\dimen0=\ht0
    \advance\dimen0 by1pt\rlap{\hbox to\wd0{\hss\raise\dimen0
    \hbox{\hskip.2em$\scriptscriptstyle\circ$}\hss}}#1\else
    {\accent"17 #1}\fi}

\newcommand{\bet}{\noalign{\vskip6pt plus 3pt minus 1pt}}

\def\qqquad{\qquad\quad}
\def\qqqquad{\qquad\qquad}

\newcounter{ind}
\def\eqlabstart{%
 \setcounter{ind}{\value{equation}}\addtocounter{ind}{1}%
 \setcounter{equation}{0}%
 \renewcommand{\theequation}{\arabic{section}.\arabic{ind}\alph{equation}}%
}
\def\eqlabend{%
 \renewcommand{\theequation}{\arabic{section}.\arabic{equation}}%
 \setcounter{equation}{\value{ind}}%
}

\newcommand{\grad}[2]{\nabla_{#1}#2} 
\newcommand{\compEmb}{\hookrightarrow \!\!\!\rightarrow}

\newtheorem{example}{Example}[section]
\newtheorem{definition}{Definition}[section]
\newtheorem{lemma}{Lemma}[section]
\newtheorem{theorem}{Theorem}[section]
\newtheorem{remark}{Remark}[section]
\newtheorem{proposition}{Proposition}[section]

\newenvironment{proof}[1][Proof]{\begin{trivlist}
\item[\hskip \labelsep {\bfseries #1}]}{ $\Box$\end{trivlist}}

\newcounter{appendix}
\renewcommand\appendix{\par
        \refstepcounter{appendix}
        \setcounter{section}{0}%
        \setcounter{theorem}{0}
        \setcounter{equation}{0}
\renewcommand\thesection{\appendixname\ \Alph{section}}
\renewcommand\thesubsection{\Alph{section}.\arabic{subsection}}%
\renewcommand\theequation{\Alph{section}.\arabic{equation}}}%
\begin{document}

\title{\textbf{\large EXISTENCE AND EQUILIBRATION OF GLOBAL WEAK SOLUTIONS TO
HOOKEAN-TYPE BEAD-SPRING CHAIN MODELS FOR DILUTE POLYMERS~\\
}}

\author{\textit{\normalsize John W. Barrett}\\
{\normalsize \textit{Department of Mathematics, Imperial College London, London SW7 2AZ, UK}}\\
{\small \texttt{j.barrett@imperial.ac.uk}}
\\
~\\
\textit{\normalsize Endre S\"uli}\\
{\normalsize \textit{Mathematical Institute, University of Oxford, Oxford OX1 3LB, UK}}\\
{\small \texttt{endre.suli@maths.ox.ac.uk}}}

\date{~}
\maketitle

\begin{abstract}
We show the existence of global-in-time weak solutions to a general class of coupled 
Hookean-type bead-spring chain models that arise from the kinetic theory
of dilute solutions of polymeric liquids with noninteracting polymer chains. The class of models
involves the unsteady incompressible
Navier--Stokes equations in a bounded domain in $\mathbb{R}^d$, $d = 2$ or $3$, for the velocity
and the pressure of the fluid, with an elastic extra-stress tensor appearing on the right-hand
side in the momentum equation. The extra-stress tensor stems from the random movement of
the polymer chains and is defined by the Kramers expression
through the associated probability density function
that satisfies a Fokker--Planck-type parabolic equation, a crucial feature of which is
the presence of a center-of-mass diffusion term. We require no structural
assumptions on the drag term in the Fokker--Planck equation; in particular, the drag term need not be corotational. With a square-integrable and divergence-free initial velocity datum $\undertilde{u}_0$ for the Navier--Stokes equation and a nonnegative initial probability density function $\psi_0$ for the Fokker--Planck equation, which has finite relative entropy with respect to the Maxwellian $M$, we prove the existence of a global-in-time weak solution $t \mapsto (\undertilde{u}(t), \psi(t))$ to the coupled Navier--Stokes--Fokker--Planck system, satisfying the initial condition $(\undertilde{u}(0), \psi(0)) = (\undertilde{u}_0, \psi_0)$,
such that $t\mapsto \undertilde{u}(t)$ belongs to the classical Leray space and $t \mapsto \psi(t)$ has bounded
relative entropy with respect to $M$ and $t \mapsto \psi(t)/M$ has integrable Fisher information
(w.r.t. the measure ${\rm d}\nu:= M(\undertilde{q})\,{\rm d}\undertilde{q}\,{\rm d}\undertilde{x}$)
over any time interval $[0,T]$, $T>0$. If the density of body forces $\undertilde{f}$ on the right-hand side of the Navier--Stokes momentum equation vanishes, then $t\mapsto (\undertilde{u}(t),\psi(t))$
decays exponentially in time to $(\undertilde{0},M)$ in the $\undertilde{L}^2 \times L^1$ norm, at a rate that is
independent of $(\undertilde{u}_0,\psi_0)$ and of the centre-of-mass diffusion coefficient.

\smallskip

\noindent
\textit{An abbreviated version of this paper has been submitted for publication in} \textit{Mathematical Models and Methods in Applied Sciences (M3AS)}.


\medskip


\end{abstract}

\section{Introduction}
\label{sec:1}

This paper establishes the existence of global-in-time weak solutions
to a large class of bead-spring chain models with Hookean-type spring potentials, ---
a system of nonlinear partial differential equations that arises
from the kinetic theory of dilute polymer solutions. The solvent is an
incompressible, viscous, isothermal Newtonian fluid confined to a
bounded open Lipschitz domain $\Omega \subset \mathbb{R}^d$, $d=2$ or $3$, with
boundary $\partial \Omega$.
For the sake of simplicity of presentation,
we shall suppose that $\Omega$ has `solid' boundary $\partial \Omega$;
the velocity field $\ut$ will then satisfy the no-slip boundary condition
$\ut=\zerot$ on $\partial \Omega$. The polymer chains, which are suspended
in the solvent, are assumed not to interact with each other. The
conservation of momentum and mass equations for the solvent then
have the form of the incompressible Navier--Stokes equations in
which the elastic {\em extra-stress} tensor $\tautt$ (i.e.,\ the
polymeric part of the Cauchy stress tensor) appears as a source
term:

Given $T \in \mathbb{R}_{>0}$, find $\ut\,:\,(\xt,t)\in
\overline\Omega \times [0,T] \mapsto
\ut(\xt,t) \in {\mathbb R}^d$ and $p\,:\, (\xt,t) \in \Omega \times (0,T]
\mapsto p(\xt,t) \in {\mathbb R}$ such that
\begin{subequations}
\begin{alignat}{2}
\ptut + (\ut \cdot \nabx)\, \ut - \nu \,\delx \ut + \nabx p
&= \ft + \nabx \cdot \tautt \qquad &&\mbox{in } \Omega \times (0,T],\label{ns1a}\\
\nabx \cdot \ut &= 0        \qquad &&\mbox{in } \Omega \times (0,T],\label{ns2a}\\
\ut &= \zerot               \qquad &&\mbox{on } \partial \Omega \times (0,T],\label{ns3a}\\
\ut(\xt,0)&=\ut_{0}(\xt)    \qquad &&\forall \xt \in \Omega.\label{ns4a}
\end{alignat}
\end{subequations}
It is assumed that each of the equations above has been written in its nondimensional form
and the momentum equation \eqref{ns1a} has been normalized so that the Strouhal number is equal to $1$;
$\ut$ denotes a nondimensional velocity, defined as the velocity field
scaled by the characteristic flow speed $U_0$; $\nu\in \mathbb{R}_{>0}$ is the reciprocal of the Reynolds number,
i.e. the ratio of the kinematic viscosity coefficient of the solvent and $L_0 U_0$, where
$L_0$ is a characteristic length-scale of the flow; $p$ is the product of
the nondimensional pressure and the Euler number; and $f$ is the product of the nondimensional density
of body forces and the Richardson number.

In a {\em bead-spring chain model}, consisting of $K+1$ beads coupled with $K$ elastic
springs to represent a polymer chain, the extra-stress tensor
$\tautt$ is defined by the \textit{Kramers expression}
as a weighted average of $\psi$, the probability
density function of the (random) conformation vector $\qt = (\qt_1,\dots, \qt_K)^{\rm T}
\in \mathbb{R}^{Kd}$ of the chain (cf. (\ref{tau1}) below), with $\qt_i$
representing the $d$-component conformation/orientation vector of the $i$th spring.
The Kolmogorov equation satisfied by $\psi$ is a second-order parabolic equation,
the Fokker--Planck equation, whose transport coefficients depend
on the velocity field $\ut$. The domain $D$ of admissible conformation
vectors $D \subset \mathbb{R}^{Kd}$ is a $K$-fold
Cartesian product $D_1 \times \cdots \times D_K$ of balanced convex
open sets $D_i \subset \mathbb{R}^d$, $i=1,\dots, K$; the term
{\em balanced} means that $\qt_i \in D_i$ if, and only if, $-\qt_i \in D_i$.
Hence, in particular, $\undertilde{0} \in D_i$, $i=1,\dots,K$.
Typically, $D_i$ is the
whole of $\mathbb{R}^d$ or a bounded open $d$-dimensional ball
centred at the origin $\zerot \in \mathbb{R}^d$ for each $i=1,\dots,K$.
When $K=1$, the model is referred to as the {\em dumbbell model.}

Let $\mathcal{O}_i\subset [0,\infty)$ denote the image of $D_i$
under the mapping $\qt_i \in D_i \mapsto
\frac{1}{2}|\qt_i|^2$, and consider the {\em spring-potential}~$U_i
\!\in\! W^{2,\infty}_{\rm loc}(\mathcal{O}_i;\mathbb {R}_{\geq 0})$, $i=1,\dots, K$.
Clearly, $0 \in \mathcal{O}_i$. We shall suppose that $U_i(0)=0$
and that $U_i$ is monotonic increasing and unbounded on $\mathcal{O}_i$ for each $i=1,\dots, K$.
The elastic spring-force $\Ft_i\,:\, D_i \subseteq \mathbb{R}^d \rightarrow \mathbb{R}^d$
of the $i$th spring in the chain is defined by
\begin{equation}\label{eqF}
\Ft_i(\qt_i) = U_i'(\textstyle{\frac{1}{2}}|\qt_i|^2)\,\qt_i, \qquad i=1,\dots,K.
\end{equation}
%


\begin{example}
\label{ex1.1} \em
In the Hookean dumbbell model $K=1$, and the spring force
is defined by ${\Ft}({\qt}) = {\qt}$, with ${\qt} \in {D}=\mathbb{R}^d$,
corresponding to ${U}(s)= s$, $s \in \mathcal{O} = [0,\infty)$.
$~~~~~\diamond$
\end{example}

Unfortunately, we are not able to deal with the pure Hookean model.
The compactness argument that
forms the core of our current existence proof requires that the
potentials $U_i$ have superlinear growth at infinity.
For example for any $s_\infty >0$ and $\vartheta > 1$,
we can deal with a potential of the form
\begin{align}
U(s) = \left\{ \begin{array}{ll}
s \qquad &\mbox{for } s \in [0,s_\infty],\\
\frac{s_\infty}{\vartheta} \left[ \left(\frac{s}{s_\infty}\right)^{\vartheta} + (\vartheta-1) \right]
\qquad & \mbox{for } s \geq s_\infty;
\end{array}
\right.
\label{eqHM}
\end{align}
which approximates the Hookean potential $U(s)=s$.

We shall assume in what follows that $D$ is a Cartesian product of 
$D_i \equiv \mathbb{R}^d$, 
$i=1,\dots, K,$ with $K \geq 1$.
We shall further suppose that
for $i=1,\dots, K$
there exist constants $c_{ij}>0$,
$j=1, 2, 3, 4$, 
such that
the (normalized) Maxwellian $M_i$, defined by
\[
M_i(\qt_i) = \frac{{\rm e}^{-U_i(\frac{1}{2}|\qt_i|^2)}}
{ \displaystyle \int_{D_i} {\rm e}^{-U_i(\frac{1}{2}|\qt_i|^2)} \dq_i}\,,
\]
and the associated monotonically increasing potentials
$U_i \in W^{2,\infty}_{\rm loc}([0,\infty);{\mathbb R}_{\geq 0})$ satisfy, %
for a $\vartheta >1$,
\eqlabstart
\begin{alignat}{2}
U_i(\textstyle{\frac{1}{2}}|\qt_i|^2) & = c_{i1} \,(\textstyle{\frac{1}{2}}|\qt_i|^2)^\vartheta
\qquad &&\mbox{as } |\qt_i| \rightarrow \infty,
 \label{growth1}\\
U_i'(\textstyle{\frac{1}{2}}|\qt_i|^2) & \leq c_{i2} + c_{i3} \,(\textstyle{\frac{1}{2}}|\qt_i|^2)^{\vartheta-1}
\qquad &&\forall \qt_i \in D_i,
 \label{growth2}\\
 \hspace*{-0.85in}
\mbox{and hence}\hspace{2in}&&\nonumber \\
 M_i(\qt_i)
& =
c_{i4}\,
{\rm e}^{-c_{i1}(\frac{1}{2}|\qt_i|^2)^\vartheta}
\qquad &&\mbox{as } |\qt_i| \rightarrow \infty.
 \label{growth3}
\end{alignat}
\eqlabend
Hence our use of the words {\em Hookean-type model} throughout the paper (instead
of {\em Hookean model}, which would have corresponded to taking $\vartheta =1$ in the above).
%

The Maxwellian in the model is then defined by
\begin{align}
M(\qt) := \prod_{i=1}^K M_i(\qt_i) \qquad \forall \qt:=(\qt_1,\ldots,\qt_K) \in D
:= D_1 \times  \cdots \times D_K.
\label{MN}
\end{align}
Observe that, for $i=1, \dots, K$,
\begin{equation}
M(\qt)\,\nabqi [M(\qt)]^{-1} = - [M(\qt)]^{-1}\,\nabqi M(\qt) =
\nabqi U_i(\textstyle{\frac{1}{2}}|\qt_i|^2)
=U_i'(\textstyle{\frac{1}{2}}|\qt_i|^2)\,\qt_i. \label{eqM}
\end{equation}
Since
$[U_i(\textstyle{\frac{1}{2}}|\qt_i|^2)]^2 = (-\log M_i(\qt_i) + {\rm Const}.)^2$,
it follows from (\ref{growth2},c) that
\begin{equation}\label{additional-1}
\int_{D_i} \left[1 + [U_i(\textstyle{\frac{1}{2}}|\qt_i|^2)]^2
+ [U_i'(\textstyle{\frac{1}{2}}|\qt_i|^2)]^2\right] M_i(\qt_i) \, \dd
\qt_i < \infty, \qquad i=1, \dots, K.
\end{equation}

The governing equations of the general class of Hookean-type chain models with centre-of-mass diffusion are
(\ref{ns1a}--d), where the extra-stress tensor $\tautt$ is defined by the \textit{Kramers expression}:
\begin{equation}\label{tau1}
\tautt(\xt,t) = k\left( \sum_{i=1}^K \int_{D} \psi(\xt,\qt,t)\, \qt_i\,
\qt_i^{\rm T}\, 
U_i'\left(\textstyle \frac{1}{2}|\qt_i|^2\right)
{\dd} \qt -
\rho(\xt,t)\,\Itt\right),
\end{equation}
with the density of polymer chains located at $\xt$ at time $t$ given by
\begin{equation}\label{rho1}
\rho(\xt,t) = \int_{D} \psi(\xt, \qt,t)
{\dd}\qt.
\end{equation}
The probability density function $\psi$ is a solution of the
Fokker--Planck equation
\begin{align}
\label{fp0}
&\frac{\partial \psi}{\partial t} + (\ut \cdot\nabx) \psi +
\sum_{i=1}^K \nabqi
\cdot \left(\sigtt(\ut) \, \qt_i\, 
\psi \right)
\nonumber
\\
&\hspace{0.1in} =
\epsilon\,\Delta_x\,\psi +
\frac{1}{2 \,\lambda}\,
\sum_{i=1}^K \sum_{j=1}^K
A_{ij}\,\nabqi \cdot \left(
M\,\nabqj\! \left(\frac{\psi}{M}\right)\right) \quad \mbox{in } \Omega \times D \times
(0,T],
\end{align}
with $\sigtt(\vt) \equiv \nabxtt \vt$, where $(\vnabtt)(\xt,t) \in {\mathbb
R}^{d \times d}$ and $\{\vnabtt\}_{ij} = \textstyle
\frac{\partial v_i}{\partial x_j}$.
The dimensionless constant $k>0$ featuring in \eqref{tau1} is a constant multiple of
the product of the Boltzmann constant $k_B$ and the absolute temperature $\mathcal{T}$.
In \eqref{fp0}, $\varepsilon>0$ is the centre-of-mass diffusion coefficient defined
as $\varepsilon := (\ell_0/L_0)^2/(4(K+1)\lambda)$ with
$\ell_0:=\sqrt{k_B \mathcal{T}/H}$ signifying the characteristic microscopic length-scale
and $\lambda :=(\zeta/4H)(U_0/L_0)$,
where $\zeta>0$ is a friction coefficient and $H>0$ is a spring-constant.
The dimensionless parameter $\lambda \in \Rplus$, called the Weissenberg number (and usually denoted by {\sf Wi}), characterizes
the elastic relaxation property of the fluid, and $A=(A_{ij})_{i,j=1}^K$ is a symmetric positive definite
matrix, the \textit{Rouse matrix}, or connectivity matrix;
for example, $A = {\tt tridiag}\left[-1, 2, -1\right]$ in the case
of a linear chain; see, for example, Nitta \cite{Nitta}.

%
\begin{definition}
The collection of equations and structural hypotheses
{\rm (\ref{ns1a}--d)}--\eqref{fp0} will be referred to throughout the paper as model
$({\rm P}_{\varepsilon})$, or as Hookean-type (bead-spring chain) models
with centre-of-mass diffusion.
\end{definition}

A noteworthy feature of equation (\ref{fp0}) in the model $({\rm P}_{\varepsilon})$
compared to classical Fokker--Planck equations for bead-spring models in the
literature is the presence of the
$\xt$-dissipative centre-of-mass diffusion term $\varepsilon
\,\Delta_x \psi$ on the right-hand side of the Fokker--Planck equation (\ref{fp0}).
We refer to Barrett \& S\"uli \cite{BS} for the derivation of (\ref{fp0})
in the case of $K=1$;
see also the article by Schieber \cite{SCHI}
concerning generalized dumbbell models with centre-of-mass diffusion,
and the recent paper of Degond \& Liu \cite{DegLiu} for a careful justification
of the presence of the centre-of-mass diffusion term through asymptotic analysis.
In standard derivations of bead-spring models the centre-of-mass
diffusion term is routinely omitted on the grounds that it is
several orders of magnitude smaller than the other terms in the
equation. Indeed, when the characteristic macroscopic
length-scale $L_0\approx 1$, (for example, $L_0 = \mbox{diam}(\Omega)$), Bhave,
Armstrong \& Brown \cite{Bh} estimate the ratio $\ell_0^2/L_0^2$ to
be in the range of about $10^{-9}$ to $10^{-7}$. However, the
omission of the term $\varepsilon \,\Delta_x \psi$ from (\ref{fp0})
in the case of a heterogeneous solvent velocity $\ut(\xt,t)$ is a
mathematically counterproductive model reduction. When $\varepsilon
\,\Delta_x\psi$ is absent, (\ref{fp0}) becomes a degenerate
parabolic equation exhibiting hyperbolic behaviour with respect to
$(\xt,t)$. Since the study of weak solutions to the coupled problem
requires one to work with velocity fields $\ut$ that have very
limited Sobolev regularity (typically $\ut \in
L^\infty(0,T;\Lt^2(\Omega)) \cap L^2(0,T; \Ht^1_0(\Omega))$), one is
then forced into the technically unpleasant framework of
hyperbolically degenerate parabolic equations with rough transport
coefficients (cf. Ambrosio \cite{Am} and DiPerna \& Lions \cite{DPL}).
For these reasons, here we
shall retain the centre-of-mass diffusion term in (\ref{fp0}).
In order to emphasize that the positive centre-of-mass diffusion coefficient
$\varepsilon$ is {\em not} a mathematical artifact but the
outcome of the physical derivation of the model, in Section \ref{sec:2} and
thereafter the variables $\ut$ and $\psi$ have been labelled with the
subscript $\varepsilon$.

Following the introductory section in the companion paper \cite{BS2010}, which is concerned with analogous 
questions to the ones considered here in the case of bead-spring chains with FENE-type potentials, we continue with
a brief literature survey. Unless otherwise stated, the center-of-mass diffusion term is absent from 
the model considered in the cited reference (i.e. $\varepsilon$ is set to $0$); also, in all references 
cited, except \cite{BS2010}, $K=1$, i.e. a simple dumbbell model is considered, with a single spring and a pair of beads, 
rather than a general bead-spring chain model.

An early contribution to the existence and uniqueness of
local-in-time solutions to a family of dumbbell type polymeric
flow models is due to Renardy \cite{R}. While the class of
potentials $\Ft(\qt)$ considered by Renardy \cite{R} (cf.\
hypotheses (F) and (F$'$) on pp.~314--315) does include the case of a
Hookean dumbbell, it excludes the practically relevant case of the
FENE dumbbell model. 
More recently, E, Li \&
Zhang \cite{E} and Li, Zhang \& Zhang \cite{LZZ} have revisited the
question of local existence of solutions for dumbbell models. A further
development in this direction is the work of Zhang \& Zhang \cite{ZZ},
where the local existence of regular solutions to FENE-type dumbbell
models has been shown. All of these papers require high regularity of the initial data.
Constantin \cite{CON} considered the Navier--Stokes equations
coupled to nonlinear Fokker--Planck equations describing the
evolution of the probability distribution of the particles
interacting with the fluid.
Otto \& Tzavaras \cite{OT} investigated the Doi model (which is
similar to a Hookean model (cf. Example \ref{ex1.1} above), except
that $D=S^2$) for suspensions of rod-like molecules in the dilute regime.
Jourdain, Leli\`evre \& Le Bris \cite{JLL2}
studied the existence of
solutions to the FENE dumbbell model in the case of a simple Couette flow.
By using tools from the theory
of stochastic differential equations, they established the existence
of a unique local-in-time
solution to the FENE dumbbell model in two space dimensions ($d=2$)
when the velocity field $\ut$ is
unidirectional and of the particular form $\ut(x_1,x_2)
= (u_1(x_2),0)^{\rm T}$.

In the case of Hookean dumbbells ($K=1$), and assuming $\varepsilon=0$, the coupled
microscopic-macroscopic model described above yields, formally,
taking the second moment of $\qt
\mapsto \psi(\qt,\xt,t)$, the fully macroscopic,
Oldroyd-B model of viscoelastic flow. Lions \& Masmoudi \cite{LM}
have shown the existence of
global-in-time weak solutions to the Oldroyd-B model
in a simplified corotational setting (i.e. with $\sigma(\ut) = \nabxtt \ut$
replaced by $\frac{1}{2}(\nabxtt \ut - (\nabxtt u)^{\rm T})$) by
exploiting the
propagation in time of the compactness of the
solution (i.e. the property that if one takes a sequence of weak solutions that
converges weakly and such that the corresponding sequence of
initial data converges strongly, then the weak
limit is also a solution) and the DiPerna--Lions \cite{DPL} theory of
renormalized solutions to linear hyperbolic
equations with nonsmooth transport coefficients. It is not
known if an identical global
existence result for the Oldroyd-B model also holds in the
absence of the crucial assumption
that the drag term is corotational. We note in passing that
with $\varepsilon>0$ the coupled microscopic-macroscopic model above yields,
taking the appropriate
moments in the case of Hookean dumbbells, a dissipative
version of the Oldroyd-B model. In this
sense, the Hookean dumbbell model has a macroscopic closure:
it is the Oldroyd-B model when
$\varepsilon=0$, and a dissipative version of Oldroyd-B
when $\varepsilon>0$ (cf. Barrett \& S\"uli \cite{BS}).
Barrett \& Boyaval \cite{barrett-boyaval-09} have proved a global existence result
for this dissipative Oldroyd-B model in two space dimensions.
In contrast, the FENE model is not known to have an
exact closure at the
macroscopic level, though Du, Yu \& Liu \cite{DU} and Yu, Du \&
Liu \cite{YU} have recently
considered the analysis of approximate closures of the FENE dumbbell model.
Lions \& Masmoudi \cite{LM2} proved the global existence of weak
solutions for the \textit{corotational} FENE dumbbell model,
once again corresponding to the
case of $\varepsilon=0$ and $K=1$, and the Doi model,
also called the rod model. As in Lions \& Masmoudi \cite{LM},
their proof is based on propagation of
compactness; see also the related paper of Masmoudi \cite{M}.
Recently, Masmoudi \cite{M10} has extended this analysis to the noncorotational case.

Previously, El-Kareh \& Leal \cite{EKL} had proposed a steady macroscopic model,
with added dissipation in the equation satisfied by the conformation tensor, defined as
\[
\Att(\xt):=\int_D\qt\,\qt^{\rm T} U'\left(\textstyle{\frac{1}{2}}|\qt|^2\right) \,\psi(\xt,\qt) {\dd}\qt,\]
in order to account for Brownian motion across streamlines; the model
can be thought of as an approximate macroscopic closure of a
FENE-type micro-macro model with
centre-of-mass diffusion.

Barrett, Schwab \& S\"uli \cite{BSS}
established the existence of, global-in-time, weak solutions to the coupled microscopic-macroscopic
model (\ref{ns1a}--d) and
(\ref{fp0}) with $\varepsilon=0$, $K=1$, an $\xt$-mollified
velocity gradient in the Fokker--Planck
equation and an $\xt$-mollified probability density function $\psi$ in the Kramers
expression---admitting a large class of potentials $U$
(including the Hookean dumbbell model as
well as general FENE-type dumbbell models); in addition to these
mollifications, $\ut$ in the
$\xt$-convective term $(\ut\cdot\nabx) \psi$ in the
Fokker--Planck equation was also mollified.
Unlike Lions \& Masmoudi \cite{LM}, the arguments in
Barrett, Schwab \& S\"uli \cite{BSS} did {\em not} require the assumption
that the drag term $\nabq\cdot(\sigtt(\ut)\,\qt\, \psi)$ in the Fokker--Planck
was corotational in the FENE case.

In Barrett \& S\"uli \cite{BS},
we derived the coupled Navier--Stokes--Fokker--Planck model
with centre-of-mass diffusion stated above, in the case of $K=1$.
The anisotropic Friedrichs mollifiers, which naturally
arise in the derivation of the model in the Kramers expression for
the extra-stress tensor and in the drag term in the Fokker--Planck equation,
were replaced by isotropic Friedrichs mollifiers.
We established the existence of global-in-time weak solutions to
the model for a general class of spring-force-potentials including in
particular the FENE potential. We justified also, through a rigorous
limiting process, certain classical reductions of this model appearing
in the literature that exclude the centre-of-mass diffusion term from the
Fokker--Planck equation on the grounds that the diffusion coefficient is
small relative to other coefficients featuring in the equation.
In the case of a corotational drag term we performed a
rigorous passage to the limit as the Helmholtz-Stokes mollifiers in the
Kramers expression and the drag term converge to identity operators.


In Barrett \& S\"uli \cite{BS2} we showed the existence of global-in-time weak solutions to general noncorotational FENE-type dumbbell models (including the standard FENE dumbbell model) with centre-of-mass diffusion, in the case of $K=1$, 
with microsropic  cut-off in the drag term
\begin{equation}\label{drag}
 \nabq\cdot(\sigtt(\ut)\,\qt\, 
 \psi) =
\nabq\cdot \left[\sigtt(\ut) \, \qt\, 
M \left(\frac{\psi}{M}\right)\right].
\end{equation}
In Barrett \& S\"uli \cite{BS2010} we took that analysis further by showing the existence of global-in-time weak solutions to general noncorotational FENE-type dumbbell models (including the standard FENE dumbbell model) with centre-of-mass diffusion, in the general case $K\geq1$, {\em without} cut-off or mollification.
The weak solution was shown to satisfy an energy inequality, and in the absence of body forces it was shown
to converge exponentially to the equilibrium solution of the problem.

The present paper extends the analysis in Barrett \& S\"uli \cite{BS2010}
to Hookean-type bead-spring models with centre-of-mass diffusion,
in the general case $K \geq 1$,
{\em without} cut-off or mollification. As was noted above our current analysis
rules out the possibility of taking $\vartheta \equiv 1$ in (\ref{growth1}--c), since
we  require that the potentials $U_i$ have superlinear growth
as $|\qt|\rightarrow \infty$.
Since the argument is long and technical, we give a brief overview of the main steps of
the proof here.

{\em Step 1.} Following the approach in Barrett \& S\"uli \cite{BS2}
and motivated by recent papers of Jourdain, Leli\`evre, Le Bris \& Otto \cite{JLLO} and
Lin, Liu \& Zhang \cite{LinLZ}
(see also  Arnold, Markowich, Toscani \& Unterreiter \cite{AMTU},
and Desvillettes \& Villani \cite{DV})
concerning the convergence of the probability density function $\psi$ to its equilibrium
value $\psi_{\infty}(\xt,\qt):=M(\qt)$
(corresponding to the equilibrium value $\ut_\infty(\xt) :=\zerot$
of the velocity field) in the absence of body forces $\ft$,
we observe that if $\psi/M$ is bounded above then, for $L \in \mathbb{R}_{>0}$
sufficiently large, the drag term (\ref{drag}) is equal to
\begin{equation}\label{cut1}
\nabq\cdot \left[\sigtt(\ut) \, \qt\,M \,
\beta^L\!\left
(\frac{\psi}{M}\right)\right],
\end{equation}
where $\beta^L \in C({\mathbb R})$ is a cut-off function defined as
\begin{align}
\beta^L(s) := \min(s,L) = \left\{\begin{array}{ll}
s \qquad & \mbox{for $s \leq L$},
\\ L \qquad & \mbox{for $L \leq s$}.
\end{array} \right.
\label{betaLa}
\end{align}
More generally, in the case of $K \geq 1$, in analogy with \eqref{cut1}, the drag term with cut-off is
defined by
\[ \sum_{i=1}^K \nabqi \cdot \left(\sigtt(\ut) \, \qt_i\,M\,
\beta^L\!\left(\frac{\psi}{M}\right)\right).\]
%
It then follows that, for $L\gg 1$, any solution $\psi$ of (\ref{fp0}),
such that $\psi/M$ is bounded above, also satisfies
\begin{eqnarray}
\label{eqpsi1aa}
&&\hspace{-6.5mm}\frac{\partial \psi}{\partial t} + (\ut \cdot\nabx)\psi
+ \sum_{i=1}^K \nabqi \cdot \left(\sigtt(\ut) \, \qt_i\,M \,
\beta^L\!\left(\frac{\psi}{M}\right)\right)
\nonumber
\\
\bet
&&= \epsilon\,\Delta_x\,\psi + \frac{1}{2 \,\lambda}\,\sum_{i=1}^K\sum_{j=1}^K A_{ij} \nabqi \cdot \left(
M\,\nabqj \left(\frac{\psi}{M}\right)\right)\quad \mbox{in } \Omega \times D \times
(0,T]. ~~~
\end{eqnarray}
We impose the following decay/boundary and initial conditions:
\begin{subequations}
\begin{alignat}{2}
&\left|M \left[ \frac{1}{2\,\lambda} \sum_{j=1}^K A_{ij}\,\nabqj \!\left(\frac{\psi}{M}\right)
- \sigtt(\ut) \,\qt_i\,
\beta^L\!\left(\frac{\psi}{M}\right)
\right] \right|
\rightarrow 0\quad \mbox{as} \quad |\qt_i| \rightarrow \infty &&~~\nonumber\\
&~ \qquad && \hspace{-6.1cm} \mbox{on }
\Omega \times \left(\bigtimes_{j=1,\, j \neq i}^K D_j\right)\times (0,T],
\mbox{~~for $i=1,\dots, K$,} \label{eqpsi2aa}\\
&\epsilon\,\nabx \psi\,\cdot\,\nt =0&&\hspace{-6.1cm}\mbox{on }
\partial \Omega \times D\times (0,T],\label{eqpsi2ab}\\
&\psi(\cdot,\cdot,0)=M(\cdot)\,\beta^L\left({\psi_{0}(\cdot,\cdot)}/{M(\cdot)}\right) \geq 0&&
\hspace{-2.7cm}\mbox{~on $\Omega\times D$},\label{eqpsi3ac}
\end{alignat}
\end{subequations}
where 
$\nt$ is the unit outward normal to $\partial \Omega$.

Here $\psi_0$ is a nonnegative function
defined on $\Omega\times D$,
with $\int_{D} \psi_0(\xt,\qt) \dd \qt = 1$ for a.e. $\xt  \in \Omega$.
We shall also assume that $\psi_0$ has finite relative entropy with
respect to the Maxwellian $M$; i.e. $\int_{\Omega \times D} \psi_0(\xt,\qt)
\log (\psi_0(\xt,\qt)/M(\qt)) \dq \dx < \infty$.
Obviously, if there exists $L>0$ such that $0 \leq \psi_0 \leq L\, M$,
then $M\,\beta^L(\psi_{0}/M) = \psi_0$. Henceforth we shall assume
that $L >1$.

\begin{definition}
The coupled problem {\rm (\ref{ns1a}--d), (\ref{tau1}), (\ref{rho1}),
(\ref{eqpsi1aa}), (\ref{eqpsi2aa}--c)} will be referred to
as model $({\rm P}_{\varepsilon,L})$, or as the
Hookean-type (bead-spring chain) models with centre-of-mass diffusion and microscopic cut-off.
\end{definition}

In order to highlight the dependence
on $\varepsilon$ and $L$, in subsequent sections the solution to
(\ref{eqpsi1aa}), (\ref{eqpsi2aa}--c) will be labelled $\psiae$.
Due to the coupling of
(\ref{eqpsi1aa}) to (\ref{ns1a}) through (\ref{tau1}), the velocity and the pressure
will also depend on $\varepsilon$ and $L$ and we shall therefore
denote them in subsequent
sections by $\ut_{\varepsilon,L}$ and $p_{\varepsilon,L}$.

The cut-off $\beta^{L}$ has several attractive properties.
We observe that the couple $(\ut_{\infty},\psi_{\infty})$,  defined by
$\ut_{\infty}(\xt) := \zerot$ and $\psi_{\infty}(\xt,\qt):=M(\qt)$,
is still an equilibrium solution of (\ref{ns1a}--d) with $\ft =\zerot$,
(\ref{tau1}), (\ref{rho1}),
(\ref{eqpsi1aa}), (\ref{eqpsi2aa}--c) for all $L>0$.
Thus, unlike the truncation of
the (unbounded) potential proposed in El-Kareh \& Leal \cite{EKL},
the introduction of
the cut-off function $\beta^L$ into the Fokker--Planck equation (\ref{fp0})
does not alter the equilibrium solution $(\ut_{\infty},\psi_\infty)$ of the
original Navier--Stokes--Fokker--Planck system.
In addition, the boundary conditions for $\psi$
on $\partial\Omega\times D\times(0,T]$ and $\Omega \times \partial D
\times (0,T]$ ensure that
\[\int_{D}\psi(\xt,\qt,t) \dd \qt  =
\int_{D}\psi(\xt,\qt,0) \dd \qt \qquad
\mbox{for a.e. $\xt \in \Omega$ and a.e. $t \in \mathbb{R}_{\geq 0}$}.\]

\textit{Step 2.} Ideally, one would like to pass to the limit $L\rightarrow \infty$ in problem $({\rm P}_{\varepsilon,L})$ to
deduce the existence of solutions to $({\rm P}_{\varepsilon})$. Unfortunately,
such a direct attack at the problem is (except in the special case of $d=2$, or in the absence of
convection terms from the model,) fraught with technical difficulties. Instead,
we shall first (semi)discretize problem $({\rm P}_{\varepsilon,L})$
by an implicit Euler scheme with respect to $t$, with step size $\Delta t$;
this results in a time-discrete version $({\rm P}^{\Delta t}_{\varepsilon,L})$ of $({\rm P}_{\varepsilon,L})$. By using Schauder's fixed point theorem, we will show in
Section \ref{sec:existence-cut-off} the existence of solutions to $({\rm P}^{\Delta t}_{\varepsilon,L})$.
In the course of the proof, for technical reasons, a further cut-off, now from below, is required,
with a cut-off parameter $\delta \in (0,1)$, which we shall let pass to $0$ to complete the proof of existence of solutions to $({\rm P}^{\Delta t}_{\varepsilon,L})$ in the limit of $\delta \rightarrow 0_+$ (cf.
Section \ref{sec:existence-cut-off}).
Ultimately, of course, our aim is to show existence of weak solutions to the
Hookean-type models with centre-of-mass diffusion, $({\rm P}_{\varepsilon})$, and that demands passing to
the limits $\Delta t \rightarrow 0_+$ and $L \rightarrow \infty$; this then brings us
to the next step in our argument.

{\em Step 3.} We shall link the time step $\Delta t$ to the cut-off parameter $L>1$ by
demanding that $\Delta t = o(L^{-1})$, as $L \rightarrow \infty$, so that the
only parameter in the problem $({\rm P}^{\Delta t}_{\varepsilon,L})$ is the cut-off parameter (the centre-of-mass
diffusion parameter $\epsilon$ being fixed). By using special energy estimates, based on testing the
Fokker--Planck equation in $({\rm P}^{\Delta t}_{\varepsilon,L})$
with the derivative of the relative entropy with respect to
the Maxwellian of the model, we show that $\uta^{\Delta t}$
can be bounded, independent of $L$. Specifically $\ut^{\Delta t}_{\varepsilon,L}$ is bounded in the
norm of the classical Leray space, independent of $L$; also, the $L^\infty$ norm in time
of the relative entropy of $\psi^{\Delta t}_{\epsilon,L}$ and the $L^2$ norm in time of the Fisher
information of
$\hat \psi^{\Delta t}_{\epsilon,L} :=\psi^{\Delta t}_{\epsilon,L}/M$
are bounded, independent of $L$. We then use these $L$-independent
bounds on the relative entropy and the Fisher information to derive $L$-independent bounds on the
time-derivatives of $\uta^{\Delta t}$ and $\hat\psi^{\Delta t}_{\epsilon,L}$ in very weak, negative-order Sobolev norms.

\textit{Step 4.} The collection of $L$-independent bounds from Step 3
then enables us to extract a weakly convergent subsequence of solutions to
problem $({\rm P}^{\Delta t}_{\varepsilon,L})$ as $L \rightarrow \infty$.
We then apply a general compactness result in seminormed sets due to Dubinski{\u\i} \cite{DUB},
which furnishes strong convergence of a subsequence of solutions
$(\ut^{\Delta t_k}_{\varepsilon,L_k}, \hat\psi^{\Delta t_k}_{\epsilon,L_k})$
to $({\rm P}^{\Delta t}_{\varepsilon,L})$ with $\Delta t = o(L^{-1})$ as $L \rightarrow \infty$,
in $L^2(0,T;L^2(\Omega))
\times L^p(0,T; L^1(\Omega \times D))$ for any $p>1$. A crucial observation is that the set of
functions with finite Fisher information is not a linear space; therefore, typical Aubin--Lions--Simon
type compactness results (see, for example, Simon \cite{Simon})
do not work in our context; however, Dubinski{\u\i}'s compactness theorem, which applies
to seminormed sets in the sense of Dubinski{\u\i}, does, enabling us to pass to the limit
with the microscopic cut-off parameter $L$ in the model $({\rm P}^{\Delta t}_{\varepsilon,L})$, with
$\Delta t = o(L^{-1})$,
as $L \rightarrow \infty$, to finally deduce the existence of weak solutions to
Hookean-type models with centre-of-mass diffusion, $({\rm P}_{\varepsilon})$.

The paper is structured as follows. We begin, in Section~\ref{sec:2}, by stating $({\rm P}_{\varepsilon,L})$,
the coupled Navier--Stokes--Fokker--Planck system with centre-of-mass diffusion and microscopic
cut-off for a general class of Hookean-type spring potentials. In Section~\ref{sec:existence-cut-off}
we show the existence of solutions to the time-discrete problem $({\rm P}^{\Delta t}_{\varepsilon,L})$.
In Section~\ref{sec:entropy} we derive a set of $L$-independent bounds on $\uta^{\Delta t}$ in the classical Leray space, together with $L$-independent bounds on the relative entropy of $\psi^{\Delta t}_{\varepsilon, L}$ and the
Fisher information of $\hat\psi^{\Delta t}_{\epsilon,L}$. We then use
these $L$-independent bounds on spatial norms to obtain $L$-independent bounds on very weak
norms of time-derivatives of $\uta^{\Delta t}$ and $\psia^{\Delta t}$.
Section~\ref{sec:dubinskii} is concerned with the application of Dubinski{\u\i}'s
theorem to our problem; and the extraction of a strongly convergent subsequence, which we shall then use
in Section~\ref{sec:passage.to.limit} to pass to the limit with the cut-off parameter $L$
in problem $({\rm P}^{\Delta t}_{\varepsilon,L})$, with $\Delta t = o(L^{-1})$,
as $L \rightarrow \infty$,
to deduce the existence of a weak solution $(\ut_{\varepsilon},\psi_\varepsilon:=M\,\hat\psi_\epsilon)$
to problem $({\rm P}_{\varepsilon})$.
Finally, in Section \ref{sec:decay}, we show using a logarithmic Sobolev inequality and the
Csisz\'ar--Kullback inequality that, when $\ft \equiv \zerot$,
global weak solutions $t \mapsto (\ut_\epsilon(t),\psi_\epsilon(t))$
thus constructed decay exponentially in time to $(\zerot,M)$, at a rate that is independent of the
choice of the initial data for the Navier--Stokes and Fokker--Planck equations and of $\varepsilon$.
We shall operate within Maxwellian-weighted Sobolev spaces that provide the
natural functional-analytic framework for the problem. Our proofs require special
density and embedding results in these spaces, which
are proved in the Appendix.

\section{The polymer model $({\rm P}_{\varepsilon,L})$}
\label{sec:2}
\setcounter{equation}{0}

Let $\Omega \subset {\mathbb R}^d$ be a bounded open set with a
Lipschitz-continuous boundary $\partial \Omega$, and suppose that
the set $D:= D_1\times \cdots \times D_K$ of admissible
elongation vectors $\qt := (\qt_1, \ldots ,\qt_K)$ in (\ref{fp0}) is
such that $D_i \equiv {\mathbb R}^d$, $i=1, \dots, K$.

Collecting (\ref{ns1a}--d), (\ref{tau1}), and (\ref{fp0}),
we then consider the following initial-boundary-value problem,
dependent on the parameter $L > 1$. As has been already emphasized in the
Introduction, the centre-of-diffusion parameter $\varepsilon>0$ is a
physical parameter and is regarded as being fixed, although we systematically
highlight its presence in the model through our subscript notation.

(${\rm P}_{\epsilon,L}$)
Find $\utae\,:\,(\xt,t)\in \overline{\Omega} \times [0,T]
\mapsto \utae(\xt,t) \in {\mathbb R}^d$ and $\pae\,:\, (\xt,t) \in
\Omega \times (0,T] \mapsto \pae(\xt,t) \in {\mathbb R}$ such that
\begin{subequations}
\begin{eqnarray}
\ptutae + (\utae \cdot \nabx) \utae - \nu \,\delx \utae + \nabx \pae
&=& \ft + \nabx \cdot \tautt(\psiae)
\nonumber\\
\bet
&& \qquad \;
\mbox{in } \Omega \times (0,T], \label{equ1}\\
\bet
\nabx \cdot \utae &=& 0 ~\hspace{0.5cm}\mbox{in } \Omega \times (0,T],
\label{equ2}\\
\bet
\utae &=& \zerot  ~\hspace{0.5cm}\mbox{on } \partial \Omega \times (0,T],
\label{equ3}\\
\bet
\utae(\xt,0)&=&\ut_{0}(\xt) \!\hspace{0.5cm}\forall \xt \in \Omega,
\label{equ4}
\end{eqnarray}
\end{subequations}
where $\nu \in \Rplus$ is the given viscosity,
$\ft(\xt,t)$ is the given body force
 and
$\tautt(\psiae)\,:\,(\xt,t) \in \Omega \times (0,T] \mapsto
\tautt(\psiae)(\xt,t)\in \mathbb{R}^{d\times d}$ is the symmetric
extra-stress tensor, dependent on a probability density function
$\psiae\,:\,(\xt,\qt,t)\in
\overline{\Omega} \times D \times [0,T]
\mapsto \psiae(\xt,\qt,t)
\in {\mathbb R}$, defined as
\begin{equation}
\tautt(\psiae) = k \left( \sum_{i=1}^K
\Ctt_i(\psiae)\right)
- k\,\rho(\psiae)\, \Itt.
\label{eqtt1}
\end{equation}
Here $k \in \Rplus$,
$\Itt$ is the unit $d
\times d$ tensor,
\begin{subequations}
\begin{eqnarray}
\Ctt_i(\psiae)(\xt,t) &=& \int_{D} \psiae(\xt,\qt,t)\, U_i'({\textstyle
\frac{1}{2}}|\qt_i|^2)\,\qt_i\,\qt_i^{\rm T} 
\dq \label{eqCtt}
\end{eqnarray}

\noindent and

\begin{eqnarray}
\rho(\psiae)(\xt,t) &=& \int_{D} \psiae(\xt,\qt,t)\dq. \label{eqrhott}
\end{eqnarray}
\end{subequations}

The Fokker--Planck equation with microscopic cut-off
satisfied by $\psiae$ is:
\begin{align}
\label{eqpsi1a}
&\frac{\partial \psiae}{\partial t} + 
(\utae \cdot\nabx)\psiae
+
\sum_{i=1}^K \nabqi
\cdot \left[\sigtt(\utae) \, \qt_i\,M\,
\beta^L
\!\left(\frac{\psiae}{M}\right)\right]
\nonumber
\\
\bet
&\hspace{0.2in} =
\epsilon\,\Delta_x\,\psiae +
\frac{1}{2 \,\lambda}\,
\sum_{i=1}^K \sum_{j=1}^K
A_{ij}\,\nabqi \cdot \left(
M\,\nabqj\! \left(\frac{\psiae}{M}\right)\right) \qquad \mbox{in } \Omega \times D \times
(0,T].
\end{align}
Here, for a given $L > 1$, $\beta^L \in C({\mathbb R})$ is defined by (\ref{betaLa}),
$\sigtt(\vt) \equiv \nabxtt \vt$, and
\begin{align}
A \in {\mathbb R}^{K \times K} \mbox{ is symmetric positive definite
with smallest eigenvalue $a_0 \in {\mathbb R}_{>0}$.}
\label{A}
\end{align}
We impose the following decay/boundary and initial conditions:
\begin{subequations}
\begin{align}
&\left| M\left[\frac{1}{2\,\lambda} \sum_{j=1}^K A_{ij}\, \nabqj \left(\frac{\psiae}{M}\right)
- \sigtt(\utae) \,\qt_i\,
\beta^L\left(\frac{\psiae}{M}\right)
\right]\right| 
\rightarrow 0 \quad \mbox{as} \quad |\qt_i| \rightarrow \infty \nonumber \\
&\hspace{1.47in}\mbox{on }
\Omega \times \left(\bigtimes_{j=1,\, j \neq i}^K D_j\right) \times (0,T],
\quad i=1, \dots, K,
\label{eqpsi2}\\
&\epsilon\,\nabx \psiae\,\cdot\,\nt =0 \qquad \,\, \quad \mbox{on }
\partial \Omega \times D\times (0,T],\label{eqpsi2a}\\
&\psiae(\cdot,\cdot,0)=M(\cdot)\beta^L(\psi_{0}(\cdot,\cdot)/M(\cdot)) \geq 0 \qquad \mbox{on  $\Omega\times D$},\label{eqpsi3}
\end{align}
\end{subequations}
where $\nt$ is the unit outward normal to $\partial \Omega$.
The boundary conditions for $\psiae$
on $\partial\Omega\times D\times(0,T]$ and the decay conditions  for $\psiae$
on $\Omega \times \left(\bigtimes_{j=1,\, j \neq i}^K D_j\right)
\times (0,T]$ as $|\qt_i| \rightarrow \infty$, $i=1,\ldots,K$,
have been chosen so as to ensure that
\begin{align}
\int_{D}\psiae(\xt,\qt,t) \dd \qt  =
\int_{D} \psiae(\xt,\qt,0) \dd \qt  \qquad \forall (\xt,t) \in \Omega_T.
\label{intDcon}
\end{align}
Henceforth, we shall write $\psia:= \psiae/M$, $\hat\psi_0 := \psi_0/M$.
Thus, for example, \eqref{eqpsi3} in terms of this compact notation becomes:
$\psia(\cdot,\cdot,0) = \beta^L(\hat\psi_0(\cdot,\cdot))$ on $\Omega \times D$.

The notation $|\cdot|$ will be used to signify one of the following. When applied to a real number $x$,
$|x|$ will denote the absolute value of the number $x$; when applied to a vector $\vt$,  $|\vt|$ will
stand for the Euclidean norm of the vector $\vt$; and, when applied to a square matrix $A$, $|A|$ will
signify the Frobenius norm, $[\mathfrak{tr}(A^{\rm T}A)]^{\frac{1}{2}}$, of the matrix $A$, where, for a square matrix
$B$, $\mathfrak{tr}(B)$ denotes the trace of $B$.



\section{Existence of a solution to 
a discrete-in-time problem}

\label{sec:existence-cut-off}
\setcounter{equation}{0}

Let
\begin{eqnarray}
&\Ht :=\{\wt \in \Lt^2(\Omega) : \nabx \cdot \wt =0\} \quad
\mbox{and}\quad \Vt :=\{\wt \in \Ht^{1}_{0}(\Omega) : \nabx \cdot
\wt =0\},&~~~ \label{eqVt}
\end{eqnarray}
where the divergence operator $\nabx\cdot$ is to be understood in
the sense of vector-valued distributions on $\Omega$. Let $\Vt'$ be
the dual of $\Vt$.
Let
$\St: \Vt' \rightarrow \Vt$ be such that $\St \,\vt$
is the unique solution to the Helmholtz--Stokes problem
\begin{eqnarray}
&\displaystyle\int_{\Omega} \St\,\vt\cdot\, \wt \dx +
\int_{\Omega}
\nabxtt (\St\,\vt)
: \wnabtt \dx
= \langle \vt,\wt \rangle_V
\qquad \forall \wt \in \Vt,
\label{eqvn1}
\end{eqnarray}
where $ \langle \cdot,\cdot\rangle_V$ denotes the duality pairing
between $\Vt'$ and $\Vt$. We note that
\begin{eqnarray}
\left\langle \vt, \St\,\vt \right\rangle_V =
\|\St\,\vt\|_{H^1(\Omega)}^2
\qquad \forall \vt \in
\Vt' \supset (\Ht^1_0(\Omega))', \label{eqvn2}
\end{eqnarray}
and $\|\St \cdot\|_{H^{1}(\Omega)}$ is a norm on $\Vt'$. More generally, let
$\Vt_\sigma$ denote the closure of the set of all divergence-free $\Ct^\infty_0(\Omega)$ functions
in the norm of $\Ht^1_0(\Omega)\cap \Ht^\sigma(\Omega)$, $\sigma \geq 1$, equipped with the Hilbert
space norm, denoted by $\|\cdot\|_{V_\sigma}$, inherited from $\Ht^\sigma(\Omega)$, and let $\Vt_\sigma'$
signify the dual space of $\Vt_\sigma$, with duality pairing $\langle \cdot , \cdot \rangle_{V_\sigma}$.

For later purposes, we recall the following well-known
Gagliardo--Nirenberg inequality. Let $r \in [2,\infty)$ if $d=2$,
and $r \in [2,6]$ if $d=3$ and $\theta = d \,\left(\frac12-\frac
1r\right)$. Then, there is a constant $C=C(\Omega,r,d)$, 
such that, for all $\eta \in H^{1}(\Omega)$:
\begin{equation}\label{eqinterp}
\|\eta\|_{L^r(\Omega)}
\leq C\,
\|\eta\|_{L^2(\Omega)}^{1-\theta} 
\,\|\eta\|_{H^{1}(\Omega)}^\theta. 
\end{equation}

Let ${\cal F}\in C(\mathbb{R}_{>0})$ be defined by
${\cal F}(s):= s\,(\log s -1) + 1$, $s>0$.
As $\lim_{s \rightarrow 0_+} \mathcal{F}(s) = 1$,
the function $\mathcal{F}$ can be considered
to be defined and continuous on $[0,\infty)$,
where it is a nonnegative, strictly convex function
with $\mathcal{F}(1)=0$.
We then introduce the following
assumptions on the data:
\begin{align}\nonumber
&\partial \Omega \in C^{0,1}; \quad \ut_0 \in \Ht; \quad 
\hat \psi_0 := \frac{\psi_0}{M}\geq 0 \ {\rm ~a.e.\ on}\ \Omega \times D\quad \mbox{with}
\\ &
\mathcal{F}(\hat\psi_0)
\in L^1_M(\Omega \times D) \quad \mbox{and}
\quad \int_D M(\qt)\,\hat\psi_0(\xt,\qt)\,\dq = 1\quad \mbox{for a.e.\ } \xt \in \Omega;
\nonumber\\
& 
\mbox{and} \quad \ft \in L^{2}(0,T;\Vt').\label{inidata}
\end{align}
Here, $L^p_M(\Omega \times D)$, for $p\in [1,\infty)$,
denotes the Maxwellian-weighted $L^p$ space over $\Omega \times D$ with norm
$$
\| \hat \varphi\|_{L^{p}_M(\Omega\times D)} :=
\left\{ \int_{\Omega \times D} \!\!M\,
|\hat \varphi|^p \dq \dx
\right\}^{\frac{1}{p}}.
$$
Similarly, we introduce $L^p_M(D)$,
the Maxwellian-weighted $L^p$ space over $D$.

On defining
%
\begin{eqnarray}
\| \hat \varphi\|_{H^{1}_M(\Omega\times D)} &:=&
\left\{ \int_{\Omega \times D} \!\!M\, \left[
|\hat \varphi|^2 + \left|\nabx \hat \varphi \right|^2 + \left|\nabq \hat
\varphi \right|^2 \,\right] \dq \dx
\right\}^{\frac{1}{2}}\!\!, \label{H1Mnorm}
\end{eqnarray}
%
we then set
%
%
\begin{eqnarray}
\quad \hat X \equiv H^{1}_M(\Omega \times D)
&:=& \left\{ \hat \varphi \in L^1_{\rm loc}(\Omega\times D): \|
\hat \varphi\|_{H^{1}_M(\Omega\times D)} < \infty \right\}. \label{H1M}
\end{eqnarray}
%
Similarly, we introduce $H^1_M(D)$, the Maxwellian-weighted $H^1$ space over $D$.
It is shown in \ref{AppendixC} that
\begin{align}
C^\infty_0(D)
\mbox{ is dense in } H^1_M(D) \quad \mbox{and hence} \quad
C^{\infty}(\overline{\Omega},C^\infty_0(D))
\mbox{ is dense in } \hat X.
\label{cal K}
\end{align}

We have from Sobolev embedding that
\begin{equation}
H^1(\Omega;L^2_M(D)) \hookrightarrow L^s(\Omega;L^2_M(D)),
\label{embed}
\end{equation}
where $s \in [1,\infty)$ if $d=2$ or $s \in [1,6]$ if $d=3$.
Similarly to (\ref{eqinterp}) we have,
with $r$ and $\theta$ as defined there,
that there exists a constant $C$, depending only  on
$\Omega$, $r$ and $d$, such that
\begin{equation}\label{MDeqinterp}
\hspace{-0mm}\|\hat \varphi\|_{L^r(\Omega;L^2_M(D))} 
\leq C\,
\|\hat \varphi\|_{L^2(\Omega;L^2_M(D))}^{1-\theta} 
\,\|\hat \varphi\|_{H^1(\Omega;L^2_M(D))}^\theta ~~~
\mbox{$\forall\hat \varphi \in
H^{1}(\Omega;L^2_M(D))$}.
\end{equation}
In addition, we note that the embeddings
\begin{subequations}
\begin{align}
 H^1_M(D) &\hookrightarrow L^2_M(D) ,\label{wcomp1}\\
H^1_M(\Omega \times D) \equiv
L^2(\Omega;H^1_M(D)) \cap H^1(\Omega;L^2_M(D))
&\hookrightarrow
L^2_M(\Omega \times D) \equiv L^2(\Omega;L^2_M(D))
\label{wcomp2}
\end{align}
\end{subequations}
are compact; 
see \ref{sec:comptensorise} and \ref{AppendixD}, respectively.

Let $\hat X'$ be the dual space of $\hat X$ with $L^2_M(\Omega\times D)$
being the pivot space.
Then, similarly to (\ref{eqvn1}),
let ${\cal G}: \hat X' \rightarrow \hat X$
be such that ${\cal G}\, \hat \eta$ is the unique solution of
\begin{align}
&\int_{\Omega \times D}
M\,\biggl[ ({\cal G}\,\hat \eta) \,\hat \varphi
+ \nabq\, ({\cal G}\,\hat \eta) \cdot \nabq\,\hat \varphi
+ \nabx\, ({\cal G}\,\hat \eta) \cdot \nabx\,\hat \varphi \biggr]
\dq \dx
\label{CalG}
\nonumber
\\
&\hspace{2.14in}
= 
\langle M\,\hat \eta,\hat \varphi \rangle_{\hat X} \qquad
\forall \hat \varphi \in \hat X,
\end{align}
where $\langle M\,\cdot, \cdot \rangle_{\hat X}$
denotes the duality pairing between $\hat X'$ and $\hat X$.
Then, similarly to (\ref{eqvn2}), we have that
\begin{align}
\langle M \,\hat \eta, {\cal G}\, \hat \eta\, \rangle_{\hat X}
= \|{\cal G}\,\hat \eta \|_{\hat X}^2 \qquad
\forall \hat \eta \in \hat X',
\label{CalG1}
\end{align}
and $\|{\cal G} \cdot\|_{\hat X}$ is a norm on ${\hat X}'$.

We recall the Aubin--Lions--Simon compactness theorem, see, e.g.,
Temam \cite{Temam} and Simon \cite{Simon}. Let ${\cal B}_0$, ${\cal B}$ and
${\cal B}_1$ be Banach
spaces, ${\cal B}_i$, $i=0,1$, reflexive, with a compact embedding ${\cal B}_0
\hookrightarrow {\cal B}$ and a continuous embedding ${\cal B} \hookrightarrow
{\cal B}_1$. Then, for $\alpha_i>1$, $i=0,1$, the embedding
\begin{eqnarray}
&\{\,\eta \in L^{\alpha_0}(0,T;{\cal B}_0): \frac{\partial \eta}{\partial t}
\in L^{\alpha_1}(0,T;{\cal B}_1)\,\} \hookrightarrow L^{\alpha_0}(0,T;{\cal B})
\label{compact1}
\end{eqnarray}
is compact.

Throughout we will assume that
(\ref{growth1}--c) and (\ref{inidata}) hold,
so that (\ref{additional-1}) and (\ref{wcomp1},b) hold.
We note for future reference that (\ref{eqCtt}) and
(\ref{additional-1}) yield that, for
$\hat \varphi \in L^2_M(\Omega \times D)$,
\begin{align}
\label{eqCttbd}
\int_{\Omega} |\Ctt_i( M\,\hat \varphi)|^2\,\dx & =
\int_{\Omega} 
\left|
\int_{D} M\,\hat \varphi \,U_i'\,\qt_{i}\,\qt_{i}^{\rm T} 
\dq \right|^2 \dx
\nonumber
\\
&\leq 
\left(\int_{D} M\,(U_i')^2 \,|\qt_i|^4 \,\dq\right)
\left(\int_{\Omega \times D} M\,|\hat \varphi|^2 \dq \dx\right)
\nonumber \\
& \leq C
\left(\int_{\Omega \times D} M\,|\hat \varphi|^2 \dq \dx\right),
\qquad i=1, \dots,  K,
\end{align}
where $C$ is a positive constant.

We now establish a simple integration-by-parts formula.
\begin{lemma}
Let $\hat\varphi \in H^1_M(D)$ and suppose that $B \in \mathbb{R}^{d\times d}$ is a square
matrix such that $\mathfrak{tr}(B)=0$; then,
\begin{subequations}
\begin{align}
\int_{D_i} M_i(\qt_i)\,
(B\qt_i) \cdot \nabqi \hat\varphi(\qt) \dd \qt_i &= \int_{D_i} M_i(\qt_i)\,
\hat\varphi(\qt) \, U_i'(\textstyle{\frac{1}{2}|\qt_i|^2})\,\qt_i \qt_i^{\rm T}
: B \dq_i
\nonumber \\
\mbox{for a.e.\ } (\qt_1,\ldots, \qt_{i-1},&\qt_{i+1},\ldots, \qt_K)^{\rm T}
\in \left( \bigtimes_{j=1,\, j \neq i}^K D_j\right),
\quad i=1,\ldots, K,
\label{intbypartsi}
\\
\label{intbyparts}
\int_D M(\qt)\,
\sum_{i=1}^K (B\qt_i) \cdot \nabqi \hat\varphi(\qt) \dd \qt &= \int_D M(\qt)\,
\hat\varphi(\qt) \sum_{i=1}^K
U_i'(\textstyle{\frac{1}{2}|\qt_i|^2})\,
\qt_i \qt_i^{\rm T}
:
B \dq.
\end{align}
\end{subequations}
\end{lemma}
\begin{proof}
The set  $C^\infty_0(D)$ is dense in
$H^1_M(D)$, see \ref{AppendixC};
hence, there exists a sequence $\{\hat\varphi_n\}_{n \geq 0} \subset
C^\infty_0(D)$, converging to $\hat\varphi$ in $H^1_M(D)$.
The identity  \eqref{intbypartsi}
with
$\hat\varphi$ replaced by $\hat\varphi_n$
is easily verified.
First, on applying the classical divergence theorem
for smooth functions, and noting \eqref{eqM} and that $\mathfrak{tr}(B)=0$,
we obtain that
\begin{align}
\int_{B_R({\footnotesize \zerot})} M_i(\qt_i)\,
(B\qt_i) \cdot \nabqi \hat\varphi_n(\qt) \dd \qt_i &
= \int_{B_R({\footnotesize \zerot})} M_i(\qt_i)\,
\hat\varphi_n(\qt) \,\qt_i \qt_i^{\rm T} U_i'(\textstyle{\frac{1}{2}|\qt_i|^2}) 
: B \dq_i
\nonumber
\\
& \hspace{1in}
+ \displaystyle
\int_{\partial B_R({\footnotesize \zerot})}
M_i\,
\hat\varphi_n \,(B\qt_i) \cdot \frac{\qt_i} {|\qt_i|}
\,{\rm d}S(\qt_i)
\nonumber \\
\qquad
\mbox{for a.e.\ } (\qt_1, \ldots, \qt_{i-1},&\qt_{i+1},\ldots, \qt_K)^{\rm T}
\in \left(\bigtimes_{j=1,\, j \neq i}^K D_j\right),
\quad i=1,\ldots, K.
\label{intbypartsibt}
\end{align}
Here we have replaced  $D_i$ by $B_R(\zerot)$, the ball of radius $R$ centred at the origin.
On noting (\ref{growth3}),
we have that, as $R \rightarrow \infty$,
\begin{align}
\left| \int_{\partial B_R({\footnotesize \zerot})}
M_i(\qt_i)\,
\hat\varphi_n(\qt) \,(B\qt_i) \cdot \frac{\qt_i} {|\qt_i|}
\,{\rm d}S(\qt_i) \right|
\leq C\,{\rm e}^{-c_{i1}(\frac{1}{2}R^2)^\vartheta} \,R^d\, |B|\,\|\hat \varphi_n\|_{L^\infty(D)}.
\end{align}
Hence, on passing to the limit $R \rightarrow \infty$ in (\ref{intbypartsibt})
the boundary term vanishes and we obtain
(\ref{intbypartsi}) with $\hat \varphi$ replaced by
$\hat \varphi_n$.
Then, \eqref{intbypartsi} itself follows by
letting $n \rightarrow \infty$, recalling the
definition of the norm in $H^1_M(D)$ and \eqref{additional-1}.
Finally, multiplying (\ref{intbypartsi}) by $M/M_i$, integrating over
$\left(\bigtimes_{j=1,\, j \neq i}^K D_j\right)$ and summing from $i=1$ to $K$ yields the
desired result (\ref{intbyparts}).
\end{proof}

We now formulate our discrete-in-time approximation of problem
(P$_{\epsilon,L})$ for fixed parameters $\epsilon \in (0,1]$
and $L > 1$. For any $T>0$, let
$N \,\Delta t=T$ and $t_n = n \, \Delta t$, $n=0, \dots,  N$.
To prove existence of a solution under minimal
smoothness requirements on the initial data, recall (\ref{inidata}),
we introduce
$\ut^0 \in \Vt$  such that
\begin{alignat}{2}
\int_{\Omega} \left[ \ut^0 \cdot \vt + \Delta t\,
\nabxtt \ut^0 : \nabxtt \vt \right] \dx
&=   \int_{\Omega} \ut_0 \cdot \vt \dx \qquad
&&\forall \vt \in \Vt;
\label{proju0}
\end{alignat}
and so
\begin{equation}
\int_{\Omega} [\,|\ut^0|^2 + \Delta t \,|\unabtt^0|^2 \,]\dx 
\leq
\int_{\Omega} |\ut_0|^2 \dx
\leq C.
\label{idatabd}
\end{equation}
In addition, we have
that $\ut^0$ converges to $\ut_0$ weakly in $\Ht$ in the limit
of $\Delta t \rightarrow 0_+$.
For $p\in [1,\infty)$, let
\begin{align}
\hat Z_p &:= \{ \hat \varphi \in L^p_M(\Omega \times D) :
\hat \varphi
\geq 0 \mbox{ a.e.\ on } \Omega \times D
\mbox{ and }
\nonumber \\
& \hspace{1.4in}
\int_D M(\qt)\,\hat \varphi(\xt,\qt) \dq \leq 1 \mbox{ for a.e. } \xt \in \Omega
\}.
\label{hatZ}
\end{align}

Analogously to defining  $\ut^0$ for a given initial velocity field
$\ut_0$, we shall assign a certain `smoothed' initial
datum, $\hat\psi^0$, to the initial datum $\hat\psi_0$. The definition of $\hat\psi^0$
is delicate; it will be given in Section \ref{sec:passage.to.limit}. All we need to know for now is
that there exists a $\hat\psi^0$, independent of the cut-off parameter $L$, such that:
\begin{equation}\label{inidata-1}
\hat\psi^0 \in \hat{Z}_1;\;\; {\small \left\{\begin{array}{rr}\mathcal{F}(\hat\psi^0) \in L^1_M(\Omega \times D);\\
                                                    \sqrt{\hat\psi^0} \in H^1_M(\Omega \times D);
                                    \end{array}\right.}
\;\;
\int_{\Omega \times D}  M\, \mathcal{F}(\hat\psi^0)\dq \dx \leq \int_{\Omega \times D}\!\!\! M\, \mathcal{F}(\hat\psi_0)\dq \dx.
\end{equation}
It follows from (\ref{inidata-1}) and (\ref{betaLa}) that
$\beta^L(\hat \psi^0) \in \hat Z_2$; in fact, $\beta^L(\hat \psi^0) \in L^\infty(\Omega \times D)$.

Our discrete-in-time  approximation of (P$_{\epsilon,L}$) is then defined as follows.



{\boldmath $({\rm P}_{\varepsilon,L}^{\Delta t})$}
Let $\utae^0= \ut^0 \in \Vt$ and $\hpsiae^0 = \beta^L(\hat \psi^0) \in \hat Z_2$.
Then, for $n =1, \dots,  N$, given
$(\utae^{n-1},\hpsiae^{n-1}) \in \Vt \times \hat Z_2$,
find $(\utae^n,\hpsiae^n 
) \in
\Vt \times (\hat X \cap \hat Z_2)
$ such that 
\begin{subequations}
\begin{align}
&\int_{\Omega}
\left[\frac{\utae^{n}-\utae^{n-1}}{\Delta t}
+ (\utae^{n-1} \cdot \nabx) \utae^{n} \right]\,\cdot\, \wt
\dx +
\nu\, \int_{\Omega}
\nabxtt \utae^n
: \wnabtt \dx
\nonumber
\\
&\hspace{1cm} = \int_{\Omega} \ft^n \cdot \wt \dx
- k\,\sum_{i=1}^K \int_{\Omega} \Ctt_i(M\,\hpsiae^n): \nabxtt
\wt \dx
\qquad \forall \wt \in \Vt,
\label{Gequn}
\end{align}
\begin{align}
&\int_{\Omega \times D} M\,\frac{\hpsiae^n
- \hpsiae^{n-1}}
{\Delta t}\,\hat \varphi \dq \dx
\nonumber
\\
\bet
&\hspace{0.1cm} + \int_{\Omega \times D}
M\,\sum_{i=1}^K  \left[\, \frac{1}{2\, \lambda}\,
\sum_{j=1}^K A_{ij}\,\nabqj \hpsiae^n
-[\,\sigtt(
\utae^n) \,\qt_i\,
]\,\beta^L(\hpsiae^{n})\right]\,\cdot\, \nabqi
\hat \varphi \dq \dx\,
\nonumber \\
\bet
&\hspace{0.1cm} + \int_{\Omega \times D} M\,
\left[\epsilon\,\nabx \hpsiae^n
- \utae^{n-1}\,
\hpsiae^n \right]\,\cdot\, \nabx
\hat \varphi\dq \dx=0
\qquad \forall \hat \varphi \in
\hat X;
\label{psiG}
\end{align}
\end{subequations}
where, for  $t \in [t_{n-1}, t_n)$, and $n=1, \dots,  N$,
\begin{align}
\ft^{\Delta t, +}(\cdot,t) =
\ft^n(\cdot) := \frac{1}{\Delta t}\,\int_{t_{n-1}}^{t_n}
\ft(\cdot,t) \dt \in \Vt'.
\label{fn}
\end{align}
It follows from (\ref{inidata}) and (\ref{fn}) that
\begin{align}
\ft^{\Delta t, +} \rightarrow \ft \quad \mbox{strongly in } L^{2}(0,T;\Vt')
\mbox { as } \Delta t \rightarrow 0_{+}.
\label{fncon}
\end{align}

We note here that since the test function $\wt$ in \eqref{Gequn} is chosen to be divergence-free,
the term containing the density $\rho$ in the definition of $\tautt$ (cf. \eqref{eqtt1})
is eliminated from \eqref{Gequn}, and will play no role in the rest of the paper.

In order to prove
existence of a solution to (P$_{\epsilon,L}^{\Delta t}$), we
require the following convex regularization
${\cal F}_{\delta}^L \in
C^{2,1}({\mathbb R})$ of ${\cal F}$ defined, for any $\delta \in (0,1)$ and
$L>1$, by
\begin{align}
 &{\cal F}_{\delta}^L(s) := \left\{
 \begin{array}{ll}
 \textstyle\frac{s^2 - \delta^2}{2\,\delta}
 + s\,(\log \delta - 1) + 1
 \quad & \mbox{for $s \le \delta$}, \\
{\cal F}(s)\ \equiv
s\,(\log s - 1) + 1 & \mbox{for $\delta \le s \le L$}, \\
  \textstyle\frac{s^2 - L^2}{2\,L}
 + s\,(\log L - 1) + 1
 & \mbox{for $L \le s$}.
 \end{array} \right. \label{GLd}
\end{align}
Hence,
\begin{subequations}
\begin{align}
\quad &[{\cal F}_{\delta}^{L}]'(s) = \left\{
 \begin{array}{ll}
 \textstyle \frac{s}{\delta} + \log \delta - 1
 \quad & \mbox{for $s \le \delta$}, \\
 \log s & \mbox{for $\delta \le s \le L$}, \\
 \textstyle \frac{s}{L} + \log L - 1
& \mbox{for $L \le s$},
 \end{array} \right. \label{GLdp}\\
\quad
 &[{\cal F}_{\delta}^{L}]''(s) = \left\{
 \begin{array}{ll}
 {\delta}^{-1} \quad & \mbox{for $s \le \delta$}, \\
 s^{-1} & \mbox{for $\delta \le s \le L$}, \\
 L^{-1} & \mbox{for $L \le s$}. \,
\end{array} \right. \label{Gdlpp}
\end{align}
\end{subequations}
We note that
\begin{align}
{\cal F}^L_\delta(s) \geq \left\{
\begin{array}{ll}
\frac{s^2}{2\,\delta} &\quad \mbox{for $s \leq 0$},
\\
\frac{s^2}{4\,L} - C(L)&\quad \mbox{for $s \geq 0$};
\end{array}
\right.
\label{cFbelow}
\end{align}
and that
$[{\cal F}_{\delta}^{L}]''(s)$
is bounded below by $1/L$ for all $s \in \mathbb{R}$.
Finally, we set
\begin{align}
\beta^L_\delta(s) := ([{\cal F}_{\delta}^{L}]'')^{-1}(s)
= \max \{\beta^L(s),\delta\},
\label{betaLd}
\end{align}
and observe that $\beta^L_\delta(s)$
is bounded above by $L$ for all $s \in \mathbb{R}$. Note also that both $\beta^L$ and $\beta^L_\delta$
are Lipschitz continuous, and that their respective Lipschitz constants are equal to $1$.

\subsection{\boldmath Existence of a solution to  $({\rm P}_{\varepsilon,L}^{\Delta t})$}
\label{sec:existence-cut-off.1}

It is convenient to rewrite (\ref{Gequn}) as
\begin{equation}
b(\utae^n,\wt) =
\ell_b(\hpsiae^n)(\wt)
\qquad \forall \wt \in \Vt;
\label{bLM}
\end{equation}
where, for all $\wt_i \in \Ht^{1}_{0}(\Omega)$, $i=1,2$,
\begin{subequations}
\begin{align}
b(\wt_1,\wt_2) &:=
\int_{\Omega} \left[ \wt_1
+ \Delta t \, (\utae^{n-1} \cdot \nabx) \wt_{1} \right]\,\cdot\,
\wt_2 \dx + \Delta t \, \nu\,\int_{\Omega}
\nabxtt \wt_1
:\nabxtt \wt_2 \dx,
\label{bgen}
\end{align}
and, for all $\wt \in \Ht^1_0(\Omega)$ and
$\hat \varphi \in L^2_M(\Omega \times D)$,
\begin{align}
\ell_b(\hat \varphi)(\wt) &:=
\Delta t \,\langle \ft^n, \wt
\rangle_V +
\displaystyle\int_{\Omega}
\left[
\utae^{n-1} \cdot \wt  - \Delta t \,k\,\sum_{i=1}^K
\Ctt_i(M\,\hat \varphi) : \nabxtt
\wt \right] \dx.
\label{lbgen}
\end{align}
\end{subequations}
We note that
\begin{align}
\int_{\Omega} \left[ (\vt \cdot \nabx) \wt_1 \right]\,\cdot\,
\wt_2 \dx\hspace{17pc}
\label{tripid}
\nonumber
\\
= - \int_{\Omega} \left[ (\vt \cdot \nabx) \wt_2
\right]\,\cdot\, \wt_1 \dx
\qquad \forall \vt \in \Vt,\quad \forall \wt_1, \wt_2 \in
\Ht^{1}(\Omega),
\end{align}
and hence $b(\cdot,\cdot)$ is a continuous nonsymmetric coercive bilinear
functional on $\Ht^1_0(\Omega) \times \Ht^1_0(\Omega)$.
In addition, on recalling (\ref{eqCttbd}), $\ell_b(\hat \varphi)(\cdot)$
is a continuous linear
functional on $\Vt$ for any $\hat\varphi \in L^2_M(\Omega\times D)$.

For $r>d$, let
\begin{align}
\Yt^r :=
\left\{\vt \in \Lt^{r}(\Omega) :
\int_{\Omega} \vt\,\cdot\,\nabx w \dx = 0 \quad \forall w
\in W^{1,\frac{r}{r-1}}(\Omega)\right\}.
\label{Ytr}
\end{align}
It is also convenient to rewrite (\ref{psiG}) as
\begin{align}
a(\hpsiae^n,\hat \varphi) = \lae(\utae^n,\beta^L(\hpsiae^n))
(\hat \varphi) \qquad \forall \hat \varphi \in \hat X,
\label{genLM}
\end{align}
where, for all $\hat \varphi_1,\,\hat \varphi_2 \in \hat X$,
\begin{subequations}
\begin{align}
a(\hat \varphi_1,\hat \varphi_2) &:= \int_{\Omega \times D} M\,\biggl(
\hat \varphi_1\,\hat \varphi_2 + \Delta t \left[
\,\epsilon\,\nabx
\hat \varphi_1 - \utae^{n-1}\,\hat \varphi_1\right]
\,\cdot\, \nabx
\hat \varphi_2 \label{agen}
\nonumber\\ & \hspace{1in}
+\, \frac{\Delta t}{2\,\lambda} \,
\sum_{i=1}^K \sum_{j=1}^K A_{ij}\,
\nabqj
\hat \varphi_1 \, \cdot\, \nabqi
\hat \varphi_2 \biggr) \dq \dx,
\end{align}
and, for all $\vt \in \Ht^1(\Omega)$, $\hat \eta \in L^\infty(\Omega\times D)$
and $\hat \varphi \in \hat X$,
\begin{align}
\lae(\vt,\hat \eta)(\hat \varphi) &:=
\int_{\Omega \times D}
M \left[\hpsiae^{n-1}
\,\hat \varphi
+ \Delta t\,\sum_{i=1}^K [\,\sigtt(\vt)
\,\qt_i\,
]\,\hat \eta\, \cdot\, \nabqi
\hat \varphi \right]\!\dq \dx.
\label{lgen}
\end{align}
\end{subequations}
It follows from
(\ref{Ytr}) and (\ref{embed}) that, for $r>d$,
\begin{eqnarray}
\int_{\Omega \times D} M\,
\vt\,\hat \varphi \,\cdot \nabx \hat \varphi
\dq \dx = 0 \qquad \forall \vt \in \Yt^r, \quad \forall \hat \varphi \in \hat
X;
\label{Ytra}
\end{eqnarray}
%
and hence $a(\cdot,\cdot)$ is a
continuous nonsymmetric coercive bilinear functional on $ \hat X \times \hat X$.
In addition,
we have that,
for all $\vt \in \Ht^1(\Omega)$, $\hat \eta \in L^\infty(\Omega \times D)$
and $\hat \varphi \in \hat X$,
\begin{align}
|\lae(\vt,\hat \eta)(\hat \varphi)| &\leq
\|\hpsiae^{n-1}\|_{L^2_M(\Omega \times D)}
\,\|\hat \varphi\|_{L^2_M(\Omega \times D)}
\nonumber \\
& \hspace{0.5in}
+ \Delta t\,\left( \int_{D} M\,|\qt|^2 \dq \right)^{\frac{1}{2}} 
\|\hat \eta\|_{L^\infty(\Omega \times D)}
\,\|\nabxtt \vt\|_{L^2(\Omega)}
\,\|\nabq \hat \varphi\|_{L^2_M(\Omega \times D)}.
\label{lgenbd}
\end{align}
Hence, on noting that $\hpsiae^{n-1} \in \hat Z_2$,
(\ref{growth3}) and (\ref{MN}), 
$\ell_a(\vt,\hat \eta)(\cdot)$ is a continuous linear functional on $\hat X$
for all $\vt \in \Ht^1(\Omega)$ and $\hat \eta \in L^\infty(\Omega \times D)$.

In order to prove existence of a solution to (\ref{Gequn},b),
i.e. (\ref{bLM}) and (\ref{genLM}),
we
consider a regularized system for a given $\delta \in (0,1)$:

Find $(\utaed^{n},\hpsiaed^n) \in \Vt \times \hat X$ such that
\begin{subequations}
\begin{alignat}{2}
b(\utaed^n,\wt) &=
\ell_b(\hpsiaed^n)(\wt)
\qquad &&\forall \wt \in \Vt,
\label{bLMd} \\
a(\hpsiaed^n,\hat \varphi) &= \lae(\utaed^n,\beta^L_\delta(\hpsiaed^n))
(\hat \varphi) \qquad &&\forall \hat \varphi \in \hat X.
\label{genLMd}
\end{alignat}
\end{subequations}

In order to prove existence of a solution to (\ref{bLMd},b),
we consider a fixed-point argument. Given
$\hat \psi \in L^2_M(\Omega \times D)$,
let $(\ut^{\star}, \hat \psi^\star ) \in
\Vt \times \hat X$ be such that
\begin{subequations}
\begin{alignat}{2}
\qquad b(\ut^{\star},\wt) &=
\ell_b(\hat \psi)(\wt)
\qquad &&\forall \wt \in \Vt,
\label{fix4} \\
\qquad a(\hat \psi^{\star},\hat \varphi) &=
\lae(\ut^\star,\beta^L_\delta(\hat \psi))(\hat \varphi) \qquad
&&\forall \hat \varphi \in \hat X. \label{fix3}
\end{alignat}
\end{subequations}
%
The Lax--Milgram theorem yields the
existence of a unique solution to (\ref{fix4},b),
and so the overall procedure
(\ref{fix4},b) is well defined.

\begin{lemma}
\label{fixlem} Let $G: L^2_M(\Omega \times D) \rightarrow \hat X
\subset L^2_M(\Omega \times D)$ denote the
nonlinear map that takes $\hat{\psi}$ to $\hat \psi^{\star} = G(\hat \psi)$
via the procedure
{\rm (\ref{fix4},b)}. Then $G$ has a fixed point. Hence there exists a solution
$(\utaed^n,\hpsiaed^n 
) \in
\Vt \times
\hat X 
$ to {\rm (\ref{bLMd},b)}.
\end{lemma}

\begin{proof}
Clearly, a fixed point of $G$ yields a
solution of (\ref{bLMd},b).
In order to show that $G$ has a fixed point, we apply
Schauder's fixed-point theorem; that is, we need to show that:
(i)~$G:
L^2_M(\Omega \times D) \rightarrow
L^2_M(\Omega \times D)$ is continuous; (ii)~$G$ is compact; and
(iii)~there exists a $C_{\star} \in {\mathbb R}_{>0}$ such that
\begin{eqnarray}
\|\hat{\psi}\|_{L^{2}_M(\Omega\times D)} \leq C_{\star}
\label{fixbound}
\end{eqnarray}
for every $\hat{\psi} \in L^2_M(\Omega \times D)$ and $\kappa \in (0,1]$
satisfying $\hat{\psi} = \kappa\, G(\hat{\psi})$.

Let $\{\hat{\psi}^{(p)}\}_{p \geq 0}$ be such that
\begin{eqnarray}
\hat{\psi}^{(p)}
\rightarrow \hat{\psi} \qquad \mbox{strongly in }
L^{2}_M(\Omega\times D)\qquad \mbox{as } p \rightarrow \infty.
\label{Gcont1}
\end{eqnarray}
It follows immediately from (\ref{betaLd})  that,
for any $r \in [2,\infty)$,
\begin{subequations}
\begin{align}
M^{\frac{1}{2}}\,\beta^L_\delta(\hat{\psi}^{(p)})
&\rightarrow M^{\frac{1}{2}}\,
\beta^L_\delta(\hat{\psi}) \qquad \mbox{strongly in }
L^{r}(\Omega\times D)\quad \mbox{as } p \rightarrow \infty,
\label{betacon}
\end{align}
and from (\ref{eqCttbd}), for  $i=1, \dots, K$,
\begin{align}
\Ctt_i(M\,\hat{\psi}^{(p)})
&\rightarrow \Ctt_i(M\,\hat{\psi}) \qquad \mbox{strongly in }
L^{2}(\Omega)\quad \mbox{as } p \rightarrow \infty.
\label{Cttcon}
\end{align}
\end{subequations}
We need to show that
\begin{eqnarray}
\hat{\eta}^{(p)}:=G(\hat{\psi}^{(p)}) \rightarrow G(\hat{\psi})
\quad \mbox{strongly in } L^2_M(\Omega\times D)\quad \mbox{as } p
\rightarrow \infty, \label{Gcont2}
\end{eqnarray}
in order to prove (i) above. We have from the definition of $G$,
see (\ref{fix4},b), that, for all $p \geq 0$,
\begin{subequations}
\begin{align}
a(\hat \eta^{(p)}, \hat \varphi)&= \ell_a(\vt ^{(p)},
\beta^L_\delta(\hat \psi^{(p)}))(\hat \varphi)
\qquad \forall \hat \varphi \in \hat X,
\label{Gcont5}
\end{align}
where $\vt^{(p)} \in \Vt$ satisfies
\begin{align}
b(\vt^{(p)},\wt) &= \ell_b(\hat \psi^{(p)})(\wt)
\qquad \forall \wt \in \Vt.
\label{Gcont5a}
\end{align}
\end{subequations}

Choosing $\hat \varphi = \hat \eta^{(p)}$ in (\ref{Gcont5})
yields, on noting the simple identity
\begin{equation}
2\,(s_1-s_2)\,s_1 = s_1^2 + (s_1 -s_2)^2 -s_2^2 \qquad \forall
s_1, s_2 \in {\mathbb R}, \label{simpid}
\end{equation}
and that $\ut^{n-1}_{\epsilon,L} \in \Vt$,
(\ref{A}),
the bound on the second term in (\ref{lgen}),
as in (\ref{lgenbd}),
(\ref{growth3}), (\ref{MN})
and (\ref{betaLd})
that, for all $p \geq 0$,
\begin{align}
&\int_{\Omega \times D}
M\,\left [ |\hat \eta^{(p)}|^2
+ | \hat \eta^{(p)} - \hpsiae^{n-1}|^2
+ \frac{a_0\,\Delta t}{2\,\lambda}\,|\nabq \hat \eta^{(p)}|^2
+ 2\,\epsilon \, \Delta t\,|\nabx \hat \eta^{(p)}|^2 \right] \dq \dx
\label{Gcont6}
\nonumber
\\
& \qquad \leq \int_{\Omega \times D} M\,
|\hpsiae^{n-1}|^2 \dq \dx
+ C(L)\,\Delta t\int_{\Omega}
|\nabxtt \vt^{(p)}|^2 \dx.
\end{align}
Choosing $\wt \equiv \vt^{(p)}$ in (\ref{Gcont5a}), and noting
(\ref{simpid}), (\ref{tripid}),
(\ref{eqCttbd}), (\ref{eqvn1}), a Poincar\'e inequality
and (\ref{Gcont1}) yields, for
all $i\geq 0$, that 
\begin{align}
&\hspace{-3mm}\int_{\Omega} \left[ |\vt^{(p)}|^2 +
|\vt^{(p)}-\utae^{n-1}|^2  \right] \dx +
\Delta t \,\nu \, \int_{\Omega} |\nabxtt \vt^{(p)}|^2 \dx
\label{Gcont4}
\nonumber\\
& \hspace{1in}
\leq \int_\Omega |\utae^{n-1}|^2 \dx + C\,\Delta t \,
\|\St\,\ft^n\|_{H^1(\Omega)}^2
+ C\, \Delta t \, \int_{\Omega \times D}
M\,|\hat{\psi}^{(p)}|^2 \dq \dx
\leq C.
\end{align}
Combining (\ref{Gcont6}) and (\ref{Gcont4}),
we have for all $p \geq 0$ that
\begin{eqnarray}
\|\hat \eta^{(p)}\|_{\hat X}
+ \|\vt^{(p)}\|_{H^1(\Omega)} \leq C(L, (\Delta t)^{-1})\,.
\label{Gcont7}
\end{eqnarray}
It follows from
(\ref{Gcont7}), 
(\ref{embed})
and the compactness of
the embedding (\ref{wcomp2})
that there exists a subsequence
$\{(\hat{\eta}^{(p_k)},\vt^{(p_k)})\}_{p_k \geq 0}$ and
functions $\hat{\eta}\in \hat X$ and $\vt \in \Vt$ such that,
as $p_k \rightarrow \infty$,
\begin{subequations}
\begin{alignat}{2}
\hat{\eta}^{(p_k)}
&\rightarrow
\hat{\eta}
\qquad
&&\mbox{weakly in }
L^{s}(\Omega ;L^2_M(D)),
\label{Gcont8a} \\
\bet
M^{\frac{1}{2}}\,\nabx \hat{\eta}^{(p_k)}
&\rightarrow M^{\frac{1}{2}}\,\nabx \hat{\eta} \qquad
&&\mbox{weakly in }
\Lt^{2}(\Omega \times D),
\label{Gcont8bx} \\
\bet
M^{\frac{1}{2}}\,\nabq \hat{\eta}^{(p_k)}
&\rightarrow M^{\frac{1}{2}}\,\nabq \hat{\eta}
\qquad
&&\mbox{weakly in }
\Lt^{2}(\Omega \times D),
\label{Gcont8b} \\
\hat{\eta}^{(p_k)}
&\rightarrow
\hat{\eta}
\qquad
&&\mbox{strongly in }
L^{2}_M(\Omega\times D),
\label{Gcont8as} \\
\bet
\vt^{(p_k)} &\rightarrow \vt \qquad
&&\mbox{weakly in }
\Vt;
\label{Gcont8c}
\end{alignat}
\end{subequations}
where $s \in [1,\infty)$ if $d=2$ or $s \in [1,6]$ if $d=3$.
We deduce from (\ref{Gcont5a}), (\ref{bgen},b), (\ref{Gcont8c})
and (\ref{Cttcon})
that $\vt \in \Vt$ and $\hat{\psi} \in \hat X$ satisfy
\begin{eqnarray}\qquad
b(\vt,\wt) &=\ell_b(\hat \psi)(\wt)
\qquad \forall \wt \in \Vt.
\label{Gcont9}
\end{eqnarray}
It follows from
(\ref{Gcont5}), (\ref{agen},b), (\ref{Gcont8a}--e), 
and (\ref{betacon})
that $\hat \eta,\,\hat \psi \in \hat X$
and $\vt \in \Vt$,
satisfy
\begin{align}
a(\hat{\eta},\hat \varphi)
&= \lae(\vt,\beta^L_\delta(\hat \psi))(\hat \varphi)
\qquad \forall \hat \varphi \in C^\infty(\overline{\Omega};C^\infty_0(D)).
\label{Gcont11}
\end{align}
Then, noting (\ref{lgenbd}), (\ref{growth3}) and (\ref{cal K}) yields that (\ref{Gcont11}) holds
for all $\hat \varphi \in \hat X$.
Combining this $\hat X$ version of (\ref{Gcont11}) and (\ref{Gcont9}), we have that
$\hat \eta = G(\hat \psi)\in \hat X$. Therefore
the whole sequence
\[\hat{\eta}^{(p)} \equiv G(\hat{\psi}^{(p)})
\rightarrow G(\hat{\psi})\qquad \mbox{strongly in $L^2_M(\Omega\times D)$},
\]
as $p \rightarrow \infty$, and so (i) holds.

As the embedding $\hat X \hookrightarrow L^{2}_M(\Omega \times D)$
is compact, it follows that (ii) holds. It therefore remains to show that (iii) holds.

As regards (iii), $\hat{\psi} =
\kappa \, G(\hat{\psi})$ implies that $\{\vt,\hat{\psi}\} \in
\Vt \times \hat X$ satisfies
\begin{subequations}
\begin{alignat}{2}
b(\vt,\wt) &=
\ell_b(\hat \psi)(\wt)
\quad &&\forall \wt \in \Vt,
\label{fix4sig}
\\
a(\hat{\psi},\hat \varphi)&=\kappa\,\lae(\vt,\beta^L_\delta(\hat \psi))(\hat \varphi)\qquad
&&\forall \hat \varphi \in \hat X. \label{fix3sig}
\end{alignat}
\end{subequations}
Choosing $\wt \equiv \vt$ in (\ref{fix4sig})
yields, similarly to (\ref{Gcont4}), that
\begin{eqnarray}
\label{Gequnbhat}
&&
\frac{1}{2}\,\displaystyle
\int_{\Omega} \left[ \,|\vt|^2 +
|\vt-\utae^{n-1}|^2 - |\utae^{n-1}|^2 \,\right]
\dx + \Delta t\, \nu \,
\int_{\Omega} |\nabxtt \vt|^2 \dx
\nonumber
 \\
\bet
&&\hspace{3cm}
= \Delta t \left[
\langle \ft^n, \vt \rangle_V
- k\,\sum_{i=1}^K
\int_{\Omega}
 \Ctt_i(M\,\hat{\psi}): \nabxtt \vt \dx \right].
\end{eqnarray}
Choosing $\hat \varphi = [{\cal F}_\delta^L]'(\hat{\psi})$ in (\ref{fix3sig}),
and noting
the convexity of ${\cal F}_\delta^L$,
(\ref{betaLd}) 
and that $\vt$ 
is divergence-free, yield
\begin{align}
&
\int_{\Omega \times D} M\,\left[ \,{\cal F}_\delta^L (\hat{\psi})
- {\cal F}_\delta^L (\kappa \,\hpsiae^{n-1})
+ \Delta t \,\left[\epsilon
\,\nabx \hat{\psi}
- \utae^{n-1} \hat \psi \right]
\cdot \nabx ([{\cal F}_{\delta}^L]'(\hat{\psi}))
\right]  \dq \dx
\label{acorstab1}
\nonumber
\\
&
\quad + \frac{\Delta t}{2\,\lambda}\,
\sum_{i=1}^K \sum_{j=1}^K A_{ij}
\int_{\Omega \times D} M\,
\,\nabqj
\hat{\psi} \cdot \nabqi ([{\cal F}_\delta^L]'(\hat{\psi}))
\dq \dx \nonumber \\
&
\hspace{2cm}
\leq \kappa\,\Delta t \,\sum_{i=1}^K
\int_{\Omega \times D}
M\,\sigtt(\vt) \,
\qt_i \cdot
\nabqi \hat{\psi}
\dq \dx
\nonumber
\\
&\hspace{2cm}
= \kappa\,\Delta t \,\sum_{i=1}^K
\int_{\Omega}
\Ctt_i(M\,\hat \psi) : \sigtt(\vt)  \dx,
\end{align}
where in the transition to the final equality we applied (\ref{intbyparts}) with $B:=\sigtt(\vt)$
(on account of it being independent of the variable $\qt$), together with the fact that
$\mathfrak{tr}(\sigtt(\vt)) = \nabx \,\cdot\, \vt=0$, and recalled (\ref{eqCtt}).
Next, on noting (\ref{betaLd}) and that $\utae^{n-1} \in \Vt$, it follows that
\begin{align}
\int_{\Omega \times D} M\,\utae^{n-1} \hat \psi
\cdot \nabx ([{\cal F}_{\delta}^L]'(\hat{\psi}))
\dq \dx &=
\int_{\Omega \times D} M\,\utae^{n-1} \frac{\hat \psi}{\beta^L_\delta(\hat \psi)}
\cdot \nabx \hat{\psi}
\dq \dx \nonumber \\
&=\int_{\Omega \times D} M\,\utae^{n-1}
\cdot \nabx ([G^L_\delta]'(\hat{\psi}))
\dq \dx =0,
\label{convGdL}
\end{align}
where $G^L_\delta \in C^{0,1}({\mathbb R})$ is defined by
 \begin{align}
G^L_\delta(s) := \left\{\begin{array}{ll}
\frac{1}{2 \delta} s^2 +\frac{(\delta-L)}{2} \qquad &\mbox{if } s \leq \delta,\\
s-\frac{L}{2} \qquad &\mbox{if } s \in [\delta,L],\\
\frac{1}{2L} s^2 \qquad &\mbox{if } s \geq L;
\end{array}
\right.
\label{GdL}
\end{align}
and so $[G^L_\delta]'(s)= s/\beta^L_\delta(s)$.
Combining (\ref{Gequnbhat}) and (\ref{acorstab1}),
and noting (\ref{convGdL}),
(\ref{Gdlpp}), (\ref{A}),
(\ref{eqvn1}) and a Poincar\'{e} inequality yields that
\begin{align}
\label{Ek}
&
\frac{\kappa}{2}\,\displaystyle
\int_{\Omega} \left[ \,|\vt|^2 +
|\vt-\utae^{n-1}|^2 \,\right]
\dx + \kappa\,\Delta t\, \nu \,
\int_{\Omega} |\nabxtt \vt|^2 \dx
+k\,
\int_{\Omega \times D} M\,{\cal F}_\delta^L (\hat{\psi})
\dq \dx \nonumber
\\
&\qquad +k\,L^{-1}\,\Delta t\,
\int_{\Omega \times D} M\,
\left[\,\epsilon
\,|\nabx \hat{\psi}|^2 
+ \frac{a_0}{2\,\lambda} \,|\nabq
\hat{\psi}|^2 
\right] \dq \dx \nonumber \\
&\hspace{1cm}\leq
\kappa\,\Delta t \,\langle \ft^n, \vt \rangle_V +
\frac{\kappa}{2}\,\int_{\Omega} |\utae^{n-1}|^2 \dx
+ k\,
\int_{\Omega \times D} M\,{\cal F}_\delta^L (\kappa \,\hpsiae^{n-1})
\dq \dx
\nonumber
\\
&\hspace{1cm}\leq
\frac{\kappa}{2}\,\Delta t\, \nu \,
\int_{\Omega} |\nabxtt \vt|^2 \dx
+ \kappa \,\Delta t \,
C(\nu^{-1})\,
\|\St\,\ft^n\|_{H^{1}(\Omega)}^2 \nonumber\\
& \hspace{2cm} +
\frac{\kappa}{2}\,\int_{\Omega} |\utae^{n-1}|^2 \dx
+ k\,
\int_{\Omega \times D} M\,{\cal F}_\delta^L (\kappa \,\hpsiae^{n-1})
\dq \dx.
\end{align}

It is easy to show that $\mathcal{F}^L_\delta(s)$ is nonnegative for all
$s \in \mathbb{R}$, with  $\mathcal{F}^L_\delta(1)=0$.
Furthermore, for any $\kappa \in (0,1]$,
\[\begin{array}{ll}
\mbox{$\mathcal{F}^L_\delta(\kappa\, s) \leq
\mathcal{F}^L_\delta(s)$} & \mbox{if $s<0$ or $1 \leq \kappa\, s$},\\
\mbox{$\mathcal{F}^L_\delta(\kappa\, s) \leq \mathcal{F}^L_\delta(0)
\leq 1$} \qquad & \mbox{if $0 \leq \kappa\, s \leq 1$.}
\end{array}
\]
Thus we deduce that
\begin{equation}\label{deltaL}
{\cal F}_\delta^L(\kappa\, s)
\leq {\cal F}_\delta^L(s)+ 1\qquad \forall s \in {\mathbb R},\quad
\forall \kappa \in (0,1].
\end{equation}
Hence, the bounds (\ref{Ek}) and (\ref{deltaL}), on noting (\ref{cFbelow}),
%
%
%
give rise to the desired bound (\ref{fixbound}) with $C_*$ dependent only on
$\delta$, $L$, $\Delta t$, $k$, $\nu$, $\ft$, $\utae^{n-1}$ and $\hpsiae^{n-1}$.
Therefore (iii) holds, and so $G$ has a fixed point.
Thus we have proved existence of a solution
to (\ref{bLMd},b).
\qquad\end{proof}

Choosing $\wt \equiv \utaed^n$ in (\ref{bLMd})
and $\hat \varphi \equiv [{\cal F}_\delta^L]'(\hpsiaed^n)$ in (\ref{genLMd}),
and combining, then yields, similarly
to (\ref{Ek}), that
\begin{align}
\label{E1}
&
\tfrac{1}{2}\,\displaystyle
\int_{\Omega} \left[ \,|\utaed^n|^2 +
|\utaed^n-\utae^{n-1}|^2 \,\right]
\dx 
+k\,
\int_{\Omega \times D} M\,{\cal F}_\delta^L (\hpsiaed^n)
\dq \dx
\nonumber
\\
&\quad\quad
+ \Delta t\, \biggl[ \frac{\nu}{2} \,
\int_{\Omega} |\nabxtt \utaed^n|^2 \dx
+k\,L^{-1}\,\epsilon\,
\int_{\Omega \times D} M
\,|\nabx \hpsiaed^n|^2 
\dq \dx
\nonumber \\
& \quad \quad + \frac{k\,L^{-1}\,a_0}{2\,\lambda} \,
\int_{\Omega \times D}
M\, |\nabq
\hpsiaed^n|^2 
 \dq \dx \biggr] \nonumber \\
&\quad\quad\quad\leq
\Delta t \,C(\nu^{-1})\,\|\St\,\ft^n \|_{H^1(\Omega)}^2 +
\tfrac{1}{2}\,\int_{\Omega} |\utae^{n-1}|^2 \dx
+ k\,
\int_{\Omega \times D} M\,{\cal F}_\delta^L (\hpsiae^{n-1})
\dq \dx
\nonumber \\
& \quad \quad \quad \leq C(L),
\end{align}
where,
on recalling that $\psiae^{n-1} \geq 0$,
$C(L)$ is a positive constant, independent of $\delta$ and $\Delta t$.
We are now in a position to prove the following convergence result.

\begin{lemma}
\label{conv}
There exists a subsequence (not indicated) of $\{(\utaed^{n},
\hpsiaed^{n})\}_{\delta >0}$, and functions $\utae^{n} \in \Vt$
and $\hpsiae^{n} \in \hat X \cap \hat Z_2$, $n \in \{1,\dots, N\}$,
such that, as $\delta \rightarrow 0_+$,
\begin{subequations}
\begin{eqnarray}
&\utaed^{n} \rightarrow \utae^{n} \qquad &\mbox{weakly in }
\Vt, \label{uwconH1}\\
\bet
&\utaed^{n} \rightarrow \utae^{n}
\qquad &\mbox{strongly in }
\Lt^{r}(\Omega), \label{usconL2}
\end{eqnarray}
\end{subequations}
where $r \in [1,\infty)$ if $d=2$ and $r \in [1,6)$ if $d=3$;
and
\begin{subequations}
\begin{alignat}{2}
M^{\frac{1}{2}}\,\hpsiaed^{n} &\rightarrow
M^{\frac{1}{2}}\,
\hpsiae^{n} &&\quad \mbox{weakly in }
L^2(\Omega\times D), \label{psiwconL2}\\
\bet
M^{\frac{1}{2}}\,\nabq \hpsiaed^{n}
&\rightarrow M^{\frac{1}{2}}\,\nabq \hpsiae^{n}
&&\quad \mbox{weakly in }
\Lt^2(\Omega\times D), \label{psiwconH1}\\
\bet
M^{\frac{1}{2}}\,\nabx \hpsiaed^{n}
&\rightarrow M^{\frac{1}{2}}\,\nabx \hpsiae^{n}
&&\quad \mbox{weakly in }
\Lt^2(\Omega\times D), \label{psiwconH1x}\\
\bet
M^{\frac{1}{2}}\,\hpsiaed^{n} &\rightarrow
M^{\frac{1}{2}}\,\hpsiae^{n}
&&\quad \mbox{strongly in }
L^{2}(\Omega\times D),\label{psisconL2}
\\
\bet
M^{\frac{1}{2}}\,\beta_\delta^L(\hpsiaed^{n}) &\rightarrow
M^{\frac{1}{2}}\,\beta^L(\hpsiae^{n})
&&\quad \mbox{strongly in }
L^{s}(\Omega\times D),\label{betaLdsconL2}
\end{alignat}
where $s \in [2, \infty)$;
and, for $i=1, \dots,  K$,
\begin{align}
\Ctt_i(M\,\hpsiaed^{n}) &\rightarrow \Ctt_i(M\,\hpsiae^{n})
\qquad \mbox{strongly in }
\Ltt^{2}(\Omega).\label{CwconL2}
\end{align}
\end{subequations}
In addition, $(\utae^n, \hpsiae^n)$ solves (\ref{Gequn},b) for $n=1,\dots, N$; consequently there
exists a solution $\{(\utae^n,$ $\hpsiae^n)\}_{n=1}^N$ to
{\em ({\rm P}$^{\Delta t}_{\epsilon,L}$)}, with $\utae^n \in \Vt$ and $\psiae^n \in \hat{Z}_2$ for all $n=1,\dots, N$.
\end{lemma}

\begin{proof}
The weak convergence results (\ref{uwconH1}) and (\ref{psiwconL2})
and the fact that
$\hpsiae^n \ge 0$ a.e.\ on $\Omega \times D$
follow immediately from the first two bounds on the left-hand side of
(\ref{E1}), on noting (\ref{cFbelow}).
The strong convergence result
(\ref{usconL2}) for $\utaed^{n}$
follows directly from (\ref{uwconH1}),
on noting that $\Vt \subset
\Ht^{1}_0(\Omega)$ is compactly embedded in $\Lt^r(\Omega)$ for
the stated values of $r$.

It follows immediately from the bound on the fifth term on the
left-hand side of (\ref{E1}) that (\ref{psiwconH1})
holds for some limit $\gt \in \Lt^2(\Omega \times D)$,
which we need to identify. However, for any $\etat \in
\Ct^{1}_0(\Omega\times D)$, it follows from
(\ref{eqM}) and the compact support of $\etat$
on $D$ that
\[ [\nabq \cdot (M^{\frac{1}{2}}
\,\etat)\,]/M^{\frac{1}{2}} \in L^2(\Omega \times D),\]
and hence the above convergence implies, noting (\ref{psiwconL2}), that
\begin{align}
\int_{\Omega \times D} \gt \cdot
\etat \dq \dx
&\leftarrow
- \int_{\Omega \times D}M^{\frac{1}{2}}\, \hpsiaed^{n}\,
\frac{\nabq \cdot (M^{\frac{1}{2}}\,\etat)}{M^{\frac{1}{2}}}
\dq \dx
\nonumber
\\
\bet
& \rightarrow
- \int_{\Omega \times D} M^{\frac{1}{2}}\,\hpsiae^n\,
\frac{\nabq \cdot (M^{\frac{1}{2}}\,\etat)}{M^{\frac{1}{2}}}
\dq \dx =
- \int_{\Omega \times D} \hpsiae^n\,
\nabq \cdot (M^{\frac{1}{2}}\,\etat)
\dq \dx \nonumber \\
\label{derivid}
\end{align}
as $\delta  \rightarrow 0_+$.
Equivalently, on dividing and multiplying by $M^{\frac{1}{2}}$ under the integral sign on the left-hand side, we have that
\[ \int_{\Omega \times D} M^{-\frac{1}{2}}\gt \cdot M^{\frac{1}{2}}\etat \dq \dx=  -  \int_{\Omega \times D} \hat{\psi}^n_{\varepsilon, L} \, \nabq \cdot (M^{\frac{1}{2}}\etat) \dq \dx\qquad \forall \etat \in \undertilde{C}^1_0(\Omega \times D).
\]
Observe that $\etat \in \undertilde{C}^1_0(\Omega \times D) \mapsto M^{\frac{1}{2}} \etat \in \undertilde{C}^1_0(\Omega \times D)$ is a bijection of $\undertilde{C}^1_0(\Omega \times D)$ onto itself; thus, the equality above is
equivalent to
\[ \int_{\Omega \times D} M^{-\frac{1}{2}}\gt \cdot \chit \dq \dx=  -  \int_{\Omega \times D} \hat{\psi}^n_{\varepsilon, L} \, (\nabq \cdot \chit) \dq \dx\qquad \forall \chit
\in \undertilde{C}^1_0(\Omega \times D).
\]
Since $\Ct^\infty_0(\Omega \times D) \subset \Ct^1_0(\Omega \times D)$, the last identity also holds for all
$\etat \in \Ct^\infty_0(\Omega \times D)$.
As $M^{\frac{1}{2}} \in L^\infty(D)$ and $M^{-\frac{1}{2}} \in L^\infty_{\rm loc}(D)$,
it follows that
$M^{-\frac{1}{2}}\gt
\in \Lt^2_{\rm loc}(\Omega \times D)$
and $\hat{\psi}^n_{\varepsilon, L} \in L^2_{\rm loc}(\Omega \times D)$.
By identification of a locally integrable function with a
distribution we deduce that $M^{-\frac{1}{2}} \gt$ is the distributional gradient of
$\hat{\psi}^n_{\varepsilon, L}$ w.r.t. $\qt$:
\[ M^{-\frac{1}{2}}\gt = \nabq \hat{\psi}^n_{\varepsilon, L} \qquad \mbox{in $\undertilde{\mathcal{D}}'(\Omega \times D)$.}\]
%
As $M^{-\frac{1}{2}} \gt \in \undertilde{L}^2_{\rm loc}(\Omega \times D)$, whereby also
$\nabq \hat{\psi}^n_{\varepsilon, L} \in \undertilde{L}^2_{\rm loc}(\Omega \times D)$,
it follows that
\[ \gt = M^{\frac{1}{2}} \nabq \hat{\psi}^n_{\varepsilon, L} \in \undertilde{L}^2_{\rm loc}(\Omega \times D).\]
However, the left-hand side belongs to $\undertilde{L}^2(\Omega \times D)$, which then implies that the right-hand side also belongs to $\undertilde{L}^2(\Omega \times D)$. Thus we have shown that
\begin{equation}\label{last}
\gt = M^{\frac{1}{2}} \nabq \hat{\psi}^n_{\varepsilon, L} \in \undertilde{L}^2(\Omega \times D),
\end{equation}
and hence the desired result (\ref{psiwconH1}) as required.

A similar argument proves (\ref{psiwconH1x})
on noting (\ref{psiwconL2}), and
the fourth bound in (\ref{E1}).

The strong convergence result
(\ref{psisconL2}) for $\hpsiaed^{n}$
follows immediately from (\ref{psiwconL2}--c) and (\ref{wcomp2}).
Finally, the desired results
(\ref{betaLdsconL2},f) follow immediately
from (\ref{psisconL2}),
(\ref{betaLd}),
(\ref{eqCtt})
and (\ref{eqCttbd}).

It follows from (\ref{uwconH1},b), (\ref{psiwconH1}--f),
(\ref{bgen},b), (\ref{agen},b), (\ref{lgenbd}) and
(\ref{cal K})
that we may pass to the limit, $\delta \rightarrow 0_+$, in
(\ref{bLMd},b) to obtain
that $(\utae^n,\hpsiae^n) \in \Vt \times \hat X$
with $\hpsiae^n \geq 0$ a.e.\ on $\Omega \times D$
solve (\ref{bLM}) and (\ref{genLM}), i.e. (\ref{Gequn},b).

Next we prove the integral constraint on $\hpsiae^n$.
First we introduce, for $m=n-1,\,n$,
\begin{align}
\zeta^m_{\epsilon,L}(\xt) := \int_D M(\qt) \hpsiae^m(\xt,\qt) \dq.
\label{zetan}
\end{align}
As $\hpsiae^n \in \hat X$ and $\hpsiae^{n-1} \in \hat Z_2$,
we deduce from the Cauchy--Schwarz inequality and Fubini's theorem that
$\zeta^n_{\epsilon,L} \in H^1(\Omega)$ and $\zeta^{n-1}_{\epsilon,L} \in L^2(\Omega)$.
We introduce
also the following closed linear subspace of $\hat X = H^1_M(\Omega \times D)$:
\begin{align}
H^1(\Omega)\otimes 1(D)
:= \left\{\hat\varphi \in H^1_M(\Omega \times D)\,:\, \hat\varphi(\cdot,\qt^*)
= \hat\varphi(\cdot,\qt^{**})\quad\mbox{for all $\qt^*, \qt^{**} \in D$}\right\}.
\label{H101}
\end{align}
Then, on choosing $\hat \varphi = \varphi \in H^1(\Omega)\otimes 1(D)$
in (\ref{psiG}), we deduce from (\ref{zetan}) and Fubini's theorem that
\begin{align}
\label{psiGI}
&\int_{\Omega}
\frac{\zeta^n_{\epsilon,L}-\zeta^{n-1}_{\epsilon,L}}{\Delta t}
\,\varphi \dx
+ \int_{\Omega} \left[
\epsilon\,\nabx \zeta^n_{\epsilon,L} -
\utae^{n-1}\,\zeta^n_{\epsilon,L} \right]\cdot\, \nabx
\varphi \dx
= 0
\qquad \forall
\varphi \in H^1(\Omega).
\end{align}

By introducing the function $z^m_{\epsilon,L}:= 1 - \zeta^m_{\epsilon,L}$,
$m=n-1,\,n$,
we deduce from (\ref{psiGI}), and as $\utae^{n-1}$ is divergence-free on $\Omega$
with zero trace on $\partial \Omega$, that
\begin{align}
\label{psiGIz}
&\int_{\Omega}
\frac{z^n_{\epsilon,L}-z^{n-1}_{\epsilon,L}}{\Delta t}
\,\varphi \dx
+ \int_{\Omega} \left[
\epsilon\,\nabx z^n_{\epsilon,L} -
\utae^{n-1}\,z^n_{\epsilon,L} \right]\cdot\, \nabx
\varphi \dx
= 0
\qquad \forall
\varphi \in H^1(\Omega).
\end{align}
Let us now define by
\[ [x]_{\pm}:={\textstyle\frac{1}{2}}\left(x\pm |x|\right)\]
the positive and negative parts, $[x]_+$ and $[x]_{-}$, of a real number $x$,
respectively.
As $\hpsiae^{n-1} \in \hat Z_2$,
we then have that $[z^{n-1}_{\epsilon,L}]_{-} =0$ a.e. on $\Omega$.
Taking $\varphi = [z^n_{\epsilon,L}]_{-}$ as a test function in  \eqref{psiGIz}, noting
that this is a legitimate choice since $[z^n_{\epsilon,L}]_{-} \in H^1(\Omega)$, decomposing
$z^m_{\epsilon,L}$, $m=n-1,n$, into their positive and negative parts,
and noting that $\utae^{n-1}$ is divergence-free on $\Omega$ and has zero trace
on $\partial \Omega$, we deduce that
\[
\|[z^n_{\epsilon,L}]_{-}\|^2 + \Delta t \, \varepsilon
\|\nabx [z^n_{\epsilon,L}]_{-}\|^2 = 0,
\]
where $\|\cdot\|$ denotes the $L^2(\Omega)$ norm.
Hence, $[z^n_{\epsilon,L}]_{-} = 0$ a.e. on $\Omega$. In other words,
$z^n_{\epsilon,L} \geq 0$ a.e. on
$\Omega$, which then gives that $\zeta^n_{\epsilon,L} \leq 1$ a.e. on $\Omega$,
i.e. $\hpsiae^n \in \hat Z_2$
as required.

As $(\utae^0,\hpsiae^0) \in \Vt \times \hat Z_2$, performing
the above existence proof at each time level $t_n$, $n=1,\ldots,N$,
yields a solution  $\{(\utae^n,\hpsiae^n)\}_{n=1}^N$  to (P$^{\Delta t}_{\epsilon,L}$).
\end{proof}

Having shown in Lemma \ref{conv} that  problem $({\rm P}^{\Delta t}_{\epsilon, L})$ has a solution, we shall next develop suitable estimates on this solution, independent of the cut-off parameter $L$.

\section{Entropy estimates}
\label{sec:entropy}
\setcounter{equation}{0}

Our starting point for the analysis here is the final result of the previous section, stated in
Lemma \ref{conv}, concerning the existence of a solution to the discrete-in-time problem  $({\rm P}^{\Delta t}_{\varepsilon,L})$. The model $({\rm P}^{\Delta t}_{\varepsilon,L})$ includes `microscopic cut-off' in the drag term of the Fokker--Planck equation,  where $L>1$ is a (fixed, but otherwise arbitrary,) cut-off parameter.
Our ultimate objective is to pass to the limits $L \rightarrow \infty$ and $\Delta t \rightarrow 0_+$ in the model $({\rm P}^{\Delta t}_{\varepsilon,L})$, with $L$ and $\Delta t$ linked by the condition $\Delta t = o(L^{-1})$,
as $L \rightarrow \infty$.
To that end, we need to develop various bounds on sequences of weak solutions of $({\rm P}^{\Delta t}_{\varepsilon,L})$ that are uniform in the cut-off parameter
$L$ and thus permit the extraction of weakly convergent subsequences, as $L \rightarrow \infty$, through the use of a weak-compactness argument. The derivation of such bounds, based on the use of the relative entropy associated with the Maxwellian $M$, is our
main task in this section.

Let us introduce the following definitions, in line with (\ref{fn}):
\begin{subequations}
\begin{equation}
\utaeD(\cdot,t):=\,\frac{t-t_{n-1}}{\Delta t}\,
\utae^n(\cdot)+
\frac{t_n-t}{\Delta t}\,\utae^{n-1}(\cdot),
\quad t\in [t_{n-1},t_n], \quad n=1,\dots,N, \label{ulin}
\end{equation}
and
\begin{equation}
\utaeDp(\cdot,t):=\ut^n(\cdot),\quad
\utaeDm(\cdot,t):=\ut^{n-1}(\cdot), 
\quad t\in(t_{n-1},t_n], \quad n=1,\dots,N. \label{upm}
\end{equation}
\end{subequations}
We shall adopt $\uta^{\Delta t (,\pm)}$ as a collective symbol for $\uta^{\Delta t}$, $\uta^{\Delta t,\pm}$. The corresponding notations $\psia^{\Delta t (,\pm)}$, $\psia^{\Delta t}$, and $\psia^{\Delta t,\pm}$ are defined analogously; recall \eqref{idatabd} and \eqref{inidata-1}.

We note for future reference that
\begin{equation}
\utaeD-\utae^{\Delta t,\pm}= (t-t_{n}^{\pm})
\,\frac{\partial \utaeD}{\partial t},
\quad t \in (t_{n-1},t_{n}), \quad n=1,\dots,N, \label{eqtime+}
\end{equation}
where $t_{n}^{+} := t_{n}$ and $t_{n}^{-} := t_{n-1}$, with an analogous relationship in the case of $\psia^{\Delta t}$.

Using the above notation, 
(\ref{Gequn},b) summed for $n=1, \dots,  N$ can be restated  in a form
that is reminiscent of a weak formulation of  (\ref{ns1a}--d):
Find $(\utaeDp(t),\hpsiae^{\Delta t,+}(t)) \in
\Vt \times (\hat X \cap \hat Z_2)$
such that
\begin{subequations}
\begin{align}
&\displaystyle\int_{0}^{T} \int_\Omega  \frac{\partial \utaeD}{\partial t}\cdot
\wt \dx \dt
\label{equncon}
\nonumber
\\
&
\hspace{0.5cm} + \int_{0}^T \int_{\Omega}
\left[ \left[ (\utaeDm \cdot \nabx) \utaeDp \right]\,\cdot\,\wt
+ \nu \,\nabxtt \utaeDp
:
\wnabtt \right] \dx \dt
\nonumber
\\
&\hspace{1cm}=\int_{0}^T
\left[ \langle \ft^{\Delta t,+}, \wt\rangle_V
- k\,\sum_{i=1}^K \int_{\Omega}
\Ctt_i(M\,\hpsiae^{\Delta t,+}): \nabxtt
\wt \dx \right] \dt
\nonumber
\\
& \hspace{3in}
\qquad \forall \wt \in L^1(0,T;\Vt ),
\\
%
%
&\int_{0}^T \int_{\Omega \times D}
M\,\frac{ \partial \hpsiae^{\Delta t}}{\partial t}\,
\hat \varphi \dq \dx \dt \label{eqpsincon}
\nonumber
\\
& \hspace{0.5cm} + \int_{0}^T \int_{\Omega \times D} M\,\left[
\epsilon\, \nabx \hpsiae^{\Delta
t,+} - \utae^{\Delta t,-}\,\hpsiae^{\Delta t,+} \right]\cdot\, \nabx
\hat \varphi
\,\dq \dx \dt
\nonumber \\
& \hspace{0.5cm} +
\frac{1}{2\,\lambda}
\int_{0}^T \int_{\Omega \times D} M\,\sum_{i=1}^K
 \,\sum_{j=1}^K A_{ij}\,\nabqj \hpsiae^{\Delta t,+}
\cdot\, \nabqi
\hat \varphi
\,\dq \dx \dt
\nonumber \\
&
\hspace{0.5cm}
-
\int_{0}^T \int_{\Omega \times D} M\,\sum_{i=1}^K
[\sigtt(\utae^{\Delta t,+})
\,\qt_i]\,
\beta^L(\hpsiae^{\Delta t,+}) \,\cdot\, \nabqi
\hat \varphi
\,\dq \dx \dt = 0
\nonumber \\
& \hspace{7cm}
\qquad \forall \hat \varphi \in L^1(0,T;\hat X);
\end{align}
\end{subequations}
subject to the initial condition
$\utaeD(\cdot,0)= \ut^0 \in \Vt$ and
$\hpsiae^{\Delta t}(\cdot,\cdot,0) = \beta^L(\hat \psi^0(\cdot,\cdot)) \in \hat Z_2$.
We emphasize that (\ref{equncon},b) 
is an equivalent restatement of problem (${\rm P}^{\Delta t}_{\epsilon,L}$), for which existence of a solution has been established (cf. Lemma \ref{conv}).

Similarly, with analogous notation for $\{\zeta_{\epsilon,L}^n\}_{n=0}^N$,
(\ref{psiGI}) summed for $n=1,\dots,N$ can be restated as follows:
Given $\ut_{\epsilon,L}^{\Delta t,-}(t) \in \Vt$ solving (\ref{equncon},b),
find $\zeta_{\epsilon,L}^{\Delta t,+}(t)\in {\cal K} := \{ \eta \in H^1(\Omega) : \eta \in [0,1] \mbox{ a.e.\ on } \Omega \}$
such that
\begin{align}
&\int_0^T \int_\Omega
\frac{\partial \zeta_{\epsilon,L}^{\Delta t}}{\partial t} \varphi \dx \dt
+ \int_0^T \int_{\Omega} \left[ \epsilon\, \nabx  \zeta_{\epsilon,L}^{\Delta t, +} -
\ut_{\epsilon,L}^{\Delta t,-}  \,\zeta_{\epsilon,L}^{\Delta t, +} \right]
\cdot \nabx \varphi \dx \dt =0
\nonumber \\
& \hspace{3.5in}
\qquad \forall \varphi \in L^1(0,T;H^1(\Omega)),
\label{zetacon}
\end{align}
subject to the initial condition $\zeta_{\epsilon,L}^{\Delta t}(\cdot,0) = \int_D M(\qt)
\beta^L(\hat \psi^0(\cdot,\qt)) \dq$; cf.\ (\ref{zetan}) and recall that $
\hat \psi_{\epsilon,L}^0 = \beta^L(\hat \psi^0)$.

Once again,
 on recalling (\ref{zetan}) and (\ref{psiGI}),
we have established the existence of a solution to
(\ref{zetacon}) and that
\begin{align}
\zeta_{\epsilon,L}^{\Delta t}(\xt,t) = \int_D M(\qt)\, \hat \psi_{\epsilon,L}^{\Delta t}
(\xt,\qt,t) \dq \qquad \mbox{for a.e. } (\xt,t) \in \Omega \times (0,T).
\label{zetancon}
\end{align}

In conjunction with $\beta^L$, defined by \eqref{betaLa},
we consider the following cut-off version $\mathcal{F}^L$
of the entropy function $\mathcal{F}\,: s \in \mathbb{R}_{\geq 0} \mapsto \mathcal{F}(s) = s (\log s - 1) + 1
\in \mathbb{R}_{\geq 0}$:
\begin{equation}\label{eq:FL}
\mathcal{F}^L(s):= \left\{\begin{array}{ll}
s(\log s - 1) + 1,   &  ~~0 \leq s \leq L,\\
\frac{s^2 - L^2}{2L} + s(\log L - 1) + 1,  &  ~~L \leq s.
\end{array} \right.
\end{equation}
Note that
\begin{equation}\label{eq:FL1}
(\mathcal{F}^L)'(s) = \left\{\begin{array}{ll}
\log s,   &  ~~0 < s \leq L,\\
\frac{s}{L} + \log L - 1,  &  ~~L \leq s,
\end{array} \right.
\end{equation}
and
\begin{equation}\label{eq:FL2}
(\mathcal{F}^L)''(s) = \left\{\begin{array}{ll}
\frac{1}{s},   &  ~~0 < s \leq L,\\
\frac{1}{L},  &  ~~L \leq s.
\end{array} \right.
\end{equation}
Hence,
\begin{equation}\label{eq:FL2a}
\beta^L(s) = \min(s,L) = [(\mathcal{F}^L)''(s)]^{-1},\quad s \in \mathbb{R}_{\geq 0},
\end{equation}
with the convention $1/\infty:=0$ when $s=0$, and
\begin{equation}\label{eq:FL2b}
(\mathcal{F}^L)''(s) \geq \mathcal{F}''(s) = s^{-1},\quad s \in \mathbb{R}_{> 0}.
\end{equation}
We shall also require the following inequality, relating $\mathcal{F}^L$ to $\mathcal{F}$:
\begin{equation}\label{eq:FL2c}
\mathcal{F}^L(s) \geq \mathcal{F}(s),\quad s \in \mathbb{R}_{\geq 0}.
\end{equation}
For $s > 1$, this follows from \eqref{eq:FL2b}, with $s$ replaced by a dummy variable $\sigma$, after integrating
twice over $\sigma \in [1,s]$, and noting that $(\mathcal{F}^L)'(1)= \mathcal{F}'(1)$ and $(\mathcal{F}^L)(1)=\mathcal{F}(1)$.
For $s \in [0,1]$, we have $\mathcal{F}^L(s) = \mathcal{F}(s)$ of course, by definition.

\subsection{$L$-independent bounds on the spatial derivatives}
\label{Lindep-space}

We are now ready to embark on the derivation of the required bounds, uniform in the cut-off parameter $L$,
on norms of $\uta^{\Delta t,+}$, $\psia^{\Delta t,+}$ and
$\zeta_{\epsilon,L}^{\Delta t, +}$. As far as $\uta^{\Delta t,+}$ is concerned, this is a relatively straightforward exercise. We select $\wt = \chi_{[0,t]}\,\uta^{\Delta t,+}$ as test function in \eqref{equncon}, with $t$ chosen as $t_n$, $n \in \{1,\dots,N\}$, and $\chi_{[0,t]}$ denoting the
characteristic function of the interval $[0,t]$. We then deduce, with $t = t_n$, that
\begin{eqnarray}
&&\|\uta^{\Delta t,+}(t)\|^2 + \frac{1}{\Delta t}\int_0^t  \|\uta^{\Delta t,+}(s) - \uta^{\Delta t,-}(s)\|^2 \dd s
+ \nu \int_0^t \|\nabxtt \uta^{\Delta t,+}(s)\|^2 \dd s \nonumber\\
&&\quad\leq \|\ut_0\|^2 + \frac{1}{\nu}\int_0^t\|\ft^{\Delta t,+}(s)\|^2_{V'} \dd s \nonumber\\
&& \quad \quad \hspace{0.3mm} -2k \int_0^t \int_{\Omega \times D} M(\qt)\,\sum_{i=1}^K \qt_i \qt_i^{\rm T} \,U_i'\big({\textstyle \frac{1}{2}}|\qt_i|^2\big)\,
\psia^{\Delta t,+} :
\nabxtt \uta^{\Delta t,+} \dq \dx \dd s,~~~~ \label{eq:energy-u}
\end{eqnarray}
where, again, $\|\cdot\|$ denotes the $L^2$ norm over $\Omega$
and we have noted (\ref{idatabd}).
We intentionally {\em did not} bound the final term on
the right-hand side of \eqref{eq:energy-u}.
As we shall see in what follows,
this simple trick will prove helpful: our bounds on $\psia^{\Delta t, +}$ below will furnish
an identical term with the opposite sign, so then by combining the bounds on $\uta^{\Delta t, +}$ and $\psia^{\Delta t, +}$
this pair of, otherwise dangerous, terms will be removed. This fortuitous cancellation reflects
the balance of total energy in the system.

Having dealt with $\uta^{\Delta t,+}$, we now embark on the less straightforward task of deriving bounds on
norms of $\psia^{\Delta t,+}$ that are uniform in the cut-off parameter $L$.
The appropriate choice of test function in \eqref{eqpsincon} for this purpose
is $\hat\varphi = \chi_{[0,t]}\,(\mathcal{F}^L)'(\psia^{\Delta t,+})$ with $t=t_n$, $n \in \{1,\dots,N\}$;
this can be seen by noting that with such a $\hat\varphi$, at
least formally, the final term on the left-hand side of
\eqref{eqpsincon} can be manipulated to become identical to
the final term in \eqref{eq:energy-u}, but with opposite sign; and this will then result in
the crucial cancellation of terms mentioned in the previous paragraph.
While Lemma \ref{conv} guarantees that $\psia^{\Delta t,+}(\cdot,\cdot,t)$ belongs to $\hat{Z}_2$ for all $t \in [0,T]$,
and is therefore nonnegative a.e. on $\Omega \times D \times [0,T]$, there is unfortunately
no reason why $\psia^{\Delta t,+}$ should be strictly positive on $\Omega\times D \times [0,T]$,
and therefore the
expression $(\mathcal{F}^L)'(\psia^{\Delta t,+})$ may in general
be undefined; the same is true of $(\mathcal{F}^L)''(\psia^{\Delta t,+})$,
which also appears in the algebraic manipulations. We shall circumvent this problem by working
with $(\mathcal{F}^L)'(\psia^{\Delta t,+} + \alpha)$ instead of $(\mathcal{F}^L)'(\psia^{\Delta t,+})$, where $\alpha>0$; since
$\psia^{\Delta t,+}$ is known to be nonnegative from Lemma \ref{conv}, $(\mathcal{F}^L)'(\psia^{\Delta t,+} + \alpha)$ and $(\mathcal{F}^L)''(\psia^{\Delta t,+} + \alpha)$
are well-defined. After deriving the relevant bounds, which will involve $\mathcal{F}^L(\psia^{\Delta t,+} + \alpha)$ only, we shall pass to the limit $\alpha \rightarrow 0_{+}$, noting that, unlike
$(\mathcal{F}^L)'(\psia^{\Delta t,+})$ and $(\mathcal{F}^L)''(\psia^{\Delta t,+})$, the function
$(\mathcal{F}^L)(\psia^{\Delta t,+})$ is well-defined for any nonnegative $\psia^{\Delta t,+}$.

Thus, we now take any $\alpha \in (0,1)$, whereby $0 < \alpha < 1 < L$, and we choose
\[\hat\varphi = \chi_{[0,t]}\,(\mathcal{F}^L)'(\psia^{\Delta t,+} + \alpha),
\qquad \mbox{with $t = t_n$, $\;n \in \{1,\dots,N\}$,}
\]
as test function in \eqref{eqpsincon}. As the calculations are quite involved, we shall, for the sake of clarity of exposition, manipulate the terms in \eqref{eqpsincon} one at a time and will then merge the resulting bounds on the individual terms with \eqref{equncon}
to obtain a single energy inequality for the pair $(\uta^{\Delta t,+},\psia^{\Delta t,+})$.

We start by considering the first term in \eqref{eqpsincon}. Clearly $\mathcal{F}^L(\cdot + \alpha)$ is twice continuously differentiable on the interval $(-\alpha,\infty)$ for any $\alpha>0$. Thus, by Taylor series expansion of $s \in [0,\infty) \mapsto \mathcal{F}^L(s +\alpha) \in [0,\infty)$ with remainder,
and $c \in [0,\infty)$,
\[ (s-c)\, (\mathcal{F}^L)'(s+\alpha) = \mathcal{F}^L(s+\alpha) - \mathcal{F}^L(c+\alpha) + \frac{1}{2}(s-c)^2\,(\mathcal{F}^L)''(\theta s + (1-\theta)c+\alpha),\]
with $\theta \in (0,1)$. Hence, on noting that $t\in [0,T]\mapsto\hat\psi^{\Delta t}_{\epsilon,L}(\cdot,\cdot,t)\in \hat{X}$ is
piecewise linear relative to the partition $\{0=t_0, t_1,\dots, t_N=T\}$ of the interval $[0,T]$,
\begin{eqnarray*}
{\rm T}_1&\!:=&\int_0^T \int_{\Omega \times D} M\,\frac{\partial \psia^{\Delta t}}{\partial s} \, \chi_{[0,t]}\,(\mathcal{F}^L)'(\psia^{\Delta t,+} + \alpha) \dq \dx \dd s \nonumber \\
&=&\int_0^t \int_{\Omega \times D} M \frac{\partial}{\partial s} (\psia^{\Delta t} + \alpha)\,  (\mathcal{F}^L)'(\psia^{\Delta t,+} + \alpha) \dq \dx \dd s
\nonumber \\
&=& \int_{\Omega \times D} M\mathcal{F}^L(\psia^{\Delta t,+}(t) + \alpha) \dq \dx - \int_{\Omega \times D} M\mathcal{F}^L(\beta^L(\hat\psi^0) + \alpha) \dq \dx
\nonumber \\
&&
+ \frac{1}{2 \Delta t}\int_0^t\int_{\Omega \times D}\!\!\!  M(\mathcal{F}^L)''(\theta\psia^{\Delta t,+}
+ (1-\theta)\psia^{\Delta t, -} + \alpha)\,(\psia^{\Delta t,+} - \psia^{\Delta t,-})^2 \dq \dx \dd s\nonumber.
\end{eqnarray*}
Noting from \eqref{eq:FL2} that $(\mathcal{F}^L)''(s + \alpha) \geq 1/L$ for all $s \in [0,\infty)$ and all $\alpha>0$, this then implies, with $t=t_n$, $n \in \{1,\dots,N\}$, that
\begin{eqnarray}\label{eq:energy-psi1}
{\rm T}_1&\geq &\int_{\Omega \times D} M\mathcal{F}^L(\psia^{\Delta t,+}(t) + \alpha) \dq \dx - \int_{\Omega \times D} M\mathcal{F}^L(\beta^L(\hat\psi^0) + \alpha) \dq \dx
\nonumber \\
&&+ \frac{1}{2 \Delta t\,L }\int_0^t\int_{\Omega \times D} M(\psia^{\Delta t,+} - \psia^{\Delta t,-})^2 \dq \dx \dd s.
\end{eqnarray}
The denominator in the prefactor of the last integral motivates us to link $\Delta t$ to $L$ so that 
\[\Delta t\, L = o(1)\qquad \mbox{as $\Delta t \rightarrow 0$}\] 
(or, equivalently, $\Delta t = o(L^{-1})$ as $L \rightarrow \infty$), in order to drive the integral multiplied by 
the prefactor to $0$ in the limit of $L \rightarrow \infty$, once the product of the two has been bounded above by a constant, independent of $L$.

Next we consider the second term in \eqref{eqpsincon}, using repeatedly that
$\nabx \cdot \uta^{\Delta t,-} = 0$ and that $\uta^{\Delta t,-}$ has zero trace on $\partial\Omega$:
\begin{eqnarray*}
&&\hspace{-2mm}{\rm T}_2:=\int_{0}^T \int_{\Omega \times D}\!\! M\!\left[ \epsilon\,\nabx \psia^{\Delta t,+} - \uta^{\Delta t,-}\,\psia^{\Delta t,+}\right]\cdot\, \nabx \chi_{[0,t]}\,(\mathcal{F}^L)'(\psia^{\Delta t,+} + \alpha)  \,\dq \dx \dd s
\nonumber\\
&&\qquad\!\hspace{-2.1mm} = \varepsilon \int_0^t \int_{\Omega \times D} M \nabx (\psia^{\Delta t,+} + \alpha) \cdot \nabx (\mathcal{F}^L)'(\psia^{\Delta t,+} + \alpha) \dq \dx \dd s\\
&&\qquad\quad - \int_0^t  \int_{\Omega \times D} M \uta^{\Delta t,-} (\psia^{\Delta t,+} + \alpha) \cdot \nabx (\mathcal F^L)'(\psia^{\Delta t,+} + \alpha) \dq \dx \dd s,
\end{eqnarray*}
where in the last line we added $0$ in the form of
\[ \alpha \int_0^t \int_{\Omega \times D} M \uta^{\Delta t,-} \cdot \nabx (\mathcal{F}^L)'(\psia^{\Delta t,+} + \alpha) \dq \dx \dd s = 0.\]
Hence, similarly to (\ref{convGdL}),
\begin{eqnarray*}
&&\hspace{-2mm}{\rm T}_2:= \varepsilon \int_0^t\!\! \int_{\Omega \times D} M (\mathcal{F}^L)''(\psia^{\Delta t,+} + \alpha)
|\nabx (\psia^{\Delta t,+} + \alpha)|^2 \dq \dx \dd s\nonumber\\
&&\hspace{-2mm} - \int_0^t\!\!  \int_{\Omega \times D} M \uta^{\Delta t,-} (\psia^{\Delta t,+} + \alpha) \cdot [(\mathcal F^L)''(\psia^{\Delta t,+} + \alpha)\nabx(\psia^{\Delta t,+} + \alpha)] \dq \dx \dd s\nonumber\\
&&\qquad\! = \varepsilon \int_0^t\!\! \int_{\Omega \times D} M (\mathcal{F}^L)''(\psia^{\Delta t,+} + \alpha)
|\nabx \psia^{\Delta t,+} |^2 \dq \dx \dd s\nonumber\\
&&\hspace{-2mm} - \int_0^t\!\!  \int_{\Omega \times D}\!\! M \uta^{\Delta t,-} (\psia^{\Delta t,+} + \alpha) \cdot \nabx \psia^{\Delta t,+}\nonumber\\
&&\qquad \hspace{3cm}
\times \left\{\begin{array}{ll} 1/(\psia^{\Delta t,+} + \alpha) &  \mbox{if $\psia^{\Delta t,+} + \alpha \leq L$} \\
                         {1}/{L}              &  \mbox{if $\psia^{\Delta t,+} + \alpha \geq L$}
                        \end{array}\! \right\} \dq \dx \dd s\nonumber\\
&&\qquad \! = \varepsilon \int_0^t\!\! \int_{\Omega \times D} M (\mathcal{F}^L)''(\psia^{\Delta t,+} + \alpha)
|\nabx \psia^{\Delta t,+} |^2 \dq \dx \dd s\nonumber\\
&&\hspace{-2mm}- \int_0^t\!\!  \int_{\Omega \times D}\! M \uta^{\Delta t,-} \cdot \nabx [G^L(\psia^{\Delta t,+} + \alpha)] \dq \dx \dd s,
\end{eqnarray*}
where $G^L$ denotes the (locally Lipschitz continuous) function defined on $\mathbb{R}$ by
\[
G^L(s):= \left\{\begin{array}{ll} s - \frac{L}{2} &  \mbox{~~if $s \leq L$},\\
                         \frac{1}{2L}s^2  &  \mbox{~~if $s \geq L$}.
                        \end{array} \right.
\]
%
On noting that the integral involving $G^L$ vanishes, \eqref{eq:FL2b} then yields the lower bound
\begin{eqnarray}\label{eq:energy-psi2}
&&{\rm T}_2 \geq  \varepsilon \int_0^t \int_{\Omega \times D} M (\psia^{\Delta t,+} + \alpha)^{-1}
|\nabx (\psia^{\Delta t,+} + \alpha) |^2 \dq \dx \dd s.
\end{eqnarray}

Next, we consider the third term in \eqref{eqpsincon}. Thanks to \eqref{A} we have, again with $t=t_n$
and $n \in \{1,\dots, N\}$:
\begin{eqnarray}\label{eq:energy-psi3}
&&\hspace{-2.3mm}{\rm T}_3 := \frac{1}{2\,\lambda} \int_{0}^T \int_{\Omega \times D}M\, \sum_{i=1}^K \sum_{j=1}^K A_{ij} \,\nabqj \psia^{\Delta t,+}
\cdot\, \nabqi \chi_{[0,t]}(\mathcal{F}^L)'(\psia^{\Delta t,+} + \alpha)  \,\dq \dx \dd s\nonumber\\
&&\quad = \frac{1}{2\,\lambda} \int_0^t \int_{\Omega \times D}M\,(\mathcal{F}^L)''(\psia^{\Delta t,+} + \alpha)\,\sum_{i=1}^K \sum_{j=1}^K A_{ij} \,\nabqj \psia^{\Delta t,+} \cdot\, \nabqi \psia^{\Delta t,+} \dq \dx \dd s\nonumber\\
&&\quad \geq  \frac{a_0}{2\,\lambda}  \int_0^t \int_{\Omega \times D}M\,(\mathcal{F}^L)''(\psia^{\Delta t,+} + \alpha)\,\sum_{i=1}^K |\nabqi \psia^{\Delta t,+} |^2 \,\dq \dx \dd s\nonumber\\
&&\quad =  \frac{a_0}{2\,\lambda}  \int_0^t \int_{\Omega \times D}M\,(\mathcal{F}^L)''(\psia^{\Delta t,+} + \alpha)\,|\nabq \psia^{\Delta t,+} |^2 \,\dq \dx \dd s.
\end{eqnarray}
We emphasize here that, unlike \eqref{eq:energy-psi2} above, in \eqref{eq:energy-psi3} we refrained
from using \eqref{eq:FL2b} to further bound $(\mathcal{F}^L)''(\psia^{\Delta t,+} + \alpha)$ from below by $\mathcal{F}''(\psia^{\Delta t,+} + \alpha)= (\psia^{\Delta t,+} + \alpha)^{-1}$. Performing this additional lower bound
will be postponed until later, after a term similar
to ${\rm T}_3$ that arises as a byproduct of manipulating term ${\rm T}_4$ below has been absorbed in
term ${\rm T}_3$.

We now consider the final term in \eqref{eqpsincon}, with $t=t_n$, $n \in \{1,\dots, N\}$:
\begin{align}
{\rm T}_4&:= - \int_{0}^T\!\! \int_{\Omega \times D}M\, \sum_{i=1}^K [\, \sigtt(\uta^{\Delta t,+}) \,\qt_i\,]\,
\beta^L(\psia^{\Delta t,+})
\,\cdot\, \nabqi \chi_{[0,t]}\,({\mathcal F}^L)'(\psia^{\Delta t,+} + \alpha) \,\dq \dx \dd s\nonumber\\
&= - \int_0^t\!\! \int_{\Omega \times D} M\, \sum_{i=1}^K [\,(\nabxtt \uta^{\Delta t,+})\,
\qt_i\,]\, 
\beta^L(\psia^{\Delta t,+})
\, \cdot\, (\mathcal{F}_L)''(\psia^{\Delta t,+} + \alpha)\, \nabqi \psia^{\Delta t,+}\, \dq \dx \dd s\nonumber \\
&= - \int_0^t\!\! \int_{\Omega \times D} M\, \sum_{i=1}^K [\,(\nabxtt \uta^{\Delta t,+})\,\qt_i\,] \,
\frac{\beta^L(\psia^{\Delta t,+})}{\beta^L(\psia^{\Delta t,+} + \alpha)}\cdot \nabqi \psia^{\Delta t,+}\, \dq \dx \dd s\nonumber\\
&= - \int_0^t\!\! \int_{\Omega \times D} M\, \sum_{i=1}^K [\,(\nabxtt \uta^{\Delta t,+})\,\qt_i\,]\,
\cdot \nabqi \psia^{\Delta t,+}\, \dq \dx \dd s\nonumber\\
&\qquad\, + \int_0^t\!\! \int_{\Omega \times D} M\, \sum_{i=1}^K [\,(\nabxtt \uta^{\Delta t,+})\,\qt_i\,] \left[1 -\frac{\beta^L(\psia^{\Delta t,+})}{\beta^L(\psia^{\Delta t,+} + \alpha)}\right] 
\cdot \nabqi \psia^{\Delta t,+}\, \dq \dx \dd s\nonumber\\
&= - \int_0^t\!\! \int_{\Omega \times D} M\,\sum_{i=1}^K \qt_i\,\qt_i^{\rm T}\,U_i'(\textstyle{\frac{1}{2}|\qt|^2})\,
\psia^{\Delta t,+} : \nabxtt \uta^{\Delta t,+}  \dq \dx \dd s\nonumber\\
&\qquad\, + \int_0^t\!\! \int_{\Omega \times D} M\, \sum_{i=1}^K [\,(\nabxtt \uta^{\Delta t,+})\,\qt_i\,] \left[1 -\frac{\beta^L(\psia^{\Delta t,+})}{\beta^L(\psia^{\Delta t,+} + \alpha)}\right] 
\cdot  \nabqi \psia^{\Delta t,+}\, \dq \dx \dd s,
\label{eq:energy-psi4}
\end{align}
where in the transition to the final equality we applied \eqref{intbyparts} with $B:= \nabxtt \ut^{\Delta t,+}_{\epsilon,L}$ (on account of it being independent of the variable $\qt$), together with the fact that
\[ \mathfrak{tr}\,(\nabxtt \uta^{\Delta t,+}) = \nabx \cdot \uta^{\Delta t,+} = 0.\]

Summing \eqref{eq:energy-psi1}, \eqref{eq:energy-psi2},
\eqref{eq:energy-psi3} and \eqref{eq:energy-psi4} yields,
with $t=t_n$ and $n \in \{1,\dots, N\}$, the following inequality:
\begin{align}
&\int_{\Omega \times D} M \,\mathcal{F}^L(\psia^{\Delta t,+}(t) + \alpha) \dq \dx
+\,\frac{1}{2 \Delta t\,L}\int_0^t\int_{\Omega \times D}  M\,(\psia^{\Delta t,+} - \psia^{\Delta t,-})^2 \dq \dx \dd s \nonumber\\
&\qquad +\, \varepsilon \int_0^t \int_{\Omega \times D} M\,
\frac{|\nabx \psia^{\Delta t,+}|^2}{\psia^{\Delta t,+} + \alpha} \dq \dx \dd s
\nonumber \\
& \qquad
+\,\frac{a_0}{2\,\lambda}  \int_0^t \int_{\Omega \times D}M\,(\mathcal{F}^L)''(\psia^{\Delta t,+} + \alpha)\,|\nabq \psia^{\Delta t,+} |^2 \,\dq \dx \dd s\nonumber
\end{align}
\begin{align}
&\quad \leq \int_{\Omega \times D} M \,\mathcal{F}^L(\beta^L(\hat\psi^0) + \alpha) \dq \dx
\nonumber\\
&\quad \qquad + \int_0^t \int_{\Omega \times D} M\,\sum_{i=1}^K \qt_i\,\qt_i^{\rm T}\,U_i'(\textstyle{\frac{1}{2}|\qt|^2})\,
\psia^{\Delta t,+} : \nabxtt \uta^{\Delta t,+}  \dq \dx \dd s\nonumber\\
&\quad \qquad - \int_0^t \int_{\Omega \times D} M\, \sum_{i=1}^K [\,(\nabxtt \uta^{\Delta t,+})\,\qt_i\,] \left[1 -\frac{\beta^L(\psia^{\Delta t,+})}{\beta^L(\psia^{\Delta t,+} + \alpha)}\right] 
\cdot \nabqi \psia^{\Delta t,+}\, \dq \dx \dd s.
\label{eq:energy-psi-summ1}
\end{align}
Comparing \eqref{eq:energy-psi-summ1} with \eqref{eq:energy-u} we see that after multiplying \eqref{eq:energy-psi-summ1} by $2k$ and adding the resulting inequality to \eqref{eq:energy-u}
the final term in \eqref{eq:energy-u} is cancelled by $2k$ times the
second term on the right-hand side of \eqref{eq:energy-psi-summ1}. Hence, for any $t=t_n$, with
$n \in \{1,\dots,N\}$, we deduce that
\begin{align}
&\|\uta^{\Delta t, +}(t)\|^2 + \frac{1}{\Delta t} \int_0^t \|\uta^{\Delta t, +} - \uta^{\Delta t,-}\|^2
\dd s + \nu \int_0^t \|\nabxtt \uta^{\Delta t, +}(s)\|^2 \dd s\nonumber\\
&\qquad +\, 2k\int_{\Omega \times D}\!\! M \mathcal{F}^L(\psia^{\Delta t, +}(t) + \alpha) \dq \dx + \frac{k}{\Delta t\,L}
\int_0^t \int_{\Omega \times D}\!\! M (\psia^{\Delta t, +} - \psia^{\Delta t, -})^2 \dq \dx \dd s
\nonumber \\
&\qquad \qquad +\, 2k\,\varepsilon \int_0^t \int_{\Omega \times D} M
\frac{|\nabx \psia^{\Delta t, +} |^2}{\psia^{\Delta t, +} + \alpha} \dq \dx \dd s\nonumber\\
&\qquad \qquad \qquad + \frac{a_0 k}{\lambda}  \int_0^t \int_{\Omega \times D}M\,(\mathcal{F}^L)''(\psia^{\Delta t, +} + \alpha)\,|\nabq \psia^{\Delta t, +} |^2 \,\dq \dx \dd s\nonumber\\
&\quad \leq \|\ut_0\|^2 + \frac{1}{\nu}\int_0^t\|\ft^{\Delta t,+}(s)\|^2_{V'} \dd s + 2k \int_{\Omega \times D} M \mathcal{F}^L(\beta^L(\hat\psi^0) + \alpha) \dq \dx
\nonumber\\
&\quad \qquad -\, 2k \int_0^t \int_{\Omega \times D} M\, \sum_{i=1}^K [\,(\nabxtt \uta^{\Delta t, +})\,\qt_i\,] \left[1 -\frac{\beta^L(\psia^{\Delta t, +})}{\beta^L(\psia^{\Delta t, +} + \alpha)}\right] 
\cdot \nabqi \psia^{\Delta t, +}\, \dq \dx \dd s.
\label{eq:energy-u+psi}
\end{align}
It remains to bound the fourth term on the right-hand side of \eqref{eq:energy-u+psi}.
Noting that
$\beta^L$ is Lipschitz continuous, with Lipschitz constant equal to $1$, and $\beta^L(s+\alpha) \geq \alpha$
for $s \geq 0$ (recall that $0<\alpha < 1 < L$), we have that
\begin{eqnarray}
0 &\leq& \left(1 - \frac{\beta^L(\psia^{\Delta t, +})}{\beta^L(\psia^{\Delta t, +} + \alpha)}\right) \frac{1}{\sqrt{({\mathcal F}^L)''(\psi^{\Delta t, +} + \alpha)}} = \frac{\beta^L(\psia^{\Delta t, +} + \alpha) - \beta^L(\psia^{\Delta t, +})}{\sqrt{\beta^L(\psia^{\Delta t, +} + \alpha)}}
\nonumber
\\
&\leq & \frac{\beta^L(\psia^{\Delta t, +} + \alpha) - \beta^L(\psia^{\Delta t, +})}{\sqrt{\alpha}} \nonumber\\
&\leq &
\left\{\begin{array}{cl}
\sqrt{\alpha}   &    \,\mbox{when $\psia^{\Delta t, +} \leq L$},\\
0               &    \,\mbox{when $\psia^{\Delta t, +} \geq L$}.
\end{array} \right.
\label{betaLbd}
\end{eqnarray}
With this bound and (\ref{growth3}),
for $t=t_n$, $n \in \{1,\dots,N\}$, we then have that
%
\begin{align}
&\left|- 2k \int_0^t \int_{\Omega \times D} M\, \sum_{i=1}^K [\,(\nabxtt \uta^{\Delta t, +})\,\qt_i\,] \left[1 -\frac{\beta^L(\psia^{\Delta t, +})}{\beta^L(\psia^{\Delta t, +} + \alpha)}\right] 
\cdot \nabqi \psia^{\Delta t, +}\, \dq \dx \dd s\right|\nonumber\\
&\qquad \leq 2k \int_0^t \int_{\Omega \times D} M\,  |\nabxtt \uta^{\Delta t, +}|\,|\qt|\, \left[1 -\frac{\beta^L(\psia^{\Delta t, +})}{\beta^L(\psia^{\Delta t, +} + \alpha)}\right]\, 
|\nabq \psia^{\Delta t, +}|\, \dq \dx \dd s \nonumber \\
&\qquad \leq 2k \sqrt{\alpha} \int_0^t \int_{\Omega \times D} M\,  |\nabxtt \uta^{\Delta t, +}|\, |\qt|\,
\sqrt{(\mathcal{F}^L)''(\psia^{\Delta t, +} + \alpha)} 
\; |\nabq \psia^{\Delta t, +}|\, \dq \dx \dd s \nonumber\\
&\qquad = 2k \sqrt{\alpha} \int_0^t \int_{\Omega}  |\nabxtt \uta^{\Delta t, +}| \left(\int_D M 
\,|\qt|
\sqrt{(\mathcal{F}^L)''(\psia^{\Delta t, +} + \alpha)}
\; |\nabq \psia^{\Delta t, +}|\, \dq \right)\dx \dd s\nonumber
\end{align}
\begin{align}
&\qquad \leq 2k \sqrt{\alpha}
\left(\int_D M\,
|\qt|^2 \dq\right)^{\frac{1}{2}}
\nonumber\\
&\qquad \qquad 
\times
\left(\int_0^t \int_{\Omega}  |\nabxtt \uta^{\Delta t, +}|
\left(\int_D M
(\mathcal{F}^L)''(\psia^{\Delta t, +} + \alpha)
\, |\nabq \psia^{\Delta t, +}|^2\, \dq \right)^{\frac{1}{2}}
\dx \dd s \right) \nonumber\\
&\qquad \leq 2k \sqrt{\alpha\,C_M} \left(\int_0^t \|\nabxtt \uta^{\Delta t, +}\|^2 \dd s\right)^{\frac{1}{2}}
\!\left(\int_0^t \int_{\Omega \times D} M (\mathcal{F}^L)''(\psia^{\Delta t, +} + \alpha)
\, |\nabq \psia^{\Delta t, +}|^2\, \dq \dx \dd s\right)^{\frac{1}{2}}\nonumber\\
&\qquad\leq \frac{a_0\, k}{2 \lambda} \left(\int_0^t \int_{\Omega \times D} M (\mathcal{F}^L)''(\psia^{\Delta t, +} + \alpha)
\, |\nabq \psia^{\Delta t, +}|^2\, \dq \dx \dd s\right)\nonumber\\
&\qquad \qquad +\alpha\, \frac{2\lambda\,k\,C_M}{a_0}
\,
\left(\int_0^t \|\nabxtt \uta^{\Delta t, +}\|^2 \dd s\right),
\label{eq:lastterm}
\end{align}
where
\begin{align}
C_M := \int_D M\,
|\qt|^2  \dq.
\label{CM}
\end{align}
Substitution of \eqref{eq:lastterm} into \eqref{eq:energy-u+psi} and using \eqref{eq:FL2b} to further bound $(\mathcal{F}^L)''(\psia^{\Delta t, +} + \alpha)$ from below by $\mathcal{F}''(\psia^{\Delta t, +} + \alpha)= (\psia^{\Delta t, +} + \alpha)^{-1}$ and
\eqref{eq:FL2c} to bound $\mathcal{F}^L(\psia^{\Delta t, +} + \alpha)$ by $\mathcal{F}(\psia^{\Delta t, +}+\alpha)$ from below finally yield, for all $t=t_n$, $n \in \{1,\dots, N\}$, that
\begin{eqnarray}\label{eq:energy-u+psi1}
&&\hspace{-2mm}\|\uta^{\Delta t, +}(t)\|^2 + \frac{1}{\Delta t} \int_0^t \|\uta^{\Delta t, +} - \uta^{\Delta t,-}\|^2
\dd s + \nu \int_0^t \|\nabxtt \uta^{\Delta t, +}(s)\|^2 \dd s\nonumber\\
&&\hspace{-2mm}+ \,2k\int_{\Omega \times D}\!\! M \mathcal{F}(\psia^{\Delta t, +}(t) + \alpha) \dq \dx + \frac{k}{\Delta t\,L}
\int_0^t \int_{\Omega \times D}\!\! M (\psia^{\Delta t, +} - \psia^{\Delta t, -})^2 \dq \dx \dd s
\nonumber \\
&&\hspace{-2mm}\quad +\, 2k\,\varepsilon \int_0^t \int_{\Omega \times D} M
\frac{|\nabx \psia^{\Delta t, +} |^2}{\psia^{\Delta t, +} + \alpha} \dq \dx \dd s
+\, \frac{a_0 k}{2\,\lambda}  \int_0^t \int_{\Omega \times D}M\, \,\frac{|\nabq \psia^{\Delta t, +}|^2}{\psia^{\Delta t, +} + \alpha} \,\dq \dx \dd s\nonumber\\
&&\hspace{-2mm}\leq \|\ut_0\|^2 + \frac{1}{\nu}\int_0^t\|\ft^{\Delta t,+}(s)\|^2_{V'} \dd s + 2k \int_{\Omega \times D} M \mathcal{F}^L(\beta^L(\hat\psi^0) + \alpha) \dq \dx\nonumber\\
&&\hspace{-2mm}\quad\quad +\, \alpha\,\frac{2\lambda\,k\,C_M}{a_0} \int_0^t \|\nabxtt \uta^{\Delta t, +}(s)\|^2 \dd s.
\end{eqnarray}
%

The only restriction we have imposed on $\alpha$ so far is that it belongs to the open interval $(0,1)$; let us now restrict the range of $\alpha$ further by demanding that, in fact,
\begin{equation}\label{alphacond}
0 < \alpha < \min \left(1 , \frac{a_0\,\nu}
{2\,\lambda\,k\,C_M} 
\right).
\end{equation}
Then, 
the last term 
on the right-hand side of
\eqref{eq:energy-u+psi1} can be absorbed into the third term
on the left-hand side, giving, for $t=t_n$ and $n \in \{1,\dots, N\}$,
\begin{align}
&\|\uta^{\Delta t, +}(t)\|^2 + \frac{1}{\Delta t} \int_0^t
\|\uta^{\Delta t, +} - \uta^{\Delta t,-}\|^2
\dd s
\nonumber \\
& \qquad + \left(\nu - \alpha\,
\frac{2\lambda\,k\,C_M 
}
{a_0}
\right) \int_0^t \|\nabxtt \uta^{\Delta t, +}(s)\|^2 \dd s\nonumber\\
&\qquad + \,2k\int_{\Omega \times D}\!\! M \mathcal{F}(\psia^{\Delta t, +}(t) + \alpha) \dq \dx + \frac{k}{\Delta t\,L}
\int_0^t \int_{\Omega \times D}\!\! M (\psia^{\Delta t, +} - \psia^{\Delta t, -})^2 \dq \dx \dd s
\nonumber \\
&\qquad + 2k\,\varepsilon \int_0^t \int_{\Omega \times D} M
\frac{|\nabx \psia^{\Delta t, +} |^2}{\psia^{\Delta t, +} + \alpha} \dq \dx \dd s
+\, \frac{a_0 k}{2\,\lambda}  \int_0^t \int_{\Omega \times D}M\, \,\frac{|\nabq \psia^{\Delta t, +}|^2}{\psia^{\Delta t, +} + \alpha} \,\dq \dx \dd s\nonumber\\
&\quad \leq 
\|\ut_0\|^2 + \frac{1}{\nu}\int_0^t\|\ft^{\Delta t,+}(s)\|^2_{V'} \dd s + 2k \int_{\Omega \times D} M \mathcal{F}^L(\beta^L(\hat\psi^0) + \alpha) \dq \dx.
\label{eq:energy-u+psi2}
\end{align}

Let us now focus our attention on the final integral on the right-hand side of \eqref{eq:energy-u+psi2}, which we label $T_5(\alpha)$ and express as follows:
\begin{eqnarray*}
{\rm T}_5(\alpha) &:=&\int_{\Omega \times D}\!\! M \mathcal{F}^L(\beta^L(\hat\psi^0) + \alpha) \dq \dx
\\
&=& \int_{\mathfrak{A}_{L,\alpha}}\!\! M \mathcal{F}^L(\beta^L(\hat\psi^0) + \alpha) \dq \dx +
\int_{\mathfrak{B}_{L,\alpha}}\!\! M \mathcal{F}^L(\beta^L(\hat\psi^0) + \alpha) \dq \dx,
\end{eqnarray*}
where
\begin{eqnarray*}
\mathfrak{A}_{L,\alpha}&:=& \{(\xt,\qt) \in \Omega \times D\,:\, 0 \leq \beta^L(\hat\psi^0(\xt,\qt)) \leq L - \alpha\},\\
\mathfrak{B}_{L,\alpha}&:=& \{(\xt,\qt) \in \Omega \times D\,:\, L-\alpha < \beta^L(\hat\psi^0(\xt,\qt)) \leq L\}.
\end{eqnarray*}
We begin by noting that
\[ \int_{\mathfrak{A}_{L,\alpha}} M \mathcal{F}^L(\beta^L(\hat\psi^0) + \alpha) \dq \dx = \int_{\mathfrak{A}_{L,\alpha}} M \mathcal{F}(\beta^L(\hat\psi^0) + \alpha) \dq \dx.\]
For the integral over $\mathfrak{B}_{L,\alpha}$ we have
\begin{eqnarray*}
&&\int_{\mathfrak{B}_{L,\alpha}} M \mathcal{F}^L(\beta^L(\hat\psi^0) + \alpha) \dq \dx\\
&&= \int_{\mathfrak{B}_{L,\alpha}} M \left[\frac{(\beta^L(\hat\psi^0) + \alpha)^2 - L^2}{2L} +
(\beta^L(\hat\psi^0) + \alpha)(\log L - 1) + 1 \right] \dq \dx\\
&&\leq  \int_{\mathfrak{B}_{L,\alpha}} M \left[\frac{(L + \alpha)^2 - L^2}{2L} +
(\beta^L(\hat\psi^0) + \alpha)(\log (\beta^L(\hat\psi^0) +\alpha) - 1) + 1 \right] \dq \dx\\
&&= \alpha\left(1 + \frac{\alpha}{2L}\right) \int_{\mathfrak{B}_{L,\alpha}} M \dq \dx +
\int_{\mathfrak{B}_{L,\alpha}} M \mathcal{F}(\beta^L(\hat\psi^0) + \alpha) \dq \dx\\
&& \leq \frac{3}{2}\alpha |\Omega| + \int_{\mathfrak{B}_{L,\alpha}} M \mathcal{F}(\beta^L(\hat\psi^0) + \alpha) \dq \dx.
\end{eqnarray*}
Thus we have shown that
\begin{equation}\label{before-two}
{\rm T}_5(\alpha) \leq \frac{3}{2}\alpha |\Omega| + \int_{\Omega \times D} M \mathcal{F}(\beta^L(\hat\psi^0) + \alpha) \dq \dx.
\end{equation}
Now, there are two possibilities:
\begin{itemize}
\item[\textit{Case 1.}] If $\beta^L(\hat\psi^0) + \alpha \leq 1$, then $0 \leq \beta^L(\hat\psi^0) \leq 1 - \alpha$. Since $L>1$ it follows that $0 \leq \beta^L(s) \leq 1$ if, and only if, $\beta^L(s)=s$. Thus we deduce that in this case $\beta^L(\hat\psi^0) = \hat\psi^0$, and therefore $0 \leq \mathcal{F}(\beta^L(\hat\psi^0) + \alpha) = \mathcal{F}(\hat\psi^0 + \alpha)$.
\item[\textit{Case 2.}] Alternatively, if $1<\beta^L(\hat\psi^0) + \alpha$, then, on noting that
$\beta^L(s) \leq s$ for all $s \in [0,\infty)$, it follows that $1 < \beta^L(\hat\psi^0) + \alpha \leq
\hat\psi^0 + \alpha$. However the function $\mathcal{F}$ is strictly monotonic increasing on the interval $[1,\infty)$, which then implies that $0 = \mathcal{F}(1) < \mathcal{F}(\beta^L(\hat\psi^0) + \alpha) \leq \mathcal{F}(\hat\psi^0 + \alpha)$.
\end{itemize}
The conclusion we draw is that, either way, 
\[0 \leq \mathcal{F}(\beta^L(\hat\psi^0) + \alpha) \leq \mathcal{F}(\hat\psi^0 + \alpha).\] 
Hence,
\begin{equation}\label{bound-on-t5}
{\rm T}_5(\alpha) \leq \frac{3}{2}\alpha |\Omega| + \int_{\Omega \times D} M \mathcal{F}(\hat\psi^0 + \alpha) \dq \dx.
\end{equation}
Substituting \eqref{bound-on-t5} into \eqref{eq:energy-u+psi2} thus yields, for $t=t_n$ and $n \in \{1,\dots, N\}$,
\begin{align}
&\|\uta^{\Delta t, +}(t)\|^2 + \frac{1}{\Delta t} \int_0^t
\|\uta^{\Delta t, +} - \uta^{\Delta t,-}\|^2
\dd s
\nonumber \\
& \qquad + \left(\nu - \alpha\,
\frac{2\lambda\,k\,C_M 
}
{a_0}
\right) \int_0^t \|\nabxtt \uta^{\Delta t, +}(s)\|^2 \dd s\nonumber\\
&\qquad + \,2k\int_{\Omega \times D}\!\! M \mathcal{F}(\psia^{\Delta t, +}(t) + \alpha) \dq \dx + \frac{k}{\Delta t\,L}
\int_0^t \int_{\Omega \times D}\!\! M (\psia^{\Delta t, +} - \psia^{\Delta t, -})^2 \dq \dx \dd s
\nonumber \\
&\qquad + 2k\,\varepsilon \int_0^t \int_{\Omega \times D} M
\frac{|\nabx \psia^{\Delta t, +} |^2}{\psia^{\Delta t, +} + \alpha} \dq \dx \dd s
+\, \frac{a_0 k}{2\,\lambda}  \int_0^t \int_{\Omega \times D}M\, \,\frac{|\nabq \psia^{\Delta t, +}|^2}{\psia^{\Delta t, +} + \alpha} \,\dq \dx \dd s\nonumber\\
&\quad \leq 
\|\ut_0\|^2 + \frac{1}{\nu}\int_0^t\|\ft^{\Delta t,+}(s)\|^2_{V'} \dd s
+3\alpha\, k\, |\Omega|
+ 2k \int_{\Omega \times D} M \mathcal{F}(\hat\psi^0 + \alpha) \dq \dx.
\label{eq:energy-u+psi3}
\end{align}
%
The key observation at this point is that the right-hand side of \eqref{eq:energy-u+psi3} is
completely independent of the cut-off parameter $L$.

We shall tidy up the bound \eqref{eq:energy-u+psi3} by passing to the limit $\alpha \rightarrow 0_+$.
The first $\alpha$-dependent term on the right-hand side of \eqref{eq:energy-u+psi3} trivially converges to $0$ as $\alpha \rightarrow 0_+$; concerning the second $\alpha$-dependent term,
Lebesgue's dominated convergence theorem implies that
\[ \lim_{\alpha \rightarrow 0_+} \int_{\Omega \times D} M \mathcal{F}(\hat\psi^0 + \alpha) \dq \dx = \int_{\Omega \times D} M \mathcal{F}(\hat\psi^0) \dq \dx.
\]
Similarly, we can easily pass to the limit on the left-hand side of \eqref{eq:energy-u+psi3}.
By applying Fatou's lemma to the fourth, sixth and seventh term on the left-hand side of \eqref{eq:energy-u+psi3} we get, for $t = t_n$, $n \in \{1,\dots,N\}$, that
\begin{eqnarray*}
&&\hspace{-6.15cm} \mbox{lim inf}_{\alpha \rightarrow 0_+}\int_{\Omega \times D} M \mathcal{F}(\psia^{\Delta t, +} (t) + \alpha) \geq \int_{\Omega \times D} M \mathcal{F}(\psia^{\Delta t, +} (t)) \dq \dx,
\\~\\~\\
\mbox{lim inf}_{\alpha \rightarrow 0_+} \int_0^t\! \int_{\Omega \times D}\!\! M
\frac{|\nabx \psia^{\Delta t, +} |^2}{\psia^{\Delta t, +}  + \alpha} \dq \dx \dd s &\geq & \int_0^t \!\int_{\Omega \times D} M
\frac{|\nabx\psia^{\Delta t, +}  |^2}{\psia^{\Delta t, +}} \dq \dx \dd s\\
&=& 4\int_0^t \!\int_{\Omega \times D}\!\! M
\big|\nabx \sqrt{\psia^{\Delta t, +} }\big|^2 \dq \dx \dd s,~~~~~
\\~\\~\\
\mbox{lim inf}_{\alpha \rightarrow 0_+} \int_0^t \!\int_{\Omega \times D}\!\! M
\frac{|\nabq \psia^{\Delta t, +} |^2}{\psia^{\Delta t, +}  + \alpha} \dq \dx \dd s &\geq& \int_0^t \!\int_{\Omega \times D} M
\frac{|\nabq \psia^{\Delta t, +}|^2}{\psia^{\Delta t, +} } \dq \dx \dd s\\
&=& 4\int_0^t \!\int_{\Omega \times D} \!\!M \big|\nabq \sqrt{\psia^{\Delta t, +} }\big|^2 \dq \dx \dd s.~~~~~
\end{eqnarray*}
%
%
%
Thus, after passage to the limit $\alpha\rightarrow 0_+$, on recalling \eqref{inidata-1}, we have, for all $t=t_n$, $n \in \{1,\dots, N\}$, that
\begin{subequations}
\begin{eqnarray}
&&\hspace{-2mm}\|\uta^{\Delta t, +}(t)\|^2 + \frac{1}{\Delta t} \int_0^t \|\uta^{\Delta t, +} - \uta^{\Delta t,-}\|^2
\dd s + \nu \int_0^t \|\nabxtt \uta^{\Delta t, +}(s)\|^2 \dd s\nonumber\\
&&\hspace{-2mm}+ \,2k\int_{\Omega \times D}\!\! M \mathcal{F}(\psia^{\Delta t, +}(t)) \dq \dx + \frac{k}{\Delta t\,L}
\int_0^t \int_{\Omega \times D}\!\! M (\psia^{\Delta t, +} - \psia^{\Delta t, -})^2 \dq \dx \dd s
\nonumber \\
&&\hspace{-2mm}\quad +\, 8k\,\varepsilon \int_0^t \int_{\Omega \times D} M
|\nabx \sqrt{\psia^{\Delta t, +}} |^2 \dq \dx \dd s\nonumber\\
&&\hspace{-2mm}\quad\quad +\, \frac{2a_0 k}{\lambda}  \int_0^t \int_{\Omega \times D}M\, \,|\nabq \sqrt{\psia^{\Delta t, +}}|^2 \,\dq \dx \dd s\nonumber
\end{eqnarray}
\begin{eqnarray}
&&\hspace{-2mm}\leq \|\ut_0\|^2 + \frac{1}{\nu}\int_0^t\|\ft^{\Delta t,+}(s)\|^2_{V'} \dd s + 2k \int_{\Omega \times D} M \mathcal{F}(\hat\psi^0) \dq \dx
\label{eq:energy-u+psi-final2a}\\
&&\hspace{-2mm} \leq \|\ut_0\|^2 + \frac{1}{\nu}\int_0^T\|\ft(s)\|^2_{V'} \dd s + 2k \int_{\Omega \times D} M \mathcal{F}(\hat\psi_0) \dq \dx  =:[{\sf B}(\ut_0,\ft, \hat\psi_0)]^2,~~~~~~~~
\label{eq:energy-u+psi-final2}
\end{eqnarray}
\end{subequations}
where, in the last line, we used \eqref{inidata-1} to bound the third term in
(\ref{eq:energy-u+psi-final2a}),
and that $t \in [0,T]$ and the definition \eqref{fn} of $\ft^{\Delta t,+}$ to bound the second term.

We select $\varphi = \chi_{[0,t]}\,\zeta_{\epsilon,L}^{\Delta t,+}$ as test function in \eqref{zetacon}, with $t$ chosen as $t_n$, $n \in \{1,\dots,N\}$.
Then, similarly to (\ref{eq:energy-u}),
we deduce, with $t = t_n$, that
\begin{align}
&\|\zeta_{\epsilon,L}^{\Delta t,+}(t)\|^2 + \frac{1}{\Delta t}\int_0^t  \|\zeta_{\epsilon,L}^{\Delta t,+}(s) - \zeta_{\epsilon,L}^{\Delta t,-}(s)\|^2 \dd s
+ 2\epsilon \int_0^t \|\nabx \zeta_{\epsilon,L}^{\Delta t,+}(s)\|^2 \dd s
\nonumber \\
& \hspace{3.5in}
\leq \|\int_D \beta^L(\hat \psi^0) \dq \|^2
\leq |\Omega|,
\label{eq:energy-zeta}
\end{align}
where we have noted (\ref{simpid}), (\ref{tripid}) and that
$\beta^L(\hat \psi^0) \in \hat Z_2$.

%
Next, we develop $L$-independent bounds on the time-derivatives of $\uta^{\Delta t}$,
$\psia^{\Delta t}$ and $\zeta_{\epsilon,L}^{\Delta t}$.

\subsection{$L$-independent bounds on the time-derivatives}
\label{Lindep-time}

We begin by bounding the time-derivative of $\psia^{\Delta t}$ using \eqref{eq:energy-u+psi-final2}; we shall then bound the time-derivative of $\uta^{\Delta t}$ in a similar manner.

\subsubsection{$L$-independent bound on the time-derivative of $\psia^{\Delta t}$}
\label{sec:time-psia}

It follows from \eqref{eqpsincon} that
\begin{align}
\label{eq:weaka2bound}
&\left|\int_{0}^T\int_{\Omega \times D} M\, \frac{\partial \psia^{\Delta t}}{\partial t}\, \hat \varphi \dq \dx \dd t\right|
\leq  \left|\varepsilon \int_{0}^T \int_{\Omega \times D} M\,
\nabx \psia^{\Delta t,+} \cdot \nabx \hat\varphi \dq \dx \dt \right|\nonumber\\
&\quad\qquad + \left|\int_{0}^T \int_{\Omega \times D} M\,
\uta^{\Delta t,-}\,\beta^L(\psia^{\Delta t,+}) \cdot\, \nabx \hat \varphi \,\dq \dx \dt\right| \nonumber\\
&\quad\qquad + \left|\frac{1}{2\,\lambda}\int_{0}^T \int_{\Omega \times D}M\,
\sum_{i=1}^K \sum_{j=1}^K A_{ij} \,\nabqj \psia^{\Delta t,+}
\cdot\, \nabqi \hat \varphi  \,\dq \dx \dt\right| \nonumber \\
&\quad\qquad  + \left| \int_{0}^T \int_{\Omega \times D}M\, \sum_{i=1}^K [\,
\sigtt(\uta^{\Delta t,+}) \,\qt_i\,]\,
\beta^L(\psia^{\Delta t,+})\,\cdot\, \nabqi
\hat \varphi \,\dq \dx \dt \right|
\nonumber\\
&\quad =: {\rm S}_1 + {\rm S_2} + {\rm S_3} + {\rm S_4}
\qqquad \forall \hat \varphi \in L^1(0,T; \hat{X}).
\end{align}

We proceed to bound each of the terms ${\rm S}_1, \dots, {\rm S}_4$, bearing in mind (cf. the last sentence in the statement of Lemma \ref{conv}) that
\begin{subequations}
\begin{align}\label{properties-a}
&\psia^{\Delta t,+} \geq 0 \quad \mbox{a.e. on $\Omega \times D \times [0,T]$}, \qquad \qquad  \int_D M(\qt) \dq = 1,&\\
& 0 \leq \int_D M(\qt) \psia^{\Delta t,+}(\xt,\qt, t)\dq \leq 1\qquad\mbox{~for a.e. $(x,t) \in \Omega \times D$}.& \label{properties-b}
\end{align}
\end{subequations}
In addition, we make use of the following result
\begin{lemma}
For $i=1,\ldots,K$ and for all $t=t_n$, $n =1,\ldots,N$,
\begin{align}
\int_{\Omega \times D} \left(\textstyle \frac{1}{2}|\qt_i|^2\right)^{\vartheta} M\,\psia^{\Delta t,\pm}(t) \dq \dx &\leq
\frac{2}{c_{i1}}\,\left[\int_{\Omega \times D} M\,{\cal F}(\psia^{\Delta t,\pm}(t)) \dq \dx
+ |\Omega| \int_{D} M\, {\rm e}^{\frac{c_{i1}}{2} (\frac{1}{2}|\qt_i|^2)^{\vartheta}} \dq \right]
\nonumber \\
& \leq \frac{1}{k\,c_{i1}} \,[{\sf B}(\ut_0,\ft, \hat\psi_0)]^2 + C_{i,\rm exp},
\label{qt2bd}
\end{align}
where $\vartheta$ and $c_{i1}$ are as defined in {\rm (\ref{growth1})}, and
$$C_{i,\rm exp}:=  \frac{2}{c_{i1}}\,|\Omega|\,\int_{D} M\,
{\rm e}^{\frac{c_{i1}}{2} (\frac{1}{2}|\qt_i|^2)^{\vartheta} } \dq
\leq C\,  \int_{D_i} {\rm e}^{-\frac{c_{i1}}{2} (\frac{1}{2}|\qt_i|^2)^{\vartheta}}
\dq_i < \infty.$$
\end{lemma}
\begin{proof}
First we recall the logarithmic Young's inequality
\begin{align}
r\,s \leq r\,\log r -r + {\rm e}^s
\qquad \mbox{for all } r,\,s \in {\mathbb R}_{\geq 0}.
\label{logY}
\end{align}
This follows from the Fenchel--Young inequality:
\[ r \, s \leq g^\ast(r) + g(s) \qquad \mbox{for all $r, s \in
\mathbb{R}$},\]
involving the convex function $g\,:\, s \in \mathbb{R} \mapsto
g(s) \in (-\infty,+\infty]$ and its convex conjugate $g^\ast$,
with $g(s) = {\rm e}^s$ and
\[ g^\ast(r) = \left\{ \begin{array}{cl}
  + \infty & \mbox{if $r < 0$;}\\
  0        & \mbox{if $r = 0$;}\\
  r\,(\log r - 1) & \mbox{if $r > 0$,}
                \end{array}
                \right.
\]
with the resulting inequality then restricted to $\mathbb{R}_{\geq 0}$.
It immediately follows from (\ref{logY}) that
$r\,s
\leq {\cal F}(r) +{\rm e}^s$
for all $r,\,s \in {\mathbb R}_{\geq 0}$.
Hence we have that
\begin{align}
\frac{c_{i1}}{2}\, \int_{\Omega \times D} M\,(\frac{1}{2}|\qt_i|^2)^{\vartheta}\,
\psia^{\Delta t,\pm}(t) \dq \dx &\leq
\int_{\Omega \times D} M \left[ {\cal F}(\psia^{\Delta t,\pm}(t)) +
{\rm e}^{\frac{c_{i1}}{2} (\frac{1}{2}|\qt_i|^2)^{\vartheta}} \right] \dq \dx.
\label{qt2bda}
\end{align}
The desired result (\ref{qt2bd}) follows immediately from
(\ref{qt2bda}), on noting (\ref{eq:energy-u+psi-final2}) and (\ref{growth3}).
\end{proof}

We shall use throughout the rest of this section tests function $\hat\varphi$ such that
\begin{equation}
\hat\varphi \in L^2(0,T;W^{1,\infty}(\Omega \times D)).
\end{equation}

We begin by considering ${\rm S}_1$, noting (\ref{properties-a},b):
\begin{eqnarray}
{\rm S}_1 &=& 2 \varepsilon \left|\int_{0}^T \int_{\Omega \times D} M\,\sqrt{\psia^{\Delta t,+}}\,\,\nabx \sqrt{\psia^{\Delta t,+}} \cdot \nabx \hat\varphi \dq \dx \dt \right|\nonumber\\
&\leq& 2 \varepsilon\int_0^T\!\!\! \int_\Omega\! \left[\left(\int_D M\psia^{\Delta t,+} \dq\right)^{\!\!\frac{1}{2}}\!
\left(\int_D M \bigg|\nabx \sqrt{\psia^{\Delta t,+}}\bigg|^2
\dq \right)^{\!\!\frac{1}{2}}\! \|\nabx \hat\varphi \|_{L^\infty(D)}\!\right]\!\!\dx \dd t \nonumber\\
&\leq& \sqrt{\frac{\varepsilon}{2k}}\,\left(\!8 k \varepsilon \!\int_0^T\!\!\! \int_{\Omega\times D}\!\! M  \bigg|\nabx \sqrt{\psia^{\Delta t,+}}\bigg|^2\! \dq \dx \dt \right)^{\!\!\frac{1}{2}}\! \left(\int_0^T\!\!\! \int_\Omega \|\nabx\hat\varphi \|^2_{L^\infty(D)}\! \dx \dd t\!\right)^{\!\!\frac{1}{2}}. \nonumber
\end{eqnarray}
Hence, by \eqref{eq:energy-u+psi-final2} with $t=t_N=T$,
\begin{equation}\label{eq:T1bound}
{\rm S}_1 \leq \sqrt{\frac{\varepsilon}{2k}}\, {\sf B}(\ut_0,\ft, \hat\psi_0)\, \left(\int_0^T \int_\Omega \|\nabx\hat\varphi \|^2_{L^\infty(D)} \dx \dd t\right)^{\!\!\frac{1}{2}}.
\end{equation}

Next, on noting (\ref{properties-a},b) and (\ref{eq:FL2a}),
we consider term ${\rm S}_2$:
\begin{eqnarray}
{\rm S}_2 &\leq& \int_0^T \int_\Omega |\uta^{\Delta t,-}| \left(\int_D M
\,\beta^L(\psia^{\Delta t,+}) \dq\right) \|\nabx\hat\varphi\|_{L^\infty(D)} \dx \dd t\nonumber\\
&\leq& \left(\int_0^T \|\uta^{\Delta t,-}\|^2 \dt \right)^{\!\!\frac{1}{2}}
\left(\int_0^T\int_\Omega \|\nabx
\hat\varphi\|^2_{L^\infty(D)} \dx \dd t \right)^{\!\!\frac{1}{2}}\nonumber\\
& \leq & C_{{\sf P}}(\Omega) \left(\int_0^T \|\nabxtt\uta^{\Delta t,-}\|^2 \dt
\right)^{\!\!\frac{1}{2}}\! \left(\int_0^T\int_\Omega \|\nabx
\hat\varphi\|^2_{L^\infty(D)} \dx \dd t \right)^{\!\!\frac{1}{2}}\!\!;~~~~~~~~~\label{Poinc}
\end{eqnarray}
where $C_{{\sf P}}(\Omega)$ denotes the (positive) constant appearing in the  Poincar\'{e} inequality $\|\vt\| \leq C_{{\sf P}}(\Omega) \linebreak \|\nabxtt \vt\|$
on $\Omega$ for any  $\vt \in \Vt \subset \Ht^1_0(\Omega)$.
%
On recalling the definitions of $\uta^{\Delta t,\pm}$
from \eqref{upm}, and noting
\eqref{idatabd} and \eqref{eq:energy-u+psi-final2} we have that
\begin{eqnarray}\label{bound-a}
\int_0^T \|\nabxtt\uta^{\Delta t,-}\|^2 \dd t
&=& \Delta t \,\|\nabxtt\ut^0\|^2 + \int_0^{T-\Delta t} \|\nabxtt \uta^{\Delta t, +} \|^2 \dt\nonumber\\
& \leq &\|\ut_0\|^2 + \int_0^T\|\nabxtt \uta^{\Delta t, +} \|^2 \dt\nonumber\\
& \leq & \left(1 + {\textstyle \frac{1}{\nu}}\right) [{\sf B}(\ut_0, \ft, \hat\psi_0)]^2.~~~~~~~~~
\end{eqnarray}
Therefore,
\begin{eqnarray}
{\rm S}_2 \leq C_{{\sf P}}(\Omega)\left(1 + {\textstyle \frac{1}{\nu}}\right)^{\frac{1}{2}} {\sf B}(\ut_0, \ft, \hat\psi_0) \left(\int_0^T\int_\Omega \|\nabx
\hat\varphi\|^2_{L^\infty(D)} \dx \dd t \right)^{\frac{1}{2}}.
\label{S2bda}
\end{eqnarray}
Alternatively, without the use of the Poincar\'e inequality, directly from the second line of
\eqref{Poinc}, we have that
\begin{eqnarray}\label{without}
{\rm S}_2 &\leq& \sqrt{T} \,\mbox{ess.sup}_{t \in [0,T]} \,\|\uta^{\Delta t,-}\|
 \left(\int_0^T\int_\Omega \|\nabx
\hat\varphi\|^2_{L^\infty(D)} \dx \dd t \right)^{\frac{1}{2}}\!\!.~~~~~~~
\end{eqnarray}
Similarly as above,
\begin{align}
\mbox{ess.sup}_{t \in [0,T]} \|\uta^{\Delta t,-}(t)\|^2
&= \max\left(\|\ut_0\|^2 , \mbox{ess.sup}_{t \in (0,T-\Delta t]}\|\uta^{\Delta t,+}(t)\|^2\right)\nonumber\\
&\leq  \max\left(\|\ut_0\|^2 , \mbox{ess.sup}_{t \in [0,T]}
\|\uta^{\Delta t,+}(t)\|^2\right)\nonumber\\
& \leq [{\sf B}(\ut_0, \ft, \hat\psi_0)]^2.
\label{bound-b}
\end{align}
Combining \eqref{S2bda}, \eqref{without} and \eqref{bound-b},
we have that
\begin{eqnarray}\label{eq:T2bound}
\!\!\!\!\!\!\!\!\!\!\!{\rm S}_2 \leq \min\left(\!C_{{\sf P}}(\Omega)\, \left(1 + {\textstyle{\frac{1}{\nu}}}\right)^{\frac{1}{2}} ,
\sqrt{T}\right) {\sf B}(\ut_0,\ft, \hat\psi_0)\!
\left(\int_0^T\!\!\int_\Omega \|\nabx
\hat\varphi\|^2_{L^\infty(D)} \dx \dd t \right)^{\frac{1}{2}}\!\!\!.
\end{eqnarray}

We are ready to consider ${\rm S}_3$; we have that
\begin{eqnarray*}
{\rm S}_3 & = & \frac{1}{2\,\lambda} \left|\int_{0}^T \int_{\Omega \times D} M\,2 \sqrt{\psia^{\Delta t, +}}\, \sum_{i=1}^K \sum_{j=1}^K A_{ij} \nabqj \sqrt{\psia^{\Delta t, +}} \cdot \nabqi \hat\varphi \,\dq \dx \dt\right|\\
&\leq& \frac{1}{\lambda}\!\left(\sum_{i,j=1}^K A_{ij}^2\right)^{\!\!\frac{1}{2}}\!\int_0^T\!\! \int_{\Omega \times D}
M \sqrt{\psia^{\Delta t, +}} \left(\sum_{i,j=1}^K \bigg|\nabqj \sqrt{\psia^{\Delta t, +}}\bigg|^2\, |\nabqi \hat\varphi|^2 \right)^{\!\!\frac{1}{2}}
\!\!\dq \dx \dd t\\
&=& \frac{1}{\lambda} |A|\, \int_0^T \int_{\Omega \times D}
M \sqrt{\psia^{\Delta t, +}} \left(\sum_{j=1}^K \bigg|\nabqj \sqrt{\psia^{\Delta t, +}}\bigg|^2\right)^{\!\!\frac{1}{2}} \left(\sum_{i=1}^K |\nabqi \hat\varphi|^2 \right)^{\!\!\frac{1}{2}}
\dq \dx \dd t\\
& = & \frac{1}{\lambda} |A|\, \int_0^T \int_{\Omega \times D}
M \sqrt{\psia^{\Delta t, +}}\, \bigg|\nabq \sqrt{\psia^{\Delta t, +}}\bigg|\, |\nabq \hat\varphi| \dq \dx \dd t\\
& \leq & \frac{1}{\lambda} |A|\, \int_0^T\!\! \int_\Omega \left(\int_D\! M \psia^{\Delta t, +} \dq\right)^{\!\!\frac{1}{2}}
\left(\int_D\! M\, \bigg |\nabq \sqrt{\psia^{\Delta t, +}}\bigg|^2 \dq \right)^{\!\!\frac{1}{2}} \|\nabq \hat\varphi\|_{L^\infty(D)}
\dx \dd t\\
& \leq & \frac{|A|}{\sqrt{2a_0 k \lambda}} \left(\!\frac{2a_0 k}{\lambda}\int_0^T\!\!\! \int_{\Omega \times D}
\!\!M\, \bigg|\nabq \sqrt{\psia^{\Delta t, +}} \bigg |^2\! \dq \dx \dd t\!\right)^{\!\!\frac{1}{2}} \!\left(\int_0^T \!\!\!\int_\Omega \|\nabq
\hat\varphi\|^2_{L^\infty(D)}\! \dx \dd t\!\right)^{\!\!\frac{1}{2}}\!\!.
\end{eqnarray*}
Thus, by \eqref{eq:energy-u+psi-final2},
\begin{eqnarray}\label{eq:T3bound}
{\rm S}_3 \leq \frac{|A|}{\sqrt{2a_0 k \lambda}}\,  {\sf B}(\ut_0,\ft, \hat\psi_0)
\left(\int_0^T\int_\Omega \|\nabq \hat\varphi\|^2_{L^\infty(D)} \dx \dd t \right)^{\!\!\frac{1}{2}}.
\end{eqnarray}

Finally, for term ${\rm S}_4$, recalling 
the inequality $\beta^L(s) \leq s$ for $s \in \mathbb{R}_{\geq 0}$,
(\ref{properties-a},b) and (\ref{qt2bd})
, we have that
{\allowdisplaybreaks
\begin{align*}
{\rm S}_4 &\leq \int_{0}^T \int_{\Omega \times D} M\,|\sigtt(\uta^{\Delta t, +})|\, \beta^L(\psia^{\Delta t, +})\,
\sum_{i=1}^K \,|\qt_i|\, |\nabqi \hat \varphi| \,\dq \dx \dt\\
& \leq  \int_{0}^T \int_{\Omega \times D} M\,|\qt|\,|\sigtt(\uta^{\Delta t, +})|\,
\beta^L(\psia^{\Delta t, +})\, |\nabq \hat \varphi| \,\dq \dx \dt\\
& \leq  \int_{0}^T \int_{\Omega} \,|\sigtt(\uta^{\Delta t, +})|\, \left(\int_D M
\,
|\qt|\,\beta^L(\psia^{\Delta t, +}) \dq\right) \,
\|\nabq \hat \varphi\|_{L^\infty(D)} \dx \dt\\
& \leq
\int_{0}^T \int_{\Omega} \,|\sigtt(\uta^{\Delta t, +})|\,
\left(\int_D M\,
|\qt|^2 \,\psia^{\Delta t, +}\dq\right)^{\frac{1}{2}}\,
\left(\int_D M\,\psia^{\Delta t, +} \dq\right)^{\frac{1}{2}} \,
\|\nabq \hat \varphi\|_{L^\infty(D)} \dx \dt
\\
& \leq  \int_{0}^T \left[
\int_{\Omega} \,|\sigtt(\uta^{\Delta t, +})|
\left(\int_D M \,
|\qt|^2 \,\psia^{\Delta t, +}\dq\right)^{\frac{1}{2}}\,
 \dx \right] \|\nabq \hat \varphi\|_{L^\infty(\Omega \times D)} \dt\\
& \leq  \left(\int_{0}^T \int_{\Omega} \,|\sigtt(\uta^{\Delta t, +})|^2 \dx \dd t\right)^{\frac{1}{2}} \left(\mbox{ess.sup}_{t \in [0,T]} \int_{\Omega \times D} M\,
|\qt|^2 \,\psia^{\Delta t, +}(t) \dq \dx \right)^{\frac{1}{2}}
\nonumber \\
& \hspace{1in} \times
\left(\int_0^T \|\nabq \hat\varphi\|^2_{L^\infty(\Omega \times D)} \dd t\right)^{\frac{1}{2}}
\\
&\leq 
\sqrt{\frac{2}{\nu \,\vartheta}\,\left[ (\vartheta-1)K +\sum_{i=1}^K\left[
\frac{1}{k\,c_{i1}}[{\sf B}(\ut_0,\ft, \hat\psi_0)]^2 + C_{i,\rm exp}\right]\right] }\,
\left(\nu \int_{0}^T \|\nabxtt \uta^{\Delta t, +}\|^2 \dd t\right)^{\frac{1}{2}}
\\
& \hspace{1in}
\times \left(\int_0^T \|\nabq \hat\varphi\|^2_{L^\infty(\Omega \times D)} \dd t\right)^{\!\!\frac{1}{2}},
\end{align*}
where we have noted, on applying a Young's inequality, that $|\qt|^2 \leq \frac{2}{\vartheta}
\,[(\vartheta-1)K + \displaystyle\sum_{i=1}^K (\textstyle\frac{1}{2}|\qt_i|^2)^\vartheta]$.
Hence, by \eqref{eq:energy-u+psi-final2},
\begin{align}
{\rm S}_4 &\leq 
\sqrt{\frac{2}{\nu \,\vartheta}\,\left[ (\vartheta-1)K +\sum_{i=1}^K\left[
\frac{1}{k\,c_{i1}}[{\sf B}(\ut_0,\ft, \hat\psi_0)]^2 + C_{i,\rm exp}\right]\right] }
\,{\sf B}(\ut_0,\ft, \hat\psi_0)
\nonumber \\ & \hspace{2in}
\times \left(\int_0^T \|\nabq \hat\varphi\|^2_{L^\infty(\Omega \times D)} \dd t\right)^{\frac{1}{2}}.
\label{eq:T4bound}
\end{align}

Upon substituting the bounds on 
${\rm S}_1$ to ${\rm S}_4$ into \eqref{eq:weaka2bound},
with $\hat\varphi \in L^2(0,T; W^{1,\infty}(\Omega \times D)),$ and
noting that the latter space is contained in $L^1(0,T;\hat{X})$ we deduce from \eqref{eq:weaka2bound}
that
\begin{align}\label{psi-time-bound}
&\left|\int_{0}^T\int_{\Omega \times D} M\, \frac{\partial \psia^{\Delta t}}{\partial t}\,\hat \varphi
\dq \dx \dt\right|
\nonumber\\ & \hspace{1in}
\leq {C}_{\ast}\, \max \left(
[{\sf B}(\ut_0,\ft, \hat\psi_0)]^2,
{\sf B}(\ut_0,\ft, \hat\psi_0)\right)
\,
\left(\int_0^T 
\|\hat \varphi\|_{W^{1,\infty}(\Omega \times D)}^2
\dd t\right)^{\frac{1}{2}}\!\!
\end{align}
for any $\hat\varphi \in L^2(0,T;W^{1,\infty}(\Omega \times D))$,
where $C_\ast$
denotes a positive constant (that can be computed by tracking the constants in \eqref{eq:T1bound}--\eqref{eq:T4bound}), which depends solely on $\epsilon$, $\nu$, $C_{\sf P}(\Omega)$, $T$, $|A|$, $a_0$, $k$,
$\lambda$, $M$, $\vartheta$, $K$ and $\{c_{i1}\}_{i=1}^K$. 

We now consider the time derivative of $\zeta_{\epsilon,L}^{\Delta t}$. It follows from
(\ref{zetacon}), (\ref{eq:energy-zeta}) and (\ref{bound-b}) that
\begin{align}
&\left|\int_0^T \int_\Omega
\frac{\partial \zeta_{\epsilon,L}^{\Delta t}}{\partial t} \varphi \dx \dt\right|
\leq
\int_0^T \int_{\Omega} \left| \left[ \epsilon\, \nabx  \zeta_{\epsilon,L}^{\Delta t, +} -
\ut_{\epsilon,L}^{\Delta t,-}  \,\zeta_{\epsilon,L}^{\Delta t, +} \right]
\cdot \nabx \varphi \right| \dx  \dt
\nonumber \\
& \quad \leq \left[ \epsilon \left( \int_0^T \|\nabx \zeta_{\epsilon,L}^{\Delta t, +}\|^2 \dt \right)^{\frac{1}{2}}
+ \mbox{ess.sup}_{t \in [0,T]} \|\zeta_{\epsilon,L}^{\Delta t, +}\|_{L^\infty(\Omega)}
\left( \int_0^T \|\ut_{\epsilon,L}^{\Delta t, -}\|^2 \dt \right)^{\frac{1}{2}}
\right]
\nonumber \\
& \hspace{3in} \times
\left( \int_0^T \|\nabx \varphi\|^2 \dt \right)^{\frac{1}{2}}
\nonumber \\
& \quad \leq \left[ \epsilon \left(\frac{|\Omega|}{2}\right)^{\frac{1}{2}}
+ {\sf B}(\ut_0, \ft, \hat\psi_0)
\right]
\left( \int_0^T \|\nabx \varphi\|^2 \dt \right)^{\frac{1}{2}}
\qquad \forall \varphi \in L^2(0,T;H^1(\Omega)).
\label{zetacondtbd}
\end{align}

\subsubsection{$L$-independent bound on the time-derivative of $\uta^{\Delta t}$}
\label{sec:time-uta}

In this section we shall derive an $L$-independent bound on the time-derivative of $\uta^{\Delta t}$. Our starting point is \eqref{equncon}, from which we deduce that

\begin{align}
&\displaystyle
\left|\int_{0}^{T}\!\! \int_\Omega  \frac{\partial \utaeD}{\partial t}\cdot
\wt \dx \dt\right|
\label{equncon1}
\nonumber
\\
&
\hspace{0.5cm} \leq \left|\int_{0}^T\!\! \int_{\Omega}
 \left[ (\utaeDm \cdot \nabx) \utaeDp \right]\,\cdot\,\wt \dx \dt \,\right|
+ \nu \left|\int_0^T\!\! \int_\Omega\,\nabxtt \utaeDp
: \wnabtt \dx \dt \,\right|
\nonumber
\\
&\hspace{1cm}+\left|\int_{0}^T  \langle \ft^{\Delta t,+}, \wt\rangle_V \dd t \right|
+ k\,\left|\sum_{i=1}^K \int_{0}^T\!\! \int_{\Omega}
\Ctt_i(M\,\hpsiae^{\Delta t,+}): \nabxtt
\wt \dx \dt \right|
\nonumber
\\
& \hspace{0.5cm} =: {\rm U_1} + {\rm U}_2 + {\rm U}_3 + {\rm U}_4 \hspace{1in}
\qquad \forall \wt \in L^1(0,T;\Vt).
\end{align}

On recalling from the discussion following \eqref{eqvn2} the definition of
$\Vt_\sigma$, we shall assume henceforth that
%
\[ \wt \in L^2(0,T;\Vt_\sigma),
\qquad \sigma > 1 + \textstyle{\frac{1}{2}}d.
\]
Clearly
with this choice of $\sigma$,
$L^2(0,T;\Vt_\sigma) \subset L^1(0,T;\Vt)$
and we will exploit
the embedding $H^\sigma(\Omega) \hookrightarrow W^{1,\infty}(\Omega)$.
Using \eqref{bound-b} and \eqref{eq:energy-u+psi-final2}, we have 
%
\begin{eqnarray}\label{time-u-t1}
{\rm U}_1 & \leq & 
\mbox{ess.sup}_{t \in [0,T]} \|\utaeDm\| \left(\int_0^T \|\nabx \utaeDp\|^2 \dd t\right)^{\!\!\frac{1}{2}} \left(\int_0^T\|\wt\|^2_{L^\infty(\Omega)}\dd t\right)^{\!\!\frac{1}{2}}\nonumber\\
& \leq & 
\sqrt{\frac{1}{\nu}} \, [{\sf B}(\ut_0, \ft, \hat\psi_0)]^2 \,
\left(\int_0^T\|\wt\|^2_{L^\infty(\Omega)}\dd t\right)^{\!\!\frac{1}{2}}.
\end{eqnarray}
For term ${\rm U}_2$ we have,
\begin{eqnarray}\label{time-u-t2}
{\rm U}_2 & \leq & \sqrt{\nu} \left(\nu\int_0^T\!\!\|\nabxtt \utaeDp\|^2 \dt \right)^{\!\!\frac{1}{2}}
\left(\int_0^T \|\wnabtt\|^2\dt \right)^{\!\!\frac{1}{2}}\nonumber\\
& \leq & \sqrt{\nu} \, {\sf B}(\ut_0, \ft, \hat\psi_0) \, \left(\int_0^T \|\wnabtt\|^2 \dt \right)^{\!\!\frac{1}{2}}.
\end{eqnarray}
Concerning the term ${\rm U}_3$, on noting the definition of the norm $\|\cdot\|_{V'}$ and
that thanks to \eqref{fn} we have $$\|\ft^{\Delta t,+}\|_{L^2(0,T;V')} \leq \|\ft\|_{L^2(0,T;V')},$$
it follows that
\begin{eqnarray}\label{time-u-t3}
{\rm U}_3 &\leq&
\sqrt{\nu}\, {\sf B}(\ut_0, \ft, \hat\psi_0) \, \left(\int_0^T \|\wnabtt\|^2 \dt \right)^{\!\!\frac{1}{2}}.
\end{eqnarray}
We now bound the term ${\rm U}_4$.
On noting (\ref{growth2}), Young's inequality, (\ref{properties-b}) and (\ref{qt2bd}),
we have that
\begin{align}
{\rm U}_4 & = k \left|\int_0^T \int_\Omega  \left[\int_{D} M\,
\hpsiae^{\Delta t,+} \sum_{i=1}^K \qt_i\, \qt_i^{\rm T}\,
U_i'\left(\textstyle{\frac{1}{2}}|\qt_i|^2\right) 
: \nabxtt \wt \dq \right] \dx \dt\right|\nonumber\\
& \leq
2k\,\sum_{i=1}^K 
\,\int_0^T \int_\Omega  \int_{D} M\,\left[ c_{i2}\,\textstyle \frac{1}{2}|\qt_i|^2
+ c_{i3}\,\textstyle \left(\frac{1}{2}|\qt_i|^{2}\right)^\vartheta\right]
\,
\hpsiae^{\Delta t,+} \,|\nabxtt \wt| \dq \dx \dt \nonumber\\
& \leq 2k
\,\sum_{i=1}^K
\int_0^T \left[ \int_{\Omega \times D} M\,\left[ \textstyle
\frac{c_{i2}\,(\vartheta-1)}{\vartheta}
+ \left(\frac{c_{i2}}{\vartheta} +c_{i3}\right)\,\left(\frac{1}{2}|\qt_i|^2\right)^\vartheta \right]
\hpsiae^{\Delta t,+} \dq \dx \right] \|\nabxtt \wt\|_{L^{\infty}(\Omega)} \dt \nonumber\\
&\leq 2k 
\left[\textstyle
\frac{\vartheta-1}{\vartheta} \displaystyle \sum_{i=1}^K c_{i2}
+ \sum_{i=1}^K \left( \textstyle \frac{c_{i2}}{\vartheta} +c_{i3}\right)\,
\left( \frac{1}{k\,c_{i1}}
\,[{\sf B}(\ut_0,\ft, \hat\psi_0)]^2 + C_{i,\rm exp}\right) \right]
\int_0^T \|\nabxtt \wt\|_{L^{\infty}(\Omega)} \dt. \label{U4a}
\end{align}
Collecting the bounds on the terms ${\rm U}_1$ to ${\rm U}_4$ and
inserting them into \eqref{equncon1} yields
\begin{align}
&\displaystyle
\left|\int_{0}^{T}\!\! \int_\Omega  \frac{\partial \utaeD}{\partial t}\cdot
\wt \dx \dt\right|\leq C_{\ast \ast}\, \max\left([{\sf B}(\ut_0,\ft, \hat\psi_0)]^2, {\sf B}(\ut_0,\ft, \hat\psi_0),1\right)
\left(\int_{0}^T \|\wt\|^2_{V_\sigma} \dt\! \right)^{\!\!\frac{1}{2}}
\label{u-time-bound}
\end{align}
for any $\wt \in L^2(0,T; \Vt_\sigma)$, $\sigma > 1 +\textstyle{\frac{1}{2}}d$, 
where $C_{\ast\ast}$ denotes a positive constant (that can be computed by tracking the constants in \eqref{time-u-t1}--\eqref{U4a}), which depends solely on
$\Omega$, $d$, $\nu$, $k$, $\lambda$, $a_0$,
$M$, $\vartheta$, $K$, $\{c_{i2}\}_{i=1}^K$ and  $\{c_{i3}\}_{i=1}^K$.

\section{Dubinski{\u\i}'s compactness theorem}
\label{sec:dubinskii}
\setcounter{equation}{0}

Having developed a collection of $L$-independent bounds in Sections \ref{Lindep-space} and \ref{Lindep-time},
we now describe the theoretical tool that will be used to set up a weak compactness argument using these
bounds, --- Dubinski{\u\i}'s compactness theorem in seminormed sets.

Let $\mathcal{A}$ be a linear space over the field $\mathbb{R}$ of real numbers, and suppose that $\mathcal{M}$ is a subset of $\mathcal{A}$ such that
\begin{equation}
\label{eq:property}
(\forall \varphi \in \mathcal{M})\; (\forall c \in \mathbb{R}_{\geq 0})\;\;\; c\, \varphi \in \mathcal{M}.
\end{equation}
In other words, whenever $\varphi$ is contained in $\mathcal{M}$, the ray through $\varphi$ from the origin of the linear space $\mathcal{A}$ is also contained in $\mathcal{M}$.  Note in particular that while any set $\mathcal{M}$ with property \eqref{eq:property} must contain the zero element of the linear space $\mathcal{A}$, the set
$\mathcal{M}$ need not be closed under summation. The linear space $\mathcal{A}$ will be referred to as the {\em ambient space} for $\mathcal{M}$.

Suppose further that each element $\varphi$ of a set $\mathcal{M}$ with property \eqref{eq:property}
is assigned a certain real number, denoted $[\varphi]_{\mathcal M}$, such that:
\begin{enumerate}
\item[(i)] $[\varphi]_{\mathcal M} \geq 0$; and $[\varphi]_{\mathcal M} = 0$ if, and only if, $\varphi=0$; and
\item[(ii)] $(\forall c \in \mathbb{R}_{\geq 0})\; [c\, \varphi]_{\mathcal M} = c\,[\varphi]_{\mathcal M}$.
\end{enumerate}
We shall then say that $\mathcal{M}$ is a {\em seminormed set}.

A subset $\mathcal{B}$ of a seminormed
set $\mathcal{M}$ is said to be {\em bounded} if there exists a positive constant $K_0$ such that $[\varphi]_{\mathcal M} \leq K_0$ for all $\varphi \in \mathcal{B}$.

A seminormed set $\mathcal{M}$ contained in a normed linear space $\mathcal{A}$ with norm $\|\cdot\|_{\mathcal A}$ is said to be {\em embedded in $\mathcal{A}$}, and we write $\mathcal{M} \hookrightarrow \mathcal{A}$, if:
\[ (\exists K_0 \in \mathbb{R}_{>0})\; (\forall \varphi \in \mathcal{M})\;\;\; \|\varphi \|_{\mathcal A} \leq K_0 [\varphi]_{\mathcal M}.
\]
Thus, bounded subsets of a seminormed set are also bounded subsets of the ambient normed linear space the seminormed set is embedded in.

The embedding of a seminormed set $\mathcal{M}$ into a normed linear space $\mathcal{A}$ is said to be {\em compact} if from any bounded, infinite set of elements of $\mathcal{M}$ one can extract a subsequence that converges in $\mathcal{A}$; we shall write $\mathcal{M} \hookrightarrow\!\!\!\rightarrow \mathcal{A}$ to denote that $\mathcal{M}$
is compactly embedded in $\mathcal{A}$.

Suppose that $T$ is a positive real number,
$\varphi$ maps the nonempty closed interval $[0,T]$ into a seminormed set $\mathcal{M}$,
and $p\in \mathbb{R}$, $p \geq 1$. We denote by $L^p(0,T; \mathcal{M})$ the set of all functions $\varphi\,:\, t \in [0,T] \mapsto \varphi(t) \in \mathcal{M}$ such that
\[ \left(\int_0^T [\varphi(t)]^p_{\mathcal M} {\rm d}t\right)^{1/p} < \infty.\]
The set $L^p(0,T; \mathcal{M})$ is then a seminormed set in the ambient linear space $L^p(0,T; \mathcal{A})$,
equipped with
\[ [\varphi]_{L^p(0,T;\mathcal{M})} := \left( \int_0^T [\varphi(t)]^p_{\mathcal M} {\rm d}t\right)^{1/p}.  \]
We shall denote by $L^\infty(0,T; \mathcal{M})$ and $[\varphi]_{L^\infty(0,T;\mathcal{M})}$ the usual modifications of these definitions when $p=\infty$. The following theorem is due to Dubinski{\u\i} \cite{DUB}.

\begin{theorem}\label{thm:Dubinski}
Suppose that $\mathcal{A}_0$ and $\mathcal{A}_1$ are normed linear spaces, $\mathcal{A}_0 \hookrightarrow \mathcal{A}_1$, and $\mathcal{M}$ is a seminormed subset of $\mathcal{A}_0$ such that $\mathcal{M} \hookrightarrow\!\!\!\rightarrow
\mathcal{A}_0$. Consider the set
\[ \mathcal{Y}:= \left\{\varphi\,:\,[0,T] \rightarrow \mathcal{M}\,:\,
[\varphi]_{L^p(0,T;\mathcal M)} + \left\|\frac{{\rm d}\varphi}{{\rm d}t} \right\|_{L^{p_1}(0,T;\mathcal{A}_1)}
< \infty   \right\},
\]
where $1 \leq p \leq \infty$, $1 \leq p_1 \leq \infty$, $\|\cdot\|_{\mathcal{A}_1}$ is the norm of $\mathcal{A}_1$, and ${\rm d}\varphi/{\rm d}t$ is understood in the sense of $\mathcal{A}_1$-valued distributions
on the open interval $(0,T)$.

Then, $\mathcal{Y}$ is a seminormed set with seminorm
\[ [\varphi]_{\mathcal Y}:= [\varphi]_{L^p(0,T;\mathcal M)} + \left\|\frac{{\rm d}\varphi}{{\rm d}t} \right\|_{L^{p_1}(0,T;\mathcal{A}_1)},\]
in the ambient linear space $L^p(0,T;\mathcal{A}_0)\cap W^{1,p_1}(0,T;\mathcal{A}_1)$, and $\mathcal{Y}
\hookrightarrow\!\!\!\rightarrow
 L^p(0,T; \mathcal{A}_0)$.
\end{theorem}

In the next section, we shall apply Dubinski{\u\i}'s theorem by selecting
\begin{eqnarray*}
&&\mathcal{A}_0 = L^1_{(1+|\qt|)^{2\vartheta}M}(\Omega \times D)\\
&&\quad\,\,  := \bigg\{ \hat\varphi \in L^1_{\rm loc}(\Omega \times D)\,:\,
\|\hat\varphi\|_{\mathcal{A}_0} := \int_{\Omega \times D} M(\qt)\,
(1+|\qt|)^{2\vartheta}\,
|\hat\varphi (\xt, \qt)| \dd \xt \dd \qt < \infty \bigg\}
\end{eqnarray*}
and
\begin{eqnarray*}
&&\hspace{-4mm}\mathcal{M} = \bigg\{ \hat\varphi \in \mathcal{A}_0\,: \hat\varphi \geq 0 \quad \mbox{with} \\
&&\qquad  \int_{\Omega \times D} M(\qt)\left( \left|\nabx \sqrt{\hat\varphi(\xt,\qt)}\right|^2 + \left|\nabq \sqrt{\hat\varphi(\xt,\qt)}\right|^2\right)\dd \xt \dd \qt < \infty \bigg\},
\end{eqnarray*}
and, for $\hat\varphi \in \mathcal{M}$,  we define
\[ [\hat\varphi]_{\mathcal M}:= \|\hat\varphi\|_{\mathcal A_0} + \int_{\Omega \times D} M(\qt)\left( \left|\nabx \sqrt{\hat\varphi(\xt,\qt)}\right|^2 + \left|\nabq \sqrt{\hat\varphi(\xt,\qt)}\right|^2\right)\dd \xt \dd \qt.
\]
Note that $\mathcal{M}$ is a seminormed subset of the ambient normed linear space $\mathcal{A}_0$. Finally, we put
\[ \mathcal{A}_1 := M^{-1} H^{s}(\Omega \times D)' := \{\hat \varphi : M \hat\varphi \in H^{s}(\Omega \times D)'\},\]
equipped with the norm
\[ \|\hat\varphi\|_{\mathcal A_1} := \|M \hat \varphi\|_{H^{s}(\Omega \times D)'},\]
and take $s>1+\frac{1}{2}(K+1)d$. With such $s$ it then follows from the Sobolev embedding theorem on $\Omega \times D
\subset \mathbb{R}^{d \times Kd} \cong \mathbb{R}^{(K+1)d}$ that, for any $\hat\varphi \in \mathcal{A}_0$,
\begin{align*}
\|\hat\varphi\|_{\mathcal A_1} &= \sup_{\!\!\!\chi \in H^s(\Omega \times D)}\! \frac{|(M\hat\varphi, \chi)|}{\|\chi\|_{H^s(\Omega \times D)}}
\leq \sup_{\!\!\!\chi \in H^s(\Omega \times D)}\!
\frac{\|\hat\varphi\|_{L^1_M(\Omega \times D)} \|\chi\|_{L^\infty(\Omega \times D)}}{\|\chi\|_{H^s(\Omega \times D)}}
\\
&\leq K_0 \|\hat \varphi\|_{L^1_M(\Omega \times D)}
\leq K_0 \|\hat\varphi\|_{\mathcal{A}_0},
\end{align*}
where $K_0$ is any positive constant that is greater than or equal to the constant $K_s$, the norm of the
continuous linear operator corresponding to the Sobolev embedding $(H^{s}(\Omega \times D)
\hookrightarrow )
  H^{s-1}(\Omega \times D) \hookrightarrow
  L^\infty(\Omega \times D)$, $s>1+ \frac{1}{2}(K+1)d$. Hence,
we have that
  $\mathcal{A}_0
  \hookrightarrow
   \mathcal{A}_1$.

Trivially, $\mathcal{M} \hookrightarrow \mathcal{A}_0$.
We shall show that in fact $\mathcal{M} \hookrightarrow\!\!\!\rightarrow \mathcal{A}_0$. Suppose to this end that $\mathcal{B}$ is
an infinite, bounded subset of $\mathcal{M}$. We can assume without loss of generality that $\mathcal{B}$ is the
infinite sequence $\{\hat\varphi_n\}_{n \geq 1}
\subset
\mathcal{M}$ with $[\hat\varphi_n]_{\mathcal M}
\leq K_0$ for all $n \geq 1$, where $K_0$ is a fixed positive constant. We define $\hat\rho_n:=
\sqrt{\hat\varphi_n}$ and note that $\hat\rho_n \geq 0$ and $\hat\rho_n \in H^1_M(\Omega \times D) \cap L^2_{(1+|\qt|)^{2 \vartheta}M}(\Omega \times D)$
for all $n\geq 1$, with
\[ \|\hat\rho_n\|^2_{L^2_{(1+|\qt|)^{2\vartheta} M}(\Omega \times D)}
+ \|\nabx\, \hat\rho_n\|^2_{L^2_M(\Omega \times D)}
+ \|\nabq\, \hat\rho_n\|^2_{L^2_M(\Omega \times D)}
= [\hat\varphi_n]_{\mathcal M} \leq K_0 \qquad
\forall n \geq 1.\]
Since $H^1_M(\Omega \times D)\cap
 L^2_{(1+|\qt|)^{2 \vartheta}M}(\Omega \times D)
$ is compactly embedded in $ L^2_{(1+|\qt|)^{2 \vartheta}M}(\Omega \times D)$
(see \ref{AppendixD} at the end of the paper for a proof of this), we deduce that
the sequence $\{\hat\rho_n\}_{n \geq 1}$ has a subsequence $\{\hat\rho_{n_k}\}_{k \geq 1}$ that is
convergent in $ L^2_{(1+|\qt|)^{2 \vartheta}M}(\Omega \times D)$; denote the limit of this subsequence by $\hat\rho$;
$\hat \rho \in  L^2_{(1+|\qt|)^{2 \vartheta}M}(\Omega \times D)$.
Then, since a subsequence of
the sequence $\{\hat\rho_{n_k}\}_{k \geq 1}$
also converges to $\hat\rho$ a.e. on $\Omega \times D$ and each $\hat\rho_{n_k}$ is nonnegative
on $\Omega \times D$, the same is true of $\hat\rho$. Now, define $\hat\varphi:= \hat\rho^{~\!2}$,
and note that $\hat\varphi \in  L^1_{(1+|\qt|)^{2 \vartheta}M}(\Omega \times D)$. Clearly,
\begin{align*}
\| \hat\varphi_{n_k} - \hat\varphi \|_{L^1_{(1+|\qt|)^{2\vartheta} M}(\Omega \times D)} &= \int_{\Omega \times D}
M \,(1+|\qt|)^{2\vartheta}\,(\,\hat\rho_{n_k} + \hat\rho\,)\, |\,\hat\rho_{n_k} - \hat\rho\,| \dd \xt \dd \qt\\
&\leq \|\, \hat\rho_{n_k} + \hat\rho \,\|_{L^2_{(1+|\qt|)^{2\vartheta} M} (\Omega \times D)}\, \|\, \hat\rho_{n_k}
- \hat\rho \,\|_{L^2_{(1+|\qt|)^{2\vartheta} M}(\Omega \times D)}\\
&\leq \left(\|\, \hat\rho_{n_k}\,\|_{L^2_{(1+|\qt|)^{2\vartheta} M}(\Omega \times D)} + \|\,\hat\rho \,\|_{L^2_{(1+|\qt|)^{2\vartheta} M}(\Omega \times D)}\right)
\\
& \hspace{2in} \times
\|\, \hat\rho_{n_k}
- \hat\rho \,\|_{L^2_{(1+|\qt|)^{2\vartheta} M}(\Omega \times D)}.
\end{align*}
As $\{\hat\rho_{n_k}\}_{k \geq 1}$ converges to $\hat\rho$ in $L^2_{(1+|\qt|)^{2\vartheta} M}(\Omega \times D)$, and is therefore
also a bounded sequence in $L^2_{(1+|\qt|)^{2\vartheta} M}(\Omega \times D)$, it follows from the last inequality that $\{\hat\varphi_{n_k}\}_{k \geq 1}$ converges to $\hat\varphi$ in $L^1_{(1+|\qt|)^{2\vartheta} M}(\Omega \times D) = \mathcal{A}_0$. This in turn implies that
the seminormed set $\mathcal{M}$ is compactly embedded in $\mathcal{A}_0$.
Thus we have shown that the triple $\mathcal{M} \hookrightarrow\!\!\!\rightarrow \mathcal{A}_0 \hookrightarrow \mathcal{A}_1$ satisfies the conditions of Theorem \ref{thm:Dubinski}.

\begin{remark}\label{rem5.1}
{\em There is a deep connection between $\mathcal{M}$ and the set of functions with finite relative entropy on $D$, exhibited by
the} logarithmic Sobolev inequality:
\begin{equation}\label{eq:logs0}
\int_{D} M(\qt) |\hat\rho(\qt)|^2 \log \frac{|\hat\rho(\qt)|^2}{\|\hat\rho\|^2_{L^2_M(D)}} \dd \qt
\leq \frac{2}{\kappa} \int_{D} M(\qt) \big|\nabq \hat\rho(\qt)\big|^2 \dd \qt \quad \forall \hat\rho \in H^1_M(D),
\end{equation}
{\em with log-Sobolev constant $\kappa>0$;
the inequality \eqref{eq:logs0} is known to hold whenever $M$ satisfies the
{\em Bakry--\'Emery condition}: ${\sf Hess}(- \log M(\qt)) \geq \kappa\, {{\sf Id}}$ (in the sense of symmetric
$Kd \times Kd$ matrices) on $D=D_1 \times \cdots \times D_K$,
asserting the logarithmic concavity of the Maxwellian on $D$.
The inequality \eqref{eq:logs0} follows from inequality (1.3)
in Arnold, Bartier \& Dolbeault \cite{ABD}.}

{\em
The validity of the Bakry--\'Emery condition for the Hookean Maxwellian, $U_i(s)=s$ for $i=1,\ldots,K$, is an easy consequence of the fact that
\begin{align}\label{bakry}
{\sf Hess}(- \log M(\qt)) &= {\sf Hess}\left(\sum_{i=1}^K U_i\left({\textstyle \frac{1}{2}}|\qt_i|^2\right)\right)\nonumber\\
&=  \mbox{\sf diag}\left({\sf Hess}\left(U_1({\textstyle \frac{1}{2}}|\qt_1|^2)\right), \dots, {\sf Hess}\left(U_K({\textstyle \frac{1}{2}}|\qt_K|^2)\right)\right)
= {{\sf Id}},
\end{align}
for all $\qt = (\qt_1,\dots, \qt_K) \in D_1 \times \cdots \times D_K = D$,
so $\kappa=1$.
%
%
%
%
If we replace $U_i(s) =s$ by the potential in (\ref{eqHM}) for $i=1,\ldots, K$,
it is easy to show, on noting that ${\rm det}(\Itt+\undertilde{a}\,\undertilde{b}^{\rm T})
= 1 + \undertilde{a}\cdot\undertilde{b}$ for all $\undertilde{a},\,\undertilde{b}
\in {\mathbb R}^d$,  that $\kappa \geq 1$.

More generally, we see from \eqref{bakry} that if $\qt_i\in D_i \mapsto U_i(\frac{1}{2}|\qt_i|^2)$ is strongly convex
(see, e.g., Hiriart-Urruty \& Lemar{\'e}chal \cite{HUL}, p.73)
on $D_i$ with modulus of convexity $\kappa_i>0$, $i=1,\dots,K$, then $M$ satisfies the Bakry--\'{E}mery condition on $D$
with $\kappa = \min\{\kappa_1,\dots, \kappa_K\}$.

On writing $\hat\varphi(\qt):= |\hat\rho(\qt)|^2\, (\,\geq 0)$ in \eqref{eq:logs0}, we have that
\begin{equation} \label{eq:logs1}
\int_{D} M(\qt) \hat\varphi(\qt) \log \frac{\hat\varphi(\qt)}{\|\hat\varphi\|_{L^1_M(D)}} \dd \qt
\leq \frac{2}{\kappa} \int_{D} M(\qt) \left|\nabq \sqrt{\hat\varphi(\qt)}\right|^2 \dd \qt,
\end{equation}
for all $\hat\varphi$ such that $\hat\varphi \geq 0$ on $D$ and $\sqrt{\hat\varphi} \in H^1_M(D)$. Taking $\hat\varphi = \varphi/M$ where $\varphi$ is a probability density function on $D$, we have that $\|\hat\varphi\|_{L^1_M(D)} = \|\varphi\|_{L^1(D)} = 1$; thus,
on denoting by $\nu$ the Gibbs measure, defined by $\dd \nu = M(\qt) \dq$, the left-hand side of \eqref{eq:logs1} becomes
\[ S(\varphi | M) := \int_D \frac{\varphi}{M} \left(\log \frac{\varphi}{M}\right)\!\dd\nu,\]
referred to as the relative entropy of $\varphi$ with respect to $M$. The expression appearing on the right-hand side of \eqref{eq:logs1} is $1/(2\kappa)$ times the
{\em Fisher information} $I(\hat \varphi)$ of $\hat\varphi$, where
\[I(\hat\varphi):= \mathbb{E}\left[\left|\nabq \log {\hat\varphi(\qt)}\right|^2\right] =
\int_{D} \left|\nabq \log {\hat\varphi(\qt)}\right|^2 \hat{\varphi}(\qt) \dd \nu = 4 \int_{D} \left|\nabq \sqrt{\hat\varphi(\qt)}\right|^2 \dd \nu,\]
where, $\mathbb{E}$ is the expectation with respect to the Gibbs measure.}
\qquad $\diamond$
\end{remark}

The following simple lemma will be helpful in the next section.

\begin{lemma}\label{le:supplementary}
Suppose that a sequence $\{\hat\varphi_n\}_{n=1}^\infty$ converges in
$L^1(0,T; L^1_{(1+|\qt|)^{2\vartheta} M}(\Omega \times D))$ to a function $\hat\varphi \in L^1(0,T; L^1_{(1+|\qt|)^{2\vartheta} M}(\Omega \times D))$,
and is bounded in $L^\infty(0,T; L^1_{(1+|\qt|)^{2\vartheta} M}(\Omega \times D))$, i.e. there exists $K_0>0$ such that
$\|\varphi_n\|_{L^\infty(0,T; L^1_{(1+|\qt|)^{2\vartheta} M}(\Omega \times D))} \leq K_0$ for all $n \geq 1$.
Then, $\hat\varphi \in L^p(0,T; L^1_{(1+|\qt|)^{2\vartheta} M}(\Omega \times D))$ for all $p \in [1,\infty)$,
and the sequence $\{\hat\varphi_n\}_{n \geq 1}$ converges to $\hat\varphi$ in
$L^p(0,T; L^1_{(1+|\qt|)^{2\vartheta} M}(\Omega \times D))$ for all $p \in [1,\infty)$.
\end{lemma}

\begin{proof}
Since $\{\hat\varphi_n\}_{n \geq 1}$ converges in $L^1(0,T; L^1_{(1+|\qt|)^{2\vartheta} M}(\Omega \times D))$, it follows that
it is a Cauchy sequence in $L^1(0,T; L^1_{(1+|\qt|)^{2\vartheta} M}(\Omega \times D))$; thus, for any $p \in [1,\infty)$,
there exists $n_0=n_0(\varepsilon,p) \in \mathbb{N}$ such that for all $m, n \geq n_0(\varepsilon,p)$ we have
\[ \int_0^T \|\hat\varphi_n - \hat\varphi_m\|_{L^1_{(1+|\qt|)^{2\vartheta} M}(\Omega \times D)} \dd t < \frac{\varepsilon^p}{(2K_0)^{p-1}}.\]
Hence, for all $m, n \geq n_0(\varepsilon,p)$,
\begin{align*}
&\left(\int_0^T \|\hat\varphi_n - \hat\varphi_m\|^p_{L^1_{(1+|\qt|)^{2\vartheta} M}(\Omega \times D)} \dd t\right)^{1/p}
\\
& \hspace{0.5in}
\leq
\mbox{ess.sup}_{t \in [0,T]} \|\hat\varphi_n - \hat\varphi_m\|^{1-(1/p)}_{L^1_{(1+|\qt|)^{2\vartheta} M}(\Omega \times D)} \left(\int_0^T \|\hat\varphi_n - \hat\varphi_m\|_{L^1_{(1+|\qt|)^{2\vartheta} M}(\Omega \times D)}\right)^{1/p} < \varepsilon.
\end{align*}
This in turn implies that $\{\hat\varphi_n\}_{n \geq 1}$ is a Cauchy sequence in the function space $L^p(0,T; L^1_{(1+|\qt|)^{2\vartheta} M}$ $(\Omega \times D))$, for each $p \in [1,\infty)$. Since $L^p(0,T; L^1_{(1+|\qt|)^{2\vartheta} M}(\Omega \times D))$ is complete,
$\{\hat\varphi_n\}_{n \geq 1}$ converges in $L^p(0,T; L^1_{(1+|\qt|)^{2\vartheta} M}(\Omega \times D))$ to a limit, which we denote by $\hat\varphi_{(p)}$, say. Since,
by assumption, $\{\hat\varphi_n\}_{n \geq 1}$ converges in $L^1(0,T; L^1_{(1+|\qt|)^{2\vartheta} M}(\Omega \times D))$, and
$L^p(0,T; L^1_{(1+|\qt|)^{2\vartheta} M}(\Omega \times D)) \subset L^1(0,T; L^1_{(1+|\qt|)^{2\vartheta} M}(\Omega \times D))$ for each $p \in [1,\infty)$,
it follows by uniqueness of the limit that $\hat\varphi_{(p)} = \hat\varphi$ for all $p \in [1,\infty)$.
Hence, also, $\hat\varphi \in L^p(0,T; L^1_{(1+|\qt|)^{2\vartheta} M}(\Omega \times D))$ for all $p \in [1,\infty)$.
\end{proof}

\section{Passage to the limit $L \rightarrow \infty$: existence of weak solutions to Hookean-type bead-spring
chain models with centre-of-mass diffusion}
\label{sec:passage.to.limit}
\setcounter{equation}{0}

The bounds \eqref{eq:energy-u+psi-final2}, \eqref{psi-time-bound} and \eqref{u-time-bound} imply the existence
of a positive constant ${C}_{\star}$, which depends only on ${\sf B}(\ut_0,\ft,\hat\psi_0)$ and the constants $C_{\ast}$ and $C_{\ast\ast}$, which in turn depend only on
$\epsilon$, $\nu$, $C_{\sf P}(\Omega)$, $T$, $|A|$, $a_0$, $k$, $\lambda$, $\Omega$,
$d$, $K$ and $M$, but {\em not} on $L$ or $\Delta t$, such that:
\begin{align}\label{eq:energy-u+psi-final5}
&\mbox{ess.sup}_{t \in [0,T]}\|\uta^{\Delta t, +}(t)\|^2 + \frac{1}{\Delta t} \int_0^T \|\uta^{\Delta t, +} - \uta^{\Delta t,-}\|^2
\dd s + \,\int_0^T \|\nabxtt \uta^{\Delta t, +}(s)\|^2 \dd s\nonumber\\
&\qquad +\, \mbox{ess.sup}_{t \in [0,T]}
\int_{\Omega \times D}\!\! M\, \mathcal{F}(\psia^{\Delta t, +}(t)) \dq 
+\, \frac{1}{\Delta t\,L}
\int_0^T\!\! \int_{\Omega \times D}\!\! M (\psia^{\Delta t, +} - \psia^{\Delta t, -})^2 \dq \dx \dd s
\nonumber \\
&\qquad +\, \int_0^T\!\! \int_{\Omega \times D} M\,
\big|\nabx \sqrt{\psia^{\Delta t, +}} \big|^2 \dq \dx \dd s
+\, \int_0^T\!\! \int_{\Omega \times D}M\,\big|\nabq \sqrt{\psia^{\Delta t, +}}\big|^2 \,\dq \dx \dd s\nonumber\\
&\qquad+\, \int_0^T\left\|\frac{\partial \uta^{\Delta t}}{\partial t}\right\|^2_{V_\sigma'}\dt + \int_0^T\left\|M\frac{\partial \psia^{\Delta t}}{\partial t}\right\|^2_{H^s(\Omega \times D)'}\dt \leq C_\star,
\end{align}
where $\|\cdot\|_{V_\sigma'}$ denotes the norm of the dual space $\Vt_\sigma'$ of $\Vt_\sigma$
with $\sigma > 1+\frac{1}{2}d$, (cf. the paragraph following \eqref{equncon1});
and $\|\cdot\|_{H^s(\Omega \times D)'}$ is the
norm of the dual space $H^s(\Omega \times D)'$ of $H^s(\Omega \times D)$, with $s>1+\frac{1}{2}(K+1)d$.
The bounds on the time-derivatives stated in the last line of \eqref{eq:energy-u+psi-final5} follow directly from \eqref{u-time-bound}, and \eqref{psi-time-bound} using the Sobolev
embedding $H^s(\Omega \times D) \hookrightarrow W^{1,\infty}(\Omega \times D)$.

By virtue of \eqref{bound-b}, \eqref{bound-a}, the definitions (\ref{ulin},b), 
and with an argument completely analogous to \eqref{bound-a} on noting
(\ref{eq:FL2a}), (\ref{inidata-1}) and
\eqref{inidata} in the case of the fourth term in \eqref{eq:energy-u+psi-final5}, we have (with a possible adjustment of the constant $C_\star$, if necessary,) that
\begin{align}\label{eq:energy-u+psi-final6}
&\mbox{ess.sup}_{t \in [0,T]}\|\uta^{\Delta t(,\pm)}(t)\|^2 + \frac{1}{\Delta t} \int_0^T \|\uta^{\Delta t,+} - \uta^{\Delta t,-}\|^2
\dd s +\,\int_0^T \|\nabxtt \uta^{\Delta t(,\pm)}(s)\|^2 \dd s
\nonumber \\
& \qquad +  \mbox{ess.sup}_{t \in [0,T]}
\int_{\Omega \times D}\!\! M\, \mathcal{F}(\psia^{\Delta t(, \pm)}(t)) \dq \dx 
+\, \frac{1}{\Delta t\,L}
\int_0^T\!\! \int_{\Omega \times D}\!\! M (\psia^{\Delta t, +} - \psia^{\Delta t, -})^2 \dq \dx \dd s
\nonumber \\
&\qquad +\, \int_0^T\!\! \int_{\Omega \times D} M\,
\big|\nabx \sqrt{\psia^{\Delta t, +}} \big|^2 \dq \dx \dd s
+\, \int_0^T\!\! \int_{\Omega \times D}M\,\big|\nabq \sqrt{\psia^{\Delta t, +}}\big|^2 \,\dq \dx \dd s\nonumber\\
&\qquad +\, \int_0^T\left\|\frac{\partial \uta^{\Delta t}}{\partial t}\right\|^2_{V_\sigma'}\dt + \int_0^T\left\|M\frac{\partial \psia^{\Delta t}}{\partial t}\right\|^2_{H^s(\Omega \times D)'}\dt \leq C_\star.
\end{align}
On noting (\ref{properties-a},b), (\ref{ulin},b),
(\ref{inidata-1})
 and \eqref{inidata},
we also have that
\begin{equation}\label{additional1}
\psia^{\Delta t(,\pm)} \geq 0\qquad \mbox{a.e. on $\Omega \times D \times [0,T]$}
\end{equation}
and
\begin{equation}\label{additional2}
\int_D M(\qt)\, \psia^{\Delta t(,\pm)}(\xt,\qt,t)\dq \leq 1\quad \mbox{for a.e. $(\xt,t) \in \Omega \times [0,T]$.}
\end{equation}

Henceforth, we shall assume that
\begin{equation}\label{LT}
\Delta t = o(L^{-1})\qquad \mbox{as $L \rightarrow \infty$}.
 \end{equation}
Requiring, for example, that $0<\Delta t \leq C_0/(L\,\log L)$, $L > 1$, with an arbitrary (but fixed)
constant $C_0$ will suffice to ensure that \eqref{LT} holds. The sequences
\[\left\{\uta^{\Delta t(,\pm)}\right\}_{L>1}, \qquad \left\{\psia^{\Delta t(,\pm)}\right\}_{L>1},\]
as well as all sequences of spatial and temporal derivatives of the entries of these two sequences,
will thus be, indirectly, indexed by $L$ alone, although for reasons of consistency with our previous notation
we shall not introduce new, compressed, notation with $\Delta t$ omitted from the superscripts. Instead, whenever $L\rightarrow \infty$ in the rest of this section, it will be understood that $\Delta t$ tends to $0$ according to \eqref{LT}.

We are now almost ready to pass to the limit with $L\rightarrow \infty$. Before doing so, however, we first need to state the definition of the function $\hat\psi^0$ that obeys \eqref{inidata-1}, for a given $\hat\psi_0$ satisfying \eqref{inidata}. We emphasize that up to this point we simply accepted without proof the existence of a function $\hat\psi^0$ obeying \eqref{inidata-1}
for a given $\hat\psi_0$.
The reason we have been evading to state the precise choice of $\hat\psi^0$ was for the sake
of clarity of exposition. The definition of $\hat\psi^0$ and the verification of the properties listed under \eqref{inidata-1}
rely on mathematical tools that were not in place at the start of Section \ref{sec:existence-cut-off} where the notation $\hat\psi^0$ was introduced, but were developed later, in the last two sections. The details of `lifting' $\hat\psi_0$ into a `smoother' function $\hat\psi^0$ are technical; they are discussed in the next subsection.

A second remark is in order.
One might wonder whether one could simply choose $\hat\psi^0$ as $\hat\psi_0$; indeed, with such
a choice all of the properties listed in \eqref{inidata} would be automatically satisfied,
bar one: there is no guarantee that $[\hat\psi_0]^{1/2} \in H^1_M(\Omega \times D)$.
Although the property $[\hat\psi^0]^{1/2} \in H^1_M(\Omega \times D)$ has not yet been used, it will play a crucial role
in our passage to the limit with $L\rightarrow \infty$ in Section \ref{passage}.
In fact, in the light of the logarithmic Sobolev inequality
\eqref{eq:logs1}, on comparing the requirements on $\hat\psi_0$ in \eqref{inidata} with those on $\hat\psi^0$ in
\eqref{inidata-1}, one can clearly see that the role of the condition $[\hat\psi^0]^{1/2} \in H^1_M(\Omega \times D)$ in \eqref{inidata-1} is to `lift' the initial datum $\hat\psi_0$
with finite relative entropy into a `smoother' initial datum $\hat\psi^0$ that
also has finite Fisher information, in analogy with the process of `lifting' the initial velocity
$\ut_0$ from $\Ht$ into $\ut^0$ in $\Vt$.
That the choice of $\hat\psi_0$ as $\hat\psi^0$ is not a good one can be
seen by noting the mismatch 
between the third term
in \eqref{eq:energy-u+psi-final6} arising from the Navier--Stokes equation on the one hand, and the sixth
and seventh term in \eqref{eq:energy-u+psi-final6} that stem from the Fokker--Planck equation. The
absence of bounds at this stage on $\hat\psi^{\Delta t,-}$ and $\hat\psi^{\Delta t}$ in those terms
is entirely due to the fact that, to derive \eqref{eq:energy-u+psi-final6}, we did not
use that $[\hat\psi^0]^{1/2} \in H^1_M(\Omega \times D)$. This shortcoming of
\eqref{eq:energy-u+psi-final6} will be rectified as soon as we have defined
$\hat\psi^0$ and shown that it possesses {\em all} of the properties listed in \eqref{inidata-1}.

\subsection{The definition of $\hat\psi^0$}

Given $\hat\psi_0$ satisfying the conditions in \eqref{inidata} and $\Lambda >1$,
we consider the following discrete-in-time problem in weak form: find
$\hat\zeta^{\Lambda,1} \in H^1_M(\Omega \times D)$ such that
\begin{align}\label{zeta-eq}
&\int_{\Omega \times D} M\, \frac{\hat\zeta^{\Lambda,1} - \hat\zeta^{\Lambda,0}}{\Delta t}\,\hat\varphi \dq \dx 
+ \int_{\Omega \times D} M \left[\nabx \hat\zeta^{\Lambda,1} \cdot \nabx \hat\varphi
+ \nabq \hat\zeta^{\Lambda,1} \cdot \nabq \hat\varphi \right] \dq \dx = 0
\end{align}
for all $\hat\varphi \in H^1_M(\Omega \times D)$, with $\hat\zeta^{\Lambda,0} := \beta^\Lambda(\hat\psi_0) \in L^2_M(\Omega \times D)$.
Here $\beta^\Lambda$ is defined by \eqref{betaLa}, with $L$ replaced by $\Lambda$. The function $\mathcal{F}^\Lambda$, which we shall
encounter below, is defined by \eqref{eq:FL}, with $L$ replaced by $\Lambda$.

The existence of a unique solution $\hat\zeta^{\Lambda,1} \in H^1_M(\Omega \times D)$ to \eqref{zeta-eq}, for each $\Delta t>0$ and $\Lambda>1$, follows immediately by applying the
Lax--Milgram theorem. The parameter $\Lambda$ plays an analogous role to the cut-off parameter $L$; however since we shall
let $\Lambda \rightarrow \infty$ in this subsection while, for the moment at least, the parameter $L$ is kept fixed,
we had to use a symbol other than $L$ in \eqref{zeta-eq} in order to avoid confusion; we chose the letter $\Lambda$ for this purpose
in order to emphasize the connection with $L$.
\begin{lemma}\label{one-bound}
Let $\hat\zeta^{\Lambda,1}$ be defined by \eqref{zeta-eq}, and consider $\gamma^{\Lambda,n}$ defined by
\begin{equation}\label{gamma-def}
\gamma^{\Lambda,n}(\xt):= \int_D M(\qt)\,\hat\zeta^{\Lambda,n}(\xt,\qt)\dq,\qquad n=0,1.
\end{equation}
Then, $\hat\zeta^{\Lambda,1}$ is nonnegative a.e. on $\Omega \times D$, and $0 \leq \gamma^{\Lambda,1} \leq 1$
a.e. on $\Omega$.
\end{lemma}
\begin{proof}
The proof of nonnegativity of $\hat\zeta^{\Lambda,1}$ is straightforward (cf. the discussion following \eqref{psiGIz}).
We have that $[{\hat{\zeta}}^{\Lambda,0}]_{-} =0$ a.e. on $\Omega\times D$, thanks to \eqref{inidata} and the definition
of $\beta^\Lambda$. We take $\varphi = [{\hat{\zeta}}^{\Lambda,1}]_{-}$ as a test function in  \eqref{zeta-eq}, noting that this is a legitimate choice since ${\hat{\zeta}}^{\Lambda,1} \in H^1_M(\Omega\times D)$ and therefore $[{\hat{\zeta}}^{\Lambda,1}]_{-} \in H^1_M(\Omega\times D)$ also (cf. Lemma 3.3 in Barrett, Schwab \& S\"uli \cite{BSS}). On decomposing ${\hat{\zeta}}^{\Lambda,1} = [{\hat{\zeta}}^{\Lambda,1}]_{+} + [{\hat{\zeta}}^{\Lambda,1}]_{-}$, and using that
\[ [{\hat{\zeta}}^{\Lambda,1}]_{+} \,[\hat\zeta^{\Lambda,1}]_{-}=0, \quad \nabx [{\hat{\zeta}}^{\Lambda,1}]_{+}\cdot \nabx [{\hat{\zeta}}^{\Lambda,1}]_{-} = 0\quad\mbox{and}\quad \nabq [{\hat{\zeta}}^{\Lambda,1}]_{+}\cdot \nabq [{\hat{\zeta}}^{\Lambda,1}]_{-} = 0\]
a.e. on $\Omega\times D$, we deduce that
\begin{eqnarray*}
&&\frac{1}{\Delta t}\|M^{\frac{1}{2}}\,[{\hat{\zeta}}^{\Lambda,1}]_{-}\|^2 + \|M^{\frac{1}{2}}\, \nabx [{\hat{\zeta}}^{\Lambda,1}]_{-}\|^2
+ \|M^{\frac{1}{2}}\,\nabq [{\hat{\zeta}}^{\Lambda,1}]_{-}\|^2\\
&&\qqqquad = \frac{1}{\Delta t}\int_{\Omega \times D} M\,\hat\zeta^{\Lambda,0}\,[\hat\zeta^{\Lambda,1}]_{-}
\dq \dx = \frac{1}{\Delta t}\int_{\Omega \times D} M\,[\hat\zeta^{\Lambda,0}]_{+}\,[\hat\zeta^{\Lambda,1}]_{-}
\dq \dx \leq 0,
\end{eqnarray*}
where $\|\cdot\|$ denotes the $L^2(\Omega\times D)$ norm. This then implies that
\[ \|M^{\frac{1}{2}}\,[{\hat{\zeta}}^{\Lambda,1}]_{-}\|^2  \leq  0. \]
Hence, $[{\hat{\zeta}}^{\Lambda,1}]_{-} = 0$ a.e. on $\Omega \times D$. In other words, ${\hat{\zeta}}^{\Lambda,1} \geq 0$ a.e.
on $\Omega \times D$, as claimed.

In order to prove the upper bound in the statement of the lemma, we proceed as follows.
With $\gamma^{\Lambda,n}$ as defined in \eqref{gamma-def}, we deduce from the definition of ${\hat{\zeta}}^{\Lambda,1}$
and Fubini's theorem that $\gamma^{\Lambda,1} \in H^1(\Omega)$. Furthermore, on selecting
$\hat\varphi = \varphi \in H^1(\Omega)\otimes 1(D)$ in \eqref{zeta-eq},
recall (\ref{H101}),
we have that
\begin{align}
\label{eq:weaka2aa}
&\int_{\Omega}
\frac{\gamma^{\Lambda,1}- \gamma^{\Lambda,0}}{\Delta t}\,\varphi\dx
+  \int_{\Omega}
\nabx {\gamma}^{\Lambda,1} \cdot\, \nabx
\varphi \dx = 0
\qquad \forall
\varphi \in
H^1(\Omega).
\end{align}
As ${\hat{\zeta}}^{\Lambda,0} = \beta^\Lambda(\hat\psi_0)$, and $0 \leq \beta^\Lambda(s) \leq s$ for all
$s \in \mathbb{R}_{\geq 0}$, we also have by \eqref{inidata} that
\begin{equation}\label{z-bound}
 0 \leq \gamma^{\Lambda,0} =\int_D M \beta^{\Lambda}(\hat\psi_0) \dq \leq \int_D M \hat\psi_0 \dq  = 1
\quad \mbox{on $\Omega$}.
\end{equation}

Consider
\[ z^{\Lambda,n}:= 1 - \gamma^{\Lambda,n},\qquad n=0,1.\]
On substituting $\gamma^{\Lambda,n}= 1 - z^{\Lambda,n}$, $n=0,1$, into \eqref{eq:weaka2aa}, we have that
\begin{align}
\label{eq:weaka2aaa}
&\int_{\Omega}
\frac{z^{\Lambda,1}- z^{\Lambda,0}}{\Delta t}\,\varphi\dx
+  \int_{\Omega}
\nabx z^{\Lambda,1} \cdot\, \nabx
\varphi \dx = 0
\qquad \forall
\varphi \in
H^1(\Omega).
\end{align}
Also, by \eqref{z-bound}, we have that $0 \leq z^{\Lambda,0} \leq  1$. By using an identical procedure to
the one in the first part of the proof, we then deduce that $[z^{\Lambda,1}]_{-} = 0$ a.e. on $\Omega$. Thus,
$z^{\Lambda,1} \geq 0$ a.e. on $\Omega$, which then implies that $\gamma^{\Lambda,1} \leq 1$ a.e. on $\Omega$, as claimed.
\end{proof}

Next, we shall pass to the limit $\Lambda \rightarrow \infty$; as we shall see in the final part of Lemma \ref{psi0properties}
below, this will require the use of smoother test functions in problem \eqref{zeta-eq}, as otherwise the term involving
$\hat\zeta^{\Lambda,0}=\beta^\Lambda(\hat \psi_0)$ is not defined in the limit. In any case, our objective is to use the
limit of the sequence $\{\zeta^{\Lambda,1}\}_{\Lambda>1}$, once it has been shown to exist, as our definition of the function $\hat\psi^0$. We shall then show that $\hat\psi^0$ thus defined has all the
properties listed in \eqref{inidata-1}.

To this end, we need to derive $\Lambda$-independent
bounds on norms of ${\hat{\zeta}}^{\Lambda,1}$, very similar to the $L$-independent bounds discussed in Section
\ref{sec:entropy}. Since the argument is almost identical to (but simpler than) the one there
(viz. \eqref{zeta-eq} can be viewed as a special case of \eqref{psiG}, with $\ft^n$, $\uta^{n-1}$ and $\uta^n$ taken to be identically zero, $\lambda=\frac{1}{2}$, $\epsilon=1$, $N=1$, and $A$ chosen as the $K\times K$ identity matrix), we shall not include the details
here. It suffices to say that, on testing \eqref{zeta-eq} with $\mathcal{F}'(\hat\zeta^{\Lambda,1} + \alpha)$ and passing to the limit
$\alpha \rightarrow 0_{+}$, analogously
as in the derivation of (\ref{eq:energy-u+psi-final2a})
in Section \ref{sec:entropy}, we obtain that
\begin{align}\label{eq:zetabd}
&\int_{\Omega \times D} M \mathcal{F}(\hat\zeta^{\Lambda,1}) \dq \dx
+\, 4\,\Delta t \int_{\Omega \times D} M
\,\big|\nabx \sqrt{\hat\zeta^{\Lambda,1}} \big|^2 \dq \dx
+\, 4\,
\Delta t \int_{\Omega \times D}M\, \big|\nabq \sqrt{\hat\zeta^{\Lambda,1}}\big|^2 \,\dq \dx\nonumber\\
& \hspace{2in}\leq \int_{\Omega \times D} M \mathcal{F}(\hat\psi_0) \dq \dx.
\end{align}
Our passage to the limit $\Lambda \rightarrow \infty$ in \eqref{zeta-eq} is based on a weak-compactness argument, using
\eqref{eq:zetabd}, and is discussed below.

We have from Lemma \ref{one-bound} that $\{[\hat\zeta^{\Lambda,1}]^{\frac{1}{2}}\}_{\Lambda>1}$
is a bounded sequence in $L^2_M(\Omega \times D)$.
Using this in conjunction with the second and third bound in \eqref{eq:zetabd} we deduce that,
for $\Delta t>0$ fixed, $\{[\hat\zeta^{\Lambda,1}]^{\frac{1}{2}}\}_{\Lambda>1}$ is a bounded sequence in
$H^1_M(\Omega \times D)$. Thanks to the compact embedding of $H^1_M(\Omega \times D)$ into
$L^2_M(\Omega \times D)$ (cf. \ref{AppendixD}), we deduce that $\{[\hat\zeta^{\Lambda,1}]^{\frac{1}{2}}\}_{\Lambda>1}$
has a strongly convergent subsequence in $L^2_M(\Omega \times D)$, whose limit we label by
$Z$, and we then let $\hat\zeta^{~\!\!1}:=Z^2$.

For future reference we note that, upon extraction
of a subsequence (not indicated), $\hat\zeta^{\Lambda,1}$ then converges to $\hat{\zeta}^{\!\!~1}$ a.e. on
$\Omega \times D$; and $\hat\zeta^{\Lambda,1}(\xt,\cdot)$ converges to $\hat{\zeta}^{\!\!~1}(\xt,\cdot)$ a.e. on
$D$, for a.e. $\xt \in \Omega$.

By definition, we have that $\hat{\zeta}^{\!\!~1} \geq 0$; furthermore, thanks to the upper bound on $\gamma^{\Lambda,1}$ stated in
Lemma \ref{one-bound}, the remark in the previous paragraph, and Fatou's lemma, we also have that
\begin{equation}\label{upperbound}
\int_D M(\qt)\,\hat{\zeta}^{\!\!~1}(\xt,\qt) \dq \leq 1\qquad \mbox{for a.e. $\xt \in \Omega$.}
\end{equation}
Further, again as a direct consequence of the definition of $\hat{\zeta}^{\!\!~1}$, we have that
\begin{equation}\label{sqrtpsi-zeta}
\sqrt{ {\hat{\zeta}}^{\Lambda,1}} \rightarrow \sqrt{{\hat{\zeta}^{~\!1}}}\qquad \mbox{strongly in $L^2_M(\Omega \times D)$}.
\end{equation}
Application of the factorization $c_1-c_2 =
(\sqrt{c_1} - \sqrt{c_2})\,(\sqrt{c_1} + \sqrt{c_2})$ 
with $c_1, c_2 \in \mathbb{R}_{\geq 0}$, the
Cauchy--Schwarz inequality and (\ref{sqrtpsi-zeta}), then yields that
\begin{equation}\label{sqrtpsi-zetaaa}
{\hat{\zeta}}^{\Lambda,1} \rightarrow \hat{\zeta}^{~\!\!1}\qquad \mbox{strongly in $L^1_M(\Omega \times D)$}.
\end{equation}

Finally, we define
\begin{equation}\label{psizerodef}
\hat\psi^0:= \hat{\zeta}^{\!\!~1}.
\end{equation}
It follows from the nonnegativity of $\hat{\zeta}^{\!\!~1}$ and \eqref{upperbound} that
\begin{equation}\label{nonnegativity}
\mbox{$\hat\psi^0\geq 0\quad$a.e. on $\Omega \times D\qquad$and}\qquad 0 \leq \int_D M(\qt)\, \hat\psi^0(\xt, \qt) \dq \leq 1\quad \mbox{for a.e. $\xt \in \Omega$.}
\end{equation}
Further, from the bound on the first term in \eqref{eq:zetabd} and Fatou's lemma, together with the
fact that, thanks to the continuity of $\mathcal{F}$,
(a subsequence, not indicated, of) $\{\mathcal{F}({\hat{\zeta}}^{\Lambda,1})\}_{\Lambda>0}$ converges
to $\mathcal{F}(\hat{\zeta}^{\!\!~1}) = \mathcal{F}(\hat\psi^0)$ a.e. on $\Omega \times D$, we also have that
\begin{equation}\label{F-stability}
\int_{\Omega \times D} M \mathcal{F}(\hat\psi^0) \dq \dx
\leq \int_{\Omega \times D} M \mathcal{F}(\hat\psi_0) \dq \dx.
\end{equation}

Next, we note that from \eqref{sqrtpsi-zeta} we have that, as $\Lambda \rightarrow \infty$,
\begin{equation}\label{strongpsi-zeta}
M^{\frac{1}{2}}\, \sqrt{\hat{\zeta}^{\Lambda,1}} \rightarrow M^{\frac{1}{2}}\,\sqrt{\hat{\zeta}^{~\!\!1}}\qquad \mbox{strongly in $L^2(\Omega \times D)$}.
\end{equation}
We shall use \eqref{strongpsi-zeta} to deduce weak convergence of the sequences of $\xt$ and $\qt$
gradients of $\hat\zeta^{\Lambda,1}$. We proceed as in the proof of
Lemma \ref{conv}. The bound on the third term on the left-hand side
of \eqref{eq:zetabd} implies the existence of a subsequence (not indicated) and an element
$\gt \in \Lt^2(\Omega \times D)$, such that
\begin{equation}\label{square}
M^{\frac{1}{2}}\, \nabq \sqrt{{\hat{\zeta}}^{\Lambda,1}}
\rightarrow \gt \qquad
\mbox{weakly in $\Lt^2(\Omega \times D)$.}
\end{equation}
Proceeding as in (\ref{derivid})--(\ref{last}) in the proof of Lemma \ref{conv}
with $\hat{\psi}^n_{\epsilon,L,\delta}$, $\hat{\psi}^n_{\epsilon,L}$ and $\delta
\rightarrow 0_+$ replaced
by $\sqrt{\hat {\zeta}^{\Lambda,1}}$, $\sqrt{\hat{\zeta}^{1}}$ and
$\Lambda \rightarrow \infty$, respectively;
we obtain
the weak convergence
result:
\begin{subequations}
\begin{align}\label{xlimitL2}
M^{\frac{1}{2}}\, \nabq\sqrt{\hat{\zeta}^{\Lambda,1}} \rightarrow  M^{\frac{1}{2}}\, \nabq\sqrt{\hat{\zeta}^{\!\!~1}}\quad \mbox{weakly in $L^2(\Omega \times D)$},
\end{align}
and similarly for the $\xt$ gradient
\begin{align}\label{qlimitL2}
M^{\frac{1}{2}}\, \nabx\sqrt{\hat{\zeta}^{\Lambda,1}}\rightarrow  M^{\frac{1}{2}}\, \nabx\sqrt{\hat{\zeta}^{\!\!~1}}\quad \mbox{weakly in $L^2(\Omega \times D)$,}
\end{align}
\end{subequations}
as $\Lambda \rightarrow \infty$. Then inequality
(\ref{eq:zetabd}), (\ref{xlimitL2},b) and the
 weak lower-semicontinuity of the $L^2(\Omega \times D)$ norm imply that
\begin{align}\label{zeta-space-bound-d}
&4\,\Delta t \int_{\Omega \times D} M\,\big|\nabx \sqrt{\hat{\zeta}^{\!\!~1}}\big|^2 \dq \dx 
+\, 4\,\Delta t \int_{\Omega \times D} M\,\big|\nabq \sqrt{\hat{\zeta}^{\!\!~1}}\big|^2 \dq \dx
\leq \int_{\Omega \times D} M \mathcal{F}(\hat\psi_0) \dq \dx.
\end{align}

After these preparations, we are now ready to state the central result of this subsection. Before we
do so, a comment is in order. Strictly speaking, we should have written $\hat\psi^0_{\Delta t}$ instead of
$\hat\psi^0$ in our definition \eqref{psizerodef}, as $\hat\psi^0$ depends on the choice of $\Delta t$.
For notational simplicity, we prefer the more compact notation, $\hat\psi^0$, with the dependence of
$\hat\psi^0$ on $\Delta t$ implicitly understood; we shall only write $\hat\psi^0_{\Delta t}$, when
it is necessary to emphasize the dependence of $\Delta t$. Of course, $\hat\psi_0$ is independent of $\Delta t$.

Next we shall show that, with our definition of $\hat\psi^0$, the properties listed under \eqref{inidata-1}
hold, together with additional properties that we extract from \eqref{psizerodef}, which were not stated
in \eqref{inidata-1} as they will only be required later, in our passage to the limit with $L\rightarrow \infty$
(and thereby $\Delta t \rightarrow 0_+$, according to $\Delta t = o(L^{-1})$,) in the next subsection.

\begin{lemma}\label{psi0properties}
The function $\hat\psi^0=\hat\psi^0_{\Delta t}$ defined by \eqref{psizerodef} has the following properties:
\begin{itemize}
\item[\ding{202}]  $~~~\hat\psi^0 \in \hat{Z}_1$;
\item[\ding{203}]  $~~~\displaystyle \int_{\Omega \times D} M \mathcal{F}(\hat\psi^0)\dq \dx \leq \int_{\Omega \times D} M \mathcal{F}(\hat\psi_0)\dq \dx$;
\item[\ding{204}]  $~~~4\,\Delta t \,\displaystyle \int_{\Omega \times D} M\,\left( |\nabx \sqrt{{\hat{\psi}^0}}|^2 + |\nabq \sqrt{{\hat{\psi}^0}}|^2\right) \dq \dx \leq \int_{\Omega \times D} M \mathcal{F}(\hat\psi_0) \dq \dx$;
\item[\ding{205}]
    $~~~\lim_{\Delta t \rightarrow 0_+} \hat\psi^0 = \hat\psi_0$, weakly in
    $L^1_{
    M}(\Omega \times D)$;
\item[\ding{206}]
    $~~~\lim_{\Delta t \rightarrow 0_+} \beta^L(\hat\psi^0) = \hat\psi_0$, weakly in
    $L^1_{
    M}(\Omega \times D)$.
\end{itemize}
\end{lemma}

\begin{proof} ~
\begin{itemize}
\item[\ding{202}] This property is an immediate consequence of \eqref{nonnegativity} and the definition (\ref{hatZ}) 
of 
$\hat Z_1$.
\item[\ding{203}] This property was established in \eqref{F-stability} above.
\item[\ding{204}]
The inequality follows by using \eqref{psizerodef}
in the left-hand side of \eqref{zeta-space-bound-d}.

\item[\ding{205}] We begin by noting that an argument, completely analogous to (but simpler than) the one in
Section \ref{sec:time-psia} that resulted in \eqref{psi-time-bound}, applied to (\ref{zeta-eq})
now, yields
\begin{align*}
\left|\int_{\Omega \times D} M \,\frac{\hat\zeta^{\Lambda,1} - \hat\zeta^{\Lambda,0}}{\Delta t}\, \hat\varphi
\dq \dx \right|
&\leq
2 \left(\int_{\Omega \times D} M \left[|\nabx \sqrt{\hat\zeta^{\Lambda,1}}|^2 + |\nabq \sqrt{\hat\zeta^{\Lambda,1}}|^2\right] \dq \dx \right)^{\frac{1}{2}}\nonumber\\
&\qquad \qquad \times
\left(\int_\Omega \left[\|\nabx \hat\varphi\|^2_{L^\infty(D)} + \|\nabq \hat\varphi\|^2_{L^\infty(D)}\right]  \dx \right)^{\frac{1}{2}},
\end{align*}
for all $\hat\varphi \in H^1(\Omega;L^\infty(D))\cap L^2(\Omega;W^{1,\infty}(D))$. On noting \eqref{eq:zetabd}
we deduce that
\begin{align}\label{firstbound}
&\left|\int_{\Omega \times D} M \,(\hat\zeta^{\Lambda,1} - \hat\zeta^{\Lambda,0})\, \hat\varphi \dq \dx \right|\nonumber\\
&\qquad \leq
(\Delta t)^{\frac{1}{2}} \left(\int_{\Omega \times D} M \mathcal{F}(\hat\psi_0) \dq \dx \right)^{\frac{1}{2}}
\!\!\left(\int_\Omega \left[\|\nabx \hat\varphi\|^2_{L^\infty(D)} + \|\nabq \hat\varphi\|^2_{L^\infty(D)}\right]  \dx \right)^{\frac{1}{2}}
\end{align}
for all $\hat\varphi \in H^1(\Omega;L^\infty(D))\cap L^2(\Omega;W^{1,\infty}(D))$.
As the right-hand side of \eqref{firstbound} is independent of $\Lambda$, we can pass
to the limit $\Lambda \rightarrow \infty$ on both sides of \eqref{firstbound}, using the strong
convergence of $\hat\zeta^{\Lambda,1}$ to $\hat\psi^0$ in $L^1_M(\Omega \times D)$ as $\Lambda \rightarrow \infty$
(see \eqref{sqrtpsi-zetaaa} and the definition of \eqref{psizerodef}) together with the strong
convergence of $\hat\zeta^{\Lambda,0}=\beta^\Lambda(\hat\psi_0)$ to $\hat\psi_0$ in $L^1_M(\Omega \times D)$, as $\Lambda \rightarrow \infty$,
with $\Delta t$ kept fixed. We deduce that
\begin{align}\label{secondbound}
&\left|\int_{\Omega \times D} M \,(\hat\psi^0 - \hat\psi_0)\, \hat\varphi \dq \dx \right|\nonumber\\
&\qquad \leq
(\Delta t)^{\frac{1}{2}} \left(\int_{\Omega \times D} M \mathcal{F}(\hat\psi_0) \dq \dx \right)^{\!\frac{1}{2}}
\!\!\left(\int_\Omega \left[\|\nabx \hat\varphi\|^2_{L^\infty(D)} + \|\nabq \hat\varphi\|^2_{L^\infty(D)}\right]  \dx \right)^{\!\frac{1}{2}}
\end{align}
for all $\hat\varphi \in H^1(\Omega;L^\infty(D))\cap L^2(\Omega;W^{1,\infty}(D))$ and therefore in particular
for all $\hat\varphi \in H^s(\Omega \times D)$ with $s > 1+ \frac{1}{2}(K+1)d$.

As the last two factors on the right-hand side of \eqref{secondbound} are independent of $\Delta t$,
we can pass to the limit $\Delta t \rightarrow 0_+$ on both sides of \eqref{secondbound}
to deduce that $\hat\psi^0= \hat\psi^0_{\Delta t}$ converges to $\hat\psi_0$ weakly
in $M^{-1}H^s(\Omega \times D)'$ for $s > 1+ \frac{1}{2}(K+1)d$, as $\Delta t \rightarrow 0_+$.

Noting \eqref{F-stability}
and the fact that $\mathcal{F}(r)/r \rightarrow \infty$ as $r\rightarrow \infty$, we deduce from de le Vall\'ee Poussin's
theorem that the family $\{\hat\psi^0_{\Delta t}\}_{\Delta t>0}$ is uniformly integrable in $L^1_M(\Omega \times D)$.
Hence, by the Dunford--Pettis theorem, the family $\{\hat\psi^0_{\Delta t}\}_{\Delta t>0}$ is weakly
relatively compact in $L^1_M(\Omega \times D)$. Consequently, one can extract a subsequence
$\{\hat\psi^0_{\Delta t_k}\}_{k=1}^\infty$ that converges
weakly in $L^1_M(\Omega \times D)$; however the uniqueness of the weak limit
together with the weak convergence of the (entire) sequence $\hat\psi^0= \hat\psi^0_{\Delta t}$ to $\hat\psi_0$ in
$M^{-1}H^s(\Omega \times D)'$, $s > 1+ \frac{1}{2}(K+1)d$, as $\Delta t \rightarrow 0_+$ then implies that the (entire)
sequence
\[\mbox{$\hat\psi^0= \hat\psi^0_{\Delta t}$ converges to $\hat\psi_0$ weakly in $L^1_M(\Omega \times D)$,
as $\Delta t \rightarrow 0_+$},\]
on noting that $L^1_M(\Omega\times D)$ 
is continuously embedded into $M^{-1}H^s(\Omega \times D)'$ for, again, $s > 1+ \frac{1}{2}(K+1)d$ (cf. the discussion
following Theorem \ref{thm:Dubinski}).

\item[\ding{206}]
It follows from $\hat \psi^0 \in \hat Z_1$ and (\ref{betaLa}) that
\begin{align}
0 \leq
\int_{\hat \psi^0 \geq L} M\,L \dq \dx \leq
\int_{\Omega \times D} M \, \beta^L(\hat \psi^0) \dq  \dx \leq
\int_{\Omega \times D} M \, \hat \psi^0 \dq  \dx \leq
|\Omega|.
\label{betapsi0bd}
\end{align}
On noting that ${\cal F}$ is non-negative and monotonically increasing on $[1,\infty)$,
and that ${\cal F}(s) \in [0,1]$ for $s \in [0,1]$, we deduce that
\begin{align}
&\int_{\Omega \times D} M\,{\cal F}([\hat \psi^0 -L]_+) \dq \dx
\nonumber \\
& \hspace{0.5in} = \int_{\hat \psi^0 \in [L,L+1)}
M\,{\cal F}([\hat \psi^0 -L]_+) \dq \dx
+\int_{\hat \psi^0 \geq L+1} M\,{\cal F}([\hat \psi^0 -L]_+) \dq \dx
\nonumber \\
& \hspace{0.5in} \leq \int_{\Omega \times D}
M \dq \dx
+\int_{\Omega \times D} M\,{\cal F}(\hat \psi^0) \dq \dx
\leq C.
\label{FpsiLbd}
\end{align}
Applying the logarithmic Young's inequality (\ref{logY}), we have that
\begin{align}
(\log L)\,[\hat \psi^0 -L]_+ \leq {\cal F}([\hat \psi^0 -L]_+) + L.
\label{logYL}
\end{align}
The bounds (\ref{betapsi0bd}), (\ref{FpsiLbd}) and
(\ref{logYL}) then imply
\begin{align}
\int_{\Omega \times D} M\,[\hat \psi^0 -L]_+ \dq \dx
&=
\int_{\hat \psi^0 \geq L} M\,[\hat \psi^0 -L]_+ \dq \dx
\nonumber \\
&\leq \frac{1}{\log L}
\left[\int_{\hat \psi^0 \geq L}  M\,{\cal F}([\hat \psi^0 -L]_+) \dq \dx +
\int_{\hat \psi^0 \geq L}  M\,L \dq \dx \right]
\nonumber \\
& \leq  \frac{C}{\log L}.
\label{Lplusbd}
\end{align}

Hence for any $\hat \varphi \in L^\infty(\Omega \times D)$, we have from
(\ref{Lplusbd}), on recalling the relationship $\Delta t =o(L^{-1})$, that
$\hat \psi^0 = \hat \psi^0_{\Delta t}$ satisfies
\begin{align}
\lim_{\Delta t \rightarrow 0_+} \left| \int_{\Omega \times D} M\,
(\hat \psi^0-\beta^L(\hat \psi^0)) \,\hat \varphi \dq \dx \right|
& =
\lim_{\Delta t \rightarrow 0_+} \left| \int_{\Omega \times D} M\,
[\hat \psi^0-L]_+ \,\hat \varphi \dq \dx \right|
\nonumber \\
&\leq \left(\lim_{\Delta t \rightarrow 0_+} \int_{\Omega \times D} M\,
[\hat \psi^0-L]_+ \dq \dx \right)
\|\hat \varphi\|_{L^\infty(\Omega \times D)}
\nonumber \\
&=0.
\label{psioLconv}
\end{align}
Therefore, similarly to (\ref{secondbound}), we have that
the sequence $\{\hat \psi^0_{\Delta t}
-\beta^L(\hat \psi^0_{\Delta t})\}_{\Delta t >0}$
converges to zero weakly in
$M^{-1}H^s(\Omega \times D)'$ for $s > \frac{1}{2}(K+1)d$, as $\Delta t \rightarrow 0_+$.

Noting \eqref{FpsiLbd}
and the fact that $\mathcal{F}(r)/r \rightarrow \infty$ as $r\rightarrow \infty$, we deduce from de le Vall\'ee Poussin's
theorem that the family $\{\hat\psi^0_{\Delta t}-\beta^L(\hat\psi^0_{\Delta t})\}_{\Delta t>0} \equiv \{[\hat\psi^0_{\Delta t}-L]_+\}_{\Delta t>0}$
is uniformly integrable in $L^1_M(\Omega \times D)$.
Hence, we can proceed as for the sequence $\{\hat\psi^0_{\Delta t}\}_{\Delta t>0}$
in the proof of \ding{205} to show that
the (entire)
sequence
\[\mbox{$\hat\psi^0-\beta^L(\hat\psi^0)= \hat\psi^0_{\Delta t}
- \beta^L(\hat\psi^0_{\Delta t})$ converges to $0$ weakly in $L^1_M(\Omega \times D)$,
as $\Delta t \rightarrow 0_+$},\]
on noting that $L^1_M(\Omega\times D)$ 
is continuously embedded into $M^{-1}H^s(\Omega \times D)'$ for $s > \frac{1}{2}(K+1)d$ (cf.\ the discussion
following Theorem \ref{thm:Dubinski}).
Hence, we have proved the desired result.
\end{itemize}

That completes the proof of the lemma.
\end{proof}

With the information contained under item 3 in Lemma \ref{psi0properties}, we can now return to
the inequality \eqref{eq:energy-u+psi-final6}, and supplement it with additional
bounds, in the sixth and seventh term on the left-hand side. The first additional bound can
be seen as the analogue of \eqref{bound-a}:

\begin{align}\label{bound-aaa}
&4\int_0^T \int_{\Omega\times D} M\left[|\nabx\sqrt{\psia^{\Delta t,-}}|^2 +
|\nabq\sqrt{\psia^{\Delta t,-}}|^2\right] \dq \dx \dt  \nonumber\\
&\qquad = 4\Delta t  \int_{\Omega\times D} M\left[|\nabx\sqrt{\beta^L(\hat\psi^0)}|^2 +
|\nabq\sqrt{\beta^L(\hat\psi^0})|^2\right] \dq \dx
\nonumber\\
&\qquad \qquad +\, 4\int_0^{T-\Delta t} \int_{\Omega\times D} M\left[|\nabx\sqrt{\psia^{\Delta t,+}}|^2 +
|\nabq\sqrt{\psia^{\Delta t,+}}|^2\right] \dq \dx \dt\nonumber\\
&\qquad \leq 4\Delta t  \int_{\Omega\times D} M\left[|\nabx\sqrt{\hat\psi^0}|^2 +
|\nabq\sqrt{\hat\psi^0}|^2\right] \dq \dx
\nonumber\\
&\qquad \qquad +\, 4\int_0^{T-\Delta t} \int_{\Omega\times D} M\left[|\nabx\sqrt{\psia^{\Delta t,+}}|^2 +
|\nabq\sqrt{\psia^{\Delta t,+}}|^2\right] \dq \dx \dt\nonumber\\
& \qquad \leq \int_{\Omega \times D}M \mathcal{F}(\hat\psi_0)\dq \dx 
+\, 4 \int_0^{T} \int_{\Omega\times D} M\left[|\nabx\sqrt{\psia^{\Delta t,+}}|^2 +
|\nabq\sqrt{\psia^{\Delta t,+}}|^2\right] \dq \dx \dt\nonumber\\
&\qquad\leq  C_\star,
\end{align}
where in the transition to the last line we used \eqref{inidata-1} and the bounds on the sixth and seventh
term in \eqref{eq:energy-u+psi-final6}; here and henceforth $C_\star$ signifies a generic positive constant,
independent of $L$ and $\Delta t$. On combining \eqref{bound-aaa} with our previous
bounds on the sixth and seventh term in \eqref{eq:energy-u+psi-final6}, we deduce that
\begin{equation}\label{bound-bbb}
\hspace{-1cm}4\int_0^T \int_{\Omega\times D} M\left[|\nabxtt\sqrt{\psia^{\Delta t,\pm}}|^2 +
|\nabq\sqrt{\psia^{\Delta t,\pm}}|^2\right] \dq \dx \dt  \leq C_\star.
\end{equation}
It remains to derive an analogous bound on $\psia^{\Delta t}$.
To this end, let $n \in \{1,\dots, N\}$
and consider $t \in (t_{n-1},t_n)$; we recall that
\begin{equation}
\psia(\cdot,\cdot, t)=\,\frac{t-t_{n-1}}{\Delta t}\,
\psia^{\Delta t,+}(\cdot,\cdot,t) + \frac{t_n-t}{\Delta t}\,\psia^{\Delta
t, -}(\cdot,\cdot,t). \label{psilin}
\end{equation}
For ease of exposition we shall write
\[\gamma_+:= \frac{t-t_{n-1}}{\Delta t} \qquad \mbox{and}\qquad
\gamma_{-}:=\frac{t_n-t}{\Delta t}\]
in the argument that follows, noting that $\gamma_{+} + \gamma_{-} = 1$
and both $\gamma_{+}$ and
$\gamma_{-}$ are positive. The functions $t \in (t_{n-1},t_n) \mapsto
\psia^{\Delta t,\pm}(\cdot,\cdot,t)$ are constant in time and
$\psia^{\Delta t,(\pm)}(\xt,\qt,t)\geq 0$ on $\Omega \times D \times
(t_{n-1},t_n)$, $n \in \{1,\dots, N\}$. For any $\alpha \in (0,1)$ we have
that
\begin{eqnarray*}
\frac{|\nabx \psia^{\Delta t}|^2}{\psia^{\Delta t}+\alpha}
&=& \frac{|\gamma_+\,\nabx \psia^{\Delta t,+} + \gamma_{-}\,\nabx
\psia^{\Delta t,-}|^2}{\gamma_+\,(\psia^{\Delta t,+}+\alpha) +
\gamma_{-}(\,\psia^{\Delta t,-}+\alpha)}
\nonumber\\
&\leq&  2\,\frac{\gamma_+^2 |\nabx \psia^{\Delta t,+}|^2 + \gamma_{-}^2
|\nabx \psia^{\Delta t,-}|^2}{\gamma_+\,(\psia^{\Delta t,+}+\alpha) +
\gamma_{-}\,(\psia^{\Delta t,-}+\alpha)}
\nonumber\\
&\leq&  2\,\frac{\gamma_{+} |\nabx \psia^{\Delta t,+}|^2}{\psia^{\Delta
t,+}+\alpha}
+ 2\,\frac{\gamma_{-} |\nabx \psia^{\Delta t,-}|^2}{\psia^{\Delta
t,-}+\alpha}.\nonumber
\end{eqnarray*}
Hence, on bounding $\gamma^{\pm}$ by 1, we deduce that
\begin{equation}\label{last-psia}
\frac{|\nabx \psia^{\Delta t}|^2}{\psia^{\Delta t}+\alpha}
\leq 2\,\frac{|\nabx \psia^{\Delta t, +}|^2}{\psia^{\Delta t, +}+\alpha} +
2\,\frac{|\nabx \psia^{\Delta t, -}|^2}{\psia^{\Delta t, -}+\alpha},
\end{equation}
for all $(\xt,\qt,t) \in \Omega \times D\times (t_{n-1},t_n)$,
$n=1,\dots,N$, and all $\alpha \in (0,1)$. On multiplying
\eqref{last-psia} by $M$, integrating over $\Omega \times D \times
(t_{n-1},t_n)$, summing over $n=1,\dots, N$,
and passing to the limit $\alpha \rightarrow 0_+$ using the monotone
convergence theorem, we deduce that
\begin{eqnarray*}
4\int_0^T\int_{\Omega \times D} M\big|\nabx \sqrt{\psia^{\Delta
t}}\big|^2\dq\dx \dt
&\leq& 2 \left[4\int_0^T \int_{\Omega \times D}
M\,|\nabx\sqrt{\psia^{\Delta t,+}}|^2 \dq \dx \dt\right.\\
&&+  \left.4\int_0^T \int_{\Omega \times D} M\,|\nabx\sqrt{\psia^{\Delta
t,-}}|^2\dq \dx \dt\right].
\end{eqnarray*}
Analogously,
\begin{eqnarray*}
4\int_0^T\int_{\Omega \times D} M\big|\nabq \sqrt{\psia^{\Delta
t}}\big|^2\dq\dx \dt
&\leq& 2 \left[4\int_0^T \int_{\Omega \times D}
M\,|\nabq\sqrt{\psia^{\Delta t,+}}|^2 \dq \dx \dt\right.\\
&&+  \left.4\int_0^T \int_{\Omega \times D} M\,|\nabq\sqrt{\psia^{\Delta
t,-}}|^2\dq \dx \dt\right].
\end{eqnarray*}
Summing the last two inequalities and recalling \eqref{bound-bbb}, we then
deduce that
\begin{equation}\label{bound-ccc}
\hspace{-1cm}4\int_0^T \int_{\Omega\times D}
M\left[|\nabx\sqrt{\psia^{\Delta t}}|^2 +
|\nabq\sqrt{\psia^{\Delta t}}|^2\right] \dq \dx \dt  \leq C_\star,
\end{equation}
where, again, $C_\ast$ denotes a generic positive constant independent of
$L$ and $\Delta t$.

On noting that $(1+|\qt|)^{2\vartheta} \leq 2^{2\vartheta-1} (1+|\qt|^{2\vartheta})
\leq 2^{2\vartheta-1} K^{\theta-1} (1+ \sum_{i=1}^K |\qt_i|^{2\vartheta})$, it follows from
(\ref{additional2}) and (\ref{qt2bd}) that for a.e. $t \in [0,T]$
\begin{align}
\int_{\Omega \times D} M\,(1+|\qt|)^{2 \vartheta} \,\psia^{\Delta t,\pm}(t) \dq \dx \leq C_\star.
\label{qthetabd}
\end{align}

Finally, on combining \eqref{bound-bbb}, \eqref{bound-ccc} and (\ref{qthetabd}) with
\eqref{eq:energy-u+psi-final6} we arrive at the following bound, which represents the starting
point for the convergence analysis that will be developed in the next subsection.

With $\sigma > 1+ \frac{1}{2}d$ 
and
$s > 1 + \frac{1}{2}(K+1)d$, we have that:
\begin{align}\label{eq:energy-u+psi-final6a}
&\mbox{ess.sup}_{t \in [0,T]}\|\uta^{\Delta t(,\pm)}(t)\|^2 + \frac{1}{\Delta t} \int_0^T \|\uta^{\Delta t,+} - \uta^{\Delta t,-}\|^2
\dd s 
+\,\int_0^T \|\nabxtt \uta^{\Delta t(,\pm)}(s)\|^2 \dd s
\nonumber \\
& \qquad
+  \mbox{ess.sup}_{t \in [0,T]}
\int_{\Omega \times D}\!\! M\, \mathcal{F}(\psia^{\Delta t(, \pm)}(t)) \dq \dx
+\, \frac{1}{\Delta t\,L}
\int_0^T\!\! \int_{\Omega \times D}\!\! M (\psia^{\Delta t,+} - \psia^{\Delta t,-})^2 \dq \dx \dd s
\nonumber \\
& \qquad
+ \mbox{ess.sup}_{t \in [0,T]}
\int_{\Omega \times D} M\,(1+|\qt|)^{2 \vartheta} \,\psia^{\Delta t(,\pm)}(t) \dq \dx
\nonumber \\
& \qquad +\, \int_0^T\!\! \int_{\Omega \times D} M\,
\big|\nabx \sqrt{\psia^{\Delta t(,\pm)}} \big|^2 \dq \dx \dd s
+\, \int_0^T\!\! \int_{\Omega \times D}M\,\big|\nabq \sqrt{\psia^{\Delta t(,\pm)}}\big|^2 \,\dq \dx \dd s\nonumber\\
&\qquad
+\, \int_0^T\left\|\frac{\partial \uta^{\Delta t}}{\partial t}\right\|^2_{V_\sigma'}\dt + \int_0^T\left\|M\frac{\partial \psia^{\Delta t}}{\partial t}\right\|^2_{H^s(\Omega \times D)'}\dt \leq C_\star.
\end{align}

\subsection{Passage to the limit $L\rightarrow \infty$}
\label{passage}

We are now ready to pass to the limit and prove the central result of the paper. In what follows,
$\langle \cdot , \cdot \rangle_{H^s(\Omega)}$ denotes the duality pairing between $H^s(\Omega)'$ and
$H^s(\Omega)$ relative to the pivot space $L^2(\Omega)$ with inner product $(\cdot,\cdot)$; similarly,
$\langle M \cdot , \cdot \rangle_{H^s(\Omega \times D)}$ denotes the duality pairing between
$M^{-1}H^s(\Omega \times D)'$ and $H^s(\Omega \times D)$ relative to the pivot space
$L^2_M(\Omega \times D)$ with inner product 
\[(\hat \phi_1,\hat \phi_2)_M := \int_{\Omega \times D} M\, \hat \phi_1\,\hat \phi_2 \dq \dx;\] 
and $\langle \cdot  , \cdot \rangle_{V_\sigma}$
denotes the duality pairing between the spaces $\Vt_{\sigma}'$ and $\Vt_\sigma$ relative to the pivot space $\Ht$.

\begin{theorem}
\label{convfinal} Suppose that the assumptions \eqref{inidata} and the condition \eqref{LT},
relating $\Delta t$ to $L$, hold. Then,
there exists a subsequence of $\{(\utae^{\Delta t}, \hpsiae^{\Delta t})\}_{L >1}$ (not indicated)
with $\Delta t = o(L^{-1})$, and a pair of functions $(\ute, \hat\psi_\epsilon)$ such that
\[\ute \in L^{\infty}(0,T;\Lt^2(\Omega))\cap L^{2}(0,T;\Vt) \cap H^1(0,T;\Vt'_\sigma),\quad \sigma > 1+ \textstyle{\frac{1}{2}}d,
\]
and
\[\hat\psi_\epsilon \in L^1(0,T;L^1_{(1+|\qt|)^{2 \vartheta}M}(\Omega \times D))\cap
H^1(0,T; M^{-1}H^s(\Omega \times D)'),
\quad s>1 + \textstyle{\frac{1}{2}}(K+1)d,\]
with $\hat\psi_\epsilon \geq 0$ a.e. on $\Omega \times D \times [0,T]$,
\begin{equation}\label{mass-conserved}
\int_D M(\qt)\,\hat\psi_\epsilon(\xt,\qt,t) \dq = 1\quad \mbox{for a.e. $(x,t) \in \Omega \times [0,T]$},
\end{equation}
and hence $\hat\psi_\epsilon \in L^\infty(0,T; L^1_M(\Omega \times D))$; and finite relative entropy and Fisher information, with
\begin{equation}\label{relent-fisher}
\mathcal{F}(\hat\psi_\epsilon) \in L^\infty(0,T;L^1_M(\Omega\times D))\quad \mbox{and}\quad \sqrt{\hat\psi_\epsilon} \in L^{2}(0,T;H^1_M(\Omega \times D)),
\end{equation}
whereby $\hat\psi_\epsilon \in L^\infty(0,T; L^1_{(1+|\qt|)^{2 \vartheta}M}(\Omega \times D))$;
such that, as $L\rightarrow \infty$ (and thereby $\Delta t \rightarrow 0_+$),
\begin{subequations}
\begin{alignat}{2}
&\utae^{\Delta t (,\pm)} \rightarrow \ute \qquad &&\mbox{weak* in }
L^{\infty}(0,T;{\Lt}^2(\Omega)), \label{uwconL2a}\\
\bet
&\utae^{\Delta t (,\pm)} \rightarrow \ute \qquad &&\mbox{weakly in }
L^{2}(0,T;\Vt), \label{uwconH1a}\\
\bet
&\utae^{\Delta t (,\pm)} \rightarrow \ute \qquad &&\mbox{strongly in }
L^{2}(0,T;\Lt^{r}(\Omega)), \label{usconL2a}\\
\bet
&~\frac{\partial \utae^{\Delta t}}{\partial t} \rightarrow  \frac{\partial \ute}{\partial t}
\qquad &&\mbox{weakly in }
L^2(0,T;\Vt_\sigma'), \label{utwconL2a}
\end{alignat}
\end{subequations}
where $r \in [1,\infty)$ if $d=2$ and $r \in [1,6)$ if $d=3$;
and
\begin{subequations}
\begin{alignat}{2}
\bet
&M^{\frac{1}{2}}\,\nabx \sqrt{\hpsiae^{\Delta t(,\pm)}} \rightarrow M^{\frac{1}{2}}\,\nabx \sqrt{\hat\psi_\epsilon}
&&\quad \mbox{weakly in } L^{2}(0,T;\Lt^2(\Omega\times D)), \label{psiwconH1a}\\
\bet
&M^{\frac{1}{2}}\,\nabq \sqrt{\hpsiae^{\Delta t(,\pm)}} \rightarrow M^{\frac{1}{2}}\,\nabq \sqrt{\hat\psi_\epsilon}
&&\quad \mbox{weakly in } L^{2}(0,T;\Lt^2(\Omega\times D)), \label{psiwconH1xa}\\
\bet
&~~~~~~~~~M\,\frac{\partial \hpsiae^{\Delta t}} {\partial t} \rightarrow
M\,\frac{\partial \hat\psi_\epsilon}{\partial t}
&&\quad \mbox{weakly in }
L^2(0,T;H^s(\Omega\times D)'), \label{psitwconL2a}\\
\bet
&~~~~~~~~~~~\hpsiae^{\Delta t (,\pm)} \rightarrow \hat\psi_\epsilon
&&\quad \mbox{strongly in }
L^{p}(0,T;L^{1}_{(1+|\qt|)^{2 \vartheta}M}(\Omega\times D)),\label{psisconL2a}
\end{alignat}
for all $p \in [1,\infty)$; and,
\begin{alignat}{2}
\bet
&\nabx\cdot\sum_{i=1}^K\Ctt_i(M\,\hpsiae^{\Delta t (,\pm)}) \rightarrow \nabx \cdot\sum_{i=1}^K\Ctt_i(M\,\hat\psi_\epsilon)
&&\quad \mbox{weakly in }
L^{2}(0,T;\Vt_\sigma').\label{CwconL2a}
\end{alignat}
\end{subequations}
The pair of functions $(\ut_\epsilon,\hat\psi_\epsilon)$ is a global weak solution to problem (P$_\epsilon$), in
the sense that
\begin{align}\label{equnconP}
&\displaystyle\int_{0}^{T} \left\langle \frac{\partial \ut_\epsilon}{\partial t},
\wt \right\rangle_{\!\!V_\sigma}
\dt
+ \int_{0}^T \int_{\Omega}
\left[ \left[ (\ut_\epsilon \cdot \nabx) \ut_\epsilon \right]\,\cdot\,\wt
+ \nu \,\nabxtt \ut_\epsilon
:
\wnabtt \right] \dx \dt
\nonumber
\\
&\hspace{0.5in}  =\int_{0}^T
\left[ \langle \ft, \wt\rangle_{V}
- k\,\sum_{i=1}^K \int_{\Omega}
\Ctt_i(M\,\hat\psi_\epsilon): \nabxtt
\wt \dx \right] \dt
\qquad \forall \wt \in L^2(0,T;\Vt_\sigma)
\end{align}
and
\begin{align}
&\int_{0}^T \left\langle M\,\frac{ \partial \hat\psi_\epsilon}{\partial t} ,
\hat \varphi \right\rangle_{\!\!H^s(\Omega \times D)} \dt \label{eqpsinconP}
+ \int_{0}^T \int_{\Omega \times D} M\,\left[
\epsilon\, \nabx \hat\psi_\epsilon - \ut_\epsilon \,\hat\psi_\epsilon \right]\cdot\, \nabx
\hat \varphi
\,\dq \dx \dt
\nonumber \\
& \qquad  +
\frac{1}{2\,\lambda}
\int_{0}^T \int_{\Omega \times D} M\,\sum_{i=1}^K
 \,\sum_{j=1}^K A_{ij}\,\nabqj \hat\psi_\epsilon
\cdot\, \nabqi
\hat \varphi
\,\dq \dx \dt
\nonumber \\
& \qquad -
\int_{0}^T \int_{\Omega \times D} M\,\sum_{i=1}^K
[\sigtt(\ut_\epsilon)
\,\qt_i]\,
\hat\psi_\epsilon \,\cdot\, \nabqi
\hat \varphi
\,\dq \dx \dt = 0
\qquad \forall \hat \varphi \in L^2(0,T;H^s(\Omega\times D)).
\end{align}

The initial conditions $\ut_\epsilon(\cdot,0) = \ut_0(\cdot)$ and
$\hat\psi_\epsilon(\cdot,\cdot,0) = \hat\psi_0(\cdot,\cdot)$ are satisfied in the sense of weakly continuous
functions, in the function spaces $C_w([0,T];\Ht)$ and
$C_w([0,T];L^1_{M}(\Omega \times D))$, respectively.

The weak solution $(\ut_\epsilon,\hat\psi_\epsilon)$ satisfies the following energy inequality for $t \in [0,T]$:
\begin{align}\label{eq:energyest}
&\|\ut_\epsilon(t)\|^2 + \nu \int_0^t \|\nabxtt \ut_\epsilon(s)\|^2 \dd s
+ \,2k\int_{\Omega \times D}\!\! M \mathcal{F}(\hat\psi_\epsilon(t)) \dq \dx
\nonumber \\
&\qquad \qquad +\, 8k\,\varepsilon \int_0^t \int_{\Omega \times D} M
|\nabx \sqrt{\hat\psi_\epsilon} |^2 \dq \dx \dd s
+\, \frac{2a_0 k}{\lambda}  \int_0^t \int_{\Omega \times D}M\, \,|\nabq \sqrt{\hat\psi_\epsilon}|^2 \,\dq \dx \dd s\nonumber\\
&\qquad \leq \|\ut_0\|^2 + \frac{1}{\nu}\int_0^t\|\ft(s)\|^2_{V'} \dd s + 2k \int_{\Omega \times D} M \mathcal{F}(\hat\psi_0) \dq \dx
\leq [{\sf B}(\ut_0,\ft, \hat\psi_0)]^2
,~~~~~~~~~
\end{align}
with $\mathcal{F}(s)= s(\log s - 1) + 1$, $s \geq 0$,
and $[{\sf B}(\ut_0,\ft, \hat\psi_0)]^2$ as defined in
{\rm (\ref{eq:energy-u+psi-final2})}.
\end{theorem}
\begin{proof} Since the proof is long, we have broken it up into a number of steps.

\smallskip

\textit{Step 1.} On noting the weak$^\ast\!$ compactness
of bounded balls in the Banach space $L^\infty(0,T;\Lt^2(\Omega))$ and recalling
the bound on
the first term on the left-hand side of \eqref{eq:energy-u+psi-final6a}, upon three successive
extractions of subsequences we deduce the existence of an unbounded
index set $\mathcal{L} \subset (1,\infty)$
such that each of the three sequences $\{\uta^{\Delta t(,\pm)}\}_{L\in \mathcal{L}}$
converges to its respective weak$^\ast$ limit in $L^\infty(0,T;\Lt^2(\Omega))$ as $L \rightarrow \infty$
with $L \in \mathcal{L}$. Thanks to (\ref{ulin},b),
\begin{equation}\label{connection}
 \int_0^T \|\uta^{\Delta t}(s) - \uta^{\Delta t,+}(s)\|^2 \dd s = {\textstyle{\frac{1}{3}}}
\int_0^T \|\uta^{\Delta t,+}(s) - \uta^{\Delta t,-}(s)\|^2 \dd s  \leq \textstyle{\frac{1}{3}}
C_\star \Delta t,
\end{equation}
where the last inequality is a consequence of the second bound in \eqref{eq:energy-u+psi-final6a}.
On passing to the limit $L \rightarrow \infty$ with $L \in \mathcal{L}$ and using \eqref{LT}
we thus deduce that
the weak$^\ast$ limits of the sequences $\{\uta^{\Delta t(,\pm)}\}_{L \in \mathcal{L}}$
coincide. We label this common limit by $\ute$; by construction then,
$\ute \in L^\infty(0,T;\Lt^2(\Omega))$. Thus we have shown \eqref{uwconL2a}.

Upon further successive extraction of subsequences from $\{\uta^{\Delta t(,\pm)}\}_{L\in \mathcal{L}}$,
and noting the bounds on the third and eighth term on the left-hand side of
\eqref{eq:energy-u+psi-final6a} the limits (\ref{uwconH1a},d) follow
directly from the weak compactness of bounded balls in the Hilbert spaces $L^2(0,T;\Vt)$ and $L^2(0,T;\Vt'_\sigma)$ and \eqref{uwconL2a} thanks to the uniqueness of limits of
sequences in the weak topology of $L^2(0,T;\Vt)$ and $L^2(0,T;\Vt'_\sigma)$, respectively.

By the Aubin--Lions--Simon compactness theorem (cf. \eqref{compact1}), we then deduce
\eqref{usconL2a} in the case of $\uta^{\Delta t}$ on noting the compact embedding of
$\Vt$ into $\Lt^r(\Omega)\cap \Ht$, with the values of $r$ as in the statement of the theorem.
In particular, with $r=2$, $\uta^{\Delta t} \rightarrow \ute$, strongly in $L^2(0,T;\Lt^2(\Omega))$.
Then, by the bound on the left-most term in \eqref{connection}, we deduce that
$\uta^{\Delta t,+}$ also converges to $\ute$, strongly in $L^2(0,T;\Lt^2(\Omega))$ as
$L \rightarrow \infty$ (and thereby $\Delta t \rightarrow 0_+$). Further,
by the bound on the middle term in \eqref{connection} we have that the same is true
of $\uta^{\Delta,-}$. Thus we have shown that the three sequences $\uta^{\Delta t(,\pm)}$ all converge to
$\ute$, strongly and $L^2(0,T;\Lt^2(\Omega))$. Since the sequences $\uta^{\Delta t(,\pm)}$ are
bounded in $L^2(0,T;\Ht^1(\Omega))$ (cf. the bound on the third term in \eqref{eq:energy-u+psi-final6a})
and strongly convergent in $L^2(0,T;\Lt^2(\Omega))$, we deduce from  \eqref{eqinterp} that
\eqref{usconL2a} holds, with the values of $r$ as in the statement of the theorem.
Thus we have proved (\ref{uwconL2a}--d).

\smallskip

\textit{Step 2.}
Dubinski{\u\i}'s theorem, with $\mathcal{A}_0$, $\mathcal{A}_1$ and $\mathcal{M}$ as
in the discussion following the statement of Theorem \ref{thm:Dubinski}, and selecting $p=1$
and $p_1=2$, 
imply that
%
\begin{align*}
&\left\{\varphi\,:\,[0,T] \rightarrow \mathcal{M}\,:\,
[\varphi]_{L^1(0,T;\mathcal M)} + \left\|\frac{{\rm d}\varphi}{{\rm d}t} \right\|_{L^{2}(0,T;\mathcal{A}_1)}
< \infty   \right\}
\\ & \hspace{2.5in}
\hookrightarrow\!\!\!\rightarrow L^1(0,T;\mathcal{A}_0)
= L^1(0,T;L^1_{(1+|\qt|)^{2 \vartheta}M}(\Omega\times D)).
\end{align*}
Using this compact embedding, together with the bounds on the sixth, the seventh, the eighth
and the last
term on the left-hand side of \eqref{eq:energy-u+psi-final6a}, in conjunction with
\eqref{additional1}, 
we deduce (upon extraction of a subsequence)
strong convergence of $\{\psia^{\Delta t}\}_{L>1}$ in $L^1(0,T; L^1_{(1+|\qt|)^{2 \vartheta}M}(\Omega \times D))$ to an element
$\hat\psi_\epsilon \in L^1(0,T; L^1_{(1+|\qt|)^{2 \vartheta}M}(\Omega \times D))$, as $L \rightarrow \infty$.

Thanks to the bound on the fifth term in \eqref{eq:energy-u+psi-final6a}, by the
Cauchy--Schwarz inequality and an argument similar to the one in \eqref{connection}, we have
\begin{align}
&\left(\int_0^T\!\!\! \int_{\Omega \times D}\!\! M\, (1+|\qt|)^{2\vartheta}\,|\psia^{\Delta t} - \psia^{\Delta t,\pm}|\dq \dx \dt
\right)^2
\nonumber \\
& \hspace{1in} \leq \frac{T\, |\Omega|}{3} \left(\int_D M \,(1+|\qt|)^{4\vartheta} \dq \right)\!\! \displaystyle \int_0^T\!\!\!
\displaystyle \int_{\Omega \times D}\!\! M\, (\psia^{\Delta t,+} - \psia^{\Delta t,-})^2
\dq \dx \dt\nonumber\\
\label{difference}
& \hspace{1in} \leq \textstyle \frac{1}{3}C_\star T\, |\Omega|\,\Delta t\, L
\displaystyle \left(\int_D M \,(1+|\qt|)^{4\vartheta} \dq \right).
\end{align}
On noting from (\ref{growth3}) that the integral over $D$ is finite and recalling \eqref{LT}, and using the triangle inequality in the $L^1(0,T;L^1_{(1+|\qt|)^{2 \vartheta}M}(\Omega \times D))$ norm, together
with \eqref{difference} and the strong  convergence of $\{\psia^{\Delta t}\}_{L>1}$ to
$\hat\psi_\epsilon$ in $L^1(0,T; L^1_{(1+|\qt|)^{2 \vartheta}M}(\Omega \times D))$,
we deduce, as $L \rightarrow \infty$,
strong convergence of $\{\psia^{\Delta t,\pm}\}_{L>1}$ in $L^1(0,T; L^1_{(1+|\qt|)^{2 \vartheta}M}(\Omega \times D))$ to the same element $\hat\psi_\epsilon \in L^1(0,T; L^1_{(1+|\qt|)^{2 \vartheta}M}(\Omega \times D))$. 
This completes the proof of \eqref{psisconL2a} for $p=1$.

From 
%
the sixth bound in (\ref{eq:energy-u+psi}) we have that
the sequences $\{\psia^{\Delta t (,\pm)}\}_{L>1}$ are bounded in
$L^\infty(0,T;L^1_{(1+|\qt|)^{2 \vartheta}M}(\Omega \times D))$. By  Lemma \ref{le:supplementary},  the strong convergence of these
to $\hat\psi_\epsilon$ in $L^1(0,T;$ $L^1_{(1+|\qt|)^{2 \vartheta}M}(\Omega \times D))$,
shown above, then implies
strong convergence in $L^p(0,T;L^1_{(1+|\qt|)^{2 \vartheta}M}(\Omega \times D))$ to the same limit for all values of $p \in [1,\infty)$. That now completes the proof of \eqref{psisconL2a}.

Strong convergence in $L^p(0,T;L^1_{(1+|\qt|)^{2 \vartheta}M}(\Omega \times D))$,
$p \geq 1$, implies strong convergence in
$L^p(0,T;L^1_M(\Omega \times D))$, $p \geq 1$, and convergence almost everywhere on $\Omega \times D \times [0,T]$ of a subsequence. Hence it follows from \eqref{additional1} that $\hat\psi_\epsilon
\geq 0$ on $\Omega \times D \times [0,T]$.
Furthermore, by Fubini's theorem, strong convergence  of $\{\psia^{\Delta t(,\pm)}\}_{L>1}$ to $\hat\psi_\epsilon$ in $L^1(0,T;L^1_M(\Omega \times D))$ implies that
\[ \int_D M(\qt)\,|\psia^{\Delta t,(\pm)}(\xt,\qt,t) -  \hat\psi_\epsilon(\xt,\qt,t)|\dq \rightarrow 0 \qquad
\mbox{as $L \rightarrow \infty$}\]
for a.e. $(\xt,t) \in \Omega \times [0,T]$. Hence we have that
\[
\mbox{$\displaystyle \int_D M(\qt)\,\psia^{\Delta t(,\pm)}(\xt,\qt,t)\dq \rightarrow
\int_D M(\qt)\,\hat\psi_\epsilon(\xt,\qt,t)\dq \qquad \mbox{as }L \rightarrow \infty$,}\]
for a.e. $(\xt,t) \in \Omega \times [0,T]$, and then \eqref{additional2} implies that
\begin{equation}\label{1boundonpsi}
\int_D M(\qt)\,\hat\psi_\epsilon(\xt,\qt,t)\dq \leq 1 \qquad \mbox{for a.e. $(x,t) \in \Omega \times [0,T]$.}
\end{equation}
We will show later that the inequality here can in fact be sharpened to an equality.

As the sequences $\{\psia^{\Delta t,(\pm)}\}_{L>1}$ converge to $\hat\psi_\epsilon$ strongly in
$L^1(0,T; L^1_M(\Omega \times D))$, it follows that (upon extraction of suitable subsequences)
they converge to $\hat\psi_\epsilon$ a.e. on
$\Omega \times D \times [0,T]$. That then, in turn, implies that the sequences
$\{\mathcal{F}(\psia^{\Delta t,(\pm)})\}_{L>1}$
converge to $\mathcal{F}(\hat\psi_\epsilon)$ a.e. on $\Omega \times D \times [0,T]$; in particular,
for a.e. $t \in [0,T]$, the sequences $\{\mathcal{F}(\psia^{\Delta t,(\pm)}(\cdot,\cdot,t))\}_{L>1}$
converge to $\mathcal{F}(\hat\psi_\epsilon(\cdot,\cdot,t))$ a.e. on $\Omega \times D$.
Since $\mathcal{F}$ is nonnegative, Fatou's lemma then implies that, for a.e. $t \in [0,T]$,
\begin{align}
\int_{\Omega \times D} M(\qt)\, \mathcal{F}(\hat\psi_\epsilon(\xt,\qt,t))\dx \dq
&\leq \mbox{lim inf}_{L \rightarrow \infty}
\int_{\Omega \times D} M(\qt)\, \mathcal{F}(\psia^{\Delta t,+}(\xt,\qt,t)) \dx \dq
\nonumber \\
&\leq \frac{1}{2k}\,[{\sf B}(\ut_0,\ft, \hat\psi_0)]^2,
\label{fatou-app}
\end{align}
where the second inequality in \eqref{fatou-app} stems from
the bound on the fourth term on the left-hand side of (\ref{eq:energy-u+psi-final2}).
As the
expression on the left-hand side of \eqref{fatou-app} is nonnegative, we deduce that $\mathcal{F}(\hat\psi_\epsilon)$
belongs to $L^\infty(0,T;L^1_M(\Omega \times D))$, as asserted in the statement of the theorem.



We observe in passing that since $|\sqrt{c_1} - \sqrt{c_2}\,|\leq \sqrt{|c_1-c_2 |}$ for any two nonnegative real
numbers $c_1$ and $c_2$, \eqref{psisconL2a} directly implies that, as $L \rightarrow \infty$ (and $\Delta t \rightarrow 0_+$),
\begin{equation}\label{sqrtpsi}
\sqrt{ \psia^{\Delta t,(\pm)}} \rightarrow \sqrt{\hat\psi_\epsilon}\qquad \mbox{strongly in $L^p(0,T;L^2_M(\Omega \times D))$}\quad \forall p \in [1,\infty),
\end{equation}
and therefore, as $L \rightarrow \infty$,
\begin{equation}\label{strongpsi}
M^{\frac{1}{2}}\, \sqrt{\psia^{\Delta t(,\pm)}} \rightarrow M^{\frac{1}{2}}\,\sqrt{\hat\psi_\epsilon}\qquad \mbox{strongly in $L^p(0,T;L^2(\Omega \times D))$}\quad \forall p \in [1,\infty).
\end{equation}
By proceeding in exactly the same way as in the previous subsection, between equations \eqref{strongpsi-zeta} and (\ref{xlimitL2},b)
with $\hat\zeta^{\Lambda,1}$, $\hat\zeta^{\!\!~1}=\hat\psi^0$ and $\Lambda$ replaced by
$\psia^{\Delta t(,\pm)}$, $\hat\psi_\epsilon$ and $L$, respectively, but now using the seventh and the eighth bounds in \eqref{eq:energy-u+psi-final6a}, 
we deduce that (\ref{psiwconH1a},b) hold.

The convergence result \eqref{psitwconL2a} follows from the bound on the last term on the left-hand
side of \eqref{eq:energy-u+psi-final6a} and the weak compactness of bounded balls in the Hilbert
space $L^2(0,T; H^s(\Omega \times D))$, $s>1 + \textstyle{\frac{1}{2}}(K+1)d$.

The proof of \eqref{CwconL2a} is considerably more complicated, and will be given below.

\smallskip
After all these technical preparations we are now ready to return to (\ref{equncon},b)
and pass to the limit $L\rightarrow \infty$ (and thereby also $\Delta t \rightarrow
0_+$); we shall also prove \eqref{CwconL2a} and will also pass to the limit on the initial
conditions for (\ref{equncon},b). 
Since there are quite a few terms to deal with,
we shall discuss them one at a time, starting with equation
\eqref{eqpsincon}, and followed by equation \eqref{equncon}.

\smallskip

\textit{Step 3.}
We begin by passing to the limit $L \rightarrow \infty$ (and $\Delta t \rightarrow 0_+$) on equation
\eqref{eqpsincon}.
In what follows, we shall take
$\hat\varphi \in C([0,T];C^\infty(\overline{\Omega};C^\infty_0(D)))$. Note
that $C^\infty(\overline{\Omega};C^\infty_0(D))$ is dense
in $L^2(\Omega;H^s(D))\cap H^s(\Omega;L^2(D))
=H^s(\Omega \times D)$, and so
$C([0,T];C^\infty(\overline{\Omega};C^\infty_0(D)))$ is a dense linear subspace of $L^2(0,T;H^s(\Omega \times D))$ for any
$s \geq 0$.
As each of the terms in \eqref{eqpsincon} has been shown to be a continuous linear
functional with respect to $\hat\varphi$ on $L^2(0,T;H^s(\Omega \times D))$ for $s>1 + \frac{1}{2}(K+1)d$,
the replacement of $L^2(0,T;H^s(\Omega \times D))$ by its dense linear subspace $C([0,T];C^\infty
(\overline{\Omega};C^\infty_0(D))$ for the purposes of the argument below
is fully justified.

\smallskip

\textit{Step 3.1.} Passing to the limit on the first term in  \eqref{eqpsincon} is easy:
using \eqref{psitwconL2a} we immediately have that
\begin{align*}
\int_{0}^T \int_{\Omega \times D} M\,\frac{ \partial \hpsiae^{\Delta t}}{\partial t}\,
\hat \varphi \dq \dx \dt &= \int_0^T \left\langle M\,\frac{ \partial \hpsiae^{\Delta t}}{\partial t} ,
\hat \varphi \right\rangle_{\!\!H^s(\Omega \times D)} \dt
\rightarrow
\int_0^T \left\langle M\,\frac{ \partial \hat\psi_\epsilon}{\partial t} ,
\hat \varphi \right\rangle_{\!\!H^s(\Omega \times D)} \dt
\end{align*}
as $L \rightarrow \infty$ (and $\Delta t \rightarrow 0_+$),
for all $\hat \varphi \in C([0,T];C^\infty(\overline{\Omega},C^\infty_0(D)))$,
as required. That completes
Step 3.1.

\smallskip

\textit{Step 3.2.} The second term will be dealt with by decomposing it into two further terms, the first of
which tends to $0$, while the second converges to the expected limiting value. We proceed as follows:
\begin{align*}
&\epsilon \int_{0}^T \int_{\Omega \times D} M\, \nabx \hpsiae^{\Delta
t,+} \cdot\, \nabx \hat \varphi \,\dq \dx \dt \nonumber\\
& \hspace{1.5in}= 2\epsilon \int_{0}^T \int_{\Omega \times D} M\,
\sqrt{\psia^{\Delta t,+}}\, \nabx \sqrt{\psia^{\Delta t,+}} \cdot\, \nabx \hat \varphi \,\dq \dx \dt
\nonumber\\
&\hspace{1.5in}= 2\epsilon \int_{0}^T \int_{\Omega \times D} M\,
\left(\sqrt{\psia^{\Delta t,+}} - \sqrt{\hat\psi_\epsilon}\,\right)\, \nabx \sqrt{\psiae^{\Delta t,+}} \cdot\, \nabx \hat \varphi \,\dq \dx \dt\nonumber\\
&\hspace{2.5in}+2\epsilon \int_{0}^T \int_{\Omega \times D} M\,
\sqrt{\hat\psi_\epsilon}\, \nabx \sqrt{\psia^{\Delta t,+}} \cdot\, \nabx \hat \varphi \,\dq \dx \dt\nonumber\\
&\hspace{1.5in}=: {\rm V}_1 + {\rm V_2}.
\end{align*}
We shall show that ${\rm V}_1$ converges to $0$ and that ${\rm V}_2$ converges to the expected limit.
\begin{align}
&|{\rm V}_1|  \leq  2 \epsilon \int_0^T \int_{\Omega}\left(\int_D M|\sqrt{\psia^{\Delta t,+}}
- \sqrt{\hat\psi_\epsilon}|^2 \dq\right)^{\frac{1}{2}}\nonumber\\
&\qqqquad\times \left(\int_D M\,|\nabx\sqrt{\psia^{\Delta t,+}}|^2 \dq\right)^{\frac{1}{2}}\|\nabx \hat\varphi\|_{L^\infty(D)} \dx \dt
\nonumber\\
&\qquad \leq 2 \epsilon\int_0^T \left(\int_{\Omega \times D} M|\sqrt{\psia^{\Delta t,+}}
- \sqrt{\hat\psi_\epsilon}|^2 \dq\dx\right)^{\frac{1}{2}} \nonumber\\
&\qqqquad \times \left(\int_{\Omega\times D}M\,|\nabx \sqrt{\psia^{\Delta t,+}}|^2\dq \dx \right)^{\frac{1}{2}}\|\nabx\hat\varphi\|_{L^\infty(\Omega \times D)}\dt\nonumber\\
&\qquad =  2 \epsilon\int_0^T \| \sqrt{M\,\psia^{\Delta t,+}}
- \sqrt{M\,\hat\psi_\epsilon}\|_{L^2(\Omega \times D)}\nonumber\\
&\qqqquad \times \|M^{\frac{1}{2}}\,\nabx \sqrt{\psia^{\Delta t,+}}\|_{L^2(\Omega \times D)}\,\|\nabx\hat\varphi\|_{L^\infty(\Omega \times D)}\dt \nonumber\\
&\qquad \leq 2 \epsilon\left(\int_0^T \|M^{\frac{1}{2}}\,\nabx \sqrt{\psia^{\Delta t,+}}\|^2_{L^2(\Omega \times D)}\dt\right)^{\frac{1}{2}}\nonumber\\
&\qqquad \times \left(\int_0^T\!\! \| \sqrt{M\,\psia^{\Delta t,+}}
- \sqrt{M\,\hat\psi_\epsilon}\|_{L^2(\Omega \times D)}^r \dt \right)^{\!\frac{1}{r}}\!\!\left(\int_0^T\!\! \|\nabx \hat\varphi\|_{L^\infty(\Omega \times D)}^{\frac{2r}{r-2}} \dt\right)^{\!\frac{r-2}{2r}}\!\!\!,
\nonumber
\end{align}
were $r \in (2,\infty)$. Using the bound on the sixth term in \eqref{eq:energy-u+psi-final6} together with the
Sobolev embedding theorem, we then have (with $C_\ast$ now denoting a possibly different constant than in
\eqref{eq:energy-u+psi-final6}, but one that is still independent of $L$ and $\Delta t$) that
\begin{align}
&|{\rm V}_1| \leq 2 C_\ast^{\frac{1}{2}} \epsilon \| \sqrt{M\,\psia^{\Delta t,+}} - \sqrt{M\,\hat\psi_\epsilon}\,\|_{L^r(0,T;L^2(\Omega \times D))}\, \|\nabx\hat\varphi\|_{L^{\frac{2r}{r-2}}(0,T;L^\infty(\Omega \times D))}\nonumber\\
&\qquad \!\leq 2 C_\ast^{\frac{1}{2}} \epsilon\, \|\psia^{\Delta t,+} - \hat\psi_\epsilon
\|_{L^{\frac{r}{2}}(0,T;L^1_M(\Omega \times D))}^{\frac{1}{2}}\, \|\nabx\hat\varphi\|_{L^{\frac{2r}{r-2}}(0,T;L^\infty(\Omega \times D))} ,\nonumber
\end{align}
where we also used the elementary inequality $|\sqrt{c_1} - \sqrt{c_2}| \leq \sqrt{|c_1-c_2|}$ with $c_1, c_2 \in \mathbb{R}_{\geq 0}$.
The norm of the difference in the last displayed line is known to converge to $0$ as $L \rightarrow \infty$
(and $\Delta t \rightarrow 0_+$), by \eqref{psisconL2a}. This then implies that
the term ${\rm V}_1$ converges to $0$ as  $L \rightarrow \infty$
(and $\Delta t \rightarrow 0_+$).

Concerning the term ${\rm V}_2$, we have that
\begin{align*}
&{\rm V}_2 = 2\epsilon \int_{0}^T \int_{\Omega \times D} M^{\frac{1}{2}}\,
\nabx \sqrt{\psia^{\Delta t,+}} \cdot\, \sqrt{M\,\hat\psi_\epsilon}\, \nabx \hat \varphi \,\dq \dx \dt.
\end{align*}
Once we have verified
that $\sqrt{M\,\hat\psi_\epsilon}\, \nabx \hat \varphi$ belongs to $L^2(0,T;\Lt^2(\Omega\times D))$,
the weak convergence result \eqref{psiwconH1a} will imply that
\begin{align*}
&{\rm V}_2 \rightarrow 2\epsilon \int_{0}^T \int_{\Omega \times D} M^{\frac{1}{2}}\,
\nabx \sqrt{\hat\psi_\epsilon} \cdot\, \sqrt{M\,\hat\psi_\epsilon}\, \nabx \hat \varphi \,\dq \dx \dt
= \epsilon \int_{0}^T \int_{\Omega \times D} M\,
\nabx \hat\psi_\epsilon \cdot\ \nabx \hat \varphi \,\dq \dx \dt
\end{align*}
as $L\rightarrow \infty$ (and $\Delta t \rightarrow 0_+$), and we will have completed Step 3.2. Let us therefore show that $\sqrt{M\,\hat\psi_\epsilon}\,\nabx\hat\varphi$
belongs to $L^2(0,T;\Lt^2(\Omega \times D))$; the justification is quite straightforward: using \eqref{1boundonpsi} we have that
\begin{align*}
\int_0^T\int_{\Omega \times D} |\sqrt{M\,\hat\psi_\epsilon}\,\nabx\hat\varphi|^2 \dx \dt
&= \int_0^T \int_{\Omega \times D} M\,\hat\psi_\epsilon\,|\nabx\hat\varphi|^2 \dq \dx \dt\nonumber\\
&\leq \int_0^T \int_\Omega \|\nabx\hat\varphi\|^2_{L^\infty(D)}\,\left(\int_D M\,\hat\psi_\epsilon \dq\right) \dx \dt\nonumber\\
&\leq \int_0^T \int_{\Omega} \|\nabx\hat\varphi\|^2_{L^\infty(D)} \dx \dt\\
&= \|\nabx \hat\varphi\|^2_{L^2(0,T;L^2(\Omega;L^\infty(D)))} < \infty.
\end{align*}
That now completes Step 3.2.

\smallskip

\textit{Step 3.3.} The third term in \eqref{eqpsincon} is dealt with as follows:
\begin{align*}
& - \int_{0}^T \int_{\Omega \times D} M\, \utae^{\Delta t,-}\,
\hpsiae^{\Delta t,+} \cdot\, \nabx
\hat \varphi \,\dq \dx \dt
\\
& \quad = - \int_{0}^T \int_{\Omega \times D} M\, \ut_\epsilon \,
\hat\psi_\epsilon \cdot\, \nabx \hat \varphi \,\dq \dx \dt
+ \int_{0}^T \int_{\Omega \times D} M\, (\ut_\epsilon- \utae^{\Delta t,-})\,
\psia^{\Delta t,+} \cdot\, \nabx \hat \varphi \,\dq \dx \dt\\
&\hspace{1in}+  \int_{0}^T \int_{\Omega \times D} M\, \ut_\epsilon\,
(\hat\psi_\epsilon - \hpsiae^{\Delta t,+}) \cdot\, \nabx \hat \varphi \,\dq \dx \dt.
\end{align*}
We label the last two terms ${\rm V}_3$ and ${\rm V}_4$ 
and we show that each of them converges to
$0$ as $L \rightarrow 0$ (and $\Delta t \rightarrow 0_+$). We start with term ${\rm V}_3$; below, we
apply H\"older's inequality with $r \in (1,\infty)$ in the case of $d=2$ and $r \in (1,6)$ when  $d=3$:
\begin{align*}
&~\!|{\rm V}_3| = \left|\int_{0}^T \int_{\Omega} (\ut_\epsilon- \utae^{\Delta t,-})\cdot\left[\int_D
M\,\psia^{\Delta t,+} \left(\nabx \hat \varphi\right) \dq\right] \dx \dt\right|\\
&\qquad \leq \int_0^T \int_\Omega |\ut_\epsilon- \utae^{\Delta t,-}| \left[\int_D
M\,\psia^{\Delta t,+}\dq\right] \|\nabx \hat \varphi\|_{L^\infty(D)}
\dx \dt\\
&\qquad \leq \int_0^T \int_\Omega |\ut_\epsilon- \utae^{\Delta t,-}| \, \|\nabx \hat \varphi\|_{L^\infty(D)} \dx \dt\\
&\qquad \leq \int_0^T\left(\int_\Omega |\ut_\epsilon- \utae^{\Delta t,-}|^r\dx\right)^{\frac{1}{r}} \left(\int_\Omega\|\nabx \hat \varphi\|^{\frac{r}{r-1}}_{L^\infty(D)}\dx \right)^{\frac{r-1}{r}}\dt\\
&\qquad \leq \int_0^T \|\ut_\epsilon- \utae^{\Delta t,-}\|_{L^r(\Omega)}\,\|\nabx \hat \varphi\|_{L^{\frac{r}{r-1}}(\Omega; L^\infty(D))}\dt\\
&\qquad \leq \|\ut_\epsilon- \utae^{\Delta t,-}\|_{L^2(0,T;L^r(\Omega))}\,\|\nabx \hat \varphi\|_{L^2(0,T;L^{\frac{r}{r-1}}(\Omega; L^\infty(D)))},
\end{align*}
where in the transition from the second line to the third line we made use of
\eqref{properties-b}.
Thanks to \eqref{usconL2a} the first factor in the last line
converges to $0$, and hence ${\rm V}_3$ converges to $0$
also, as $L\rightarrow \infty$ (and $\Delta t \rightarrow 0_+$).

For ${\rm V}_4$, we have, by using 
Fubini's theorem, the factorization
\begin{equation}\label{psi-fact}
M\left(\hat\psi_\epsilon - \hpsiae^{\Delta t,+}\right) = M^{\frac{1}{2}}\left(\sqrt{\hat\psi_\epsilon} - \sqrt{\hpsiae^{\Delta t,+}}\right)\,  M^{\frac{1}{2}}\left(\sqrt{\hat\psi_\epsilon} + \sqrt{\hpsiae^{\Delta t,+}}\right),
\end{equation}
together with the Cauchy--Schwarz inequality, \eqref{properties-b}, \eqref{1boundonpsi} and the elementary inequality \linebreak
$|\sqrt{c_1} - \sqrt{c_2}~\!| \leq \sqrt{|c_1-c_2|}$ with $c_1, c_2 \in \mathbb{R}_{\geq 0}$, that
\begin{align*}
&~|{\rm V}_4| \leq
\int_{0}^T \int_{\Omega \times D} M\, |\ut_\epsilon|\,|\hat\psi_\epsilon - \hpsiae^{\Delta t,+}|\,|\nabx \hat \varphi| \,\dq \dx \dt\\
&\qquad \leq \int_{0}^T \int_{\Omega} |\ut_\epsilon|\,\left(\int_D M\, |\hat\psi_\epsilon - \hpsiae^{\Delta t,+}| \dq\right) \|\nabx \hat \varphi\|_{L^\infty(D)}\dx \dt\\
&\qquad \leq {2}\,\int_{0}^T \int_{\Omega} |\ut_\epsilon|\,\left(\int_D M\, |\sqrt{\hat\psi_\epsilon} - \sqrt{\hpsiae^{\Delta t,+}}|^2 \dq\right)^{\frac{1}{2}} \|\nabx \hat \varphi\|_{L^\infty(D)}\dx \dt\\
&\qquad \leq {2}\,\int_{0}^T \left(\int_{\Omega} |\ut_\epsilon|^2\dx\right)^{\frac{1}{2}} \,\left(\int_{\Omega \times D} M\, |\sqrt{\hat\psi_\epsilon} - \sqrt{\hpsiae^{\Delta t,+}}|^2\dq \dx \right)^{\frac{1}{2}}
\,\|\nabx \hat \varphi\|_{L^\infty(\Omega \times D)}\dt\\
&\qquad \leq {2}\,\int_{0}^T \|\ut_\epsilon\|_{L^2(\Omega)} \,\left(\int_{\Omega \times D} M\, |\hat\psi_\epsilon - \hpsiae^{\Delta t,+}|\dq \dx \right)^{\frac{1}{2}}
\,\|\nabx \hat \varphi\|_{L^\infty(\Omega \times D)}\dt\\
&\qquad \leq {2}\, \|\ut_\epsilon\|_{L^\infty(0,T;L^2(\Omega))} \|\hat\psi_\epsilon - \hpsiae^{\Delta t,+}\|^{\frac{1}{2}}_{L^1(0,T;L^1_M(\Omega \times D))}  \|\nabx \hat \varphi\|_{L^2(0,T;L^\infty(\Omega \times D))}.
\end{align*}
By \eqref{uwconL2a} the first factor in the last line is finite while, according to \eqref{psisconL2a} (with $p=1$),
the middle factor converges to $0$ as $L \rightarrow \infty$ (and $\Delta t \rightarrow 0_+$). This
proves that ${\rm V}_4$ converges to $0$ as $L \rightarrow \infty$ (and $\Delta t \rightarrow 0_+$),
also.
That completes Step 3.3.

\smallskip

\textit{Step 3.4.} Thanks to \eqref{psiwconH1xa},
\begin{align*}
&M^{\frac{1}{2}}\,\nabq \sqrt{\hpsiae^{\Delta t(,\pm)}} \rightarrow M^{\frac{1}{2}}\,\nabq \sqrt{\hat\psi_\epsilon}
&&\quad \mbox{weakly in } L^{2}(0,T;\Lt^2(\Omega\times D)).
\end{align*}
This, in turn, implies that, componentwise,
\begin{align*}
&M^{\frac{1}{2}}\,\nabqj \sqrt{\hpsiae^{\Delta t(,\pm)}} \rightarrow M^{\frac{1}{2}}\,\nabqj \sqrt{\hat\psi_\epsilon}
&&\quad \mbox{weakly in } L^{2}(0,T;\Lt^2(\Omega\times D)),
\end{align*}
for each $j=1,\dots, K$, whereby also,
\begin{align*}
&M^{\frac{1}{2}}\, \sum_{j=1}^K A_{ij}\nabqj \sqrt{\hpsiae^{\Delta t(,\pm)}} \rightarrow M^{\frac{1}{2}}\, \sum_{j=1}^KA_{ij}\nabqj \sqrt{\hat\psi_\epsilon}
&&\quad \mbox{weakly in } L^{2}(0,T;\Lt^2(\Omega\times D)).
\end{align*}
That places us in a very similar position as in the case of Step 3.2, and we can argue in an
identical manner as there to show that
\begin{align*}
&\frac{1}{2\,\lambda}
\int_{0}^T \int_{\Omega \times D} M\,\sum_{i=1}^K
 \,\sum_{j=1}^K A_{ij}\,\nabqj \hpsiae^{\Delta t,+}
\cdot\, \nabqi
\hat \varphi
\,\dq \dx \dt \nonumber\\
&\hspace{1.5in}\rightarrow \frac{1}{2\,\lambda}
\int_{0}^T \int_{\Omega \times D} M\,\sum_{i=1}^K
 \,\sum_{j=1}^K A_{ij}\,\nabqj \hat\psi_\epsilon
\cdot\, \nabqi \hat \varphi
\,\dq \dx \dt
\end{align*}
as $L \rightarrow \infty$ and $\Delta t \rightarrow 0_+$, for all $\hat\varphi \in L^{\frac{2r}{r-2}}(0,T;W^{1,\infty}(\Omega \times D))$, where $r \in (2,\infty)$,
and in particular for all $\hat\varphi \in C([0,T];
C^\infty(\overline{\Omega};C^\infty_0(D)))$. That completes Step 3.4.

\smallskip

\textit{Step 3.5.} The final term in \eqref{eqpsincon}, the drag term, is the one in the
equation that is the most difficult to deal with.
We shall break it up into four subterms, three of which will be shown to converge to
$0$ in the limit of $L\rightarrow \infty$ (and $\Delta t\rightarrow 0_{+}$), leaving the
fourth term as the (expected) limiting value:
\begin{align}\label{bound-3.5}
&-
\int_{0}^T \int_{\Omega \times D} M\,\sum_{i=1}^K
[\sigtt(\utae^{\Delta t,+})
\,\qt_i]\,
\beta^L(\hpsiae^{\Delta t,+}) \,\cdot\, \nabqi
\hat \varphi
\,\dq \dx \dt \nonumber\\
&\hspace{1in}= - \int_{0}^T \int_{\Omega \times D} M\,\sum_{i=1}^K
\left[\left(\nabx\utae^{\Delta t,+}\right) \,\qt_i\right]
\beta^L(\hpsiae^{\Delta t,+}) \,\cdot\, \nabqi
\hat \varphi
\,\dq \dx \dt\nonumber\\
&\hspace{1in}= - \int_{0}^T \int_{\Omega \times D} M\,\sum_{i=1}^K
\left[\left(\nabx\utae^{\Delta t,+}\right) \,\qt_i\right] 
\left(\beta^L(\hpsiae^{\Delta t,+})
- \beta^L(\hat\psi_\epsilon)\right)\,\cdot\, \nabqi
\hat \varphi
\,\dq \dx \dt\nonumber\\
&\hspace{1in}\qquad\quad- \int_{0}^T \int_{\Omega \times D} M\,\sum_{i=1}^K
\left[\left(\nabx\utae^{\Delta t,+}\right) \,\qt_i\right] 
\left(\beta^L(\hat\psi_\epsilon)
- \hat\psi_\epsilon\right)\,\cdot\, \nabqi
\hat \varphi
\,\dq \dx \dt\nonumber\\
&\hspace{1in}\qquad \quad- \int_{0}^T \int_{\Omega \times D} M\,\sum_{i=1}^K
\left[\left(\nabx\utae^{\Delta t,+}-\nabx\ut_\epsilon\right) \,\qt_i\right] 
\hat\psi_\epsilon\,\cdot\, \nabqi
\hat \varphi
\,\dq \dx \dt\nonumber\\
&\hspace{1in}\qquad \quad- \int_{0}^T \int_{\Omega \times D} M\,\sum_{i=1}^K
\left[\left(\nabx\ut_\epsilon\right) \,\qt_i\right] 
\hat\psi_\epsilon\,\cdot\, \nabqi
\hat \varphi
\,\dq \dx \dt.
\end{align}
We label the first three terms on the right-hand side by ${\rm V}_5$, ${\rm V}_6$, ${\rm V}_7$, respectively,
and we
proceed to bound each of them. We shall show that each of the three terms converges to $0$, leaving the fourth
term as the limit of the left-most expression in the chain, as $L\rightarrow \infty$ (and $\Delta t
\rightarrow 0_+$).

We begin by bounding the term ${\rm V}_5$, noting that
$\hat \varphi$ is fixed with compact support in $D$,
$\beta^L$ is Lipschitz
continuous, 
using the
factorization \eqref{psi-fact} together with \eqref{properties-b} and
(\ref{1boundonpsi}), and then proceeding as
in the case of term ${\rm V}_4$ in Step 3.3:
\begin{align*}
|{\rm V}_5| &\leq \int_{0}^T \int_{\Omega \times D} M\, |\qt|\,|\nabx\utae^{\Delta t,+}|\, |\hpsiae^{\Delta t,+} - \hat\psi_\epsilon|\,|\nabq \hat \varphi|\,\dq \dx \dt\\
&\leq \int_{0}^T \left[\int_\Omega |\nabx\utae^{\Delta t,+}|\,\left(\int_{D} M\, \, |\hpsiae^{\Delta t,+} - \hat\psi_\epsilon|\,\dq\right)\dx \right]\|\,
|\qt|\,\nabq \hat \varphi\|_{L^\infty(\Omega \times D)} \dt\\
&\leq 2\,\|\nabx\utae^{\Delta t,+}\|_{L^2(0,T;L^2(\Omega))}\,
\|\hpsiae^{\Delta t,+} - \hat\psi_\epsilon\|^{\frac{1}{2}}_{L^1(0,T;L^1_M(\Omega\times D))}
\,\|\,
|\qt|\,\nabq \hat \varphi\|_{L^\infty((0,T)\times \Omega \times D))}.
\end{align*}
On noting the bound on the third term on the left-hand side of \eqref{eq:energy-u+psi-final6a}
and the convergence result \eqref{psisconL2a} that was proved in Step 2, we deduce that term
${\rm V}_5$ converges to $0$ as $L\rightarrow \infty$ (and $\Delta t \rightarrow 0_+$).

We move on to term ${\rm V}_6$,
using arguments similar to those used for terms ${\rm V}_5$: 
\begin{align*}
|{\rm V}_6| &\leq \int_{0}^T \int_{\Omega \times D} M\,|\qt|\,|\nabx\utae^{\Delta t,+}| \,|\beta^L(\hat\psi_\epsilon)
- \hat\psi_\epsilon|\, |\nabq \hat \varphi| \dq \dx \dt\\
& \leq \int_{0}^T \left[\int_{\Omega} |\nabx\utae^{\Delta t,+}| \left(\int_D
M\,|\beta^L(\hat\psi_\epsilon)
- \hat\psi_\epsilon|\,\dq\right) \dx \right] \|\,
|\qt|\,\nabq \hat \varphi\|_{L^\infty(\Omega \times D)} \dt\\
& \leq 2\,\|\nabx\utae^{\Delta t,+}\|_{L^2(0,T;L^2(\Omega))}\,\|\beta^{L}(\hat\psi_\epsilon) - \hat\psi_\epsilon\|^{\frac{1}{2}}_{L^1(0,T;L^1_M(\Omega\times D))}
\,\|\,
|\qt|\,\nabq \hat \varphi\|_{L^\infty((0,T)\times \Omega \times D))}.
\end{align*}
Note that $0 \leq \hat\psi_\epsilon - \beta^{L}(\hat\psi_\epsilon) \leq \hat\psi_\epsilon$
and that $\hat\psi_\epsilon - \beta^{L}(\hat\psi_\epsilon)$ converges to $0$ almost everywhere on
$\Omega \times D \times (0,T)$ as $L \rightarrow \infty$.
Note further that, thanks to \eqref{psisconL2a} with $p=1$,
$\hat\psi_\epsilon \in L^1(0,T;L^1_M(\Omega \times D))$. Thus, Lebesgue's dominated convergence
theorem implies that,
the middle factor in the last displayed
line converges to $0$
as $L \rightarrow \infty$.
Hence, recalling the bound on the third term on the left-hand side of
\eqref{eq:energy-u+psi-final6a}, we thus deduce that ${\rm V}_6$ converges to $0$ as
$L\rightarrow \infty$ (and $\Delta t \rightarrow 0_+$).

Finally, we consider term ${\rm V}_7$:
\begin{align*}
&\,\quad {\rm V}_7:=- \int_{0}^T \int_{\Omega \times D} M\,\sum_{i=1}^K
\left[\left(\nabx\utae^{\Delta t,+}-\nabx\ut_\epsilon\right) \,\qt_i\right] 
\hat\psi_\epsilon\,\cdot\, \nabqi \hat \varphi \,\dq \dx \dt.
\end{align*}
We observe that, before starting to bound ${\rm V}_7$, we should perform an integration by
parts in order to transfer the $x$ gradient from the difference $$\nabx\utae^{\Delta t,+}-\nabx\ut_\epsilon$$ onto the other factors
under the integral sign, as we only have weak, but not strong, convergence of $\nabx\utae^{\Delta t,+}-\nabx\ut_\epsilon$ to $0$, (cf. \eqref{uwconH1a}) whereas the difference $\utae^{\Delta t,+}-\ut_\epsilon$ converges to $0$
strongly, by virtue of \eqref{usconL2a}.

We note is this respect
that the function $\xt \in \Omega \mapsto \hat\psi_\varepsilon(\xt,\qt,t)$ has a well-defined trace
on $\partial \Omega$ for a.e. $(\qt,t) \in D \times (0,T)$, since, thanks to \eqref{psiwconH1a},
\[
\mbox{$\sqrt{\hat\psi_\epsilon(\cdot,\qt,t)} \in H^1(\Omega),\qquad$ and therefore $\qquad\left.\sqrt{\hat\psi_\epsilon(\cdot,\qt,t)}\right|_{\partial \Omega} \in H^{1/2}(\partial
\Omega)$,}\]
for a.e. $(\qt,t) \in D \times (0,T)$, implying that
\[ \mbox{$\left.\sqrt{\hat\psi_\epsilon(\cdot,\qt,t)}\right|_{\partial\Omega} \in L^{2p}(\partial \Omega)$}\]
for a.e. $(\qt,t) \in D \times (0,T)$, with $2p \in [1,\infty)$, when $d=2$ and $2p \in [1,4]$ when $d=3$, whereby
$$\left. \hat\psi_\epsilon(\cdot,\qt,t)\right|_{\partial\Omega} \in L^{p}(\partial \Omega)$$
for a.e. $(\qt,t) \in D \times (0,T)$,
with $p \in [1,\infty)$, when $d=2$ and $p \in [1,2]$ when $d=3$. As the functions
$\ut_\epsilon$, $\utae^{\Delta t,+}$ have zero trace on $\partial\Omega$, the boundary integral that arises in the course of integration by parts then vanishes. First, we write
\begin{align*}
&{\rm V}_7=- \!\int_{0}^T\!\!\! \int_{\Omega \times D}\!\!\! M\sum_{i=1}^K \sum_{m,n=1}^d\!\frac{\partial}{\partial  x_m} \left[\left((\utae^{\Delta t,+})_n-(\ut_\epsilon)_n\right) (\qt_i)_m\right] 
 \hat\psi_\epsilon \, (\nabqi \hat \varphi)_n \dq \dx \dt.
\end{align*}
Here, $(\utae^{\Delta t,+})_n$ and $(\ut_\epsilon)_n$ denote the $n$th among the $d$ components of the
vectors $\utae^{\Delta t,+}$ and $\ut_\epsilon$, $1 \leq n \leq d$, respectively,
and $(\nabqi \hat \varphi)_n$ denotes the $n$th among the $d$ components of the vector
$\nabqi \hat \varphi$, $1 \leq n \leq d$, for each $i \in \{1,\dots, K\}$. Similarly,
$(\qt_i)_m$ denotes the $m$th component, $1 \leq m \leq d$, of the $d$-component vector
$\qt_i$ for $i \in \{1,\dots, K\}$. Now, on integrating by parts,
and cancelling the boundary integral terms, with the justification given above, we have that

\begin{align*}
\,{\rm V}_7 &= \int_{0}^T\!\!\! \int_{\Omega\times D}\!\!M\,
\sum_{i=1}^K\sum_{m,n=1}^d
\left[\left((\utae^{\Delta t,+})_n-(\ut_\epsilon)_n\right) \,(\qt_i)_m\right] \frac{\partial}{\partial  x_m}\left(\hat\psi_\epsilon\, (\nabqi \hat \varphi)_n\right)\!\dq \dx \dt\\
&= \int_{0}^T\!\!\! \int_{\Omega\times D}\!\!M\,
\sum_{i=1}^K\sum_{m,n=1}^d
\left[\left((\utae^{\Delta t,+})_n-(\ut_\epsilon)_n\right) \,(\qt_i)_m\right] \frac{\partial \hat\psi_\epsilon}{\partial  x_m}\, (\nabqi \hat \varphi)_n\dq \dx \dt\\
&\qquad + \int_{0}^T\!\!\! \int_{\Omega\times D}\!\!M\,
\sum_{i=1}^K\sum_{m,n=1}^d
\left[\left((\utae^{\Delta t,+})_n-(\ut_\epsilon)_n\right) \,(\qt_i)_m\right] \left(\hat\psi_\epsilon\,\frac{\partial}{\partial  x_m} (\nabqi \hat \varphi)_n\right)\!\dq \dx \dt\\
&=:{\rm V}_{7,1} + {\rm V}_{7,2}.
\end{align*}
For term ${\rm V}_{7,1}$, we have that
\begin{align*}
|{\rm V}_{7,1}| &\leq \int_{0}^T\!\!\! \int_{\Omega\times D}\!\!M\,
\sum_{i=1}^K\sum_{m,n=1}^d
\left|(\utae^{\Delta t,+})_n-(\ut_\epsilon)_n\right|\,|(\qt_i)_m|\, \left|\frac{\partial \hat\psi_\epsilon}{\partial  x_m}\right|\, \left|(\nabqi \hat \varphi)_n\right|\,\dq \dx \dt\\
&\leq \int_{0}^T\!\!\! \int_{\Omega\times D}\!\!M\,
\left(\sum_{i=1}^K\sum_{m,n=1}^d
\left|(\utae^{\Delta t,+})_n-(\ut_\epsilon)_n\right|^2\,|(\qt_i)_m|^2\right)^{\frac{1}{2}}\\
&\hspace{5.6cm}\times
\left(\sum_{i=1}^K\sum_{m,n=1}^d \left|\frac{\partial \hat\psi_\epsilon}{\partial  x_m}\right|^2\, \left|(\nabqi \hat \varphi)_n\right|^2\right)^{\frac{1}{2}}\,\dq \dx \dt\\
&= \int_{0}^T \int_{\Omega\times D}\!\!M\,
|\utae^{\Delta t,+}- \ut_\epsilon|\,|\qt|
|\nabx\hat\psi_\epsilon|\, |\nabq \hat \varphi\,| \dq \dx \dt\\
&\leq \int_{0}^T \left(\int_{\Omega\times D} M\,|\utae^{\Delta t,+}- \ut_\epsilon|\,
|\nabx\hat\psi_\epsilon| \dq \dx\right)
\|\,
|\qt|\,\nabq \hat \varphi\,\|_{L^\infty(\Omega \times D)}\dt\\
&= 2 \int_{0}^T \left[\int_{\Omega} |\utae^{\Delta t,+}- \ut_\epsilon|\,
\left(\int_D M\, \sqrt{\hat\psi_\epsilon}\,|\nabx\sqrt{\hat\psi_\epsilon}| \dq\right) \dx\right]
\|\,
|\qt|\,\nabq \hat \varphi\,\|_{L^\infty(\Omega \times D)}\dt\\
&\leq 2 \int_{0}^T \left[\int_{\Omega} |\utae^{\Delta t,+}- \ut_\epsilon|\, \left(\int_D M\, |\nabx\sqrt{\hat\psi_\epsilon}|^2 \dq\right)^{\frac{1}{2}} \dx\right]
\|\,
|\qt|\,\nabq \hat \varphi\,\|_{L^\infty(\Omega \times D)}\dt,
\end{align*}
where in the transition to the last line we used the Cauchy--Schwarz inequality in conjunction with
the upper bound \eqref{1boundonpsi}. Hence,
\begin{align*}
&|{\rm V}_{7,1}| \leq
2\, \|\utae^{\Delta t,+}- \ut_\epsilon\|_{L^2(0,T;L^2(\Omega))}\,
\|\nabx\sqrt{\hat\psi_\epsilon}\|_{L^2(0,T;L^2_M(\Omega \times D))}
\,\|\,
|\qt|\,\nabq \hat \varphi\,\|_{L^\infty(0,T;L^\infty(\Omega \times D))}.
\end{align*}
Noting \eqref{usconL2a} with $r=2$
and (\ref{psiwconH1a})
we then deduce that ${\rm V}_{7,1}$ converges to $0$ as
$L \rightarrow 0$ (and $\Delta t \rightarrow 0_+$).

Let us now consider term ${\rm V}_{7,2}$. We proceed similarly as in the case of term ${\rm V}_{7,1}$:
\begin{align*}
|{\rm V}_{7,2}| &\leq \int_{0}^T\!\!\! \int_{\Omega\times D}\!\!M\,
\sum_{i=1}^K\sum_{m,n=1}^d
|(\utae^{\Delta t,+})_n-(\ut_\epsilon)_n| \,|(\qt_i)_m|\, \hat\psi_\epsilon\,\left|\frac{\partial}{\partial  x_m} (\nabqi \hat \varphi)_n\right|\!\dq \dx \dt\\
&\leq \int_{0}^T \int_{\Omega}
|\utae^{\Delta t,+}-\ut_\epsilon| \left(\int_D\,M\,
|\qt|\, \hat\psi_\epsilon\,|\nabx \nabq \hat \varphi|\dq \right) \dx \dt\\
& \leq \int_{0}^T \int_{\Omega}
|\utae^{\Delta t,+}-\ut_\epsilon|\, \left(\int_D M \,\hat\psi_\epsilon \dq\right) \,
\|\,
|\qt|\,\nabx \nabq \hat \varphi\|_{L^\infty(D)}\dx \dt\\
&\leq \int_{0}^T \int_{\Omega}
|\utae^{\Delta t,+}-\ut_\epsilon|
\,\|\,
|\qt|\,\nabx \nabq \hat \varphi\|_{L^\infty(D)}\dx \dt,
\end{align*}
where in the transition to the last line we used \eqref{1boundonpsi}. Hence,
\begin{align*}
&|{\rm V}_{7,2}| \leq
\|\utae^{\Delta t,+}-\ut_\epsilon\|_{L^2(0,T;L^2(\Omega))}\,
\|\,
|\qt|\,\nabx \nabq \hat \varphi\|_{L^2(0,T;L^2(\Omega;L^\infty(D)))}.
\end{align*}
Noting \eqref{usconL2a} with $r=2$, we deduce that ${\rm V}_{8,2}$ converges to $0$ as
$L \rightarrow 0$ (and $\Delta t \rightarrow 0_+$).

Having shown that both ${\rm V}_{7,1}$ and  ${\rm V}_{7,2}$ converge to $0$ as
$L \rightarrow 0$ (and $\Delta t \rightarrow 0_+$), it follows that the same is
true of ${\rm V}_7 = {\rm V}_{7,1} + {\rm V}_{7,2}$. We have already shown that
${\rm V}_5$ and ${\rm V}_{6}$ converge to $0$ as $L \rightarrow 0$ (and $\Delta t
\rightarrow 0_+$). Since the sum of the first three terms on the left-hand side of
\eqref{bound-3.5} converges to $0$, it follows that the left-most expression in
the chain \eqref{bound-3.5} converges to the right-most term, in the
limit of $L \rightarrow \infty$ (and $\Delta t \rightarrow 0_+$).
That completes Step 3.5.

Having dealt with \eqref{eqpsincon}, we now turn our attention to \eqref{equncon},
with the aim to pass to the limit with $L$ (and $\Delta t$). In Steps 3.6 and 3.7
below we shall choose as our test function
\[ \wt \in C([0,T];\Ct^\infty_0(\Omega)), \qquad \mbox{with $\quad\;\nabx \cdot \wt = 0$ on $\Omega$,
for all $t \in [0,T]$}.\]
Clearly, any such $\wt$ belongs to $L^1(0,T;\Vt)$ and is therefore a legitimate choice of
test function in \eqref{equncon}. Furthermore, the set of such smooth functions $\wt$ is dense
in $L^2(0,T;\Vt_\sigma)$, $\sigma > 1+  \frac{1}{2}d$. As each term in
\eqref{equncon} has been shown before to be a continuous linear functional on $L^2(0,T;\Vt_\sigma)$, $\sigma > 1 + \frac{1}{2}d$, the replacement of $L^2(0,T;\Vt_\sigma)$, $\sigma > 1+
\frac{1}{2}d$, with such smooth test functions for the purposes of the argument below is fully justified.

\smallskip

\textit{Step 3.6.} The terms on the left-hand side of \eqref{equncon} are handled routinely,
using \eqref{eq:energy-u+psi-final6a} and, respectively, \eqref{utwconL2a}, \eqref{usconL2a}
with $r=2$ and \eqref{uwconH1a}. In particular, the second (nonlinear) term on the left-hand
side of \eqref{equncon} is quite simple to deal with on rewriting it as
\[  - \int_0^T(\uta^{\Delta t,+} \otimes \uta^{\Delta t,-}, \nabxtt\wt)\dt,\]
and then considering the difference
\[  \int_0^T(\ut_\epsilon \otimes \ut_\epsilon - \uta^{\Delta t,+} \otimes \uta^{\Delta t,-}, \nabxtt\wt)\dt,\]
which is bounded by
\[ \left(\int_0^T\|\ut_\epsilon \otimes \ut_\epsilon - \uta^{\Delta t,+} \otimes \uta^{\Delta t,-}\|_{L^1(\Omega)}\dt\right)  \|\nabxtt\wt\|_{L^\infty(0,T;L^\infty(\Omega))}.\]
By adding and subtracting $\ut_\epsilon \otimes \uta^{\Delta t,-}$ inside the first norm sign, using
the triangle inequality, followed by the Cauchy--Schwarz inequality in each of the resulting terms,
and then applying the first bound in \eqref{eq:energy-u+psi-final6a}, and \eqref{usconL2a} with $r=2$,
we deduce that the above expression converges to $0$ as $L\rightarrow \infty$ (and $\Delta t\rightarrow
0_+$).  The convergence of the first term on the right-hand side of \eqref{equncon} to the correct limit, as $L \rightarrow \infty$ (and $\Delta t \rightarrow 0_+$), is an immediate consequence of \eqref{fncon}.
We refer the reader, for a similar argument in the case of the incompressible Navier--Stokes equations, to Chapter 3, Section 4 of Temam \cite{Temam}. That completes Step 3.6.

\smallskip

\textit{Step 3.7.} The extra-stress tensor appearing on the right-hand side of
\eqref{equncon} is
dealt with as follows. 
We have from (\ref{growth2}), and similarly to (\ref{U4a}), that
\begin{align*}
{\rm V}_8 &:=\left|k\,\int_{0}^T\!\sum_{i=1}^K \int_{\Omega}
\Ctt_i(M\,\hpsiae^{\Delta t,+}): \nabxtt \wt \dx \dt - k\,\int_{0}^T\!\sum_{i=1}^K \int_{\Omega}
\Ctt_i(M\,\hat\psi_\epsilon): \nabxtt \wt \dx \dt\right|\\
&= k \left|\int_{0}^T\!\sum_{i=1}^K \int_{\Omega \times D}
M\,U_i'\,\qt_i \,\qt_i^T(\hpsiae^{\Delta t,+}-\hat\psi_\epsilon):
\nabxtt \wt \dq \dx \dt\right| \\
&\leq 2k \,\|\nabxtt \wt\|_{L^\infty(0,T;L^\infty(\Omega))} \int_{0}^T\!\sum_{i=1}^K \int_{\Omega \times D} \! \! \textstyle
M \left[c_{i2}  \left(\frac{1}{2}|\qt_i|^2\right)
+ c_{i3} \left(\frac{1}{2}|\qt_i|^2\right)^\vartheta \right]
\!|\hpsiae^{\Delta t,+}-\hat\psi_\epsilon|
 \dq \dx \dt
\\
&\leq C \,\|\nabxtt \wt\|_{L^\infty(0,T;L^\infty(\Omega))}
\int_{0}^T\!\int_{\Omega \times D} \textstyle
M\,(1 + |\qt|)^{2\vartheta}\,
|\hpsiae^{\Delta t,+}-\hat\psi_\epsilon|
 \dq \dx \dt.
\end{align*}
By (\ref{psisconL2a}) with $p=1$ the integral in the last line converges to zero as $L \rightarrow
\infty$ (and $\Delta t \rightarrow 0_+$).
The fact that ${\rm V}_{8}$ converges to $0$ then directly implies \eqref{CwconL2a}, thanks
to the denseness of the set of divergence-free functions contained in $C([0,T];\Ct^\infty_0(\Omega))$ in the
function space $L^2(0,T;\Vt_\sigma)$, $\sigma > 1+ \frac{1}{2}d$. 
That completes Step 3.7, and the proof of \eqref{CwconL2a}.

\smallskip

\textit{Step 3.8.} Steps 3.1--3.7 allow us to pass to the limits $L \rightarrow \infty$
and $\Delta t \rightarrow 0_+$, with $L$ and $\Delta t$ linked by the condition
$\Delta t = o(L^{-1})$ as $L \rightarrow \infty$, to deduce the existence of a pair
$(\ut_\epsilon,\hat\psi_\epsilon)$ satisfying \eqref{equnconP}, \eqref{eqpsinconP}
for smooth test functions $\hat\varphi$ and $\wt$, as above. On noting the denseness of
the set of divergence-free functions contained in $C([0,T];\Ct^\infty_0(\Omega))$ in
$L^2(0,T;\Vt_\sigma)$, $\sigma > 1+ \frac{1}{2}d$, and the denseness of $C([0,T];
C^\infty(\overline{\Omega};C^\infty_0(D)))$ in $L^2(0,T;H^s(\Omega \times D))$, $s > 1 + \frac{1}{2}(K+1)d$, respectively. That completes Step 3.8.

\smallskip

\textit{Step 3.9.} The weak continuity of $\ut_\epsilon$ and $\hat\psi_\epsilon$ stated in the theorem is shown
as follows. First we consider $\hat\psi_\epsilon$.

As $\hat\psi_\epsilon \in H^1(0,T;M^{-1}H^s(\Omega \times D)')$, $s> 1 + \frac{1}{2}(K+1)d$,
it follows by the Sobolev embedding theorem that $\hat\psi_\epsilon \in C([0,T];M^{-1}H^s(\Omega \times D)')$.
In fact, since $\hat\psi_\epsilon \in L^\infty(0,T; L^1_{
M}(\Omega \times D))$, it is possible to
show that $\hat\psi_\epsilon \in C_{w}([0,T]; L^1_{
M}
(\Omega \times D))$, i.e. $\hat\psi_\epsilon$
is weakly continuous as a function from $[0,T]$ into $L^1_{
M}(\Omega \times D)$; in particular,
$\langle \hat\psi_\epsilon(\cdot,\cdot,t) - \hat\psi_0(\cdot,\cdot) , \hat\varphi \rangle \rightarrow 0$ as
$t \rightarrow 0_+$ for all $\hat\varphi \in L^1_{
M}(\Omega \times D)'$,
where $\langle \cdot , \cdot \rangle$ denotes the duality pairing between
$L^1_{
M}(\Omega \times D)$ and $L^1_{
M}(\Omega \times D)'$. 
In order to prove the statements in the last sentence we invoke Lemma 1.4 in Chapter 3 of Temam \cite{Temam},
which we quote here.

\begin{lemma}\label{lem-zeta}
Suppose that $X$ and $Y$ are Banach spaces, with $X$ continuously embedded into $Y$. If
$\varphi \in L^\infty(0,T;X)$ and is weakly continuous as a function with values in $Y$, then $\varphi$
is weakly continuous as a function with values in $X$.
\end{lemma}

By taking $X=L^1_{
M}(\Omega \times D)$ and $Y=M^{-1} H^{s}(\Omega \times D)'$ with, once again,
$s > 1 + \frac{1}{2}(K+1)d$, the continuous embedding of $X$ into $Y$ follows
(cf. the discussion following Theorem \ref{thm:Dubinski} for a proof); Lemma \ref{lem-zeta} then
implies that $\hat\psi_\epsilon \in C_{w}([0,T]; L^1_{
M}
(\Omega \times D))$. On recalling point  \ding{206} of Lemma \ref{psi0properties}
and (\ref{psitwconL2a},d) and following the abstract framework in
Temam \cite{Temam}, Chapter 3, Section 4, we deduce that $\hat\psi$ satisfies the initial condition
in the sense of weakly continuous functions from $[0,T]$ into $L^1_M(\Omega \times D)$.
The proof for $\ut_\epsilon$ is similar, by taking $X=\Ht$ and
$Y=\Vt_\sigma'$ and noting that $\ut^0 \rightarrow \ut_0$ weakly in $\Ht$ as
$\Delta t \rightarrow 0_+$. That completes Step 3.9.

\smallskip

\textit{Step 3.10.} The energy inequality \eqref{eq:energyest} is a direct consequence of (\ref{uwconL2a}-c)
and (\ref{psiwconH1a},b,d), on noting the (weak) lower-semicontinuity of the terms on the left-hand side
of \eqref{eq:energy-u+psi-final2}, and (\ref{fatou-app}).
That completes Step 3.10.

\smallskip

\textit{Step 3.11.} It remains to prove \eqref{mass-conserved}.
To this end, it follows from (\ref{eq:energy-zeta}), (\ref{zetacondtbd})
and on recalling that $\zeta_{\epsilon,L}^{\Delta t} \in L^2(0,T;{\cal K})$
that
there exists a subsequence of $\{\zeta_{\epsilon,L}^{\Delta t}\}_{L>1}$
(not indicated) with $\Delta t =o(L^{-1})$ and  a function $\zeta_\epsilon
\in 
L^2(0,T;{\cal K})\cap L^2(0,T;H^1(\Omega)')$ such that,
as $L \rightarrow \infty$ (and thereby $\Delta t \rightarrow 0_+$)
\begin{subequations}
\begin{alignat}{2}
\zeta_{\epsilon,L}^{\Delta t(,\pm)} &\rightarrow \zeta_\epsilon \qquad&&
\mbox{weak* in } L^\infty(0,T;L^\infty(\Omega)),
\label{zwconL2} \\
\zeta_{\epsilon,L}^{\Delta t,+} &\rightarrow \zeta_\epsilon \qquad&&
\mbox{weakly in } L^2(0,T;H^1(\Omega)),
\label{zwconH1}\\
\frac{\partial\zeta_{\epsilon,L}^{\Delta t}}{\partial t}
&\rightarrow \frac{\partial \zeta_\epsilon}
{\partial t} \qquad&&
\mbox{weakly in } L^2(0,T;H^1(\Omega)').
\label{ztwconH-1}
\end{alignat}
\end{subequations}
Noting (\ref{zwconL2}--c) and (\ref{usconL2a}), we can pass to the limit
as $L \rightarrow \infty$ in (\ref{zetacon}) to obtain that
\begin{align}
&\int_0^T \langle
\frac{\partial \zeta_{\epsilon}}{\partial t}, \varphi \rangle_{H^1(\Omega)} \dt
+ \int_0^T \int_{\Omega} \left[ \epsilon\, \nabx  \zeta_{\epsilon} -
\ut_{\epsilon}  \,\zeta_{\epsilon} \right]
\cdot \nabx \varphi \dx \dt =0
\qquad \forall \varphi \in L^2(0,T;H^1(\Omega)).
\label{zetaconlim}
\end{align}
Noting Fubini's theorem, point  \ding{206} of Lemma \ref{psi0properties}, see also (\ref{psioLconv}),
and (\ref{inidata}),
we have for any $\varphi \in C^\infty(\overline{\Omega})
\subset L^\infty(\Omega \times D)$ that
as $L \rightarrow \infty$
\begin{align}
\int_{\Omega} \zeta^{\Delta t}_{\epsilon,L}(\xt,0)\, \varphi(\xt) \dx
&= \int_{\Omega} \left(\int_D M(\qt) \,\beta^L (\hat \psi^0(\xt,\qt)) \dq \right)
\varphi(\xt) \dx
\nonumber \\
&= \int_{\Omega\times D} M(\qt) \,\beta^L (\hat \psi^0(\xt,\qt)) \,\varphi(\xt) \dq \dx
\nonumber \\
& \rightarrow
\int_{\Omega\times D} M(\qt) \,\hat \psi_ 0(\xt,\qt)) \,\varphi(\xt) \dq \dx
\nonumber \\
& = \int_{\Omega} \left(\int_D M(\qt)\,\hat \psi_ 0(\xt,\qt)) \dq \right)
\varphi(\xt) \dx =
\int_{\Omega} \varphi(\xt) \dx.
\label{zeta0conv}
\end{align}
Hence we deduce from (\ref{zwconL2}--c) and (\ref{zeta0conv}) that $\zeta_\epsilon(\xt,0) = 1$ for a.e.\ $\xt \in \Omega$.
We note also from Fubini's theorem, (\ref{zetancon}) and (\ref{psisconL2a}) that
\begin{align}
\int_0^T \int_\Omega |\zeta_{\epsilon,L}^{\Delta t} - \int_D M \,\hat \psi_{\epsilon} \dq|
\dx \dt &= \int_0^T \int_\Omega |\int_D M \,(\hat \psi_{\epsilon,L}^{\Delta t}-
\hat \psi_{\epsilon}) \dq|
\dx \dt \nonumber \\
& \leq \int_0^T \int_{\Omega\times D} M \,|\hat \psi_{\epsilon,L}^{\Delta t}-
\hat \psi_{\epsilon}|
\dq \dx \dt \rightarrow 0 \quad \mbox{as} \quad L \rightarrow \infty.
\label{equivcon}
\end{align}
Hence $\zeta_{\epsilon,L}^{\Delta t} \rightarrow \int_D M \,\hat \psi_{\epsilon} \dq$
strongly in $L^1(0,T;L^1(\Omega))$.
Comparing this with (\ref{zwconL2}), we deduce that
\begin{align}
\zeta_\epsilon(\xt,t) = \int_D M(\qt) \,\hat \psi_{\epsilon}(\xt,\qt,t) \dq
\qquad \mbox{for a.e. } (\xt,t) \in \Omega \times (0,T).
\label{zetaequiv}
\end{align}
Clearly the linear parabolic problem \eqref{zetaconlim}
with initial data $\zeta_\epsilon(\xt,0) = 1$ for a.e.\ $\xt \in \Omega$
has the unique solution $\zeta_\epsilon \equiv 1$ on $\Omega \times [0,T]$. This implies \eqref{mass-conserved}, and completes Step 3.11 and the proof.
\end{proof}


\section{Exponential decay to the equilibrium solution}\label{sec:decay}
\setcounter{equation}{0}

We shall show that, in the absence of a body force (i.e. with $\ft \equiv \zerot$),
weak solutions $(\ut_\epsilon, \hat\psi_\epsilon)$ to (P$_\epsilon$), whose existence we have proved in the
previous section, decay exponentially in time to the trivial solution of the steady counterpart of problem
(P$_\epsilon$) at a rate that is independent of the specific choice of the initial data for the Navier--Stokes
and Fokker--Planck equations.
Our result is similar to the one derived by Jourdain, Leli\`evre, Le Bris \& Otto \cite{JLLO}, except that
the arguments there were partially formal in the sense that the existence of a unique global-in-time solution,
which was required to be regular enough, was assumed; in fact, the probability
density function was supposed to be a classical solution to the Fokker--Planck equation;
(cf. p.105, 
therein; as well as the recent paper
of Arnold, Carrillo \& Manzini \cite{ACM} for refinements and extensions).
In contrast, we require no additional regularity hypotheses here.

\begin{theorem} Suppose that the assumptions of Theorem \ref{convfinal} hold and $M$ satisfies the
Bakry--\'{E}mery condition (cf. Remark \ref{rem5.1}) with $\kappa>0$; then, for any $T>0$,
\begin{align}\label{eq:bd55a}
&\|\ut_\epsilon(T)\|^2 + \frac{k}{|\Omega|} \|\hat\psi_\epsilon(T) - 1\|^2_{L^1_M(\Omega \times D)}
\nonumber\\&\hspace{1in}
\leq {\rm e}^{-\gamma_0 T}\left[ \|\ut_0\|^2 + 2k\int_{\Omega \times D} M \mathcal{F}(\hat\psi_0) \dq \dx \right]
+ \frac{1}{\nu}\int_{0}^{T}\!\!\|\ft\|^2_{V'} \dd s,
\end{align}
where $\gamma_0 := \min\left(\frac{\nu}{C_{\sf P}^2}\,,\,\frac{\kappa\,a_0}{2\lambda}\right)$. In particular if $\ft \equiv 0$,
the following inequality holds:
\begin{align}\label{eq:bd55aa}
&\|\ut_\epsilon(T)\|^2 + \frac{k}{|\Omega|} \|\hat\psi_\epsilon(T) - 1\|^2_{L^1_M(\Omega \times D)}
\leq {\rm e}^{-\gamma_0 T}\left[ \|\ut_0\|^2 + 2k\int_{\Omega \times D} M \mathcal{F}(\hat\psi_0) \dq \dx \right].
\end{align}
\end{theorem}

\begin{proof}
We take $t = t_1 = \Delta t$ and write $0=t_0$ in \eqref{eq:energy-u+psi2}, and we replace the function $\mathcal{F}$ on the left-hand side of \eqref{eq:energy-u+psi2}
by $\mathcal{F}^L$, noting that, prior to \eqref{eq:energy-u+psi2}, in \eqref{eq:energy-u+psi} we in fact had $\mathcal{F}^L$ on the left-hand side of the inequality, which was subsequently bounded below by $\mathcal{F}$; thus
we reinstate the $\mathcal{F}^L$ we previously had. We
recall that $\ut^0 = \uta^{\Delta t,-}(t_1)$ and $\beta^L(\hat\psi^0) = \psia^{\Delta t, -}(t_1)$ and adopt the notational convention $t_{-1}:=-\infty$ (say), which allows us to write $\uta^{\Delta t,+}(t_0)$ instead of $\uta^{\Delta t,-}(t_1)$
and $\psia^{\Delta t,+}(t_0)$ instead of $\psia^{\Delta t,-}(t_1)$.
Hence we have that
\begin{align}\label{eq:energy-u+psi2-local1}
&\|\uta^{\Delta t, +}(t_1)\|^2
+ \frac{1}{\Delta t} \int_{t_0}^{t_1} \|\uta^{\Delta t, +} - \uta^{\Delta t,-}\|^2 \dd s \nonumber \\
& \hspace{0.5in} + \left(\nu - \alpha\,\frac{
2\lambda\,k\,C_M} 
{a_0}\right)
\int_{t_0}^{t_1} \|\nabxtt \uta^{\Delta t, +}(s)\|^2 \dd s
\nonumber\\
& \hspace{0.5in}
+ 2k\int_{\Omega \times D}\!\!\! M \mathcal{F}^L(\psia^{\Delta t, +}(t_1) + \alpha)
\dq \dx + \frac{k}{\Delta t\,L} \int_{t_0}^{t_1} \int_{\Omega \times D}\!\!\! M
(\psia^{\Delta t, +} - \psia^{\Delta t, -})^2 \dq \dx \dd s
\nonumber \\
&\hspace{0.5in}
+\, 2k\,\varepsilon \int_{t_0}^{t_1} \int_{\Omega \times D} M
\frac{|\nabx \psia^{\Delta t, +} |^2}{\psia^{\Delta t, +} + \alpha} \dq \dx \dd s
+\, \frac{a_0 k}{2\,\lambda}  \int_{t_0}^{t_1}
\int_{\Omega \times D}M\, \,\frac{|\nabq \psia^{\Delta t, +}|^2}{\psia^{\Delta t, +} + \alpha} \,\dq \dx \dd s\nonumber\\
&\quad \leq 
\|\uta^{\Delta t,+}(t_0)\|^2 + \frac{1}{\nu}\int_{t_0}^{t_1}\!\!\|\ft^{\Delta t,+}(s)\|^2_{V'} \dd s
+2k \int_{\Omega \times D} M \mathcal{F}^L(\psia^{\Delta t,+}(t_0) + \alpha)\dq \dx.
\end{align}
Closer inspection of the procedure that resulted in inequality \eqref{eq:energy-u+psi2} reveals that
\eqref{eq:energy-u+psi2} could have been, equivalently, arrived at by repeating the argument that gave
us \eqref{eq:energy-u+psi2-local1} on each time interval $[t_{n-1},t_n]$, $n=1,\dots,N$; viz.,
\begin{align}\label{eq:energy-u+psi2-local2}
&\|\uta^{\Delta t, +}(t_n)\|^2
+ \frac{1}{\Delta t}\! \int_{t_{n-1}}^{t_n} \|\uta^{\Delta t, +} - \uta^{\Delta t,-}\|^2 \dd s\nonumber\\
&\hspace{0.5in}+
 \left(\nu - \alpha\,\frac{
 2\lambda\,k\,C_M 
 }
{a_0}\right)
\int_{t_{n-1}}^{t_n} \|\nabxtt \uta^{\Delta t, +}(s)\|^2 \dd s\nonumber\\
&\hspace{0.5in}+ 2k\int_{\Omega \times D}\!\! M \mathcal{F}^{L}(\psia^{\Delta t, +}(t_n) + \alpha) \dq \dx
+ \frac{k}{\Delta t\,L} \int_{t_{n-1}}^{t_n}
\int_{\Omega \times D}\!\! M (\psia^{\Delta t, +} - \psia^{\Delta t, -})^2 \dq \dx \dd s
\nonumber \\
&\hspace{0.5in}+2k\,\varepsilon \int_{t_{n-1}}^{t_n} \int_{\Omega \times D} M
\frac{|\nabx \psia^{\Delta t, +} |^2}{\psia^{\Delta t, +} + \alpha} \dq \dx \dd s
+ \frac{a_0 k}{2\,\lambda}  \int_{t_{n-1}}^{t_n}
\int_{\Omega \times D}M\, \,\frac{|\nabq \psia^{\Delta t, +}|^2}{\psia^{\Delta t, +} + \alpha} \,\dq \dx \dd s \nonumber\\
&\qquad \leq
\|\uta^{\Delta t,+}(t_{n-1})\|^2 + \frac{1}{\nu}\int_{t_{n-1}}^{t_n}\!\!\|\ft^{\Delta t,+}(s)\|^2_{V'} \dd s
\nonumber \\
& \hspace{1.5in} + 2k \int_{\Omega \times D}\!\! M \mathcal{F}^L(\psia^{\Delta t,+}(t_{n-1}) + \alpha) \dq \dx,\qquad n=1,\dots, N,
\end{align}
summing these through $n$ and then bounding $\mathcal{F}^L$ on the left-hand side below by $\mathcal{F}$.

Here we proceed differently: we shall retain $\mathcal{F}^L$ on both sides of \eqref{eq:energy-u+psi2-local2},
and omit the second, fifth and sixth term from the
left-hand side of \eqref{eq:energy-u+psi2-local2}. Thus we
have that
\begin{align}\label{eq:energy-u+psi2-local3}
&\|\uta^{\Delta t, +}(t_n)\|^2
+
\left(\nu - \alpha\,\frac{
2\lambda\,k\,C_M 
}
{a_0}\right)
\int_{t_{n-1}}^{t_n} \|\nabxtt \uta^{\Delta t, +}(s)\|^2 \dd s
\nonumber\\
&\hspace{0.5in}+ 2k\int_{\Omega \times D}\!\! M
\mathcal{F}^L(\psia^{\Delta t, +}(t_n) + \alpha) \dq \dx
\nonumber\\
&\hspace{1in} +\frac{2 a_0 k}{\lambda}  \int_{t_{n-1}}^{t_n}
\int_{\Omega \times D}M\, \,|\nabq \sqrt{\psia^{\Delta t, +}+\alpha}|^2 \,\dq \dx \dd s\nonumber\\
& \qquad \leq
\|\uta^{\Delta t,+}(t_{n-1})\|^2
+ \frac{1}{\nu}\int_{t_{n-1}}^{t_n}\!\!\|\ft^{\Delta t,+}(s)\|^2_{V'} \dd s\nonumber\\
&\hspace{1in } + 2k
\int_{\Omega \times D}\!\! M \mathcal{F}^L(\psia^{\Delta t,+}(t_{n-1})+\alpha) \dq \dx,
\qquad n=1,\dots, N.
\end{align}
Thanks to Poincar\'e's inequality, recall (\ref{Poinc}), there exists a positive constant
$C_{\sf P}=C_{\sf P}(\Omega)$, such that
\begin{align}\label{bd1}
\| \uta^{\Delta t, +}(\cdot,s) \| \leq C_{{\sf P}}(\Omega)\, \|\nabxtt \uta^{\Delta t, +}
(\cdot,s)\|
\end{align}
for $s \in (t_{n-1},t_n]$; $n=1,\dots, N$.
Also, by the logarithmic Sobolev inequality \eqref{eq:logs1}, we have
for a.e.\ $\xt \in \Omega$ that
\begin{align*}
&\int_D M(\qt) [\psia^{\Delta t, +}(\xt,\qt,s)+\alpha]
\log \frac{[\psia^{\Delta t, +}(\xt,\qt,s)+\alpha]}{\int_D M(\qt)[\psia^{\Delta t, +}(\xt,\qt,s) +\alpha]\dq} \dq\\
&\hspace{3in} \leq \frac{2}{\kappa} \int_D M(\qt)
\left|\nabq \sqrt{\psia^{\Delta t, +}(\xt,\qt,s)+\alpha}\right|^2 \dq,
\end{align*}
for $s \in (t_{n-1},t_n]$; $n=1,\dots, N$. Hence, for a.e.\ $\xt \in \Omega$,
\begin{align}\label{last-bb}
&\int_{D} M(\qt) [\psia^{\Delta t, +}(\xt,\qt,s) +\alpha] \log [\psia^{\Delta t, +}(\xt,\qt,s)+\alpha]\dq\nonumber\\
&\qquad \leq \frac{2}{\kappa} \int_D M(\qt) \left|\nabq \sqrt{\psia^{\Delta t, +}(\xt,\qt,s)+\alpha}
\right|^2 \dq\nonumber\\
&\qquad \qquad + \left(\int_D M(\qt) [\psia^{\Delta t, +}(\xt,\qt,s)+\alpha]\dq\right)\,\log\left( \int_D M(\qt) [\psia^{\Delta t, +}(\xt,\qt,s)+\alpha]\dq\right),
\end{align}
for $s \in (t_{n-1},t_n]$, $n=1,\dots, N$. Note that, thanks to \eqref{properties-b}
and the monotonicity of the mapping $s\in \mathbb{R}_{>0} \mapsto \log s\in \mathbb{R}$,
the second factor in the second term on the right-hand
side of \eqref{last-bb} is $\leq \log(1+\alpha)$. Since $\alpha \in (0,1)$, we have $\log(1+\alpha)>0$;
also, the first factor in the
second term on the right-hand side of \eqref{last-bb} is positive thanks to \eqref{properties-a} and
by \eqref{properties-b} it is bounded above by $(1+\alpha)$. Hence the
second term on the right-hand side of \eqref{last-bb} is bounded above by
the product $(1+\alpha) \log(1+\alpha)$. We integrate the resulting inequality over $\Omega$ to deduce that
\begin{align*}
&\int_{\Omega \times D}\!\!\! M(\qt) [\psia^{\Delta t, +}(\xt,\qt,s)+\alpha] \log [\psia^{\Delta t, +}(\xt,\qt,s)+\alpha]\dq\dx \\
&\hspace{1in} \leq \frac{2}{\kappa} \int_{\Omega \times D}\!\!\! M(\qt) |\nabq \sqrt{\psia^{\Delta t, +}(\xt,\qt,s)+\alpha}|^2 \dq \dx + |\Omega|\,(1+\alpha) \log(1+\alpha),
\end{align*}
for $s \in (t_{n-1},t_n]$, $n=1,\dots, N$. Equivalently, on noting that $s \log s = \mathcal{F}(s) - (1 - s)$, we can rewrite the last
inequality in the following form:
\begin{align}\label{penultimate}
&\int_{\Omega \times D} M(\qt) \mathcal{F}(\psia^{\Delta t, +}(\xt,\qt,s)+\alpha) \dq\dx \nonumber\\
&\hspace{0.5in} \leq \frac{2}{\kappa} \int_{\Omega \times D} M(\qt) |\nabq \sqrt{\psia^{\Delta t, +}(\xt,\qt,s)+\alpha}|^2 \dq \dx\nonumber\\
&\hspace{1in}+ \int_{\Omega \times D} M(\qt)(1 - \psia^{\Delta t, +}(\xt,\qt,s)-\alpha) \dq \dx
+ |\Omega|\,(1+\alpha) \log(1+\alpha),
\end{align}
for $s \in (t_{n-1},t_n]$, $n=1,\dots, N$.
This then in turn implies, thanks to the fact that $\psia^{\Delta t, +}(\xt,\qt,\cdot)$ is constant on $(t_{n-1},t_n]$ for
all $(\xt,\qt) \in \Omega \times D$, that
\begin{align}\label{bd2}
&\frac{\kappa\, a_0\, k}{\lambda}\, \Delta t \,\int_{\Omega \times D} M(\qt) \mathcal{F}(\psia^{\Delta t, +}(t_n)+\alpha) \dq\dx \nonumber\\
&\quad\leq \frac{2a_0\,k}{\lambda} \int_{t_{n-1}}^{t_n} \int_{\Omega \times D}  M\, |\nabq \sqrt{\psia^{\Delta t, +}+\alpha}|^2 \dq \dx \,{\rm d}s \nonumber\\
&\qquad +\frac{\kappa\, a_0\, k}{\lambda}\, \int_{t_{n-1}}^{t_n}\int_{\Omega \times D}\!\! M\, (1 - \psia^{\Delta t, +} -\alpha) \dq \dx \dd s
\nonumber \\
&\qquad
+ \frac{\kappa\,a_0\, k}{\lambda}\,\Delta t\,|\Omega|\,(1+\alpha) \log(1+\alpha),
\end{align}
for $n=1,\dots, N$.

Using \eqref{bd1} and \eqref{bd2} in \eqref{eq:energy-u+psi2-local3} yields
\begin{align}
&\left(1+ \frac{\Delta t}{C_{\sf P}^2}
\left(\nu - \alpha\,\frac{
2\lambda\,k\,C_M 
}
{a_0}\right) \right)\|\uta^{\Delta t, +}(t_n)\|^2
\nonumber\\
&\hspace{0.5in}+ \left(1+ \frac{\kappa\, a_0}{2\lambda}\Delta t\right) 2k\,\int_{\Omega \times D}\!\! M \mathcal{F}(\psia^{\Delta t, +}(t_n)+\alpha) \dq \dx
\nonumber \\
& \hspace{0.5in}+\, 2k \int_{\Omega \times D}M\,[\mathcal{F}^L(\psia^{\Delta t,+}(t_n) + \alpha) - \mathcal{F}(\psia^{\Delta t,+}(t_n)+\alpha)]\dq \dx\nonumber\\
&\qquad \leq
\|\uta^{\Delta t,+}(t_{n-1})\|^2 + 2k \int_{\Omega \times D}\!\! M \mathcal{F}(\psia^{\Delta t,+}(t_{n-1})+\alpha) \dq \dx
\nonumber \\
&\hspace{0.2in}\qquad +
 2k \int_{\Omega \times D}M\,[\mathcal{F}^L(\psia^{\Delta t,+}(t_{n-1}) + \alpha) - \mathcal{F}(\psia^{\Delta t,+}(t_{n-1})+\alpha)]\dq \dx
\nonumber\\
&\hspace{0.2in}\qquad+\, \frac{\kappa\,a_0\, k}{\lambda}\, \int_{t_{n-1}}^{t_n}\int_{\Omega \times D} M (1 - \psia^{\Delta t, +}-\alpha) \dq \dx \dt\nonumber\\
&\hspace{0.2in}\qquad+\,\frac{\kappa\,a_0\, k}{\lambda}\, \Delta t\,|\Omega|\,(1+\alpha) \log(1+\alpha)\,  +\, \frac{1}{\nu}\int_{t_{n-1}}^{t_n}\!\!\|\ft^{\Delta t,+}\|^2_{V'} \dd s,
%
\qquad\mbox{for $n=1,\dots,N$.}
\label{eq:energy-u+psi2-local4}
\end{align}%
We now introduce, for $n=1,\dots,N$, the following notation:
\begin{align*}
\gamma(\alpha)&:= \min\left(\frac{1}{C_{\sf P}^2}
\left(\nu - \alpha\,\frac{
2\lambda\,k\,C_M 
}
{a_0}\right)
,\frac{\kappa\,a_0}{2\lambda}\right),\\
A_n(\alpha)   &:= \|\uta^{\Delta t, +}(t_n)\|^2 + 2k\,\int_{\Omega \times D}\!\! M \mathcal{F}(\psia^{\Delta t, +}(t_n)+\alpha) \dq \dx,\\
B_n(\alpha)   &:= 2k \int_{\Omega \times D}M\,[\mathcal{F}^L(\psia^{\Delta t,+}(t_n) + \alpha) - \mathcal{F}(\psia^{\Delta t,+}(t_n)+\alpha)]\dq \dx,\\
C_n(\alpha)   &:= \frac{\kappa\,a_0\,k}{\lambda}\, \int_{t_{n-1}}^{t_n}\int_{\Omega \times D} M (1 - \psia^{\Delta t, +}(\xt,\qt,s)-\alpha) \dq \dx \dd s\\
&\hspace{0.7cm} +\,\frac{\kappa\,a_0\, k}{\lambda}\, \Delta t\,|\Omega|\,(1+\alpha) \log(1+\alpha) + \frac{1}{\nu}\int_{t_{n-1}}^{t_n}\!\!\|\ft^{\Delta t,+}\|^2_{V'} \dd s.
\end{align*}
We shall assume henceforth that $\alpha$ is sufficiently small in the sense that
\eqref{alphacond} holds. For all such $\alpha$,
$\gamma(\alpha) 
> 0$;
further, trivially, $A_n(\alpha)$ is nonnegative; by \eqref{eq:FL2c},  we have that
$B_n(\alpha)$ is nonnegative, and by \eqref{additional2}, $C_n(\alpha)$ is also nonnegative. In terms of this notation \eqref{eq:energy-u+psi2-local4} can be rewritten in the following compact form:
\[ \left(1 + \gamma(\alpha) \, \Delta t\right) A_n(\alpha) + B_n(\alpha) \leq
A_{n-1}(\alpha) + B_{n-1}(\alpha) + C_n(\alpha),\qquad n=1,\dots, N.\]
Equivalently, we can write this as follows:
\[ \left(1 + \gamma(\alpha) \, \Delta t\right) A_n(\alpha) \leq
A_{n-1}(\alpha) + D_n(\alpha),\qquad n=1,\dots, N,\]
where $D_n(\alpha):= C_n(\alpha) - (B_{n}(\alpha) - B_{n-1}(\alpha))$. It then follows by induction that
\[ A_n(\alpha) \leq (1 + \gamma(\alpha) \,\Delta t
)^{-n} A_0(\alpha) + \sum_{j=1}^n D_j(\alpha),\qquad n = 1,\dots, N.\]
That is,
\[ A_n(\alpha) + B_n(\alpha) \leq 
( 
1 + \gamma(\alpha) \,\Delta t
)^{-n} A_0(\alpha) + \left\{B_0(\alpha) + \sum_{j=1}^n C_j(\alpha)\right\},\qquad n = 1,\dots, N.\]
In particular, with $n=N$, by omitting the nonnegative term $B_N(\alpha)$ from the left-hand side of the resulting inequality, and recalling that $T = t_N = N \Delta t$, we have that
\begin{align}\label{bd3}
&\|\uta^{\Delta t, +}(T)\|^2 + 2k\,\int_{\Omega \times D}\!\! M \mathcal{F}(\psia^{\Delta t, +}(T)+\alpha) \dq \dx \nonumber\\
&\quad \leq 
\left(1+ \frac{\gamma(\alpha) \,T}
{N}
\right)^{-N}
\left[ \|\uta^{\Delta t, +}(0)\|^2 + 2k\,\int_{\Omega \times D}\!\! M \mathcal{F}(\psia^{\Delta t, +}(0)+\alpha) \dq \dx \right]\nonumber\\
&\quad\quad +\, 2k\, \int_{\Omega \times D} M\,\left[\mathcal{F}^L(\psia^{\Delta t, +}(0)+\alpha)
- \mathcal{F}(\psia^{\Delta t, +}(0)+\alpha)\right] \dq\dx\nonumber\\
&\quad\quad + \frac{\kappa\,a_0\, k}{\lambda}\, \int_{0}^{T}\int_{\Omega \times D} M (1 - \psia^{\Delta t, +}(\xt,\qt,s)-\alpha) \dq \dx \dd s
\nonumber\\
&\quad\quad +\,\frac{\kappa\, a_0\, k}{\lambda}\, T\,|\Omega|\,(1+\alpha) \log(1+\alpha) + \frac{1}{\nu}\int_{0}^{T}\!\!\|\ft^{\Delta t,+}\|^2_{V'} \dd s.
\end{align}
Using that
\[\uta^{\Delta t, +}(0) = \ut^0\qquad \mbox{and}\qquad \psia^{\Delta t,+} = \beta^L(\hat\psi^0),\]
we then obtain from \eqref{bd3} that
\begin{align}\label{bd3a}
&\|\uta^{\Delta t, +}(T)\|^2 + 2k\,\int_{\Omega \times D}\!\! M \mathcal{F}(\psia^{\Delta t, +}(T)+\alpha) \dq \dx \nonumber\\
&\hspace{0.5in} \leq 
\left(1+ \frac{\gamma(\alpha) \,T}
{N}
\right)^{-N}
\left[ \|\ut^0\|^2 + 2k\,\int_{\Omega \times D}\!\! M \mathcal{F}(\beta^L(\hat\psi^0)+\alpha) \dq \dx \right]\nonumber\\
&\hspace{1in} +\, 2k\, \int_{\Omega \times D} M\,\left[\mathcal{F}^L(\beta^L(\hat\psi^0)+\alpha)
- \mathcal{F}(\beta^L(\hat\psi^0)+\alpha)\right] \dq\dx\nonumber\\
&\hspace{1in} + \frac{\kappa\,a_0\, k}{\lambda}\, \int_{0}^{T}\int_{\Omega \times D} M (1 - \psia^{\Delta t, +}(\xt,\qt,s)-\alpha) \dq \dx \dd s
\nonumber\\
&\hspace{1in}
 +\,\frac{\kappa\,a_0\, k}{\lambda}\, T\,|\Omega|\,(1+\alpha) \log(1+\alpha) + \frac{1}{\nu}\int_{0}^{T}\!\!\|\ft^{\Delta t,+}\|^2_{V'} \dd s.
\end{align}
Applying \eqref{eq:FL2c} and \eqref{bound-on-t5} in the second factor in the first term on the right-hand side  of \eqref{bd3a}
and applying \eqref{before-two} in the square brackets in the second term on the right-hand side, we have that
\begin{align}\label{bd4}
&\|\uta^{\Delta t, +}(T)\|^2 + 2k\,\int_{\Omega \times D}\!\! M \mathcal{F}(\psia^{\Delta t, +}(T)+\alpha) \dq \dx \nonumber\\
& \hspace{0.5in}\leq 
\left(1+ \frac{\gamma(\alpha) \,T}
{N}
\right)^{-N}
\left[ \|\ut^0\|^2 + 3\alpha\,k\, |\Omega| + 2k\int_{\Omega \times D} M \mathcal{F}(\hat\psi^0 + \alpha) \dq \dx \right]\nonumber\\
&\hspace{1in}
+ 3 \alpha\,k\,|\Omega|\,+ \frac{\kappa\,a_0\, k}{\lambda}\, \int_{0}^{T}\int_{\Omega \times D} M (1 - \psia^{\Delta t, +}(\xt,\qt,s)-\alpha) \dq \dx \dd s
\nonumber\\
&\hspace{1in} +\,\frac{\kappa\,a_0\, k}{\lambda}\, T \,|\Omega|\,(1+\alpha) \log(1+\alpha) + \frac{1}{\nu}\int_{0}^{T}\!\!\|\ft^{\Delta t,+}\|^2_{V'} \dd s.
\end{align}
We now pass to the limit $\alpha \rightarrow 0_+$, with $L$ and $\Delta t$ fixed, in much the same way as in the previous
section. Noting that $\lim_{\alpha \rightarrow 0_+}\gamma(\alpha)= \gamma_0$,
we thus obtain from \eqref{bd4}, \eqref{idatabd}
and \eqref{inidata-1}, that
\begin{align}\label{bd5}
&\|\uta^{\Delta t, +}(T)\|^2 + 2k\,\int_{\Omega \times D}\!\! M \mathcal{F}(\psia^{\Delta t, +}(T)) \dq \dx \nonumber\\
&\hspace{0.5in} \leq \left(1 + \frac{\gamma_0\, T}{N}\right)^{-N}\left[ \|\ut_0\|^2 + 2k\int_{\Omega \times D} M \mathcal{F}(\hat\psi_0) \dq \dx \right]\nonumber\\
&\hspace{1in}+ \frac{\kappa\,a_0\, k}{\lambda}\, \int_{0}^{T}\int_{\Omega \times D} M (1 - \psia^{\Delta t, +}) \dq \dx \dd s
+ \frac{1}{\nu}\int_{0}^{T}\!\!\|\ft^{\Delta t,+}\|^2_{V'} \dd s.
\end{align}
Using \eqref{usconL2a}, \eqref{fatou-app}, \eqref{psisconL2a}, \eqref{mass-conserved} and \eqref{fncon},
we pass to the limit with $L \rightarrow \infty$ (whereby $\Delta t \rightarrow 0_+$ according to $\Delta t = o(L^{-1})$ and therefore $N=T/\Delta t \rightarrow \infty$), to deduce from \eqref{bd5} that
\begin{align}\label{bd7}
&\|\ut_\epsilon(T)\|^2 + 2k\,\int_{\Omega \times D}\!\! M \mathcal{F}(\hat\psi_\epsilon(T)) \dq \dx \nonumber\\
&\hspace{1.5in}\leq {\rm e}^{-\gamma_0 T}\left[ \|\ut_0\|^2 + 2k\int_{\Omega \times D} M \mathcal{F}(\hat\psi_0) \dq \dx \right]
+ \frac{1}{\nu}\int_{0}^{T}\!\!\|\ft\|^2_{V'} \dd s.
\end{align}

Now, the Csisz\'ar--Kullback inequality (cf., for example, (1.1) and (1.2) in the work of
Unterreiter, Arnold, Markowich \& Toscani \cite{UAMT})
with respect to the Gibbs measure $\undertilde{\mu}$
defined by $\dd \undertilde{\mu}= M(\qt)\dq$,  yields, on noting \eqref{mass-conserved}, for a.e.\ $\xt \in \Omega$, that
\[ \|\hat\psi_\epsilon(\xt,\cdot,T) - 1\|_{L^1_M(D)} \leq \left[2 \int_{D} M \mathcal{F}(\hat\psi_\epsilon(\xt,\qt,T))
\dq\right]^{\frac{1}{2}},\]
which, after integration over $\Omega$ implies, by the Cauchy--Schwarz inequality and subsequent squaring of both sides,
that
\[ \|\hat\psi_\epsilon(T) - 1\|^2_{L^1_M(\Omega \times D)} \leq 2|\Omega| \int_{\Omega \times D} M \mathcal{F}(\hat\psi_\epsilon(T))
\dq\dx.\]
By combining this inequality with \eqref{bd7}, we deduce (\ref{eq:bd55a}).
%
%
Taking $\ft\equiv 0$, the stated exponential decay in time of $(\ut_\epsilon,\hat\psi_\epsilon)$ to $(\zerot,1)$
in the $\Lt^2(\Omega) \times L^1_M(\Omega \times D)$ norm then follows from \eqref{eq:bd55a}.
\end{proof}

\begin{remark}{\em
By recalling our notational convention $\psi_\epsilon = M \hat\psi_\epsilon$, we see from \eqref{eq:bd55aa} that,
in the absence of an external force, from any initial datum $(\ut_0, \psi_0)$ with initial velocity $\ut_0 \in \Ht$
and initial probability density function $\psi_0$ that has finite relative entropy with respect to the Maxwellian $M$,
the solution will, as $t \rightarrow \infty$, evolve to the trivial solution $(\zerot, M)$ of the steady counterpart of
the unsteady problem at an exponential rate that is independent of the choice of the initial datum.
By introducing the {\em free energy} (the sum of the kinetic energy and the relative entropy):
\[\mathfrak{E}(t):= \frac{1}{2}\|\ut_\epsilon(t)\|^2 + k\,\int_{\Omega \times D}\!\! M \mathcal{F}(\hat\psi_\epsilon(t)) \dq \dx,\]
we deduce from \eqref{bd7} that, for any $T>0$,
\[ \mathfrak{E}(T) \leq {\rm e}^{-\gamma_0 T} \mathfrak{E}(0) + \frac{1}{2\nu}\int_{0}^{T}\!\!\|\ft\|^2_{V'} \dd s.\]
Thus in particular when $\ft=\zerot$, the free energy decays to $0$ as a function of time.}\qquad $\diamond$
\end{remark}

~\vspace{-6mm}

\begin{remark}{\em
It is interesting to note the dependence of \[\gamma_0 = \min\left(\frac{\nu}{C_{\sf P}^2}\,,\,\frac{\kappa\,a_0}{2\lambda}\right),\]
the rate
at which the fluid relaxes to equilibrium, on the dimensionless viscosity coefficient $\nu$ of the solvent, the minimum eigenvalue $a_0$
of the Rouse matrix $A$, the geometry of the flow domain encoded in the Poincar\'e constant $C_{\sf P}(\Omega)$, the Weissenberg number $\lambda$, and the Bakry--\'Emery constant $\kappa$ for the Maxwellian $M$ of the model;
as we have shown in Remark \ref{rem5.1},
in the case of a Hookean chain $\kappa=1$.
We observe in particular that the right-hand side of the energy inequality \eqref{eq:energyest} and the rate  $\gamma_0$
at which the fluid relaxes to equilibrium are independent of
the centre-of-mass diffusion coefficient $\varepsilon$ appearing in \eqref{fp0}.}\qquad $\diamond$
\end{remark}


\noindent
\textbf{Acknowledgement}
ES was supported by the EPSRC Science and Innovation award to the Oxford Centre
for Nonlinear PDE (EP/E035027/1).


\bibliographystyle{siam}

\bibliography{polyjwbesrefs}

\appendix

\section{Density of $C^\infty_0(D)$ in $L^2_w(D)$ and $H^1_w(D)$}
\label{AppendixC}
\setcounter{equation}{0}

Suppose that $w$ is a positive measurable function defined on $D:=\mathbb{R}^{Kd}$ such that
$w, w^{-1} \in L^\infty(B)$ for any bounded open ball $B\subset D$ centred at the origin
of $D$.

Consider, for $j = 1, 2, \dots$, the sequence of
concentric open balls $B_j:= B(\zerot, j)$ in $D$. Consider further the sequence of functions
$\{\zeta_j\}_{j\geq 1} \subset C^\infty_0(D)$ such that $0 \leq \zeta_j \leq 1$ on $D$;
$\zeta_j \equiv 1$ on $\overline{B}_j$; $\zeta_j \equiv 0$ on $D\setminus B_{j+1}$;
$|\zeta_j|_{W^{1,\infty}(B_{j+1}\setminus B_j)} \leq C$, where $C$ is a fixed positive
constant, $C>1$, independent of $j$.

For $v \in L^2_w(D)$, define $v_j := \zeta_j\, v$.  Clearly,
\[ \| v - v_j\|^2_{L^2_w(D)} = \int_D (1 - \zeta_j)^2 |v|^2\, w(\qt) \dq = \int_{D\setminus B_j}
(1 - \zeta_j)^2 |v|^2\, w(\qt) \dq \leq \int_{D\setminus B_j} |v|^2\, w(\qt) \dq.\]
Hence, given $\delta > 0$, there exists $j \geq 1$ such that the right-most expression in this
chain is bounded above by $(\delta/2)^2$; therefore, also,
\[ \| v - v_j\|_{L^2_w(D)} < \textstyle{\frac{1}{2}}\delta.\]
Note further that $v_j \in L^2(B_{j+1})$; this follows by observing that
\begin{eqnarray*}
\|v_j\|^2_{L^2(B_{j+1})} &=& \int_{B_{j+1}} |v_j|^2 \dq \leq \max_{\qt \in B_{j+1}} w^{-1}(\qt)
\int_{B_{j+1}} |v_j|^2\, w(\qt) \dq \\
&\leq& \max_{\qt\in B_{j+1}} w^{-1}(\qt)
\int_{B_{j+1}} |v|^2\, w(\qt) \dq \leq \max_{\qt \in B_{j+1}} w^{-1}(\qt)\, \|v\|^2_{L^2_w(D)}
< \infty.
\end{eqnarray*}
As $C^\infty_0(B_{j+1})$ is dense in $L^2(B_{j+1})$, there exists a function
$\varphi_j \in C^\infty_0(B_{j+1})$ such that
\[ \| v_j - \varphi_j \|_{L^2(B_{j+1})} < \frac{\delta}{2\, [\max_{~\!\qt \in B_{j+1}} w(\qt)]^{\frac{1}{2}}}.\]
On extending $\varphi_j$ by $0$ from $B_{j+1}$ to the whole of $D$ and noting that $\varphi_j \in C^\infty_0(D)$,
we have that
\[ \int_D |v_j - \varphi_j|^2 w(\qt) \dq = \int_{B_{j+1}} |v_j - \varphi_j|^2 w(\qt) \dq
\leq \max_{\qt \in B_{j+1}} w(\qt) \int_{B_{j+1}} |v_j - \varphi_j|^2 \dq,\]
from which we deduce that
\[ \| v_j - \varphi_j\|_{L^2_w(D)} < \textstyle{\frac{1}{2}}\delta, \]
whereby, using the triangle inequality,
\[ \|v - \varphi_j \|_{L^2_w(D)} \leq \|v - v_j\|_{L^2_w(D)} + \|v_j - \varphi_j\|_{L^2_w(D)}
< \textstyle{\frac{1}{2}}\delta + \textstyle{\frac{1}{2}}\delta = \delta.\]
Thus we have shown that for each $\delta>0$ there exists $\varphi_j \in C^\infty_0(D)$ such that $\|v - \varphi_j \|_{L^2_w(D)} < \delta$. This then implies that $C^\infty_0(D)$ is dense in $L^2_w(D)$.

An analogous argument yields the density of $C^\infty_0(D)$ in $H^1_w(D)$. Given $v \in H^1_w(D)$,
we define $v_j:= \zeta_j v$, with $\zeta_j$ as above, and note that, since $B_{j+1}\setminus B_j \subset D \setminus B_j$ and $\max(1+2C^2, 2) < 3C^2$, we have
\begin{eqnarray*} \| v - v_j\|^2_{H^1_w(D)} &\leq& \int_{D\setminus B_{j}}|v|^2\, w(\qt) \dq + 2 \int_{D \setminus B_{j}}|\nabq v|^2\, w(\qt)
\dq + 2\, C^2 \int_{B_{j+1} \setminus B_j}|v|^2\, w(\qt) \dq\\
&\leq & 3\, C^2 \int_{D\setminus B_{j}} \left(|v|^2 + |\nabq v|^2\right) w(\qt) \dq.
\end{eqnarray*}
Hence, given $\delta > 0$, there exists $j \geq 1$ such that the right-most expression in this
chain is bounded above by $(\delta/2)^2$; therefore, also,
\[ \| v - v_j\|_{H^1_w(D)} < \textstyle{\frac{1}{2}}\delta.\]
As $v_j \in H^1_0(B_{j+1})$ and $C^\infty_0(B_{j+1})$ is dense in $H^1_0(B_{j+1})$, there exists
a function $\varphi_j \in C^\infty_0(B_{j+1})$ such that
\[ \| v_j - \varphi_j \|_{H^1(B_{j+1})} < \frac{\delta}{2\, [\max_{~\!\qt \in B_{j+1}} w(\qt)]^{\frac{1}{2}}}.\]
On extending $\varphi_j$ by $0$ from $B_{j+1}$ to the whole of $D$ and noting that $\varphi_j \in C^\infty_0(D)$, we have that
\[ \|v_j - \varphi_j\|^2_{H^1_w(D)} = \int_{B_{j+1}}\left( |v_j - \varphi_j|^2 + |\nabq(v_j - \varphi_j)|^2\right) w(\qt) \dq
\leq \max_{\qt \in B_{j+1}} w(\qt)\, \|v_j - \varphi_j\|^2_{H^1(B_{j+1})},\]
from which we deduce that
\[ \| v_j - \varphi_j\|_{H^1_w(D)} < \textstyle{\frac{1}{2}}\delta, \]
whereby, using the triangle inequality,
\[ \|v - \varphi_j \|_{H^1_w(D)} \leq \|v - v_j\|_{H^1_w(D)} + \|v_j - \varphi_j\|_{H^1_w(D)}
< \textstyle{\frac{1}{2}}\delta + \textstyle{\frac{1}{2}}\delta = \delta.\]
This then implies that $C^\infty_0(D)$ is dense in $H^1_w(D)$.

\begin{remark}~
\begin{itemize}
\item[1.] An argument identical to the one above shows that $C^\infty_0(D)$ is dense in $H^k_w(D)$
for any $k \geq 0$.

\item[2.]
In the special case when $w$ is the Gaussian weight function, the density of $C^\infty_0(D)$ in $H^1_w(D)$
follows by noting Definition 1.5.1 on p.13 and Proposition 1.5.2 on p.14 in
Bogachev \cite{bogachev}.

\item[3.]
For the purposes of the present paper, the relevant choices of $w$ are
$M$ and $(1+|\qt|)^{2 \vartheta}M$.

\item[4.]
Thanks to Theorem 8.10.2 on p.418 in the monograph of Kufner, John \& Fu{\v c}ik \cite{KJF},
\[L^2_{M}(D),\quad L^2_M(\Omega \times D),\quad L^2_{(1+|\qt|)^{2\vartheta} M}(D),\quad  L^2_{(1+|\qt|)^{2\vartheta} M}(\Omega \times D),\quad H^1_{M}(D),\quad H^1_M(\Omega \times D)
\]
are separable Hilbert spaces for any $\vartheta>0$.
\end{itemize}
\end{remark}

\section{$\!\!H^1_{M_i}(\mathbb{R}^d) \cap L^2_{(1+|\qt_i|)^{2\vartheta} M_i}(\mathbb{R}^d)
\hookrightarrow \hspace{-3mm} \rightarrow L^2_{(1+|\qt_i|)^{2\vartheta} M_i}(\mathbb{R}^d)$,
 $\vartheta>1$ 
 }
 \label{AppendixC1}

 \setcounter{equation}{0}

We begin by introducing our notational conventions.
For $r>0$, let
\[ D_{(r)} := \{\qt \in \mathbb{R}^d\,:\, |\qt| \leq r\}\quad \mbox{and}\quad
D^{(r)} := \{\qt \in \mathbb{R}^d\,:\, |\qt| \geq r\}.\]
Suppose that $G$ is an unbounded
domain in $\mathbb{R}^d$ such that $G \cap D_{(r)}$ and $G\cap D^{(r)}$ satisfy the segment
property and the cone property (see 4.5 and 4.6 in Chapter 4, p.82, of
Adams \& Fournier \cite{AF:2003}).
Let $\mu$ be a positive finite measure on $G$ defined by $\dd\mu = w(\qt)\dq$,
such that $w \in L^\infty(G)$
and $w$ is locally bounded away from zero on $\overline{G}$; i.e. for any compact set
$K \subset \overline{G}$, $w(\qt) \geq \delta_K>0$ a.e. on $K$. We consider $\mu$ to be a
measure on all of $\mathbb{R}^d$ by defining $\mu(U) = \mu(U \cap G)$ for all Borel subsets
$U$ or $\mathbb{R}^d$. For $1 \leq p < \infty$ and any domain $E$ contained in $\mathbb{R}^d$
we denote by $L^p_w(E):= L^p(E ; \mu)$ the set of all (equivalence classes of $\mu$-almost everywhere equal)
functions defined on $E$ whose $p$th power is Lebesgue-integrable on $E$ with respect to the measure $\mu$.

\begin{proposition}[Hooton \cite{hooton}]\label{prop-hooton}
Let $\mathcal{M}$ be a subset of $L^p_w(G)$, $1 \leq p < \infty$, such that:
\begin{enumerate}
\item[(i)] for each $m \in \mathbb{Z}_{>0}$, the set of restrictions of
functions in $\mathcal{M}$ to $G \cap D_{(m)}$ is a relatively compact subset of
$L^p_w(G \cap D_{(m)})$;
\item[(ii)] given $\delta >0$, there exists $m \in \mathbb{Z}_{>0}$ such that
 \[\int_{G \cap D^{(m)}} |\hat\varphi|^p \dd \mu < \delta \qquad \forall \hat\varphi \in
 \mathcal{M}.\]
\end{enumerate}
Then $\mathcal{M}$ is relatively compact in $L^p_w(G)$.
\end{proposition}
The proof of this proposition is a simple modification of the proof of
Theorem 2.33 in \cite{AF:2003}.

\begin{theorem}
For $\vartheta>1$ the following compact embeddings hold:
\[H^1_{M_i}(\mathbb{R}^d) \cap L^2_{(1+|\qt_i|)^{2\vartheta} M_i}(\mathbb{R}^d)
\compEmb L^2_{(1+|\qt_i|)^{2\vartheta} M_i}(\mathbb{R}^d),\qquad i=1,\dots,K.\]
\end{theorem}

\begin{proof}
Our argument follows closely the proof of Theorem 3.1 in the work of Hooton \cite{hooton}.

Since $M_i(s) = c_{i4} \exp ( - c_{i1} s^\vartheta )$ as $s \rightarrow +\infty$,
with $\vartheta>1$,
where $c_{i1},\, c_{i4}$, $i=1,\dots, K$, are positive constants whose actual values are of no relevance in the argument that follows, we shall, for the sake of clarity of exposition and ease of writing, set these constants to $1$, and assume below that
\[ M_i(\qt_i) := {\rm e}^{-  \left(\frac{1}{2}|\qt_i|^2\right)^\vartheta},\qquad i=1,\dots,K.\]

Suppose that $\hat\varphi\in H^1_{M_i}(\mathbb{R}^d) \cap L^2_{(1+|\qt_i|)^{2\vartheta} M_i}(\mathbb{R}^d)$,
define $\hat\Phi(\qt_i) := [\hat\varphi(\qt_i)]^2$, and consider the integral
\[ I_{\vartheta,i}(r,t):= \int_{D^{i,(r)}}{\rm e}^{-\left(\frac{1}{2}|\qt_i|^2\right)^\vartheta}
( 1 + |\qt_i|)^{2\vartheta}\, \hat\Phi\left(\qt_i - t \frac{\qt_i}{|\qt_i|}\right)\, \dq_i,\]
where $0 \leq t \leq r$, $\vartheta>0$, and
\[ D^{i,(r)} :=\{ \qt_i \in \mathbb{R}^d\,:\, |\qt_i| \geq r\},\qquad i=1,\dots,K.\]

By performing the change of variable
\[ \undertilde{\hat{q}}_i = \left( 1 - \frac{t}{|\qt_i|}\right) \qt_i,\]
we deduce that $(1 - (t/|\qt_i|))\,|\qt_i| = |\undertilde{\hat{q}}_i|$, and therefore $|\qt_i| = |\undertilde{\hat{q}}_i| + t$,
with $0 \leq t \leq r$ and $|\qt_i| \geq r$. We observe that in terms of the radial co-ordinate
$|\qt_i|$ and the $d-1$ angular co-ordinates $\phi_{i,1}, \dots, \phi_{i,d-1}$ the $d$ components of $\qt_i$ can be expressed as follows:
\begin{eqnarray*}
q_{i,1} &=& |\qt_i|\, \cos(\phi_{i,1}),\\
q_{i,2} &=& |\qt_i|\, \sin(\phi_{i,1})\,\cos(\phi_{i,2}),\\
q_{i,3} &=& |\qt_i|\, \sin(\phi_{i,1})\,\sin(\phi_{i,2})\,\cos(\phi_{i,3}),\\
\ldots &~& \ldots\\
q_{i,d-1} &=& |\qt_i|\, \sin(\phi_{i,1})\,\cdots\, \sin(\phi_{i,d-2})\,\cos(\phi_{i,d-1}),\\
q_{i,d} &=& |\qt_i|\, \sin(\phi_{i,1})\,\cdots\, \sin(\phi_{i,d-2})\, \sin(\phi_{i,d-1}),
\end{eqnarray*}
and thereby, letting $c_{d}(\phi_{i,1}, \dots, \phi_{i,d-1}):= \sin^{d-2}(\phi_{i,1})\, \sin^{d-3}(\phi_{i,2})\, \cdots \,\sin(\phi_{i,d-2})$, we have that
\begin{eqnarray*}
\dq_i &=& |\qt_i|^{d-1}\, c_d(\phi_{i,1}, \dots, \phi_{i,d-1})\,\dd|\qt_i|\,
\dd \phi_{i,1}\,\dd\phi_{i,2}\, \cdots \dd \phi_{i,d-1}.
\end{eqnarray*}
With this notation, and observing that $\dd |\qt_i| = \dd |\undertilde{\hat{q}}_i|$, we have that
\begin{eqnarray*}
\dq_i =  \frac{|\qt_i|^{d-1}}{|\undertilde{\hat{q}}_i|^{d-1}}\,
|\undertilde{\hat{q}}_i|^{d-1}\, c_d(\phi_{i,1}, \dots, \phi_{i,d-1})\,  \dd |\undertilde{\hat{q}}_i| \dd \phi_{i,1}\,\dd\phi_{i,2}\, \cdots \dd \phi_{i,d-1}
= \left[\frac{|\undertilde{\hat{q}}_i| + t}{|\undertilde{\hat{q}}_i|}\right]^{d-1}\,\dd \undertilde{\hat{q}}_i.
\end{eqnarray*}
Thus,
\[I_{\vartheta,i}(r,t)= \int_{D^{i,(r-t)}}{\rm e}^{-\left(\frac{1}{2}(|\undertilde{\hat{q}}_i|+ t)^2\right)^\vartheta}
( 1 + |\undertilde{\hat{q}}_i| + t)^{2\vartheta}\, \hat\Phi(\undertilde{\hat{q}}_i)\, \left[\frac{|\undertilde{\hat{q}}_i| + t}{|\undertilde{\hat{q}}_i|}\right]^{d-1}\,\dd \undertilde{\hat{q}}_i. \]
We deduce by bounding the integrand of this integral and renaming the dummy integration
variable $\undertilde{\hat{q}}_i$ into $\qt_i$ that
\[ I_{\vartheta,i}(r,t) \leq \gamma(r,t) \int_{D^{i,(r-t)}} {\rm e}^{-\left(\frac{1}{2}|\qt_i|^2\right)^\vartheta} (1+|\qt_i|)^{2\vartheta}\, \hat\Phi(\qt_i) \dq_i,\]
where, independent of $i$,
\[ \gamma(r,t) := \sup_{\qt\in D^{i,(r-t)}} \frac{{\rm e}^{-\left(\frac{1}{2}(|\qt_i|+t)^2\right)^\vartheta}}{{\rm e}^{-\left(\frac{1}{2}|\qt_i|^2\right)^\vartheta}} \left[\frac{1+|\qt_i|+t}{1 + |\qt_i|}\right]^{2\vartheta} \left[\frac{|\qt_i|+t}{|\qt_i|}\right]^{d-1}.\]
Taking in particular $t=1$ yields that
\begin{eqnarray}
I_{\vartheta,i}(r,1)&=& \int_{D^{i,(r)}}{\rm e}^{-\left(\frac{1}{2}|\qt_i|^2\right)^\vartheta}
( 1 + |\qt_i|)^{2\vartheta}\, \hat\Phi\left(\qt_i - \frac{\qt_i}{|\qt_i|}\right) \dq_i\nonumber\\
&\leq& \gamma(r,1) \int_{D^{i,(r-1)}} {\rm e}^{-\left(\frac{1}{2}|\qt_i|^2\right)^\vartheta} (1+|\qt_i|)^{2\vartheta}\, \hat\Phi(\qt_i) \dq_i.\label{comp1-1}
\end{eqnarray}

Next, observe that
\begin{eqnarray}
&&\int_{D^{i,(r)}}{\rm e}^{-\left(\frac{1}{2}|\qt_i|^2\right)^\vartheta} (1 + |\qt_i|)^{2\vartheta}
\left|\int_0^1 \frac{\dd}{\dd t}\, \hat\Phi \left(\qt_i - t \frac{\qt_i}{|\qt_i|}\right)\dd t\right| \dq_i \nonumber\\
&&\quad\leq \int_0^1 \left[\int_{D^{i,(r)}} {\rm e}^{-\left(\frac{1}{2}|\qt_i|^2\right)^\vartheta}
( 1 + |\qt_i|)^{2\vartheta}\,\left|\frac{\dd}{\dd t}\, \hat\Phi\left(\qt_i - t \frac{\qt_i}{|\qt_i|}\right)\right| \dq_i   \right] \dd t \nonumber\\
&&\quad \leq \int_0^1 \gamma_1(r,t) \left[\int_{D^{i,(r-t)}} {\rm e}^{-\left(\frac{1}{2}|\qt_i|^2\right)^\vartheta} ( 1 + |\qt_i|)^\vartheta\, |\nabqi \hat\Phi(\qt_i)|_{\ell_1} \dq_i  \right] \dd t, \label{comp1-2}
\end{eqnarray}
where $|\cdot|_{\ell_1}$ denotes the 1-norm on $\mathbb{R}^d$, defined as the sum of the moduli of the components of the
vector in $\mathbb{R}^d$ whose norm is taken, and
\[ \gamma_1(r,t) := \sup_{\qt_i\in D^{i,(r-t)}} \frac{{\rm e}^{-\left(\frac{1}{2}(|\qt_i|+t)^2\right)^\vartheta}}{{\rm e}^{-\left(\frac{1}{2}|\qt_i|^2\right)^\vartheta}}\, \frac{(1+|\qt_i|+t)^{2\vartheta}}{(1 + |\qt_i|)^\vartheta} \,\left[\frac{|\qt_i|+t}{|\qt_i|}\right]^{d-1}.\]
The Newton--Leibnitz formula implies that
\[ \hat\Phi(\qt_i) \leq \hat\Phi\left(\qt_i - \frac{\qt_i}{|\qt_i|}\right) + \left|\int_0^1 \frac{\dd}{\dd t} \hat
\Phi\left(\qt_i - t \frac{\qt_i}{|\qt_i|}\right) \dd t\right|,\]
and we thus deduce from \eqref{comp1-1} and \eqref{comp1-2} that
\begin{eqnarray}\label{comp1-3}
&&\int_{D^{i,(r)}} {\rm e}^{-\left(\frac{1}{2}|\qt_i|^2\right)^\vartheta}
(1 + |\qt_i|)^{2\vartheta}\, \hat\Phi(\qt_i) \dq_i \nonumber\\
&&\leq \gamma(r,1) \int_{D^{i,(r-1)}} {\rm e}^{-\left(\frac{1}{2}|\qt_i|^2\right)^\vartheta} (1+|\qt_i|)^{2\vartheta}\, \hat\Phi(\qt_i)\dq_i\nonumber\\
&&\qquad\qquad+ \int_0^1 \gamma_1(r,t) \left[\int_{D^{i,(r-t)}} {\rm e}^{-\left(\frac{1}{2}|\qt_i|^2\right)^\vartheta} ( 1 + |\qt_i|)^\vartheta\, |\nabqi \hat\Phi(\qt_i)|_{\ell_1} \dq_i  \right] \dd t\nonumber\\
&&\leq \delta(r) \left[\int_{\mathbb{R}^d} {\rm e}^{-\left(\frac{1}{2}|\qt_i|^2\right)^\vartheta} (1+|\qt_i|)^{2\vartheta}\,\hat\Phi(\qt_i) \dq_i \right.\nonumber\\
&&\qquad\qquad+ \left.\int_{\mathbb{R}^d} {\rm e}^{-\left(\frac{1}{2}|\qt_i|^2\right)^\vartheta}
(1 + |\qt_i|)^\vartheta\, |\nabqi \hat\Phi(\qt_i)|_{\ell_1} \dq_i  \right],~~~~~~~~~~
\end{eqnarray}
where
\begin{equation}\label{def-delta}
\delta(r):= \max \left\{ \gamma(r,1) , \int_0^1 \gamma_1(r,t) \dd t \right\}.
\end{equation}
We recall that $\hat\Phi:=[\hat\varphi]^2$; hence, $\nabqi \hat\Phi = 2 \hat\varphi\,\, \nabqi \hat\varphi$. Use of Young's inequality $2ab \leq a^2 + b^2$ under the second
integral sign in \eqref{comp1-3} with $a=(1+|\qt_i|)^\vartheta |\hat\varphi|$ and $b=|\nabqi \hat\varphi|_{\ell_1}$ yields that
\begin{align}\label{comp1-4}
&\int_{D^{i,(r)}} {\rm e}^{-\left(\frac{1}{2}|\qt_i|^2\right)^\vartheta}
(1 + |\qt_i|)^{2\vartheta}\, [\hat\varphi(\qt_i)]^2 \dq_i \nonumber\\
&\quad\leq 2 \delta(r) \left[\int_{\mathbb{R}^d} {\rm e}^{-\left(\frac{1}{2}|\qt_i|^2\right)^\vartheta} (1+|\qt_i|)^{2\vartheta}\, [\hat\varphi(\qt_i)]^2 \dq_i + \int_{\mathbb{R}^d} {\rm e}^{-\left(\frac{1}{2}|\qt_i|^2\right)^\vartheta}
|\nabqi \hat\varphi(\qt_i)|^2_{\ell_1} \dq_i  \right].
\end{align}

The inequality \eqref{comp1-4} is the analogue of inequality (8) on p.575 in the work of Hooton \cite{hooton}. Next, we show that if $\vartheta>1$ then $\lim_{r \rightarrow \infty} \delta(r) = 0$. Let us suppose to this end that $0 \leq t \leq 1$, $\vartheta>\frac{1}{2}$ and  $2 \leq r$; then,
\[ \left[\frac{1}{2} (|\qt_i| + t)^2 \right]^\vartheta = \left(\frac{1}{2}\right)^\vartheta
(|\qt_i| + t)^{2\vartheta} \geq \left( \frac{1}{2}\right)^\vartheta \left[|\qt_i|^{2\vartheta} + 2 \vartheta t \, |\qt_i|^{2\vartheta -1}\right].
\]
Hence we deduce that
\[ 0 \leq \gamma(r,t) \leq {\rm e}^{-\left(\frac{1}{2}\right)^{\vartheta-1} \vartheta \,t \,(r-t)^{2\vartheta -1}}
2^{d-1} \left(\frac{3}{2}\right)^{2\vartheta} = \left(2^{d-1-2\vartheta} 3^{2\vartheta}\right)
\,{\rm e}^{-2^{1-\vartheta}\, \vartheta\, t\, (r-t)^{2\vartheta -1}}.\]
This implies that, for any $\vartheta> \frac{1}{2}$,
\begin{equation}\label{comp1-5}
\lim_{r \rightarrow \infty} \gamma(r,1) = 0.
\end{equation}

Similarly, for $0 \leq t \leq 1$, $2 \leq r$, we have that
\begin{eqnarray*}
\gamma_1(r,t) &\leq& \left(2^{d-1-\vartheta} 3^{\vartheta}\right) \sup_{\qt_i \in D^{i, (r-t)}}
\,{\rm e}^{-2^{1-\vartheta}\, \vartheta\, t\, |\qt_i|^{2\vartheta -1}} (1 + |\qt_i| + t)^\vartheta\\
&\leq& \left(2^{d-1-\vartheta} 3^{\vartheta}\right) \sup_{\qt_i \in D^{i, (r-1)}}
\,{\rm e}^{-2^{1-\vartheta}\, \vartheta\, t\, |\qt_i|^{2\vartheta -1}} (1 + |\qt_i| + t)^\vartheta.
%
\end{eqnarray*}
We shall assume in the rest of this section that $\vartheta>1$.

For $r \geq 2$, consider the function $G$ defined on $[r-1,\infty) \times [0,1]$ by
\[ G(x,t):= \,{\rm e}^{- c\, t\, x^{2\vartheta -1}} (1 + x + t)^\vartheta,\qquad \mbox{where $~~x \geq r-1$, $~~0 \leq t \leq 1$},\]
with $c:=2^{1-\vartheta}\, \vartheta>0$ and $x:=|\qt_i|$. Our aim is to show that,
in the limit of $r \rightarrow \infty$,
\[ \int_0^1 \sup_{x \geq r-1} G(x,t) \dd t\]
converges to $0$. As $\gamma_1(r,t)$ is positive for all $r\geq 2$ and $t \in [0,1]$, this will then imply that
\[\lim_{r \rightarrow \infty} \int_0^1 \gamma_1(r,t) \dd t = 0.\]

Clearly, for any $t \in [0,1]$ and any $r \geq 2$ fixed,
$x \in [r-1,\infty) \mapsto G(x,t) \in (0,\infty)$ is a $C^\infty$ function, and, for any $t \in (0,1]$ we have that
$\lim_{x \rightarrow +\infty} G(x,t) = 0$. Thus, for $t \in (0,1]$ the function $x \mapsto G(x,t)$ either attains its maximum value over the interval $x \in [r-1,\infty)$ at the left-hand endpoint $x=r-1$ of the interval $[r-1,\infty)$, or at an internal point $x_0(t) \in (r-1, \infty)$ whose location is potentially $t$-dependent.

In the former case,
$G(r-1,t):= \,{\rm e}^{- c t\, (r-1)^{2\vartheta -1}} (r + t)^\vartheta$. If on the other hand the maximum
of the function $x \mapsto G(x,t)$ is attained at an internal point $x_0(t) \in (r-1, \infty)$, then $x_0(t)$
annihilates the $x$-partial derivative of $G$:
\[ G_x(x,t) = - c\,t\, (2 \vartheta - 1)\, x^{2\vartheta - 2}\, {\rm e}^{- c t\, x^{2\vartheta -1}} (1 + x + t)^\vartheta + \vartheta\,{\rm e}^{- c t\, x^{2\vartheta -1}} (1 + x + t)^{\vartheta-1}.\]
Now, $G_x(x,t)=0$ if, and only if,
\[ t\, x^{2\vartheta -2}\,(1 + x + t) = \frac{\vartheta}{c\,(2\vartheta - 1)}.\]
As, by hypothesis, $\vartheta>1$, it follows that $\frac{\vartheta}{c\,(2\vartheta - 1)}>0$; on the other hand, for each $t \in (0,1]$, the function $x \mapsto t x^{2\vartheta -2}\,(1 + x + t)$ is a strictly monotonic increasing
bijection of $[0,\infty)$ onto itself. Hence, for each $t \in (0,1]$ there exists a unique positive real number
$x_0(t)$ such that
\begin{equation}\label{last-identity}
t\, [x_0(t)]^{2\vartheta -2}\,(1 + [x_0(t)] + t) = \frac{\vartheta}{c\,(2\vartheta - 1)};
\end{equation}
i.e. $G_x(x_0(t),t)=0$, which means that $x_0(t) \in (0,\infty)$ is the unique point of maximum of the function
$x \mapsto G(x,t)$ in the interval $(0,\infty)$, with $t \in (0,1]$. We emphasize that, at this stage, it is
not yet clear whether or not, for a fixed value of $r\geq 2$,
$x_0(t) \in (r-1,\infty)$. Differentiating
both sides of \eqref{last-identity} with respect to $t\in (0,1]$, we deduce that
\[ x_0'(t) = - \frac{x_0(t)\,(1 + x_0(t) + 2t)}{2t\, (\vartheta-1)\,(1 + x_0(t) + t) + t x_0(t)}.\]
As $x_0(t)>0$ for all $t \in (0,1]$, and $\vartheta>1$ by hypothesis, it follows that $x_0'(t)<0$ for all $t \in (0,1]$,
and $\lim_{t \rightarrow 0_+} x_0(t) = +\infty$. We define $c_\vartheta := x_0(1)$; as we have noted above
$x_0(1) > 0$ for all $\vartheta>1$; and therefore, $c_\vartheta>0$. In summary then, $x_0(t) \in [c_\vartheta, \infty)$
for all $t \in (0,1]$ and all $\vartheta>1$.

For any $r \geq c_\vartheta+1$, there exists $t_{r-1} \in (0,1]$ such that $x_0(t_{r-1}) = r-1$.
Hence, for any $r \geq c_{\vartheta} + 1$,
\begin{equation} \label{start-ident}
\int_0^1 \sup_{x \geq r-1} G(x,t) \dd t = \int_0^{t_{r-1}} \sup_{x \geq r-1} G(x,t) \dd t + \int_{t_{r-1}}^1 \sup_{x \geq r-1} G(x,t) \dd t.
\end{equation}
Now, for $t \in [t_{r-1},1]$, we have that $x_0(t) \in [x_0(1), x_0(t_{r-1})] = [c_\vartheta, r-1]$, and so the point
$x_0(t)$ at which the maximum value of the function $x \mapsto G(x,t)$ is attained over the interval $(0,\infty)$
does not belong to the open interval $(r-1,\infty)$. Hence,
\[ \int_{t_{r-1}}^1 \sup_{x \geq r-1} G(x,t) \dd t = \int_{t_{r-1}}^1 G(r-1,t)\dd t
\leq \int^1_{t_{r-1}} \,{\rm e}^{- c t\, (r-1)^{2\vartheta -1}} (r + 1)^\vartheta \dd t\leq \, \frac{(r + 1)^\vartheta}{c\,(r-1)^{2\vartheta -1}}.\]
Thanks to our assumption that $\vartheta>1$, it follows that the right-hand side converges to $0$ as $r \rightarrow
\infty$, and therefore,
\begin{equation} \label{first-bound}
\lim_{r \rightarrow \infty} \int_{t_{r-1}}^1 \sup_{x \geq r-1} G(x,t) \dd t = 0.
\end{equation}

On the other hand for $t \in (0,t_{r-1}]$, the (unique) stationary point $x_0(t)$ of the positive function
$x \mapsto G(x,t)$
over the interval $(0,\infty)$ belongs to the open interval $(r-1,\infty)$, and the maximum value of $x \mapsto G(x,t)$ is then
equal to $G(x_0(t),t)$.
Hence,
%
\begin{eqnarray}\label{second-bound}
\int_0^{t_{r-1}} \sup_{x \geq r-1} G(x,t) \dd t
&\leq&\int_0^{t_{r-1}} G(x_0(t),t)\dd t. 
\end{eqnarray}
It remains to understand the behaviour of
\[ \int_0^{t_{r-1}} G(x_0(t),t)\dd t\]
in the limit of $r \rightarrow \infty$. We begin by noting that, as $r \rightarrow \infty$, necessarily $t_{r-1}\rightarrow 0_+$ (cf. the definition of $t_{r-1}$ above, and recall that $x_0$ is a strictly monotonic
decreasing function of $t$ on $(0,1]$, with $\lim_{t\rightarrow 0_+}x_0(t) = +\infty$). Extracting $x_0(t)$
from the final factor on the left-hand side of \eqref{last-identity} and passing to the limit
$t \rightarrow 0_+$, it follows that
\[ \lim_{t \rightarrow 0_+} t\, [x_0(t)]^{2\vartheta -1} = \frac{\vartheta}{c\,(2\vartheta - 1)}.\]
Equivalently,
\[ \lim_{t \rightarrow 0_+} t^{\frac{1}{2\vartheta-1}}\, x_0(t) = \left[\frac{\vartheta}{c\,(2\vartheta - 1)}\right]^{\frac{1}{2\vartheta-1}}.\]
Thus,
\[
\lim_{t \rightarrow 0_+} t^{\frac{\vartheta}{2\vartheta-1}} G(x_0(t),t) = {\rm e}^{-\frac{\vartheta}{2\vartheta-1}} \left[\frac{\vartheta}{c\,(2\vartheta - 1)}\right]^{\frac{\vartheta}{2\vartheta-1}}.\]
This implies the existence of a positive constant $C(\vartheta)$ such that
\[ 0 < G(x_0(t),t) \leq C(\vartheta)\, t^{-\frac{\vartheta}{2\vartheta-1}} \qquad \forall t \in (0,1].\]
As $\vartheta>1$, and therefore $\frac{\vartheta}{2\vartheta-1}<1$, we then deduce that
\[ \lim_{r \rightarrow \infty} \int_0^{t_{r-1}} G(x_0(t),t)\dd t  = \lim_{t \rightarrow 0_+} \int_0^t G(x_0(s),s)
\dd s = 0.\]
Returning with this information to \eqref{second-bound}, we have that
\begin{equation}\label{third-bound}
\lim_{r \rightarrow \infty} \int_0^{t_{r-1}} \sup_{x \geq r-1} G(x,t) \dd t = 0.
\end{equation}
Using \eqref{first-bound} and \eqref{third-bound} in \eqref{start-ident} implies that
\[ \lim_{r \rightarrow \infty} \int_0^1 \sup_{x \geq r-1} G(x,t) \dd t = 0.\]
Hence, also,
\begin{equation}\label{comp1-6}
\lim_{r \rightarrow \infty} \int_0^1 \gamma_1(r,t) \dd t = 0.
\end{equation}
Thus we deduce from \eqref{comp1-5} and \eqref{comp1-6} and the definition \eqref{def-delta} of $\delta(r)$
that
\begin{equation}\label{comp1-7}
\vartheta > 1\qquad \Longrightarrow \qquad \lim_{r \rightarrow \infty} \delta(r) =0.
\end{equation}

Having established \eqref{comp1-4} and \eqref{comp1-7}, we suppose that
$\mathcal{M}$ is a closed and bounded set in the normed linear space
$H^1_{M_i}(\mathbb{R}^d) \cap L^2_{(1+|\qt_i|)^{2\vartheta}M_i}(\mathbb{R}^d)$, the norm of the space being
\[\left[\|\cdot\|^2_{L^2_{(1+|\qt_i|)^{2\vartheta}M_i}(\mathbb{R}^d)} + \|\cdot\|^2_{H^1_{M_i}(\mathbb{R}^d)} \right]^{\frac{1}{2}}.\]%
The boundedness of $\mathcal{M}$ then means that there exists $C_\dagger>0$ such that
\[\left[\int_{\mathbb{R}^d} {\rm e}^{-\left(\frac{1}{2}|\qt_i|^2\right)^\vartheta} (1+|\qt_i|)^{2\vartheta}\, [\hat\varphi(\qt_i)]^2 \dq_i + \int_{\mathbb{R}^d} {\rm e}^{-\left(\frac{1}{2}|\qt_i|^2\right)^\vartheta}
|\nabqi \hat\varphi(\qt_i)|^2_{\ell_1} \dq_i  \right] \leq C_\dagger\]
for all $\hat\varphi \in \mathcal{M}$.

One can now easily show that $\mathcal{M}$ satisfies conditions (i) and (ii) of Proposition \ref{prop-hooton} with $G=\mathbb{R}^{d}$ and $p=2$. Indeed, (i) is a direct consequence of the Rellich--Kondrachov theorem (cf. Adams \& Fournier \cite{AF:2003}, p.168, Theorem 6.3, Part I, eq. (3)), applied on a bounded ball $D_{i,(m)}$ of radius $m$ centred at the origin, thanks to the fact that on $D_{i,(m)}$ the weight functions $M_i$ and $(1+|\qt_i|)^{2\vartheta} M_i$ are bounded above and below by positive constants,
whereby $H^1_M(D_{i,(m)})$ and $L^2_{(1+|\qt_i|)^{2\vartheta} M_i}(D_{i,(m)})$ coincide with $H^1(D_{i,(m)})$ and $L^2(D_{i,(m)})$, respectively.

To verify
condition (ii) of Proposition \ref{prop-hooton}, we take $\delta>0$ and recall the definition \eqref{def-delta} of
 $\delta(r)$. Thanks to \eqref{comp1-7}, there exists a positive integer $m$ such that
$2C_\dagger \delta(m) < \delta$.  It then follows from \eqref{comp1-4} that
\[ \int_{D^{i,(m)}} {\rm e}^{-\left(\frac{1}{2}|\qt_i|^2\right)^\vartheta}
(1 + |\qt_i|)^{2\vartheta}\, [\hat\varphi(\qt_i)]^2 \dq_i < \delta,\qquad i = 1,\dots,K.\]
Thus we have verified condition (ii) of Proposition \ref{prop-hooton}.

Therefore,
as $\mathcal{M}$ is
closed in $L^2_{(1+|\qt_i|)^{2\vartheta} M_i}(\mathbb{R}^d)$,
$\mathcal{M}$ is relatively compact in $L^2_{(1+|\qt_i|)^{2\vartheta} M_i}(\mathbb{R}^d)$;  it follows that $\mathcal{M}$ is compact in
$L^2_{(1+|\qt_i|)^{2\vartheta} M_i}(\mathbb{R}^d)$. This completes the proof of the compact embedding
$$H^1_{M_i}(\mathbb{R}^d) \cap L^2_{(1+|\qt_i|)^{2\vartheta} M_i}(\mathbb{R}^d) \compEmb L^2_{(1+|\qt_i|)^{2\vartheta} M_i}(\mathbb{R}^d)$$ for $\vartheta>1$ and $i=1,\dots,K$.
\end{proof}

In the next section, we extend this argument to the case of $K$ coupled dumbbells, $K\geq 1$.

\section{\!\!\!\!$H^1_{M}(\mathbb{R}^{Kd}) \cap L^2_{(1+|\qt|)^{2\vartheta} M}(\mathbb{R}^{Kd})
\!\hookrightarrow \hspace{-3mm} \rightarrow \!L^2_{(1+|\qt|)^{2\vartheta} M}(\mathbb{R}^{Kd})$,
$\!\vartheta\!>\!1$}
\label{AppendixC2}

\setcounter{equation}{0}

\begin{theorem}
For $\vartheta>1$, the following compact embedding holds:
\[H^1_{M}(\mathbb{R}^{Kd}) \cap L^2_{(1+|\qt|)^{2\vartheta} M}(\mathbb{R}^{Kd})
\compEmb L^2_{(1+|\qt|)^{2\vartheta} M}(\mathbb{R}^{Kd}).\]
\end{theorem}

\begin{proof}
We proceed in the same way as in the previous section.

Once again,
since $c_{i1} \exp ( - C_1 s^\vartheta ) \leq M_i(s) \leq c_{i2} \exp ( - C_1 s^\vartheta )$ as $s \rightarrow +\infty$,
with $\vartheta>1$,
where $c_{i1}$, $c_{i2}$, $i=1,\dots, K$, and $C_1$ are positive constants whose actual values are of no relevance in the argument that follows, we shall, for the sake of clarity of exposition and ease of writing, set these constants to $1$, and assume below that
\[ M_i(\qt_i) := {\rm e}^{-  \left(\frac{1}{2}|\qt_i|^2\right)^\vartheta},\qquad i=1,\dots, K,\qquad \mbox{and}
\qquad M(\qt):=M_1(\qt_1)\cdots M_K(\qt_K),\]
where $\qt = (\qt_1,\dots,\qt_K)^{\rm T}$.

Suppose that $\hat\varphi\in H^1_{M}(\mathbb{R}^{Kd}) \cap L^2_{(1+|\qt|)^{2\vartheta} M}(\mathbb{R}^{Kd})$,
define $\hat\Phi(\qt) := [\hat\varphi(\qt)]^2$, and consider the integral
\[ I_{\vartheta}(r,t):= \int_{D^{(r)}}{\rm exp}\left[-\sum_{i=1}^K\left(\frac{1}{2}|\qt_i|^2\right)^\vartheta\right]
( 1 + |\qt|)^{2\vartheta}\, \hat\Phi\left(\qt - t \frac{\qt}{|\qt|}\right)\, \dq,\]
where $0 \leq t \leq r$, $\vartheta>0$, and
\[ D^{(r)} :=\{ \qt \in \mathbb{R}^{Kd}\,:\, |\qt| \geq r\}.\]

By performing the change of variable
\[ \undertilde{\hat{q}} = \left( 1 - \frac{t}{|\qt|}\right) \qt,\]
we deduce that $(1 - (t/|\qt|))\,|\qt| = |\undertilde{\hat{q}}|$, and therefore $|\qt| = |\undertilde{\hat{q}}| + t$, with $0 \leq t \leq r$ and $|\qt| \geq r$. Also,
$|\undertilde{\hat{q}}_i|= (1 - (t/|\qt|))\,|\qt_i| = (|\undertilde{\hat{q}}| /(|\undertilde{\hat{q}}| + t ))\,|\qt_i|$ for $i=1,\dots, K$.
Analogously as in the previous section,
\begin{eqnarray*}
\dq = \left[\frac{|\undertilde{\hat{q}}| + t}{|\undertilde{\hat{q}}|}\right]^{K(d-1)}\,\dd \undertilde{\hat{q}}.
\end{eqnarray*}
Thus,
\[I_{\vartheta}(r,t)= \int_{D^{(r-t)}}{\rm exp}\left[-\sum_{i=1}^K\left(\frac{1}{2}|\undertilde{\hat{q}}_i|^2\right)^{\vartheta}
\left[\frac{|\undertilde{\hat{q}}|+t}{|\undertilde{\hat{q}} |}\right]^{2 \vartheta}\right]
( 1 + |\undertilde{\hat{q}}| + t)^{2\vartheta}\, \hat\Phi(\undertilde{\hat{q}})\, \left[\frac{|\undertilde{\hat{q}}| + t}{|\undertilde{\hat{q}}|}\right]^{K(d-1)}\,\dd \undertilde{\hat{q}}. \]
We deduce by bounding the integrand of this integral and renaming the dummy integration
variable $\undertilde{\hat{q}}$ into $\qt$ that
\[I_{\vartheta}(r,t)\leq \gamma(r,t) \int_{D^{(r-t)}}{\rm exp}\left[-\sum_{i=1}^K\left(\frac{1}{2}|\undertilde{{q}}_i|^2\right)^{\vartheta}
\right]
( 1 + |\undertilde{{q}}|)^{2\vartheta}\, \hat\Phi(\undertilde{{q}})\, \,\dd \undertilde{{q}}, \]
where,
\[ \gamma(r,t) := \sup_{\qt\in D^{(r-t)}} \frac{{\rm exp}\left[-\sum_{i=1}^K\left(\frac{1}{2}|\undertilde{{q}}_i|^2\right)^{\vartheta}
\left[\frac{|\undertilde{{q}}|+t}{|\undertilde{{q}} |}\right]^{2 \vartheta}\right] }{
{\rm exp}\left[-\sum_{i=1}^K\left(\frac{1}{2}|\undertilde{{q}}_i|^2\right)^{\vartheta}
\right]
}
\left[\frac{1+|\qt|+t}{1 + |\qt|}\right]^{2\vartheta} \left[\frac{|\qt|+t}{|\qt|}\right]^{K(d-1)}.\]
Taking in particular $t=1$ yields that
\begin{eqnarray}
I_{\vartheta}(r,1)\leq \gamma(r,1)\int_{D^{(r-1)}}
{\rm exp}\left[-\sum_{i=1}^K\left(\frac{1}{2}|\undertilde{{q}}_i|^2\right)^{\vartheta}
\right]
( 1 + |\undertilde{{q}}|)^{2\vartheta}\, \hat\Phi(\undertilde{{q}}) \dd \undertilde{{q}}.\label{comp2-1}
\end{eqnarray}

Next, observe that
\begin{eqnarray}
&&\int_{D^{(r)}}{\rm exp}\left[-\sum_{i=1}^K\left(\frac{1}{2}|\undertilde{{q}}_i|^2\right)^{\vartheta}
\right] (1 + |\qt|)^{2\vartheta}
\left|\int_0^1 \frac{\dd}{\dd t}\, \hat\Phi \left(\qt - t \frac{\qt}{|\qt|}\right)\dd t\right| \dq \nonumber\\
&&\quad\leq \int_0^1 \left\{\int_{D^{(r)}}
{\rm exp}\left[-\sum_{i=1}^K\left(\frac{1}{2}|\undertilde{{q}}_i|^2\right)^{\vartheta}
\right]
( 1 + |\qt|)^{2\vartheta}\,\left|\frac{\dd}{\dd t}\, \hat\Phi\left(\qt - t \frac{\qt}{|\qt|}\right)\right| \dq   \right\} \dd t \nonumber\\
&&\quad \leq \int_0^1 \gamma_1(r,t) \left\{\int_{D^{(r-t)}}
 {\rm exp}\left[-\sum_{i=1}^K\left(\frac{1}{2}|\undertilde{{q}}_i|^2\right)^{\vartheta}
\right]( 1 + |\qt|)^\vartheta\, |\nabq \hat\Phi(\qt)|_{\ell_1} \dq  \right\} \dd t, \label{comp2-2}
\end{eqnarray}
where
\[ \gamma_1(r,t) := \sup_{\qt\in D^{(r-t)}}
\frac{{\rm exp}\left[-\sum_{i=1}^K\left(\frac{1}{2}|\undertilde{{q}}_i|^2\right)^{\vartheta}
\left[\frac{|\undertilde{{q}}|+t}{|\undertilde{{q}} |}\right]^{2 \vartheta}\right] }{
{\rm exp}\left[-\sum_{i=1}^K\left(\frac{1}{2}|\undertilde{{q}}_i|^2\right)^{\vartheta}
\right]
}\,
\frac{(1+|\qt|+t)^{2\vartheta}}{(1 + |\qt|)^\vartheta} \,\left[\frac{|\qt|+t}{|\qt|}\right]^{K(d-1)}.\]
The Newton--Leibnitz formula implies that
\[ \hat\Phi(\qt) \leq \hat\Phi\left(\qt - \frac{\qt}{|\qt|}\right) + \left|\int_0^1 \frac{\dd}{\dd t} \hat
\Phi\left(\qt - t \frac{\qt}{|\qt|}\right) \dd t\right|,\]
and we thus deduce from \eqref{comp2-1} and \eqref{comp2-2} that
\begin{eqnarray}\label{comp2-3}
&&\!\!\!\!\!\int_{D^{(r)}} {\rm e}^{-\sum_{i=1}^K\left(\frac{1}{2}|\qt_i|^2\right)^\vartheta}
(1 + |\qt|)^{2\vartheta}\, \hat\Phi(\qt) \dq\nonumber\\
&&\leq \gamma(r,1) \int_{D^{(r-1)}} {\rm e}^{-\sum_{i=1}^K\left(\frac{1}{2}|\qt_i|^2\right)^\vartheta} (1+|\qt|)^{2\vartheta}\, \hat\Phi(\qt)\dq\nonumber\\
&&\!\!\!\!\!\qquad\qquad+ \int_0^1 \gamma_1(r,t) \left[\int_{D^{(r-t)}} {\rm e}^{-\sum_{i=1}^K\left(\frac{1}{2}|\qt|^2\right)^\vartheta} ( 1 + |\qt|)^\vartheta\, |\nabq \hat\Phi(\qt)|_{\ell_1} \dq  \right] \dd t\nonumber\\
&&\!\!\!\!\!\leq \delta(r) \left[\int_{\mathbb{R}^{Kd}} {\rm e}^{-\sum_{i=1}^K\left(\frac{1}{2}|\qt_i|^2\right)^\vartheta} (1+|\qt|)^{2\vartheta}\,\hat\Phi(\qt) \dq \right. \nonumber\\
&&\qquad\qquad+ \left. \int_{\mathbb{R}^{Kd}} {\rm e}^{-\sum_{i=1}^K\left(\frac{1}{2}|\qt_i|^2\right)^\vartheta}
(1 + |\qt|)^\vartheta\, |\nabq \hat\Phi(\qt)|_{\ell_1} \dq  \right],
\end{eqnarray}
where
\begin{equation}\label{delta-def1}
\delta(r):= \max \left\{ \gamma(r,1) , \int_0^1 \gamma_1(r,t) \dd t \right\}.
\end{equation}
We recall that $\hat\Phi:=[\hat\varphi]^2$; hence, $\nabq \hat\Phi = 2 \hat\varphi\,\, \nabq \hat\varphi$. Use of Young's inequality $2ab \leq a^2 + b^2$ under the second
integral sign in \eqref{comp2-3} with $a=(1+|\qt|)^\vartheta |\hat\varphi|$ and $b=|\nabq \hat\varphi|_{\ell_1}$ yields that
\begin{eqnarray}\label{comp2-4}
&&\int_{D^{(r)}} {\rm e}^{-\sum_{i=1}^K\left(\frac{1}{2}|\qt_i|^2\right)^\vartheta}
(1 + |\qt|)^{2\vartheta}\, [\hat\varphi(\qt)]^2 \dq \nonumber\\
&&\leq 2 \delta(r) \left[\int_{\mathbb{R}^{Kd}} {\rm e}^{-\sum_{i=1}^K\left(\frac{1}{2}|\qt_i|^2\right)^\vartheta}
(1+|\qt|)^{2\vartheta}\, [\hat\varphi(\qt)]^2 \dq + \int_{\mathbb{R}^{Kd}}
{\rm e}^{-\sum_{i=1}^K\left(\frac{1}{2}|\qt_i|^2\right)^\vartheta}
|\nabq \hat\varphi(\qt)|^2_{\ell_1} \dq  \right].\nonumber\\
\end{eqnarray}

The inequality \eqref{comp2-4} is the analogue of inequality (8) on p.575 in the work of Hooton \cite{hooton}. Next, we show that if $\vartheta>1$ then $\lim_{r \rightarrow \infty} \delta(r) = 0$. Let us suppose to this end that $0 \leq t \leq 1$, $\vartheta> \frac{1}{2}$ and  $2 \leq r$; then,
\[ \left[\frac{1}{2} (|\qt| + t)^2 \right]^\vartheta = \left(\frac{1}{2}\right)^\vartheta
(|\qt| + t)^{2\vartheta} \geq \left( \frac{1}{2}\right)^\vartheta \left[|\qt|^{2\vartheta} + 2 \vartheta t \, |\qt|^{2\vartheta -1}\right].
\]
Also,
\[ \sum_{i=1}^K |\qt_i|^{2\vartheta} \geq \max_{1\leq i \leq K} |\qt_i|^{2\vartheta}
= \left[\max_{1 \leq i \leq K}|\qt_i|^2\right]^\vartheta \geq \left[\frac{1}{K}\sum_{i=1}^K |\qt_i|^2\right]^\vartheta = K^{-\vartheta} |\qt|^{2\vartheta}.\]
Hence we deduce that
\[ 0 \leq \gamma(r,t) \leq {\rm e}^{-\left(\frac{1}{2}\right)^{\vartheta-1}\, K^{-\vartheta}\,\vartheta \,t \,(r-t)^{2\vartheta -1}}
2^{K(d-1)} \left(\frac{3}{2}\right)^{2\vartheta} = \left(2^{K(d-1)-2\vartheta} 3^{2\vartheta}\right)
\,{\rm e}^{-2^{1-\vartheta} K^{-\vartheta}\, \vartheta\, t\, (r-t)^{2\vartheta -1}}.\]
This implies that, for any $\vartheta> \frac{1}{2}$,
\begin{equation}\label{comp2-5}
\lim_{r \rightarrow \infty} \gamma(r,1) = 0.
\end{equation}

Similarly, for $0 \leq t \leq 1$, $2 \leq r$, and assuming that $\vartheta>\frac{1}{2}$, we have that
\begin{eqnarray*}
\gamma_1(r,t) &\leq& \left(2^{K(d-1)-\vartheta} 3^{\vartheta}\right)
\sup_{\qt \in D^{(r-t)}} \,{\rm e}^{-2^{1-\vartheta}\,K^{-\vartheta}\, \vartheta\, t\, |\qt|^{2\vartheta -1}} (1 + |\qt| + t)^\vartheta
\\
&\leq& \left(2^{K(d-1)-\vartheta} 3^{\vartheta}\right)
\sup_{\qt \in D^{(r-1)}} \,{\rm e}^{-2^{1-\vartheta}\,K^{-\vartheta}\, \vartheta\, t\, |\qt|^{2\vartheta -1}} (1 + |\qt| + t)^\vartheta.
\end{eqnarray*}
Arguing in an identical manner as in the previous section, we then deduce that if $\vartheta>1$ then
\begin{equation}\label{comp2-6}
\lim_{r \rightarrow \infty} \int_0^1 \gamma_1(r,t) \dd t = 0.
\end{equation}
Thus we deduce from \eqref{comp2-5} and \eqref{comp2-6} and the definition \eqref{delta-def1} of $\delta(r)$
that
\begin{equation}\label{comp2-7}
\vartheta > 1\qquad \Longrightarrow \qquad \lim_{r \rightarrow \infty} \delta(r) =0.
\end{equation}

Having established \eqref{comp2-4} and \eqref{comp2-7}, we suppose that
$\mathcal{M}$ is a closed and bounded set in the normed linear space
$H^1_{M}(\mathbb{R}^{Kd}) \cap L^2_{(1+|\qt|)^{2\vartheta}M}(\mathbb{R}^{Kd})$. The boundedness
of $\mathcal{M}$ means that there exists $C_\dagger>0$ such that
\[\left[\int_{\mathbb{R}^{Kd}} {\rm e}^{-\sum_{i=1}^K\left(\frac{1}{2}|\qt_i|^2\right)^\vartheta} (1+|\qt|)^{2\vartheta}\, [\hat\varphi(\qt)]^2 \dq + \int_{\mathbb{R}^{Kd}} {\rm e}^{-\sum_{i=1}^K\left(\frac{1}{2}|\qt_i|^2\right)^\vartheta}
|\nabq \hat\varphi(\qt)|^2_{\ell_1} \dq  \right] \leq C_\dagger\]
for all $\hat\varphi \in \mathcal{M}$.

As in the case of $K=1$, one can now easily show that $\mathcal{M}$ satisfies conditions (i) and (ii) of Proposition \ref{prop-hooton} with $G=\mathbb{R}^{Kd}$ and $p=2$ ($q=2$ in the notation of Hooton). Indeed, (i) is a direct consequence of the Rellich--Kondrachov theorem (cf. Adams \& Fournier \cite{AF:2003}, p.168, Theorem 6.3, Part I, eq. (3)), applied on a bounded ball $D_{(m)}$ of radius $m$ centred at the origin, thanks to the fact that on $D_{(m)}$ the weight functions $M$ and $(1+|\qt|)^{2\vartheta} M$ are bounded above and below by positive constants,
whereby $H^1_M(D_{(m)})$ and $L^2_{(1+|\qt|)^{2\vartheta} M}(D_{(m)})$ coincide with $H^1(D_{(m)})$ and $L^2(D_{(m)})$, respectively.

To verify
condition (ii), we take $\delta>0$. Noting \eqref{delta-def1} and \eqref{comp2-7}, there exists a positive integer $m$ such that $2 C_\dagger \delta(m) < \delta$.  It then follows from \eqref{comp2-4} that
\[ \int_{D^{(m)}} {\rm e}^{-\sum_{i=1}^K\left(\frac{1}{2}|\qt_i|^2\right)^\vartheta}
(1 + |\qt|)^{2\vartheta}\, [\hat\varphi(\qt)]^2 \dq < \delta.\]
Thus we have verified condition (ii) of Proposition \ref{prop-hooton}.

Consequently, $\mathcal{M}$ is closed and relatively compact in $L^2_{(1+|\qt|)^{2\vartheta} M}(\mathbb{R}^{Kd})$,
and thereby compact in $L^2_{(1+|\qt|)^{2\vartheta} M}(\mathbb{R}^{Kd})$.
This completes the proof of the compact embedding
$$H^1_{M}(\mathbb{R}^{Kd}) \cap L^2_{(1+|\qt|)^{2\vartheta} M}(\mathbb{R}^{Kd}) \compEmb L^2_{(1+|\qt|)^{2\vartheta} M}(\mathbb{R}^{Kd})$$ for $\vartheta>1$.
\end{proof}

\begin{remark}
When $\vartheta = 1$, the Maxwellian $M$ becomes a normalized Gaussian. It can then
be shown using Gross' logarithmic Sobolev inequality \cite{gross} and the logarithmic Fenchel--Young inequality that
\[H^1_{M}(\mathbb{R}^{Kd}) \cap L^2_{(1+|\qt|)^{2\vartheta} M}(\mathbb{R}^{Kd})
= H^1_{M}(\mathbb{R}^{Kd}).\]
Hence in particular $H^1_{M}(\mathbb{R}^{Kd})$ is contained in $L^2_{(1+|\qt|)^{2} M}(\mathbb{R}^{Kd})$. The question arises whether this inclusion
is in fact a compact embedding. We suspect that the answer is negative.
On the other hand, one
can still show that $H^1_{M}(\mathbb{R}^{Kd})\compEmb L^2_{M}(\mathbb{R}^{Kd})$ for all $\vartheta>\frac{1}{2}$. This we shall do in the next section.
\end{remark}

\section{$H^1_M(D) \compEmb L^2_{M}(D)$, $\vartheta > \frac{1}{2}$}\label{sec:comptensorise}

\setcounter{equation}{0}

Let $D:=D_1 \times \cdots \times D_K$, where $D_i = \mathbb{R}^d$, $i=1,\dots,K$, and
write  $M(\qt):=M_1(\qt_1)\cdots M_K(\qt_K)$. We shall prove that, for $\vartheta>\frac{1}{2}$,
\begin{equation}\label{compEmbNu}
     H^1_M(D)\compEmb L^2_{M}(D).
\end{equation}
We begin by noting that from Theorem 3.1 in the work of Hooton \cite{hooton} and a slight modification
of Example on p.576 therein --- the only difference being that there the weight-function is
${\rm exp}(-|\qt_i|^{2\vartheta})$ whereas here it is equivalent to ${\rm \exp}\left(-(\frac{1}{2}|\qt_i|^2)^{\vartheta}\right)$ --- we have that
$H^1(D_i;\mu_i)$ is compactly embedded in $L^2(D_i;\mu_i)$ for $\vartheta >\frac{1}{2}$,
where $\mu_i$ is the measure on $D_i$ defined by $\dd \mu_i := {\rm \exp}\left(-(\frac{1}{2}|\qt_i|^2)^{\vartheta}\right) \dq_i$.
Hence, by
\eqref{growth1},
we have that $H^1_{M_i}(D_i)$ is compactly embedded in $L^2_{M_i}(D_i)$; i.e.,
\begin{equation}\label{partial-compEmb}
H^1_{M_i}(D_i) \compEmb L^2_{M_i}(D_i), \qquad i=1,\dots,K;\qquad \vartheta >\textstyle{\frac{1}{2}}.
\end{equation}
The proof of the compactness of the embedding of $H^1_M(D)$ into $L^2_M(D)$ then proceeds
identically as in \cite{BS2010}, using Theorem 2.4 in the work of Opic \cite{Opic}. We shall prove \eqref{compEmbNu} for the case of $K=2$, with $D=D_1 \times D_2$ and $M(\qt) = M_1(\qt_1) M_2(\qt_2)$. For $K>2$ the proof is completely analogous.

Let $u \in H^1_M(D)$. As $M = M_1 \times M_2$, it follows from Fubini's theorem that, for almost all $\qt_1 \in D_1$,
\begin{equation*}
u(\qt_1,\cdot) \in L^1_{\mathrm{loc}}(D_2) \quad\text{and}\quad \partial^\alpha u(\qt_1,\cdot) \in L^2_{M_2}(D_2),\\
\end{equation*}
where $\alpha$ is any $d$-component multi-index with $0 \leq |\alpha| \leq 1$. Fubini's theorem also
implies that, given $\varphi_2 \in C^\infty_0(D_2)$ and a $d$-component multi-index $\alpha_2$, with
$0 \leq |\alpha_2| \leq 1$, we have
\begin{equation*}
\int_{D_1} \left[ (-1) \int_{D_2} u(\qt_1,\cdot) \partial^{\alpha_2}\varphi_2 \dd  \qt_2 \right] \varphi_1 \dd \qt_1
= \int_{D_1} \left[ \int_{D_2} \partial^{(0,\alpha_2)}u(\qt_1,\cdot) \varphi_2 \dd \qt_2 \right] \varphi_1 \dd \qt_1,
\end{equation*}
for all $\varphi_1 \in C^\infty_0(D_1)$. Therefore,
$\partial^{\alpha_2}[u(\qt_1,\cdot)] = \partial^{(0,\alpha_2)}
u(\qt_1,\cdot)$ in the sense of weak derivatives on $D_2$ for almost all $\qt_1 \in D_1$.
As $\partial^{(0,\alpha_2)} u(\qt_1,\cdot)$ belongs to $L^2_{M_2}(D_2)$ for
almost all $\qt_1 \in D_1$ we have that
\begin{equation}\label{regularity-a.e.}
u(\qt_1,\cdot) \in H^1_{M_2}(D_2) \quad\text{for almost all }\qt_1 \in D_1.
\end{equation}
Analogously,
\[ u(\cdot,\qt_2) \in H^1_{M_1}(D_1) \quad\text{for almost all }\qt_2 \in D_2. \]
As each of the partial Maxwellians, $M_1$ and $M_2$, is bounded from above and below by positive
constants on compact subsets of their respective domains, there exists a sequence $(D_{i,(n)}\,:\,n \in \mathbb{N})$ of open proper Lipschitz subsets of $D_i$, $i=1,2$, such that
\begin{equation*}
D_{i,(n)} \subset D_{i,(n+1)},\ n \in \mathbb{N},\qquad \bigcup_{n=1}^\infty D_{i,(n)} = D_i\quad\text{and}\quad H^1_{M_i}(D_{i,(n)}) \compEmb L^2_{M_i}(D_{i,(n)});
\end{equation*}
e.g., $D_{i,(n)} = B\big(0,n\big)$; the compact embeddings stated here
follow by the Rellich--Kondrachov theorem (cf. Adams \& Fournier \cite{AF:2003}, p.168, Theorem 6.3, Part I, eq. (3)) applied on $D_{i,(n)}$, $i=1,2$, $n \in \mathbb{N}$.
Letting, for $n \in \mathbb{N}$, $D_{(n)} := \left(\bigtimes_{i=1}^2 D_{i,(n)}\right) \subsetneq D$, and noting
that, by Appendix A in \cite{BS2010}, $D_{(n)}$ is a Lipschitz domain, the above properties get inherited by $D_{(n)}$
from $D_{i,(n)}$:
\begin{equation*}
D_{(n)} \subset D_{(n+1)},\ n \in \mathbb{N},\qquad \bigcup_{n=1}^\infty D_{(n)} = D\quad\text{and}\quad H^1_{M}(D_{(n)}) \compEmb L^2_{M}(D_{(n)}).
\end{equation*}
Let $D_i^{(n)} := D_i \setminus D_{i,(n)}$ and $D^{(n)} := D \setminus
D_{(n)}$. It follows from Opic \cite{Opic}, Theorem 2.4, that the above compact
embeddings on members of a nested covering imply the following characterizations
(the first, for $i \in \{1, 2\}$):
\begin{gather}
\label{equiv-i}
\hspace{-3mm}H^1_{M_i}(D_i) \compEmb L^2_{M_i}(D_i) \iff
\lim_{n \to \infty} \sup_{u \in H^1_{M_i}(D_i) \setminus \{0\}}
\int_{D_i^{(n)}} u^2 \dd \mu_i / \|u\|_{H^1_{M_i}(D_i)}^2 = 0,\\
\label{equiv-full}
\hspace{-3mm}H^1_M(D) \compEmb L^2_M(D) \iff
\lim_{n \to \infty} \sup_{u \in H^1_M(D) \setminus \{0\}}
\int_{D^{(n)}} u^2 \dd \mu / \|u\|_{H^1_M(D)}^2 = 0,
\end{gather}
where $\dd \mu_i:=M_i(\qt_i)\dq_i$, $i=1,2$, and $\dd \mu:= M(\qt) \dq$.
By virtue of \eqref{partial-compEmb}, the left-hand side of \eqref{equiv-i} holds;
hence, its right-hand side also holds. Using \eqref{regularity-a.e.} and
\eqref{equiv-i} with $i = 2$, we deduce that for any $\delta > 0$
there exists $n=n(\delta) \in \mathbb{N}$ such that
\begin{equation*}
\begin{split}
\int_{D_1 \times D_2^{(n)}} u^2 \dd \mu & = \int_{D_1} \left[ \int_{D_2^{(n)}} u^2(\qt_1,\cdot) \dd \mu_2 \right] \dd \mu_1 \leq \delta \int_{D_1} \|u(\qt_1,\cdot)\|_{H^1_{M_2}(D_2)}^2 \dd \mu_1\\
& = \delta \int_{D_1} \left[ \int_{D_2} u^2(\qt_1,\cdot)\dd \mu_2 + \int_{D_2} |{\grad{\qt_2}u(\qt_1,\cdot)}|^2 \dd \mu_2 \right] \dd \mu_1\\
& \leq \delta \|u\|_{H^1_M(D)}^2;
\end{split}
\end{equation*}
and similarly for $\int_{D_1^{(n)} \times D_2} u^2 \dd \mu$. Then,
as $D^{(n)} = (D_1 \times D_2^{(n)}) \cup (D_1^{(n)} \times D_2)$, the right-hand side
of \eqref{equiv-full} holds; therefore, so does its left-hand side; hence \eqref{compEmbNu}.

\section{Compact embeddings of Banach-space valued Sobolev spaces}
\label{AppendixC3}
\setcounter{equation}{0}

We begin by proving the following theorem, which can be seen as a generalization of the classical
Kolmogorov--Riesz theorem to the case of Bochner spaces $L^r(\mathbb{R}^d; E)$, $1 \leq r < \infty$.

\begin{theorem}[Kolmogorov--Riesz]\label{th-kr} Let $1 \leq r < \infty$ and suppose that $E_0$ and $E_1$
are Banach spaces, with $E_0 \compEmb E_1$. Suppose that $\mathcal{F}$ is a subset of
$L^r(\mathbb{R}^d; E_0)$, such that:
\begin{itemize}
\item[(i)] $\mathcal{F}$ is bounded in $L^r(\mathbb{R}^d; E_0)$;
\item[(ii)] for every $\delta>0$ there exists an $R>0$ such that,
for every $f \in \mathcal{F}$,
\[ \int_{|\xt| > R} \|f(\xt)\|^r_{E_1} \dd \xt < \delta^r;\]
\item[(iii)] for every $\delta>0$ there exists $\rho>0$ such that,
for every $f \in \mathcal{F}$ and every $\yt\in \mathbb{R}^d$ with
$|\yt|<\rho$,
\[ \int_{\mathbb{R}^d}\|f(\xt+\yt) - f(\xt)\|^r_{E_1} \dd \xt < \delta^r.\]
\end{itemize}
Then, $\mathcal{F}$ is totally bounded in $L^r(\mathbb{R}^d; E_1)$.
\end{theorem}

\begin{proof}
%
Suppose that $\mathcal{F} \subset L^r(\mathbb{R}^d; E_0)$ satisfies the three conditions
of the theorem. For any $\delta'>0$, choose $\delta>0$ such that
$(2^d + 1)^{\frac{1}{r}} \delta < \frac{1}{2}\delta'$. Then, for such a $\delta>0$, choose $R>0$ as
in the second condition, and $\rho$ as in the third condition.

Suppose that $Q$ is an open cube centred at the origin of $\mathbb{R}^d$ such that
$|\yt| < \frac{1}{2}\rho$ for all $\yt \in Q$. Let further $Q_1, \dots, Q_N$ be mutually
nonoverlapping translates of $Q$ such that
\[ \overline{\cup_{i=1}^N Q_i}\]
contains a ball with radius $R$ centred at the origin of $\mathbb{R}^d$. For any $f \in
L^r(\mathbb{R}^d; E_0)$, we define
\[ P_{Q_i}f(\xt) := \left\{\begin{array}{cl}  \frac{1}{|Q_i|} \int_{Q_i} f(\zt) \dd \zt, & \mbox{for $\xt \in Q_i,
\quad i=1,\dots,N$},\\
0, & \mbox{otherwise}.
\end{array}   \right.\]
Hence, by Jensen's inequality,
\[ \|P_{Q_i}f(\xt)\|_{E_0} \leq \left\{\begin{array}{cl}  \frac{1}{|Q_i|} \int_{Q_i} \|f(\zt)\|_{E_0} \dd \zt, & \mbox{for $\xt \in Q_i,
\quad i=1,\dots,N$},\\
0, & \mbox{otherwise},
\end{array}   \right.\]
which, in turn, implies that  $\|P_{Q_i}f\|_{L^r(Q_i; E_0)} \leq \|f\|_{L^r(Q_i;E_0)}$. Let us define,
for $f \in L^r(\mathbb{R}^d; E_0)$,
\begin{equation}\label{P-def}
Pf:= \sum_{k=1}^N P_{Q_k}f.
\end{equation}
Then,
\begin{eqnarray} \|Pf\|_{L^r(\mathbb{R}^d; E_0)}^r &=& \sum_{k=1}^N \|P_{Q_k}f\|_{L^r(Q_k; E_0)}^r\nonumber\\
 &\leq& \sum_{k=1}^N \|f\|_{L^r(Q_k; E_0)}^r = \|f\|_{L^r(\cup_{k=1}^N Q_k; E_0)}^r \leq \|f\|_{L^r(\mathbb{R}^d; E_0)}^r.
\label{proj-bound1}
\end{eqnarray}
Thus,
the mapping $P : f \in L^r(\mathbb{R}^d; E_0) \mapsto Pf \in L^r(\mathbb{R}^d; E_0)$ is a
bounded linear operator, whose norm is equal to $1$.

Let us take any $i \in \{1,\dots, N\}$ and consider the set $\mathcal{A}_i:=\{P_{Q_i}f(\xt)\,:\, f \in \mathcal{F}, \; \xt \in Q_i\}$. Noting that for $f \in L^r(\mathbb{R}^d; E_0)$ fixed the value of
$P_{Q_i}f(\xt)$ is independent of $\xt \in Q_i$, and using \eqref{proj-bound1}, we have that
\begin{equation}\label{proj-bound2}
\|P_{Q_i}f\|_{E_0}^r\,|Q_i| = \|P_{Q_i}f\|_{L^r(Q_i; E_0)}^r \leq  \sum_{k=1}^N \|P_{Q_k}f\|_{L^r(Q_k; E_0)}^r  \leq \|f\|_{L^r(\mathbb{R}^d; E_0)}^r\qquad \forall i \in \{1,\dots,N\}.
\end{equation}
As $\mathcal{F}$ is, by hypothesis, a bounded set in $L^r(\mathbb{R}^d; E_0)$, it follows from \eqref{proj-bound2}
that $\mathcal{A}_i$ is a bounded set in $E_0$. Since $E_0 \compEmb E_1$, we then deduce that $\mathcal{A}_i$ is totally bounded in $E_1$.
Thus, for each $\delta'>0$,
there exist a positive integer $M_i=M_i(\delta')$ and balls $A_{ij}$, $j=1,\dots, M_i(\delta')$, in $E_1$
whose centres $\mathcal{C}_{ij}$,
$j=1,\dots, M_i(\delta')$, are elements of $\mathcal{A}_i$, whose radii (in the norm of $E_1$) are smaller than $\delta'/(2(N|Q|)^{\frac{1}{r}})$ for each $j=1,\dots, M_i(\delta')$,
and whose union covers $\mathcal{A}_i$, in the sense that
\begin{equation}\label{cover-1}
\{P_{Q_i}f(\xt)\,:\, f \in \mathcal{F},\; \xt \in Q_i\} \subset \bigcup_{j=1}^{M_i(\delta')} A_{ij}.
\end{equation}
In other words, for every each $i \in \{1,\dots,N\}$ and for each $P_{Q_i}f$, where $f \in \mathcal{F}$, there exists
a $j=j(f,i) \in \{1, \dots, M_i(\delta')\}$ and $\mathcal{C}_{i\,j(f,i)} \in \mathcal{A}_i$, and a ball in $E_1$ of radius $\delta'/(2(N|Q|)^{\frac{1}{r}})$ centred at $\mathcal{C}_{i\,j(f,i)}$, such that
\[ \|P_{Q_i}f - \mathcal{C}_{i\,j(f,i)}\|_{E_1} < \frac{\delta'}{2 (N|Q|)^{\frac{1}{r}}}. \]
%
%
%
Let $\chi_{Q_i}$ denote the characteristic function of the set $Q_i$, $i=1,\dots, N$. Hence, by the definition
\eqref{P-def} of $P$, 
\begin{equation}\label{prelim-1}
\|Pf - \sum_{i=1}^N \chi_{Q_i}\mathcal{C}_{i\, j(f,i)}\|_{L^r(\mathbb{R}^d;E_1)}  < \frac{1}{2}\delta'.
\end{equation}

Noting again that $Pf(\xt) = P_{Q_i}f(\xt)$ for all $x \in Q_i$, $i=1,\dots, N$, it follows
from \eqref{cover-1} that
\[ \{Pf(\xt)\,:\, f \in \mathcal{F},\;\xt \in
\bigcup_{i=1}^N Q_i\} \subset \bigcup_{i=1}^N\bigcup_{j=1}^{M_i(\delta')} A_{ij}.\]
For any $\xt$ contained in the complement of $\bigcup_{i=1}^N Q_i$, we then have, by the definition of $P$, that $Pf(\xt)=0$; therefore,
\begin{equation*}
\{Pf(\xt)\,:\, f \in \mathcal{F},\; \xt \in \mathbb{R}^d\} \subset \{0\} \cup \bigcup_{i=1}^N\bigcup_{j=1}^{M_i(\delta')} A_{ij}\subset E_1.
\end{equation*}

Consider the (finite) set
\[\frak{C}_{\delta'}:= \{0\}  \cup \left\{\sum_{i=1}^N \chi_{Q_i} \lambda_{i}(f)\,:\, \lambda_{i}(f) := \mathcal{C}_{i\,j(f,i)},\,i=1,\dots,N,\; f \in \mathcal{F}\right\}\subset E_1 \]
of cardinality $\leq 1 + \sum_{i=1}^N M_i(\delta')$.

Then, for $f \in \mathcal{F}$, we have from \eqref{prelim-1} that
\begin{eqnarray}
\|f - \sum_{i=1}^N \chi_{Q_i}\lambda_{i}(f)\|_{L^r(\mathbb{R}^d; E_1)} &\leq & \|f - Pf\|_{L^r(\mathbb{R}^d; E_1)}
+ \|Pf - \sum_{i=1}^N \chi_{Q_i}\lambda_{i}(f)\|_{L^r(\mathbb{R}^d; E_1)}\nonumber\\
&\leq & \|f - Pf\|_{L^r(\mathbb{R}^d; E_1)} + \frac{1}{2}\delta'.\label{p-bound}
\end{eqnarray}
It remains to bound the first term on the right-hand side or \eqref{p-bound}.

It follows from property (ii) and the definition of $P$ that, for any $f \in \mathcal{F}$, we have
\begin{eqnarray*}
\|f - Pf\|^r_{L^r(\mathbb{R}^d; E_1)} &<& \delta^r + \sum_{i=1}^N \int_{Q_i}\|f(\xt) - Pf(\xt)\|^r_{E_1}\dd \xt\\
&=& \delta^r + \sum_{i=1}^N \int_{Q_i}\left\|\frac{1}{|Q_i|}\int_{Q_i}(f(\xt)-f(\zt)) \dd \zt\right\|^r_{E_1}\dd \xt.
\end{eqnarray*}
By applying Jensen's inequality again,
performing a change of variable, and noting that $\xt-\zt \in 2Q$ when $\xt, \zt \in Q_i$, we have that
\begin{eqnarray*}
\|f - Pf\|^r_{L^r(\mathbb{R}^d; E_1)} &<&  \delta^r + \sum_{i=1}^N \int_{Q_i}\frac{1}{|Q_i|}\int_{Q_i}\|f(\xt)-f(\zt)\|^r_{E_1}\dd \zt \dd \xt\\
&\leq& \delta^r + \sum_{i=1}^N \int_{Q_i} \frac{1}{|Q_i|}\int_{2Q}\|f(\xt) - f(\xt+\yt)\|^r_{E_1} \dd \yt \dd \xt\\
&\leq& \delta^r + \frac{1}{|Q|} \int_{2Q} \int_{\mathbb{R}^d}\|f(\xt) - f(\xt+\yt)\|^r_{E_1} \dd \xt \dd \yt\\
&<& \delta^r + \frac{1}{|Q|}\int_{2Q} \delta^r \dd \yt = (2^d + 1)\delta^r,
\end{eqnarray*}
thanks to property (iii). Thus, by our choice of $\delta$ for a given $\delta'$ at the start of the proof,
\begin{equation} \label{last-fpf}
\|f - Pf\|_{L^r(\mathbb{R}^d; E_1)} < (2^d + 1)^{\frac{1}{r}} \delta < \frac{1}{2}\delta'.
\end{equation}
By inserting \eqref{last-fpf} into \eqref{p-bound} we deduce that for any $\delta'>0$ and any $f \in \mathcal{F}$ there exists an element of the finite set $\frak{C}_{\delta'}$,
which in the norm of $L^r(\mathbb{R}^d; E_1)$ is at a
distance $<\delta'$ from $f$. Hence $\frak{C}_{\delta'}$ is a finite $\delta'$-net for $\mathcal{F}$ in $L^r(\mathbb{R}^d; E_1)$, whereby $\mathcal{F}$ is totally bounded in $L^r(\mathbb{R}^d; E_1)$.
\end{proof}

In a complete metric space a set $\mathcal{F}$ is totally bounded if, and only if, it is
relatively compact. Since $L^p(\mathbb{R}^d; E)$ is a Banach space, the words {\em totally
bounded} in Theorem \ref{th-kr} are synonymous to the words {\em relatively compact}.

In the sequel, if necessary, we shall consider any function $f$ defined almost everywhere on an
open set $\Omega\subset \mathbb{R}^d$ to be extended by zero to the whole of $\mathbb{R}^d$; i.e.,
we shall introduce the function $\tilde{f}$ defined for almost all $x \in \mathbb{R}^d$ by
\[ \tilde{f}(\xt) := \chi_\Omega(\xt)\,f(\xt) =  \left\{\begin{array}{cl}
f(\xt) & \mbox{if $\xt \in \Omega$},\\
0    & \mbox{otherwise}.
\end{array} \right.\]
Instead of writing $\tilde{f}(\xt)$ we shall often simply write $f(\xt)$ also when $\xt \notin \Omega$.

\begin{theorem}\label{th-kr1}
Let $\Omega$ be a bounded open set in $\mathbb{R}^d$, $1 \leq r < \infty$, and let $E_0$ and $E_1$ be Banach spaces,
with $E_0 \compEmb E_1$. Suppose that
\begin{itemize}
\item[(i)] $\mathcal{F}$ is a bounded subset of $L^r(\Omega; E_0)$;
\item[(ii)] For every $\delta>0$ there exists $\rho>0$ such that
\[ \int_{\Omega} \|f(\xt+\yt) - f(\xt)\|^r_{E_1} \dd \xt< \delta^r,\]
for each $f \in \mathcal{F}$ and all $\yt \in \mathbb{R}^d$ with $|\yt|<\rho$.
\end{itemize}
Then, $\mathcal{F}$ is relatively compact in $L^r(\Omega;E_1)$.
\end{theorem}

\begin{proof}
The set $\chi_\Omega \mathcal{F}$ satisfies the conditions of Theorem \ref{th-kr}, condition
(ii) of Theorem \ref{th-kr} being satisfied trivially since $\Omega$ is assumed to be a bounded
subset of $\mathbb{R}^d$ here.
\end{proof}

\begin{theorem}\label{th-kr2}
Suppose that $\Omega$ is a bounded open Lipschitz domain in $\mathbb{R}^d$. Let
$E_0$ and $E_1$ be Banach spaces, with $E_0 \compEmb E_1$. Suppose further that $1 \leq p < \infty$ and $\mathcal{F}$ is a bounded subset of
\[ W^{1,p}(\Omega; E_0, E_1):= \{v \in L^p(\Omega; E_0)\,:\, D^{{\footnotesize \undertilde\alpha}} v \in L^p(\Omega; E_1)\quad \forall \undertilde{\alpha} \in \mathbb{N}_{\geq 0}^d\quad\mbox{such that}\quad |\undertilde{\alpha}| = 1\}.\]
Here $|\undertilde{\alpha}| := |\undertilde{\alpha}|_{\ell_1} = \sum_{i=1}^d \alpha_i$ is the {\em length} of the multi-index $\undertilde{\alpha}$. Then,
$\mathcal{F}$ is relatively compact in $L^r(\Omega; E_1)$ for all $r$ such that $1 \leq r < \infty$ and
$\frac{1}{r} > \frac{1}{p} - \frac{1}{d}$.
\end{theorem}

\begin{proof}

In a metric space a set is relatively compact if, and only if, it is sequentially relatively compact.
%
It therefore suffices to prove that any bounded sequence $\{f_n\}_{n=1}^\infty \subset W^{1,p}(\Omega;E_0, E_1)$, i.e.,
such sequence that there exists a $c_0>0$ for which
\begin{equation}\label{comp1}
\sup_{n \in \mathbb{N}}\|f_n\|_{W^{1,p}(\Omega;E_0, E_1)} =  \sup_{n \in \mathbb{N}}\left(\|f_n\|_{L^{p}(\Omega;E_0)}^p +
\sum_{|\footnotesize \undertilde \alpha| = 1} \|D^{\footnotesize \undertilde \alpha} f_n\|^p_{L^{p}(\Omega;E_1)}\right)^{\frac{1}{p}}
=: c_0,
\end{equation}
forms a relatively compact subset of $L^{r}(\Omega;E_1)$, with $r$ as in the statement of the theorem.

We shall first suppose that $1 \leq p < d$. Then, by the Sobolev embedding theorem, we have that $W^{1,p}(\Omega;E_0, E_1) \hookrightarrow
W^{1,p}(\Omega; E_1)\hookrightarrow L^{r^\ast}(\Omega;E_1)$, where $r^\ast: = \frac{dp}{d-p}$; clearly, $r^\ast>1$. Thus,
\begin{equation}\label{def-c1}
\sup_{n \in \mathbb{N}} \|f_n\|_{L^{r^\ast}(\Omega;E_1)} := c_1 < \infty.
\end{equation}
As $p \geq 1$, it follows from \eqref{comp1} and H\"older's inequality
that
\[ \sup_{n \in \mathbb{N}} \|f_n\|_{W^{1,1}(\Omega;E_0, E_1)} := c_2 \leq|\Omega|^{1-\frac{1}{p}}c_0 < \infty.\]
Now, let $\delta>0$ be arbitrary but fixed, and let $\Omega^\ast \subset \overline{\Omega^\ast} \subset \Omega$
be a domain such that
\[ |\Omega \setminus \Omega^\ast| < \left(\frac{\delta}{3c_1}\right)^{\frac{r^\ast}{r^\ast-1}}.\]
Let $\rho>0$ be such that $\xt \in \Omega^\ast$ and $\yt \in \mathbb{R}^d$, $|\yt|<\rho$, implies that
$\xt + \yt \in \Omega$. Assume that $\rho < \frac{\delta}{3c_2}$. Then, by Jensen's inequality
and the triangle inequality,
\begin{eqnarray*}
\|f_n(\xt+\yt) - f_n(\xt)\|_{E_1} = \left\| \int_0^1 \sum_{i=1}^d \frac{\partial f_n}{\partial x_i}(\xt + t \yt)y_i \dd t\right\|_{E_1}               \leq \int_0^1 \sum_{i=1}^d \left\|\frac{\partial f_n}{\partial x_i}(\xt + t \yt)\right\|_{E_1} |y_i| \dd t,
\end{eqnarray*}
which then implies that
\[ \int_{\Omega^\ast} \|f_n(\xt+\yt) - f_n(\xt)\|_{E_1} \dd \xt \leq |\yt|\,\|f_n\|_{W^{1,1}(\Omega; E_1)} < \rho\,c_2 < \frac{\delta}{3}
,\]
and therefore
\begin{eqnarray*}
\int_\Omega\|f_n(\xt + \yt) - f_n(\xt)\|_{E_1}\dd \xt &\leq& \int_{\Omega \setminus \Omega^\ast}\|f_n(\xt + \yt)\|_{E_1} \dd \xt
+ \int_{\Omega \setminus \Omega^\ast}\|f_n(\xt)\|_{E_1} \dd \xt\nonumber\\
&& + \int_{\Omega^\ast}\|f_n(\xt + \yt)- f_n(\xt)\|_{E_1} \dd \xt < \delta.
\end{eqnarray*}
Hence, the assertion of the theorem for $r=1$ follows from Theorem \ref{th-kr1}.
Thus we have shown that the bounded sequence $\{f_n\}_{n=1}^\infty$ in $W^{1,p}(\Omega;E_0, E_1)$ has a subsequence
$\{f_{n_k}\}_{k=1}^\infty$ that is convergent in $L^1(\Omega;E_1)$.

Let $1 < r < r^\ast$. Then, noting that $0< \frac{r^\ast-r}{r^\ast -1} < 1$, using H\"older's inequality and
\eqref{def-c1}, we have that
\begin{eqnarray*}
\|f_{n_k} - f_{n_l}\|^r_{L^r(\Omega;E_1)} &=& \int_\Omega \|f_{n_k}(\xt) - f_{n_l}(\xt)\|_{E_1}^{\frac{r^\ast(r-1)}{r^\ast -1}}
\|f_{n_k}(\xt) - f_{n_l}(\xt)\|_{E_1}^{\frac{r^\ast-r}{r^\ast -1}}\dd \xt\\
& \leq & (2c_1)^{\frac{r^\ast(r-1)}{r^\ast -1}} \|f_{n_k} - f_{n_l}\|_{L^1(\Omega;E_1)}^{\frac{r^\ast-r}{r^\ast -1}}.
\end{eqnarray*}
Hence $\{f_{n_k}\}_{k=1}^\infty$ is a Cauchy sequence in $L^r(\Omega;E_1)$; since the latter is a Banach space,
the sequence $\{f_{n_k}\}_{k=1}^\infty$ is convergent in $L^r(\Omega;E_1)$. Thus, combining the outcomes of the
 cases $r=1$ and $1< r < r^\ast$, we have shown that the bounded
sequence $\{f_n\}_{n=1}^\infty$ in $W^{1,p}(\Omega;E_0, E_1)$ has a subsequence $\{f_{n_k}\}_{k=1}^\infty$ that is
convergent in $L^r(\Omega; E_1)$ for $1 \leq r < r^\ast$, whenever $1 \leq p < d$.

We have thus shown that a bounded set $\mathcal{F}$ in $W^{1,p}(\Omega;E_0, E_1)$ is relatively compact
in $L^r(\Omega;E_1)$ for $1 \leq r < r^\ast$, whenever $1 \leq p < d$.

Suppose now that $p \geq d$, and $\mathcal{F}$ is a bounded set in $W^{1,p}(\Omega;E_0, E_1)$. As $\Omega$ is a bounded set in $\mathbb{R}^d$, the set $\mathcal{F}$ is then automatically a bounded set in $W^{1,s}(\Omega;E_0, E_1)$ for all $s \in [1,d)$ by H\"older's inequality, and the stated result follows for all $r \in [1,\infty)$ from what was proved above.
\end{proof}

%

Next, we shall extend Theorem \ref{th-kr2} to the case of three Banach spaces, $E_0$, $E$ and $E_1$, where
$E_0\compEmb E \hookrightarrow E_1$ and $E_0$ is reflexive. We shall confine ourselves to considering the case
when $1 < p < \infty$, as this range of $p$ will suffice for our purposes here.

\begin{theorem}\label{th-kr3}
Suppose that $E_0$, $E$ and $E_1$ are Banach spaces, with $E_0\compEmb E \hookrightarrow E_1$ and $E_0$ reflexive, and let $\Omega$ be a bounded open Lipschitz domain in $\mathbb{R}^d$. For $1 < p < \infty$, we define
\[ W^{1,p}(\Omega; E_0, E_1):= \{v \in L^p(\Omega; E_0)\,:\, D^{{\footnotesize \undertilde\alpha}} v \in L^p(\Omega; E_1)\quad \forall \undertilde{\alpha} \in \mathbb{N}_{\geq 0}^d\quad\mbox{such that}\quad |\undertilde{\alpha}| = 1\}.\]
Here, again, $|\undertilde{\alpha}| := |\undertilde{\alpha}|_{\ell_1} = \sum_{i=1}^d \alpha_i$ is the {\em length} of the multi-index $\undertilde{\alpha}$. Then,
\[ W^{1,p}(\Omega; E_0, E_1) \compEmb L^p(\Omega; E).\]
\end{theorem}

\begin{proof}
Suppose that $\mathcal{F}$ is a bounded set in $W^{1,p}(\Omega; E_0, E_1)$. As $E_0 \compEmb E \hookrightarrow E_1$, it follows that $E_0 \compEmb E_1$. We thus deduce from Theorem \ref{th-kr2} that $\mathcal{F}$ is relatively compact in $L^p(\Omega;E_1)$. In order to complete the proof, we need to show that $\mathcal{F}$ is in fact relatively compact
in $L^p(\Omega;E)$. Once again we shall rely in the proof on the fact that in a metric space
relative compactness and sequential relative compactness are equivalent.

Suppose that $\{f_n\}_{n=1}^\infty$ is a bounded sequence in $W^{1,p}(\Omega; E_0, E_1)$. Since $W^{1,p}(\Omega; E_0, E_1)
\subset L^p(\Omega;E_0)$, the sequence $\{f_n\}_{n=1}^\infty$ is bounded in $L^p(\Omega;E_0)$. As $1<p< \infty$ and
$E_0$ is reflexive, also $L^p(\Omega;E_0)$ is reflexive. Therefore $\{f_n\}_{n=1}^\infty$ has a weakly
convergent subsequence $\{f_{n_k}\}_{k=1}^\infty$ in $L^p(\Omega;E_0)$; we denote the weak limit by $f_\ast$; $f_\ast
\in L^p(\Omega;E_0)$.

Further, if follows from the discussion in the first paragraph of the proof that $\{f_{n_k}\}_{k=1}^\infty$
is relatively compact in $L^p(\Omega; E_1)$. We can therefore extract a subsubsequence $\{f_{n_{k_l}}\}_{l=1}^\infty$
from $\{f_{n_k}\}_{k=1}^\infty$, which strongly converges in $L^p(\Omega;E_1)$
to a limit $f \in L^p(\Omega;E_1)$. It
follows from the uniqueness of the limit that $f_\ast = f$, and therefore $f \in L^p(\Omega;E_0)
(\subset L^p(\Omega;E))$, in fact.
It remains to show that $\{f_{n_{k_l}}\}_{l=1}^\infty$ converges to $f$ in $L^p(\Omega;E)$.

As $E_0 \compEmb E \hookrightarrow E_1$ we deduce from Ehrling's Lemma (cf. Temam \cite{Temam} Ch. III, Lemma 2.1)
that for every $\delta>0$ there exists $c_\delta>0$ such that, for all $v \in E_0$,
\begin{equation}\label{ehrling}
\|v\|_{E} \leq \delta \|v\|_{E_0} + c_\delta \|v\|_{E_1}.
\end{equation}
As $f - f_{n_{k_l}}$ is known to belong to $L^p(\Omega;E_0) (\subset L^p(\Omega;E) \subset L^p(\Omega;E_1))$ for all $l \geq 1$, we deduce from \eqref{ehrling} that
\[ \|f(\xt) - f_{n_{k_l}}(\xt)\|_E \leq \delta \|f(\xt) - f_{n_{k_l}}(\xt)\|_{E_0} + c_\delta \|f(\xt) - f_{n_{k_l}}(\xt)\|_{E_1},
\qquad \mbox{for a.e. $\xt \in \Omega$}.\]
Thus, by the triangle inequality in $L^p(\Omega)$,
\begin{equation}\label{triangle-bd}
\|f - f_{n_{k_l}}\|_{L^p(\Omega;E)} \leq \delta \|f - f_{n_{k_l}}\|_{L^p(\Omega;E_0)} + c_\delta \|f - f_{n_{k_l}}\|_{L^p(\Omega;E_1)}.
\end{equation}

Since $\{f_{n_{k_l}}\}_{l=1}^\infty$ is a bounded sequence in $L^p(\Omega; E_0)$ and $f \in L^p(\Omega;E_0)$, there exists a positive real number $K$ such that
\[ \|f - f_{n_{k_l}}\|_{L^p(\Omega; E_0)} \leq K\qquad \forall l \geq 1.\]
Further, for $\delta>0$ fixed (but otherwise arbitrary) and therefore $c_\delta$ (as above) also fixed, there exists
$l_0 = l_0(\delta)$ such that
\[ \| f - f_{n_{k_l}}\|_{L^p(\Omega; E_1)} \leq \frac{\delta}{c_\delta}\qquad \forall l \geq l_0.\]
Consequently, by substituting the last two inequalities into \eqref{triangle-bd}, we deduce that
\[ \|f - f_{n_{k_l}}\|_{L^p(\Omega;E)} \leq \delta K + \delta = (K+1)\delta \qquad \forall l \geq l_0. \]
Thus we have shown that for each $\delta>0$ there exists $l_0 = l_0(\delta)$ such that
\[ \|f - f_{n_{k_l}}\|_{L^p(\Omega;E)} \leq (K+1)\delta \qquad \forall l \geq l_0. \]
In other words, $\{f_{{n_k}_l}\}_{l=1}^\infty$ converges to $f$ strongly in $L^p(\Omega;E)$.
\end{proof}

\section{$H^1_M(\Omega \times D) \compEmb L^2_M(\Omega \times D)$, $\vartheta> \frac{1}{2}$}
\label{AppendixD}

\setcounter{equation}{0}

With the compact embedding of $H^1_M(D)$ in $L^2_M(D)$ established in Section
\ref{sec:comptensorise}, the proof of the compact embedding
$H^1_M(\Omega \times D) \compEmb L^2_M(\Omega \times D)$, $\vartheta> \frac{1}{2}$, proceeds,
verbatim,
as in \cite{BS2010}: first the isometric isomorphism of $L^2_M(\Omega)$ and $L^2(\Omega; L^2_M(D))$
is proved using the separability of the Hilbert space $L^2_M(D)$; then the isometric isomorphism
$H^{0,1}_M(\Omega \times D)$ and $L^2(\Omega; H^1_M(D))$ is shown, where
\[ H^{0,1}_M(\Omega \times D) := \{V \in L^2_M(\Omega \times D)\,:\,
\nabq V \in L^2_M(\Omega \times D)\},   \]
and the isometric isomorphism of $H^{1,0}_M(\Omega \times D)$ and
$H^1(\Omega;L^2_M(D))$ is also shown, where
\[ H^{1,0}_M(\Omega \times D) := \{V \in L^2_M(\Omega \times D)\,:\,
\nabx V \in L^2_M(\Omega \times D)\}.  \]
With these definitions, we then identify the space $L^2_M(\Omega \times D)$ with $L^2(\Omega;L^2_M(D))$
and the space
$H^1_M(\Omega \times D) =  H^{1,0}_M(\Omega \times D) \cap
H^{0,1}_M(\Omega \times D)$ with
$H^1(\Omega; L^2_M(D)) \cap L^2(\Omega; H^1_M(D))$.
Upon doing so, the compact embedding of
$H^1_M(\Omega \times D)$ into $L^2_M(\Omega \times D)$
directly follows from the compact embedding of $H^1(\Omega;L^2_M(D))\cap L^2(\Omega;H^1_M(D))$
into $L^2(\Omega;L^2_M(D))$, implied by Theorem \ref{th-kr3} above,
thanks to the compact embedding of $E_0:=H^1_M(D)$ into $E=E_1 = L^2_M(D)$.

~\\

~\hfill{\footnotesize \textit{London \& Oxford, 16 August 2010.}

\end{document}